\documentclass[10pt]{article}

\usepackage[utf8]{inputenc}

\usepackage{epsf}
\usepackage{amsmath}

\allowdisplaybreaks

\usepackage[showframe=false]{geometry}
\usepackage{changepage}

\usepackage{epsfig}
\usepackage{amssymb}

\usepackage{amsthm}
\usepackage{setspace}
\usepackage{cite}
\usepackage{mcite}

\usepackage{algorithmic}  % This package provides an algorithmic environment fo describing algorithms.
\usepackage{algorithm}

\usepackage{shadow}
\usepackage{fancybox}
\usepackage{fancyhdr}

\usepackage{color}
\usepackage[usenames,dvipsnames,svgnames,table]{xcolor}

\definecolor{mypurple}{rgb}{.4,.0,.5}

\usepackage[hyphens]{url}

\usepackage[colorlinks=true,
            linkcolor=black,
            urlcolor=blue,
            citecolor=purple]{hyperref}

\usepackage{breakurl}
\def\y{{\bf y}}

\def\x{{\bf x}}

\def\x{{\mathbf x}}

\def\u{{\bf u}}

\def\x{{\bf x}}
\def\y{{\bf y}}

\def\q{{\bf q}}
\def\m{{\bf m}}

\def\h{{\bf h}}

\def\be{\begin{equation}}
\def\ee{\end{equation}}
\def\ba{\left[\begin{array}}
\def\ea{\end{array}\right]}

\def\u{{\bf u}}

\def\x{{\bf x}}
\def\y{{\bf y}}

\def\q{{\bf q}}

\def\p{{\bf p}}

\def\1{{\bf 1}}

\def\0{{\bf 0}}

\def\calX{{\cal X}}
\def\calY{{\cal Y}}

\def\mR{{\mathbb R}}
\def\mN{{\mathbb N}}
\def\mE{{\mathbb E}}
\def\mS{{\mathbb S}}

\def\lp{\left (}
\def\rp{\right )}

\newtheorem{theorem}{Theorem}
\newtheorem{proposition}{Proposition}

\begin{document}

\begin{singlespace}

\title {Fully lifted interpolating comparisons of bilinearly indexed random processes  %A tight variant of Gordon's escape through a mesh theorem
%\footnote{ This work was supported in
%part.}
}
\author{
\textsc{Mihailo Stojnic
\footnote{e-mail: {\tt flatoyer@gmail.com}} }}
\date{}
\maketitle

%%%%%%%%%%%%%%%%%%%%%%%%%%%%%%%%%%%%%%%%%%%%%%%%%%%%%%%%%%%%%%%%%%%%%%%%%%%%%%%%
%%%%%%%%%%%%%%%%%%%%%%%%%%%%%%%%%%%%%%%%%%%%%%%%%%%%%%%%%%%%%%%%%%%%%%%%%%%%%%%%
\centerline{{\bf Abstract}} \vspace*{0.1in}
%%%%%%%%%%%%%%%%%%%%%%%%%%%%%%%%%%%%%%%%%%%%%%%%%%%%%%%%%%%%%%%%%%%%%%%%%%%%%%%%
%%%%%%%%%%%%%%%%%%%%%%%%%%%%%%%%%%%%%%%%%%%%%%%%%%%%%%%%%%%%%%%%%%%%%%%%%%%%%%%%

A powerful statistical interpolating concept, which we call \emph{fully lifted} (fl), is introduced and presented while establishing a connection between bilinearly indexed random processes and their corresponding fully decoupled (linearly indexed) comparative alternatives. Despite on occasion very involved technical considerations, the final interpolating forms and their underlying relations admit rather elegant expressions that provide conceivably highly desirable and useful tool for further studying various different aspects of random processes and their applications. We also discuss the generality of the considered models and show that they encompass many well known random structures and optimization problems to which then the obtained results automatically apply.

\vspace*{0.25in} \noindent {\bf Index Terms: Random processes; comparison principles, lifting}.

\end{singlespace}

%%%%%%%%%%%%%%%%%%%%%%%%%%%%%%%%%%%%%%%%%%%%%%%%%%%%%%%%%%%%%%%%%
\section{Introduction}
\label{sec:back}
%%%%%%%%%%%%%%%%%%%%%%%%%%%%%%%%%%%%%%%%%%%%%%%%%%%%%%%%%%%%%%%%%

In this paper, we study random processes interpolations and resulting comparisons.  Many excellent results on these topics have been  obtained in various directions over the last several decades. Probably two of the most prominent ones are
the Slepian's max \cite{Slep62} and the Gordon's minmax \cite{Gordon85} principle (see also, e.g.,  \cite{Sudakov71,Fernique74,Fernique75,Kahane86}). Brief history and development reviews, with a hint at a rather massive range of applications spanning the most diverse scientific fields, can be found in, e.g., \cite{Stojnicgscomp16,Adler90,Lifshits85,LedTal91,Tal05}. The methods turned out to be particularly successful over the last 20 years when utilized as the main probabilistic foundation in the analysis of random structures and optimization problems (see, e.g., \cite{Guerra03,Tal06,Pan10,Pan10a,Pan13,Pan13a} for the Slepian  and \cite{StojnicISIT2010binary,StojnicCSetam09,StojnicUpper10,StojnicICASSP10knownsupp,StojnicICASSP10block,StojnicICASSP10var}  for both Slepian and Gordon principles based mechanisms). As is well known, in such structures, in large dimensional settings, one often encounters the appearance of the so-called \emph{phase-transition} (PT) phenomenon. This phenomenon means that the underlying random structure or optimization problem exhibits a sharp change in behavior, as one moves from one region of the system parameters to another. For example, for a particular algorithm, a random optimization problem is highly likely solvable throughout a range of system parameters  and then highly unlikely solvable outside such a range. Quite surprisingly, not only were the techniques from, say, \cite{StojnicISIT2010binary,StojnicCSetam09,StojnicUpper10,StojnicICASSP10block,StojnicICASSP10knownsupp} strong enough to provide an excellent qualitative performance characterizations of the underlying problems, they also turned out to be capable of doing so on an ultimate, PT, precision level. In other words, methods employed in \cite{StojnicISIT2010binary,StojnicCSetam09,StojnicUpper10,StojnicICASSP10block,StojnicICASSP10knownsupp} were strong enough to precisely pinpoint both the existence of the PT phenomena and their exact locations in the space of problem parameters.

On the other hand, in some of the scenarios such a methodology (despite providing useful performance characterizations) still needs a substantial upgrading to achieve the full PT level of precision (more on this can be found in, e.g., \cite{Guerra03,Tal06,Pan10,Pan10a,Pan13,Pan13a}  for quadratic (or particularly polynomial/tensorial) max type forms, and in, e.g., \cite{StojnicLiftStrSec13,StojnicMoreSophHopBnds10,StojnicRicBnds13,StojnicAsymmLittBnds11,StojnicGardSphNeg13,StojnicGardSphErr13} and references therein for various other forms including both quadratic and bilinear max and minmax types of forms). As discussed in \cite{Guerra03,Tal06,Pan10,Pan10a,Pan13,Pan13a}  regarding the quadratic/tensorial max forms, and in \cite{Stojnicgscompyx16} regarding the forms from \cite{StojnicLiftStrSec13,StojnicMoreSophHopBnds10,StojnicRicBnds13,StojnicAsymmLittBnds11,StojnicGardSphNeg13,StojnicGardSphErr13}, the foundational blocks of the needed upgrades are rooted in the core improvements of the underlying random processes' comparisons. As many of generic random structures still lack a full PT level analytical characterization, the need for establishing optimal comparisons' upgradings is rather pressing. The machinery of \cite{Stojnicgscomp16,Stojnicgscompyx16} and the introduction of a \emph{partial lifting} concept made a solid progress in that direction. Here though, we move things to a whole another level and create a new interpolating mechanism to which we refer as \emph{fully lifted} (fl). As it will be clear throughout the presentation, the key idea will be in successively  increasing the number of the so-called nested levels of lifting. To make the reasoning and the underlying mathematics more instructive and easier to follow and to slowly successively introduce all key technical ingredients, we, in Sections \ref{sec:gencon} and \ref{sec:seclev}, present separately the results for the first and second level of lifting and then, in Section \ref{sec:rthlev}, the corresponding generalization for any $r$-th ($r\in\mN$), level of lifting.

%%%%%%%%%%%%%%%%%%%%%%%%%%%%%%%%%%%%%%%%%%%%%%%%%%%%%%%%%%%%%%%%%
\section{A bilinear comparison form -- first level of lifting}
\label{sec:gencon}
%%%%%%%%%%%%%%%%%%%%%%%%%%%%%%%%%%%%%%%%%%%%%%%%%%%%%%%%%%%%%%%%%

For given two sets, $\calX=\{\x^{(1)},\x^{(2)},\dots,\x^{(l)}\}$ with $\x^{(i)}\in \mR^n$ and $\calY=\{\y^{(1)},\y^{(2)},\dots,\y^{(l)}\}$ with $\y^{(i)}\in \mR^m$, vector $\p=[\p_0,\p_1,\p_2]$ with $\p_0\geq \p_1\geq \p_2= 0$, vector $\q=[\q_0,\q_1,\q_2]$ with $\q_0\geq \q_1\geq \q_2= 0$, and real parameters $\beta>0$ and $s$, we consider function
\begin{equation}\label{eq:genanal1}
 f(G,u^{(4,1)},u^{(4,2)},\calX,\calY,\p,\q,\beta,s)= \frac{1}{\beta|s|\sqrt{n}} \log\lp \sum_{i_1=1}^{l}\lp\sum_{i_2=1}^{l}e^{\beta \lp (\y^{(i_2)})^T
 G\x^{(i_1)}+\|\x^{(i_1)}\|_2\|\y^{(i_2)}\|_2 (u^{(4,1)}+u^{(4,2)})\rp} \rp^{s}\rp.
\end{equation}
Behavior of this function in random mediums is of our prevalent interest. In particular, we consider $(m\times n)$ dimensional matrices  $G\in \mR^{m\times n}$ with i.i.d. standard normal components and independent (of $G$ and among themselves) random variables $u^{(4,1)}\sim {\mathcal N}(0,\p_0\q_0-\p_1\q_1)$ and $u^{(4,2)}\sim {\mathcal N}(0,\p_1\q_1)$. For a scalar $\m=[\m_1]$, the following function turns out to be critically important for studying $f(G,u^{(4,1)},u^{(4,2)},\calX,\calY,\q,\beta,s)$
\begin{equation}\label{eq:genanal2}
\xi(\calX,\calY,\p,\q,\m,\beta,s)  \triangleq   \mE_{G,u^{(4,2)}}\frac{1}{\beta|s|\sqrt{n}\m_1} \log \mE_{u^{(4,1)}}\lp \sum_{i_1=1}^{l}\lp\sum_{i_2=1}^{l}e^{\beta \lp (\y^{(i_2)})^T
 G\x^{(i_1)}+\|\x^{(i_1)}\|_2\|\y^{(i_2)}\|_2 u^{(4)}\rp} \rp^{s}\rp^{\m_1}.
\end{equation}
Throughout the paper, $\mE$ with a subscript denotes expectation with respect to subscript specified randomness. On the other hand, $\mE$ without subscript denotes the expectation with respect to any underlying randomness. To study properties of $\xi(\calX,\calY,\q,\m,\beta,s)$, we
find it quite convenient to follow into the footsteps of \cite{Stojnicgscomp16,Stojnicgscompyx16} and consider the following interpolating function $\psi(\cdot)$
\begin{equation}\label{eq:genanal3}
\psi(\calX,\calY,\p,\q,\m,\beta,s,t)  =  \mE_{G,u^{(4,2)},\u^{(2,2)},\h^{(2)}} \frac{1}{\beta|s|\sqrt{n}\m_1} \log \mE_{u^{(4,1)},\u^{(2,1)},\h^{(1)}} \lp \sum_{i_1=1}^{l}\lp\sum_{i_2=1}^{l}e^{\beta D_0^{(i_1,i_2)}} \rp^{s}\rp^{\m_1},
\end{equation}
where
\begin{eqnarray}\label{eq:genanal3a}
 D_0^{(i_1,i_2)} & \triangleq & \sqrt{t}(\y^{(i_2)})^T
 G\x^{(i_1)}+\sqrt{1-t}\|\x^{(i_1)}\|_2 (\y^{(i_2)})^T(\u^{(2,1)}+\u^{(2,2)})\nonumber \\
 & & +\sqrt{t}\|\x^{(i_1)}\|_2\|\y^{(i_2)}\|_2(u^{(4,1)}+u^{(4,2)}) +\sqrt{1-t}\|\y^{(i_2)}\|_2(\h^{(1)}+\h^{(2)})^T\x^{(i_1)}.
 \end{eqnarray}
In (\ref{eq:genanal3}), $\u^{(2,1)}$ and $\u^{(2,2)}$ are $m$ dimensional vectors of i.i.d. zero-mean Gaussians with variances $\p_0-\p_1$ and $\p_1$, respectively. Similarly, $\h^{(1)}$ and $\h^{(2)}$ are $n$ dimensional vectors of i.i.d. zero-mean Gaussians with variances $\q_0-\q_1$ and $\q_1$, respectively. These four vectors are assumed to be independent among themselves and of $G$, $u^{(4,1)}$, and $u^{(4,2)}$ as well.

It is not that difficult to see that one basically has $D_0^{(i_1,i_2)}$ as a \emph{bilinearly} indexed (Gaussian) random process. One also easily observes that $\xi(\calX,\calY,\p,\q,\m,\beta,s)=\psi(\calX,\calY,\p,\q,\m,\beta,s,1)$ and since $\psi(\calX,\calY,\p,\q,\m,\beta,s,0)$ is typically easier to handle than $\psi(\calX,\calY,\p,\q,\m,\beta,s,1)$, one would like to connect $\psi(\calX,\calY,\p,\q,\m,\beta,s,1)$ to $\psi(\calX,\calY,\p,\q,\m,\beta,s,0)$ as a way of connecting $\xi(\calX,\calY,\p,\q,\m,\beta,s)$ to $\psi(\calX,\calY,\p,\q,\m,\beta,s,0)$. This would essentially connect the original, \emph{bilinearly} indexed, process to two decoupled, \emph{linearly} indexed, comparative alternatives.

For the convenience of the exposition that follows below, we set
\begin{eqnarray}\label{eq:genanal4}
\u^{(i_1,1)} & =  & \frac{G\x^{(i_1)}}{\|\x^{(i_1)}\|_2} \nonumber \\
\u^{(i_1,3,1)} & =  & \frac{(\h^{(1)})^T\x^{(i_1)}}{\|\x^{(i_1)}\|_2} \nonumber \\
\u^{(i_1,3,2)} & =  & \frac{(\h^{(2)})^T\x^{(i_1)}}{\|\x^{(i_1)}\|_2}.
\end{eqnarray}
Assuming that $G_{j,1:n}$ denotes the $j$-th row of $G$ and that $\u_j^{(i_1,1)}$ is the $j$-th component of $\u^{(i_1,1)}$, one from (\ref{eq:genanal4}) has
\begin{eqnarray}\label{eq:genanal5}
\u_j^{(i_1,1)} & =  & \frac{G_{j,1:n}\x^{(i_1)}}{\|\x^{(i_1)}\|_2},1\leq j\leq m.
\end{eqnarray}
Also, it is rather trivial that for any fixed $i_1$, the elements of $\u^{(i_1,1)}$ are i.i.d. standard normals, the elements of $\u^{(2,1)}$ and $\u^{(2,2)}$ are zero-mean Gaussians with respective variances $\p_0-\p_1$ and $\p_1-\p_2$, the elements of $\u^{(i_1,3,1)}$ and $\u^{(i_1,3,2)}$ are zero-mean Gaussians with  respective variances $\q_0-\q_1$ and $\q_1-\q_2$. Setting ${\mathcal U}_r=\{u^{(4,r)},\u^{(2,r)},\h^{(r)}\},r\in\{1,2\}$, one then easily rewrites (\ref{eq:genanal3}) as
\begin{equation}\label{eq:genanal6}
\psi(\calX,\calY,\p,\q,\m,\beta,s,t)  =  \mE_{G,{\mathcal U}_2} \frac{1}{\beta|s|\sqrt{n}\m_1} \log \mE_{{\mathcal U}_1} \lp \sum_{i_1=1}^{l}\lp\sum_{i_2=1}^{l}A^{(i_1,i_2)} \rp^{s}\rp^{\m_1},
\end{equation}
where $\beta_{i_1}=\beta\|\x^{(i_1)}\|_2$ and
\begin{eqnarray}\label{eq:genanal7}
B^{(i_1,i_2)} & \triangleq &  \sqrt{t}(\y^{(i_2)})^T\u^{(i_1,1)}+\sqrt{1-t} (\y^{(i_2)})^T(\u^{(2,1)}+\u^{(2,2)}) \nonumber \\
D^{(i_1,i_2)} & \triangleq &  (B^{(i_1,i_2)}+\sqrt{t}\|\y^{(i_2)}\|_2 (u^{(4,1)}+u^{(4,2)})+\sqrt{1-t}\|\y^{(i_2)}\|_2(\u^{(i_1,3,1)}+\u^{(i_1,3,2)})) \nonumber \\
A^{(i_1,i_2)} & \triangleq &  e^{\beta_{i_1}D^{(i_1,i_2)}}\nonumber \\
C^{(i_1)} & \triangleq &  \sum_{i_2=1}^{l}A^{(i_1,i_2)}\nonumber \\
Z & \triangleq & \sum_{i_1=1}^{l} \lp \sum_{i_2=1}^{l} A^{(i_1,i_2)}\rp^s =\sum_{i_1=1}^{l}  (C^{(i_1)})^s.
\end{eqnarray}

The readers familiar with the statistical physics replica theory related fields, may recognize the nesting type of averaging of successively added Gaussian sequences often seen in the so-called replica symmetry breaking (rsb) schemes, pioneered by Parisi in \cite{Parisi80,Par79,Par80,Par83}. As is well known in these fields, the idea of replicated systems, in both replica symmetric (rs) and replica symmetry breaking (rsb) regimes, is to eventually utilize the \emph{index decoupling} enabled by the averaging over large number of replicas. While replicated systems appear exactly nowhere in (\ref{eq:genanal6}) and (\ref{eq:genanal7}), a trace of their decoupling flair  might be recognized if one notes the existence of the (bilinearly) coupled $\psi(\calX,\calY,\p,\q,\m,\beta,s,1)$ and fully decoupled $\psi(\calX,\calY,\p,\q,\m,\beta,s,0)$. In particular, if an analytical relation between $\psi(\calX,\calY,\p,\q,\m,\beta,s,1)$ and $\psi(\calX,\calY,\p,\q,\m,\beta,s,0)$ can be established, then the coupled $\psi(\calX,\calY,\p,\q,\m,\beta,s,1)$ might turn out to be, partially or even completely, characterizable via the decoupled $\psi(\calX,\calY,\p,\q,\m,\beta,s,0)$. This would, in a way, be similar to how decoupled replicated systems, rather magically, faithfully characterize the non-replicated coupled ones in various statistical physics considerations. While no rigorous mathematical justification can be put behind such a reasoning, the underlying intuition seems rather convincing and imposes, almost as an imperative, further exploring of the existence of a potential connection between $\psi(\calX,\calY,\p,\q,\m,\beta,s,1)$ and $\psi(\calX,\calY,\p,\q,\m,\beta,s,0)$.

Since comparing $\psi(\calX,\calY,\p,\q,\m,\beta,s,1)$ and $\psi(\calX,\calY,\p,\q,\m,\beta,s,0)$ naturally presents itself as a logical path in such an exploration, studying monotonicity of $\psi(\calX,\calY,\p,\q,\m,\beta,s,t)$ seems as a good starting point on such a path. To that end, we start by considering its derivative
\begin{eqnarray}\label{eq:genanal9}
\frac{d\psi(\calX,\calY,\q,\m,\beta,s,t)}{dt} & = &  \mE_{G,{\mathcal U}_2} \frac{1}{\beta|s|\sqrt{n}\m_1} \log \mE_{{\mathcal U}_1} Z^{\m_1}\nonumber \\
& = &  \mE_{G,{\mathcal U}_2} \frac{1}{\beta|s|\sqrt{n}\m_1\mE_{{\mathcal U}_1} Z^{\m_1}} \frac{d \mE_{{\mathcal U}_1} Z^{\m_1} }{dt}\nonumber \\
& = &  \mE_{G,{\mathcal U}_2} \frac{\m_1}{\beta|s|\sqrt{n}\m_1\mE_{{\mathcal U}_1} Z^{\m_1}}\mE_{{\mathcal U}_1} \frac{1}{Z^{1-\m_1}}\frac{d Z}{dt}\nonumber \\
& = &  \mE_{G,{\mathcal U}_2} \frac{\m_1}{\beta|s|\sqrt{n}\m_1\mE_{{\mathcal U}_1} Z^{\m_1}}\mE_{{\mathcal U}_1} \frac{1}{Z^{1-\m_1}} \frac{d\lp \sum_{i_1=1}^{l} \lp \sum_{i_2=1}^{l} A^{(i_1,i_2)}\rp^s \rp }{dt}\nonumber \\
& = &   \mE_{G,{\mathcal U}_2} \frac{s\m_1}{\beta|s|\sqrt{n}\m_1\mE_{{\mathcal U}_1} Z^{\m_1}}\mE_{{\mathcal U}_1} \frac{1}{Z^{1-\m_1}}  \sum_{i=1}^{l} (C^{(i_1)})^{s-1} \nonumber \\
& & \times \sum_{i_2=1}^{l}\beta_{i_1}A^{(i_1,i_2)}\frac{dD^{(i_1,i_2)}}{dt},
\end{eqnarray}
where
\begin{eqnarray}\label{eq:genanal9a}
\frac{dD^{(i_1,i_2)}}{dt}= \lp \frac{dB^{(i_1,i_2)}}{dt}+\frac{\|\y^{(i_2)}\|_2 (u^{(4,1)}+u^{(4,2)})}{2\sqrt{t}}-\frac{\|\y^{(i_2)}\|_2 (\u^{(i_1,3,1)}+\u^{(i_1,3,2)})}{2\sqrt{1-t}}\rp.
\end{eqnarray}
Utilizing (\ref{eq:genanal7}) we find
\begin{eqnarray}\label{eq:genanal10}
\frac{dB^{(i_1,i_2)}}{dt} & = &   \frac{d\lp\sqrt{t}(\y^{(i_2)})^T\u^{(i_1,1)}+\sqrt{1-t} (\y^{(i_2)})^T(\u^{(2,1)}+\u^{(2,2)})\rp}{dt} \nonumber \\
 & = &
\sum_{j=1}^{m}\lp \frac{\y_j^{(i_2)}\u_j^{(i_1,1)}}{2\sqrt{t}}-\frac{\y_j^{(i_2)}\u_j^{(2,1)}}{2\sqrt{1-t}}-\frac{\y_j^{(i_2)}\u_j^{(2,2)}}{2\sqrt{1-t}}\rp.
\end{eqnarray}
Combining (\ref{eq:genanal9a}) and (\ref{eq:genanal10}) we obtain
\begin{eqnarray}\label{eq:genanal10a}
\frac{dD^{(i_1,i_2)}}{dt} & = & \sum_{j=1}^{m}\lp \frac{\y_j^{(i_2)}\u_j^{(i_1,1)}}{2\sqrt{t}}-\frac{\y_j^{(i_2)}\u_j^{(2,1)}}{2\sqrt{1-t}}-\frac{\y_j^{(i_2)}\u_j^{(2,2)}}{2\sqrt{1-t}}\rp \nonumber \\
& & +\frac{\|\y^{(i_2)}\|_2 u^{(4,1)}}{2\sqrt{t}}+\frac{\|\y^{(i_2)}\|_2 u^{(4,2)}}{2\sqrt{t}}-\frac{\|\y^{(i_2)}\|_2 \u^{(i_1,3,1)}}{2\sqrt{1-t}}-\frac{\|\y^{(i_2)}\|_2 \u^{(i_1,3,2)}}{2\sqrt{1-t}}.
\end{eqnarray}
%\begin{eqnarray}\label{eq:genanal11}
%\frac{d\psi(\calX,\calY,\beta,s,t)}{dt} & = & \mE_{G,{\mathcal U}_2} \frac{1}{\beta|s|\sqrt{n}\m_1} \log \mE_{{\mathcal U}_1} Z^{\m_1}\nonumber \\
%& = &   \mE_{G,{\mathcal U}_2} \frac{s\m_1}{\beta|s|\sqrt{n}\m_1\mE_{{\mathcal U}_1} Z^{\m_1}}\mE_{{\mathcal U}_1} \frac{1}{Z^{1-\m_1}}  \sum_{i=1}^{l} (C^{(i_1)})^{s-1} \nonumber \\
%& & \times \sum_{i_2=1}^{l}\beta_{i_1}A^{(i_1,i_2)}\lp \sum_{j=1}^{m}\lp \frac{\y_j^{(i_2)}\u_j^{(i_1,1)}}{2\sqrt{t}}-\frac{\y_j^{(i_2)}\u_j^{(2)}}{2\sqrt{1-t}}\rp+\frac{\|\y^{(i_2)}\|_2 u^{(4)}}{2\sqrt{t}}-\frac{\|\y^{(i_2)}\|_2 \u^{(i_1,3)}}{2\sqrt{1-t}}\rp.\nonumber \\
%\end{eqnarray}
The above terms can be rearranged into three clearly distinguishable groups that depend on $G$, ${\mathcal U}_2$, and ${\mathcal U}_1$, respectively
\begin{eqnarray}\label{eq:genanal10b}
\frac{dD^{(i_1,i_2)}}{dt} & = & \bar{T}_G+\bar{T}_2+\bar{T}_1,
\end{eqnarray}
where
\begin{eqnarray}\label{eq:genanal10c}
\bar{T}_G & = & \sum_{j=1}^{m} \frac{\y_j^{(i_2)}\u_j^{(i_1,1)}}{2\sqrt{t}}\nonumber \\
\bar{T}_2 & = & -\sum_{j=1}^{m}\frac{\y_j^{(i_2)}\u_j^{(2,2)}}{2\sqrt{1-t}}+\frac{\|\y^{(i_2)}\|_2 u^{(4,2)}}{2\sqrt{t}} -\frac{\|\y^{(i_2)}\|_2 \u^{(i_1,3,2)}}{2\sqrt{1-t}} \nonumber\\
\bar{T}_1 & = &  -\sum_{j=1}^{m} \frac{\y_j^{(i_2)}\u_j^{(2,1)}}{2\sqrt{1-t}}  +\frac{\|\y^{(i_2)}\|_2 u^{(4,1)}}{2\sqrt{t}}-\frac{\|\y^{(i_2)}\|_2 \u^{(i_1,3,1)}}{2\sqrt{1-t}}.
\end{eqnarray}
Combining (\ref{eq:genanal9}) and (\ref{eq:genanal10b}) we have
\begin{equation}\label{eq:genanal10d}
\frac{d\psi(\calX,\calY,\q,\m,\beta,s,t)}{dt}  =     \mE_{G,{\mathcal U}_2} \frac{s\m_1}{\beta|s|\sqrt{n}\m_1\mE_{{\mathcal U}_1} Z^{\m_1}}\mE_{{\mathcal U}_1} \frac{1}{Z^{1-\m_1}}  \sum_{i_1=1}^{l} (C^{(i_1)})^{s-1} \sum_{i_2=1}^{l}\beta_{i_1}A^{(i_1,i_2)}
\lp \bar{T}_G + \bar{T}_2+ \bar{T}_1\rp.
\end{equation}
Finally, we recognize the following seven objects placed in the three mentioned groups that will play a key role in the ensuing computations.
\begin{equation}\label{eq:genanal10e}
\frac{d\psi(\calX,\calY,\q,\m,\beta,s,t)}{dt}  =       \frac{\mbox{sign}(s)}{2\beta\sqrt{n}} \sum_{i_1=1}^{l}  \sum_{i_2=1}^{l}
\beta_{i_1}\lp T_G + T_2+ T_1\rp,
\end{equation}
where
\begin{eqnarray}\label{eq:genanal10f}
T_G & = & \sum_{j=1}^{m}\frac{T_{G,j}}{\sqrt{t}}  \nonumber\\
T_2 & = & -\sum_{j=1}^{m}\frac{T_{2,1,j}}{\sqrt{1-t}}-\|\y^{(i_2)}\|_2\frac{T_{2,2}}{\sqrt{1-t}}+\|\y^{(i_2)}\|_2\frac{T_{2,3}}{\sqrt{t}} \nonumber\\
T_1 & = & -\sum_{j=1}^{m}\frac{T_{1,1,j}}{\sqrt{1-t}}-\|\y^{(i_2)}\|_2\frac{T_{1,2}}{\sqrt{1-t}}+\|\y^{(i_2)}\|_2\frac{T_{1,3}}{\sqrt{t}},
\end{eqnarray}
and
\begin{eqnarray}\label{eq:genanal10g}
T_{G,j} & = &  \mE_{G,{\mathcal U}_2} \lp\frac{1}{\mE_{{\mathcal U}_1} Z^{\m_1}}\mE_{{\mathcal U}_1}\frac{(C^{(i_1)})^{s-1} A^{(i_1,i_2)} \y_j^{(i_2)}\u_j^{(i_1,1)}}{Z^{1-\m_1}} \rp \nonumber \\
T_{2,1,j} & = &  \mE_{G,{\mathcal U}_2}\lp \frac{1}{\mE_{{\mathcal U}_1} Z^{\m_1}}\mE_{{\mathcal U}_1}\frac{(C^{(i_1)})^{s-1} A^{(i_1,i_2)} \y_j^{(i_2)}\u_j^{(2,2)}}{Z^{1-\m_1}} \rp \nonumber \\
T_{2,2} & = &  \mE_{G,{\mathcal U}_2} \lp\frac{1}{\mE_{{\mathcal U}_1} Z^{\m_1}}\mE_{{\mathcal U}_1}\frac{(C^{(i_1)})^{s-1} A^{(i_1,i_2)} \u^{(i_1,3,2)}}{Z^{1-\m_1}} \rp \nonumber \\
T_{2,3} & = &  \mE_{G,{\mathcal U}_2}\lp \frac{1}{\mE_{{\mathcal U}_1} Z^{\m_1}}\mE_{{\mathcal U}_1}\frac{(C^{(i_1)})^{s-1} A^{(i_1,i_2)} u^{(4,2)}}{Z^{1-\m_1}} \rp \nonumber \\
T_{1,1,j} & = &  \mE_{G,{\mathcal U}_2} \lp\frac{1}{\mE_{{\mathcal U}_1} Z^{\m_1}}\mE_{{\mathcal U}_1}\frac{(C^{(i_1)})^{s-1} A^{(i_1,i_2)} \y_j^{(i_2)}\u_j^{(2,1)}}{Z^{1-\m_1}}\rp \nonumber \\
T_{1,2} & = &  \mE_{G,{\mathcal U}_2} \lp\frac{1}{\mE_{{\mathcal U}_1} Z^{\m_1}}\mE_{{\mathcal U}_1}\frac{(C^{(i_1)})^{s-1} A^{(i_1,i_2)} \u^{(i_1,3,1)}}{Z^{1-\m_1}}\rp \nonumber \\
T_{1,3} & = &  \mE_{G,{\mathcal U}_2}\lp \frac{1}{\mE_{{\mathcal U}_1} Z^{\m_1}}\mE_{{\mathcal U}_1}\frac{(C^{(i_1)})^{s-1} A^{(i_1,i_2)} u^{(4,1)}}{Z^{1-\m_1}}\rp.
\end{eqnarray}
The above equations (\ref{eq:genanal10d})-(\ref{eq:genanal10g}) are, in principle, sufficient to determine the $\psi(\calX,\calY,\q,\m,\beta,s,t)$'s derivative. We will handle each of the above seven terms separately. While doing so, we will heavily rely on  \cite{Stojnicgscompyx16}. Moreover, to facilitate the presentation, whenever given a chance, we will try to parallel the presentation of  \cite{Stojnicgscompyx16} as much as possible.

%%%%%%%%%%%%%%%%%%%%%%%%%%%%%%%%%%%%%%%%%%%%%%%%%%%%%%%%%%%%%%%%%%%%%%%%
\subsection{Computing $\frac{d\psi(\calX,\calY,\p,\q,\m,\beta,s,t)}{dt}$}
\label{sec:compderivative}
%%%%%%%%%%%%%%%%%%%%%%%%%%%%%%%%%%%%%%%%%%%%%%%%%%%%%%%%%%%%%%%%%%%%%%%%

As mentioned above, we will separately handle all the terms appearing in (\ref{eq:genanal10g}). However, we also carefully choose the order in which we do so. In particular, we first handle the three terms from the last group ($T_{1,1,j}$, $T_{1,2}$, and $T_{1,3}$), then the three terms from the middle group ($T_{2,1,j}$, $T_{2,2}$, and $T_{2,3}$), and then $T_{G,j}$ as well. We call the last group $T_1$--group, the middle one $T_2$--group, and the first one $T_G$--group.

%%%%%%%%%%%%%%%%%%%%%%%%%%%%%%%%%%%%%%%%%%%%%%%%%%%%%%%%%%%%%%%%%%%%%%%%
\subsubsection{Handling $T_1$--group}
\label{sec:handlT1}
%%%%%%%%%%%%%%%%%%%%%%%%%%%%%%%%%%%%%%%%%%%%%%%%%%%%%%%%%%%%%%%%%%%%%%%%

We handle separately each of the three terms from $T_1$--group.

\vspace{.0in}
%%%%%%%%%%%%%%%%%%%%%%%%%%%%%%%%%%%%%%%%%%%%%%%%%%%%%%%%%%%%%%%%%%%%%%%%
\underline{\textbf{\emph{Determining}} $T_{1,1,j}$}
\label{sec:hand1T11}
%%%%%%%%%%%%%%%%%%%%%%%%%%%%%%%%%%%%%%%%%%%%%%%%%%%%%%%%%%%%%%%%%%%%%%%%

We start by relying on the Gaussian integration by parts to obtain

 \begin{eqnarray}\label{eq:liftgenAanal19}
T_{1,1,j} & = & \mE_{G,{\mathcal U}_2}\lp \frac{1}{\mE_{{\mathcal U}_1} Z^{\m_1}}\mE_{{\mathcal U}_1}  \frac{(C^{(i_1)})^{s-1} A^{(i_1,i_2)}\y_j^{(i_2)}}{Z^{1-\m_1}}\rp \nonumber \\
& = & \mE_{G,{\mathcal U}_2} \lp \frac{1}{\mE_{{\mathcal U}_1} Z^{\m_1}}\mE_{{\mathcal U}_1}\lp \mE_{{\mathcal U}_1}(\u_j^{(2,1)}\u_j^{(2,1)}) \frac{d}{du_j^{(2,1)}}\lp\frac{(C^{(i_1)})^{s-1} A^{(i_1,i_2)}\y_j^{(i_2)}}{Z^{1-\m_1}}\rp\rp\rp \nonumber \\
& = & \mE_{G,{\mathcal U}_2}\lp \frac{1}{\mE_{{\mathcal U}_1} Z^{\m_1}}\mE_{{\mathcal U}_1}(\u_j^{(2,1)}\u_j^{(2,1)})\mE_{{\mathcal U}_1}\lp  \frac{d}{du_j^{(2,1)}}\lp\frac{(C^{(i_1)})^{s-1} A^{(i_1,i_2)}\y_j^{(i_2)}}{Z^{1-\m_1}}\rp\rp\rp \nonumber \\
& = & \mE_{G,{\mathcal U}_2} \lp\frac{(\p_0-\p_1)}{\mE_{{\mathcal U}_1} Z^{\m_1}}\mE_{{\mathcal U}_1}\lp  \frac{d}{du_j^{(2,1)}}\lp\frac{(C^{(i_1)})^{s-1} A^{(i_1,i_2)}\y_j^{(i_2)}}{Z^{1-\m_1}}\rp\rp\rp.
\end{eqnarray}
Now, we observe that the inner expectation on the most right hand side is structurally identical to the one considered in the second part of Section 3.1 in \cite{Stojnicgscompyx16}. That quantity was computed based on the analysis from Section 2.1.2 of \cite{Stojnicgscompyx16}. One notes that the above minimal scaling adjustment (together with a trivial replacement of $c_3$ by $\m_1$) then immediately gives
\begin{eqnarray}\label{eq:liftgenAanal19a}
T_{1,1,j} & = &  \mE_{G,{\mathcal U}_2} \lp\frac{(\p_0-\p_1)}{\mE_{{\mathcal U}_1} Z^{\m_1}}\mE_{{\mathcal U}_1}\lp  \frac{d}{du_j^{(2,1)}}\lp\frac{(C^{(i_1)})^{s-1} A^{(i_1,i_2)}\y_j^{(i_2)}}{Z^{1-\m_1}}\rp\rp\rp \nonumber \\
& = &  (\p_0-\p_1)\mE_{G,{\mathcal U}_2} \lp\frac{1}{\mE_{{\mathcal U}_1} Z^{\m_1}}\lp \Theta_1+\Theta_2 \rp\rp,
\end{eqnarray}
where from \cite{Stojnicgscompyx16},
\begin{eqnarray}\label{eq:liftgenAanal19b}
\mE_{{\mathcal U}_1}\lp  \frac{d}{du_j^{(2,1)}}\lp\frac{(C^{(i_1)})^{s-1} A^{(i_1,i_2)}\y_j^{(i_2)}}{Z^{1-\m_1}}\rp\rp
=\Theta_1+\Theta_2,
\end{eqnarray}
with
{\small\begin{eqnarray}\label{eq:liftgenAanal19c}
\Theta_1 &  = &  \mE_{{\mathcal U}_1} \Bigg( \Bigg. \frac{\y_j^{(i_2)} \lp (C^{(i_1)})^{s-1}\beta_{i_1}A^{(i_1,i_2)}\y_j^{(i_2)}\sqrt{1-t} +A^{(i_1,i_2)}(s-1)(C^{(i_1)})^{s-2}\beta_{i_1}\sum_{p_2=1}^{l}A^{(i_1,p_2)}\y_j^{(p_2)}\sqrt{1-t}\rp}{Z^{1-\m_1}}\Bigg. \Bigg)\Bigg. \Bigg) \nonumber \\
\Theta_2 & = & -(1-\m_1)\mE_{{\mathcal U}_1} \lp\sum_{p_1=1}^{l}
\frac{(C^{(i_1)})^{s-1} A^{(i_1,i_2)}\y_j^{(i_2)}}{Z^{2-\m_1}}
s  (C^{(p_1)})^{s-1}\sum_{p_2=1}^{l}\beta_{p_1}A^{(p_1,p_2)}\y_j^{(p_2)}\sqrt{1-t}\rp\Bigg.\Bigg).\nonumber \\
\end{eqnarray}}
We find it convenient to observe the following as well
{\footnotesize\begin{eqnarray}\label{eq:liftgenAanal19d}
\sum_{i_1=1}^{l}\sum_{i_2=1}^{l}\sum_{j=1}^{m} \lp\frac{1}{\mE_{{\mathcal U}_1} Z^{\m_1}} \frac{\beta_{i_1}\Theta_1}{\sqrt{1-t}}\rp
&  = & \lp  \mE_{{\mathcal U}_1}\frac{Z^{\m_1}}{\mE_{{\mathcal U}_1} Z^{\m_1}} \sum_{i_1=1}^{l}\frac{(C^{(i_1)})^s}{Z}\sum_{i_2=1}^{l}\frac{A^{(i_1,i_2)}}{C^{(i_1)}}\beta_{i_1}^2\|\y^{(i_2)}\|_2^2\rp \nonumber\\
& & +  \lp \mE_{{\mathcal U}_1}\frac{Z^{\m_1}}{\mE_{{\mathcal U}_1} Z^{\m_1}}  \sum_{i_1=1}^{l}\frac{(s-1)(C^{(i_1)})^s}{Z}\sum_{i_2=1}^{l}\sum_{p_2=1}^{l}\frac{A^{(i_1,i_2)}A^{(i_1,p_2)}}{(C^{(i_1)})^2}\beta_{i_1}^2(\y^{(p_2)})^T\y^{(i_2)}\rp.\nonumber \\
 \end{eqnarray}}

\noindent Consider the operator
\begin{eqnarray}\label{eq:genAanal19e}
 \Phi_{{\mathcal U}_1} & \triangleq &  \mE_{{\mathcal U}_{1}} \frac{Z^{\m_1}}{\mE_{{\mathcal U}_{1}}Z^{\m_1}}  \triangleq  \mE_{{\mathcal U}_{1}} \lp \frac{Z^{\m_1}}{\mE_{{\mathcal U}_{1}}Z^{\m_1}}\lp \cdot \rp\rp,
 \end{eqnarray}
and set
\begin{eqnarray}\label{eq:genAanal19e1}
  \gamma_0(i_1,i_2) & = &
\frac{(C^{(i_1)})^{s}}{Z}  \frac{A^{(i_1,i_2)}}{C^{(i_1)}} \nonumber \\
\gamma_{01}^{(1)}  & = &  \Phi_{{\mathcal U}_1} (\gamma_0(i_1,i_2)) \nonumber \\
\gamma_{02}^{(1)}  & = &  \Phi_{{\mathcal U}_1} (\gamma_0(i_1,i_2)\times \gamma_0(i_1,p_2)) \nonumber \\
\gamma_{1}^{(1)}   & = &  \Phi_{{\mathcal U}_1}  \lp \gamma_0(i_1,i_2)\times \gamma_0(p_1,p_2) \rp \nonumber \\.
\gamma_{2}^{(1)}   & = &  \Phi_{{\mathcal U}_1}  \gamma_0(i_1,i_2)  \times   \Phi_{{\mathcal U}_1} \gamma_0(p_1,p_2).
\end{eqnarray}
Clearly, all the above $\gamma$'s are valid measures. For example, one has the following
\begin{eqnarray}\label{eq:genAanal19f}
 \sum_{i_1=1}^{l}  \sum_{i_2=1}^{l} \gamma_{01}^{(1)}  & = & \sum_{i_1=1}^{l}  \sum_{i_2=1}^{l}\mE_{{\mathcal U}_1}\frac{Z^{\m_1}}{\mE_{{\mathcal U}_1} Z^{\m_1}}\frac{(C^{(i_1)})^{s}}{Z}  \frac{A^{(i_1,i_2)}}{C^{(i_1)}} \nonumber \\
& = &
 \mE_{{\mathcal U}_1}\frac{Z^{\m_1}}{\mE_{{\mathcal U}_1} Z^{\m_1}}   \sum_{i_1=1}^{l} \frac{(C^{(i_1)})^{s}}{Z}  \sum_{i_2=1}^{l} \frac{A^{(i_1,i_2)}}{C^{(i_1)}}\nonumber\\
 & = &
 \mE_{{\mathcal U}_1}\frac{Z^{\m_1}}{\mE_{{\mathcal U}_1} Z^{\m_1}}  \nonumber \\
 & = & 1,
 \end{eqnarray}
where the second to last equality follows by the definitions of $Z$ and $C^{(i_1)}$  from (\ref{eq:genanal7}). Combining (\ref{eq:genAanal19f}) with the trivial fact $\gamma_{01}^{(1)}\geq 0$, one indeed has that $\gamma_{01}^{(1)}$ is a measure. The proofs for the remaining $\gamma$'s are identical and we omit them. Throughout the paper, we denote by $\langle \cdot \rangle_a$ the average with respect to measure $a$. Also, all the above $\gamma$ measures (and $\Phi(\cdot)$ operators) are clearly functions of $t$. Consequently, all functions that depend on $\gamma$'s throughout  the paper will be functions of $t$. To lighten the notation we skip stating this explicitly all the time and, instead, assume it as implicitly trivial.

From (\ref{eq:liftgenAanal19d})
\begin{eqnarray}\label{eq:liftgenAanal19g}
\sum_{i_1=1}^{l}\sum_{i_2=1}^{l}\sum_{j=1}^{m} \lp\frac{1}{\mE_{{\mathcal U}_1} Z^{\m_1}} \frac{\beta_{i_1}\Theta_1}{\sqrt{1-t}}\rp
&  = & \beta^2 \lp \langle \|\x^{(i_1)}\|_2^2\|\y^{(i_2)}\|_2^2\rangle_{\gamma_{01}^{(1)}} +  (s-1) \langle \|\x^{(i_1)}\|_2^2(\y^{(p_2)})^T\y^{(i_2)}\rangle_{\gamma_{02}^{(1)}} \rp. \nonumber \\
 \end{eqnarray}
From (\ref{eq:liftgenAanal19c}), we also have
\begin{eqnarray}\label{eq:liftgenAanal19h}
\sum_{i_1=1}^{l}\sum_{i_2=1}^{l}\sum_{j=1}^{m} \lp\frac{1}{\mE_{{\mathcal U}_1} Z^{\m_1}} \frac{\beta_{i_1}\Theta_2}{\sqrt{1-t}}\rp
 & = & -s(1-\m_1) \mE_{G,{\mathcal U}_2} \Bigg( \Bigg. \frac{Z^{\m_1}}{\mE_{{\mathcal U}_1} Z^{\m_1}} \sum_{i_1=1}^{l}\frac{(C^{(i_1)})^s}{Z}\sum_{i_2=1}^{l}
\frac{A^{(i_1,i_2)}}{C^{(i_1)}} \nonumber \\
& & \times
 \sum_{p_1=1}^{l} \frac{(C^{(p_1)})^s}{Z}\sum_{p_2=1}^{l}\frac{A^{(p_1,p_2)}}{C^{(p_1)}} \beta_{i_1}\beta_{p_1}(\y^{(p_2)})^T\y^{(i_2)} \Bigg.\Bigg)\nonumber \\
& =& -s\beta^2(1-\m_1) \mE_{G,{\mathcal U}_2} \langle \|\x^{(i_1)}\|_2\|\x^{(p_1)}\|_2(\y^{(p_2)})^T\y^{(i_2)} \rangle_{\gamma_{1}^{(1)}}.
\end{eqnarray}
Combining  (\ref{eq:liftgenAanal19a}), (\ref{eq:liftgenAanal19g}), and (\ref{eq:liftgenAanal19h}) we obtain
\begin{eqnarray}\label{eq:liftgenAanal19i}
\sum_{i_1=1}^{l}\sum_{i_2=1}^{l}\sum_{j=1}^{m} \beta_{i_1}\frac{T_{1,1,j}}{\sqrt{1-t
}}
& = &  (\p_0-\p_1) \mE_{G,{\mathcal U}_2} \lp\frac{1}{\mE_{{\mathcal U}_1} Z^{\m_1}}\lp \frac{\beta_{i_1}\Theta_1}{\sqrt{1-t}}+\frac{\beta_{i_1}\Theta_2}{\sqrt{1-t}} \rp\rp\nonumber \\
& = & (\p_0-\p_1)\beta^2 \nonumber \\
 & &
 \times \lp \mE_{G,{\mathcal U}_2} \langle \|\x^{(i_1)}\|_2^2\|\y^{(i_2)}\|_2^2\rangle_{\gamma_{01}^{(1)}} +  (s-1)\mE_{G,{\mathcal U}_2}\langle \|\x^{(i_1)}\|_2^2(\y^{(p_2)})^T\y^{(i_2)}\rangle_{\gamma_{02}^{(1)}} \rp \nonumber \\
& & - (\p_0-\p_1)s\beta^2(1-\m_1)\langle \|\x^{(i_1)}\|_2\|\x^{(p_1)}\|_2(\y^{(p_2)})^T\y^{(i_2)} \rangle_{\gamma_{1}^{(1)}}.
\end{eqnarray}

%%%%%%%%%%%%%%%%%%%%%%%%%%%%%%%%%%%%%%%%%%%%%%%%%%%%%%%%%%%%%%%%%%%%%%%%
\underline{\textbf{\emph{Determining}} $T_{1,2}$}
\label{sec:hand1T12}
%%%%%%%%%%%%%%%%%%%%%%%%%%%%%%%%%%%%%%%%%%%%%%%%%%%%%%%%%%%%%%%%%%%%%%%%

As above, we follow the Gaussian integration by parts path to find
\begin{eqnarray}\label{eq:liftgenBanal20}
T_{1,2} & = & \mE_{G,{\mathcal U}_2} \lp \frac{1}{\mE_{{\mathcal U}_1} Z^{\m_1}}\mE_{{\mathcal U}_1} \frac{(C^{(i_1)})^{s-1} A^{(i_1,i_2)}\u^{(i_1,3,1)}}{Z^{1-\m_1}}\rp \nonumber \\
& = & \mE_{G,{\mathcal U}_2} \lp \frac{1}{\mE_{{\mathcal U}_1} Z^{\m_1}}
\mE_{{\mathcal U}_1} \sum_{p_1=1}^{l}\mE_{{\mathcal U}_1}(\u^{(i_1,3,1)}\u^{(p_1,3,1)}) \frac{d}{d\u^{(p_1,3,1)}}\lp\frac{(C^{(i_1)})^{s-1} A^{(i_1,i_2)}}{Z^{1-\m_1}}\rp\rp \nonumber \\
& = & \mE_{G,{\mathcal U}_2} \lp\frac{(\q_0-\q_1)}{\mE_{{\mathcal U}_1} Z^{\m_1}}
\mE_{{\mathcal U}_1} \sum_{p_1=1}^{l}\frac{(\x^{(i_1)})^T\x^{(p_1)}}{\|\x^{(i_1)}\|_2\|\x^{(p_1)}\|_2} \frac{d}{d\u^{(p_1,3,1)}}\lp\frac{(C^{(i_1)})^{s-1} A^{(i_1,i_2)}}{Z^{1-\m_1}}\rp\rp.
\end{eqnarray}
The inner expectation on the most right hand side is structurally identical to the one considered in the third part of Section 3.1 in \cite{Stojnicgscompyx16}. Following the analysis from Section 2.1.3 of \cite{Stojnicgscompyx16}, one, after trivially replacing $c_3$ by $\m_1$, immediately has

$ $
{\footnotesize \begin{eqnarray}\label{eq:liftgenBanal20a}
T_{1,2}  & = & \mE_{G,{\mathcal U}_2} \lp\frac{(\q_0-\q_1)}{\mE_{{\mathcal U}_1} Z^{\m_1}}
\mE_{{\mathcal U}_1} \sum_{p_1=1}^{l}\frac{(\x^{(i_1)})^T\x^{(p_1)}}{\|\x^{(i_1)}\|_2\|\x^{(p_1)}\|_2} \frac{d}{d\u^{(p_1,3,1)}}\lp\frac{(C^{(i_1)})^{s-1} A^{(i_1,i_2)}}{Z^{1-\m_1}}\rp\rp \nonumber \\
& = &
\mE_{G,{\mathcal U}_2} \Bigg( \Bigg. \frac{(\q_0-\q_1)}{\mE_{{\mathcal U}_1} Z^{\m_1}} \nonumber \\
& & \times \mE_{{\mathcal U}_1} \lp\frac{(C^{(i_1)})^{s-1}\beta_{i_1}A^{(i_1,i_2)}\|\y^{(i_2)}\|_2\sqrt{1-t}+A^{(i_1,i_2)}(s-1)(C^{(i_1)})^{s-2}\beta_{i_1}\sum_{p_2=1}^{l}A^{(i_1,p_2)}\|\y^{(p_2)}\|_2\sqrt{1-t} }{Z^{1-\m_1}}\rp\Bigg.\Bigg) \nonumber  \\
  & & -\mE_{G,{\mathcal U}_2}\Bigg( \Bigg. \frac{(\q_0-\q_1)(1-\m_1)}{\mE_{{\mathcal U}_1} Z^{\m_1}} \nonumber \\
& & \times\mE_{{\mathcal U}_1} \lp\sum_{p_1=1}^{l}\frac{(\x^{(i_1)})^T\x^{(p_1)}}{\|\x^{(i_1)}\|_2\|\x^{(p_1)}\|_2}
\frac{(C^{(i_1)})^{s-1} A^{(i_1,i_2)}}{Z^{2-\m_1}}
s  (C^{(p_1)})^{s-1}\sum_{p_2=1}^{l}\beta_{p_1}A^{(p_1,p_2)}\|\y^{(p_2)}\|_2\sqrt{1-t}\rp\Bigg. \Bigg).
\end{eqnarray}}
Following the above considerations leading to (\ref{eq:liftgenAanal19i}), we can then also immediately write
\begin{eqnarray}\label{eq:liftgenBanal20b}
\sum_{i_1=1}^{l}\sum_{i_2=1}^{l} \beta_{i_1}\|\y^{(i_2)}\|_2 \frac{T_{1,2}}{\sqrt{1-t}} & = & (\q_0-\q_1)\beta^2 \nonumber \\
& & \times
\lp\mE_{G,{\mathcal U}_2}\langle \|\x^{(i_1)}\|_2^2\|\y^{(i_2)}\|_2^2\rangle_{\gamma_{01}^{(1)}} +   (s-1)\mE_{G,{\mathcal U}_2}\langle \|\x^{(i_1)}\|_2^2 \|\y^{(i_2)}\|_2\|\y^{(p_2)}\|_2\rangle_{\gamma_{02}^{(1)}}\rp\nonumber \\
& & - (\q_0-\q_1)s\beta^2(1-\m_1)\mE_{G,{\mathcal U}_2}\langle (\x^{(p_1)})^T\x^{(i_1)}\|\y^{(i_2)}\|_2\|\y^{(p_2)}\|_2 \rangle_{\gamma_{1}^{(1)}}.
\end{eqnarray}

$ $

%%%%%%%%%%%%%%%%%%%%%%%%%%%%%%%%%%%%%%%%%%%%%%%%%%%%%%%%%%%%%%%%%%%%%%%%
\underline{\textbf{\emph{Determining}} $T_{1,3}$}
\label{sec:hand1T13}
%%%%%%%%%%%%%%%%%%%%%%%%%%%%%%%%%%%%%%%%%%%%%%%%%%%%%%%%%%%%%%%%%%%%%%%%

We again proceed via Gaussian integration by parts to obtain
\begin{eqnarray}\label{eq:liftgenCanal21}
T_{1,3} & = & \mE_{G,{\mathcal U}_2} \lp \frac{1}{\mE_{{\mathcal U}_1} Z^{\m_1}}\mE_{{\mathcal U}_1}  \frac{(C^{(i_1)})^{s-1} A^{(i_1,i_2)}u^{(4,1)}}{Z^{1-\m_1}} \rp \nonumber \\
& = & \mE_{G,{\mathcal U}_2} \lp \frac{1}{\mE_{{\mathcal U}_1} Z^{\m_1}}\mE_{{\mathcal U}_1} \lp\mE_{{\mathcal U}_1} (u^{(4,1)}u^{(4,1)})\lp\frac{d}{du^{(4,1)}} \lp\frac{(C^{(i_1)})^{s-1} A^{(i_1,i_2)}u^{(4,1)}}{Z^{1-\m_1}}\rp \rp\rp\rp \nonumber \\
& = & \mE_{G,{\mathcal U}_2} \lp \frac{(\p_0\q_0-\p_1\q_1)}{\mE_{{\mathcal U}_1} Z^{\m_1}}\mE_{{\mathcal U}_1} \lp\frac{d}{du^{(4,1)}} \lp\frac{(C^{(i_1)})^{s-1} A^{(i_1,i_2)}u^{(4,1)}}{Z^{1-\m_1}}\rp\rp\rp.
\end{eqnarray}
The last inner expectation is structurally identical to the one considered in the fourth part of Section 3.1 in \cite{Stojnicgscompyx16} which was computed based on the analysis from Section 2.1.4 of \cite{Stojnicgscompyx16}. As a result of that analysis we immediately have
\begin{eqnarray}\label{eq:liftgenCanal21a}
T_{1,3}  & = & \mE_{G,{\mathcal U}_2} \lp \frac{(\p_0\q_0-\p_1\q_1)}{\mE_{{\mathcal U}_1} Z^{\m_1}}\mE_{{\mathcal U}_1} \lp\frac{d}{du^{(4,1)}} \lp\frac{(C^{(i_1)})^{s-1} A^{(i_1,i_2)}u^{(4,1)}}{Z^{1-\m_1}}\rp\rp\rp \nonumber \\
 & = &\mE_{G,{\mathcal U}_2} \Bigg( \Bigg. \frac{(\p_0\q_0-\p_1\q_1)}{\mE_{{\mathcal U}_1} Z^{\m_1}} \nonumber \\
& & \times  \mE_{{\mathcal U}_1} \lp\frac{(C^{(i_1)})^{s-1}\beta_{i_1}A^{(i_1,i_2)}\|\y^{(i_2)}\|_2\sqrt{t}+A^{(i_1,i_2)}(s-1)(C^{(i_1)})^{s-2}\beta_{i_1}\sum_{p_2=1}^{l}A^{(i_1,p_2)}\|\y^{(p_2)}\|_2\sqrt{t}}{Z^{1-\m_1}} \rp \Bigg.\Bigg)\nonumber \\
& & -\mE_{G,{\mathcal U}_2} \Bigg( \Bigg. \frac{(\p_0\q_0-\p_1\q_1)(1-\m_1)}{\mE_{{\mathcal U}_1} Z^{\m_1}} \nonumber \\
& & \times \mE_{{\mathcal U}_1} \lp
\frac{(C^{(i_1)})^{s-1} A^{(i_1,i_2)}}{Z^{2-\m_1}}
s  \sum_{p_1=1}^{l} (C^{(p_1)})^{s-1}\sum_{p_2=1}^{l}\beta_{p_1}A^{(p_1,p_2)}\|\y^{(p_2)}\|_2\sqrt{t}\rp \Bigg. \Bigg).
\end{eqnarray}
Analogously to (\ref{eq:liftgenBanal20b}), we have
\begin{eqnarray}\label{eq:liftgenCanal21b}
\sum_{i_1=1}^{l}\sum_{i_2=1}^{l} \beta_{i_1}\|\y^{(i_2)}\|_2 \frac{T_{1,3}}{\sqrt{t}} & = & (\p_0\q_0-\p_1\q_1)\beta^2 \nonumber \\
& & \times
\lp \mE_{G,{\mathcal U}_2}\langle \|\x^{(i_1)}\|_2^2\|\y^{(i_2)}\|_2^2\rangle_{\gamma_{01}^{(1)}} +   (s-1)\mE_{G,{\mathcal U}_2}\langle \|\x^{(i_1)}\|_2^2 \|\y^{(i_2)}\|_2\|\y^{(p_2)}\|_2\rangle_{\gamma_{02}^{(1)}}\rp\nonumber \\
& & - (\p_0\q_0-\p_1\q_1)s\beta^2(1-\m_1)\mE_{G,{\mathcal U}_2}\langle \|\x^{(i_1)}\|_2\|\x^{(p_`)}\|_2\|\y^{(i_2)}\|_2\|\y^{(p_2)}\|_2 \rangle_{\gamma_{1}^{(1)}}. \nonumber \\
\end{eqnarray}

%%%%%%%%%%%%%%%%%%%%%%%%%%%%%%%%%%%%%%%%%%%%%%%%%%%%%%%%%%%%%%%%%%%%%%%%
%%%%%%%%%%%%%%%%%%%%%%%%%%%%%%%%%%%%%%%%%%%%%%%%%%%%%%%%%%%%%%%%%%%%%%%%
%%%%%%%%%%%%%%%%%%%%%%%%%%%%%%%%%%%%%%%%%%%%%%%%%%%%%%%%%%%%%%%%%%%%%%%%
%%%%%%%%%%%%%%%%%%%%%%%%%%%%%%%%%%%%%%%%%%%%%%%%%%%%%%%%%%%%%%%%%%%%%%%%
\subsubsection{Handling $T_2$--group}
\label{sec:handlT2}
%%%%%%%%%%%%%%%%%%%%%%%%%%%%%%%%%%%%%%%%%%%%%%%%%%%%%%%%%%%%%%%%%%%%%%%%
%%%%%%%%%%%%%%%%%%%%%%%%%%%%%%%%%%%%%%%%%%%%%%%%%%%%%%%%%%%%%%%%%%%%%%%%
%%%%%%%%%%%%%%%%%%%%%%%%%%%%%%%%%%%%%%%%%%%%%%%%%%%%%%%%%%%%%%%%%%%%%%%%
%%%%%%%%%%%%%%%%%%%%%%%%%%%%%%%%%%%%%%%%%%%%%%%%%%%%%%%%%%%%%%%%%%%%%%%%

Similarly to above, we handle separately each of the three terms that $T_2$'s contribution is comprised of.

%%%%%%%%%%%%%%%%%%%%%%%%%%%%%%%%%%%%%%%%%%%%%%%%%%%%%%%%%%%%%%%%%%%%%%%%
\underline{\textbf{\emph{Determining}} $T_{2,1,j}$}
\label{sec:hand1T21}
%%%%%%%%%%%%%%%%%%%%%%%%%%%%%%%%%%%%%%%%%%%%%%%%%%%%%%%%%%%%%%%%%%%%%%%%

Relying on the Gaussian integration by parts we write
\begin{eqnarray}\label{eq:genDanal19}
T_{2,1,j}& = &  \mE_{G,{\mathcal U}_2}\lp \frac{1}{\mE_{{\mathcal U}_1} Z^{\m_1}}\mE_{{\mathcal U}_1}\frac{(C^{(i_1)})^{s-1} A^{(i_1,i_2)} \y_j^{(i_2)}\u_j^{(2,2)}}{Z^{1-\m_1}} \rp \nonumber \\
 &  = & \mE_{G,{\mathcal U}_2,{\mathcal U}_1}  \frac{(C^{(i_1)})^{s-1} A^{(i_1,i_2)}\y_j^{(i_2)}\u_j^{(2,2)}}{Z^{1-\m_1}\mE_{{\mathcal U}_1} Z^{\m_1}} \nonumber \\
& = &
\mE_{G,{\mathcal U}_1}\lp\mE_{{\mathcal U}_2}\lp\mE_{{\mathcal U}_2} (\u_j^{(2,2)}\u_j^{(2,2)})\frac{d}{d\u_j^{(2,2)}}\lp \frac{(C^{(i_1)})^{s-1} A^{(i_1,i_2)}\y_j^{(i_2)}}{Z^{1-\m_1}\mE_{{\mathcal U}_1} Z^{\m_1}}\rp\rp\rp \nonumber \\
& = &
\mE_{G,{\mathcal U}_2,{\mathcal U}_1}\lp \frac{\p_1}{\mE_{{\mathcal U}_1} Z^{\m_1}}\frac{d}{d\u_j^{(2,2)}}\lp \frac{(C^{(i_1)})^{s-1} A^{(i_1,i_2)}\y_j^{(i_2)}}{Z^{1-\m_1}}\rp\rp \nonumber \\
& & + \mE_{G,{\mathcal U}_2,{\mathcal U}_1}\lp \frac{\p_1\lp(C^{(i_1)})^{s-1} A^{(i_1,i_2)}\y_j^{(i_2)} \rp}{Z^{1-\m_1}}\frac{d}{d\u_j^{(2,2)}}\lp \frac{1}{\mE_{{\mathcal U}_1} Z^{\m_1}}\rp\rp.\nonumber \\
\end{eqnarray}
For the convenience of further exposition, we write the above as
\begin{eqnarray}\label{eq:genDanal19a}
T_{2,1,j}   =   T_{2,1,j}^{c} +  T_{2,1,j}^{d},
\end{eqnarray}
where
\begin{eqnarray}\label{eq:genDanal19b}
T_{2,1,j}^c &  = &
\mE_{G,{\mathcal U}_2,{\mathcal U}_1}\lp \frac{\p_1}{\mE_{{\mathcal U}_1} Z^{\m_1}}\frac{d}{d\u_j^{(2,2)}}\lp \frac{(C^{(i_1)})^{s-1} A^{(i_1,i_2)}\y_j^{(i_2)}}{Z^{1-\m_1}}\rp\rp \nonumber \\
T_{2,1,j}^d &  = & \mE_{G,{\mathcal U}_2,{\mathcal U}_1}\lp \frac{\p_1\lp(C^{(i_1)})^{s-1} A^{(i_1,i_2)}\y_j^{(i_2)} \rp}{Z^{1-\m_1}}\frac{d}{d\u_j^{(2,2)}}\lp \frac{1}{\mE_{{\mathcal U}_1} Z^{\m_1}}\rp\rp.
\end{eqnarray}
One now observes that $T_{2,1,j}^c$ scaled by $\p_1$ is structurally identical to the term considered in the first part of Section \ref{sec:hand1T11} scaled by $(\p_0-\p_1)$. That means that we have
\begin{eqnarray}\label{eq:genDanal19b1}
\sum_{i_1=1}^{l}\sum_{i_2=1}^{l}\sum_{j=1}^{m} \beta_{i_1}\frac{T_{2,1,j}^c}{\sqrt{1-t
}}
& = &  \sum_{i_1=1}^{l}\sum_{i_2=1}^{l}\sum_{j=1}^{m} \beta_{i_1}\frac{T_{1,1,j}}{\sqrt{1-t}}\frac{\p_1}{\p_0-\p_1}\nonumber\\
& = & \p_1\beta^2
\lp \mE_{G,{\mathcal U}_2}\langle \|\x^{(i_1)}\|_2^2\|\y^{(i_2)}\|_2^2\rangle_{\gamma_{01}^{(1)}} +   (s-1)\mE_{G,{\mathcal U}_2}\langle \|\x^{(i_1)}\|_2^2(\y^{(p_2)})^T\y^{(i_2)}\rangle_{\gamma_{02}^{(1)}} \rp\nonumber \\
& & - \p_1s\beta^2(1-\m_1)\mE_{G,{\mathcal U}_2}\langle \|\x^{(i_1)}\|_2\|\x^{(p_1)}\|_2(\y^{(p_2)})^T\y^{(i_2)} \rangle_{\gamma_{1}^{(1)}}.
\end{eqnarray}
We now focus on $T_{2,1,j}^d$. To that end we have
\begin{eqnarray}\label{eq:genDanal20}
\frac{d}{d\u_j^{(2,2)}}\lp \frac{1}{\mE_{{\mathcal U}_1} Z^{\m_1}}\rp = -\frac{1}{\lp\mE_{{\mathcal U}_1} Z^{\m_1}\rp^2}
\mE_{{\mathcal U}_1}\frac{d Z^{\m_1}}{d\u_j^{(2,2)}}.\nonumber \\
\end{eqnarray}
Moreover, we find
\begin{eqnarray}\label{eq:genDanal21}
\frac{dZ^{\m_1}}{d\u_j^{(2,2)}} & = & \frac{\m_1}{Z^{1-\m_1}}\frac{d\lp \sum_{p_1=1}^{l}  (C^{(p_1)})^s \rp}{d\u_j^{(2,2)}}
=\frac{\m_1}{Z^{1-\m_1}} s \sum_{p_1=1}^{l}  (C^{(p_1)})^{s-1}\sum_{p_2=1}^{l}\frac{d(A^{(p_1,p_2)})}{d\u_j^{(2)}} \nonumber \\
& = & \frac{\m_1}{Z^{1-\m_1}}s  \sum_{p_1=1}^{l}  (C^{(p_1)})^{s-1}\sum_{p_2=1}^{l}
\beta_{p_1}A^{(p_1,p_2)}\y_j^{(p_2)}\sqrt{1-t}.
\end{eqnarray}
A combination of (\ref{eq:genDanal20}) and (\ref{eq:genDanal21}) gives
\begin{eqnarray}\label{eq:genDanal22}
\frac{d}{d\u_j^{(2,2)}}\lp \frac{1}{\mE_{{\mathcal U}_1} Z^{\m_1}}\rp = -\frac{1}{\lp\mE_{{\mathcal U}_1} Z^{\m_1}\rp^2}
\mE_{{\mathcal U}_1}  \frac{\m_1}{Z^{1-\m_1}}s \sum_{p_1=1}^{l}  (C^{(p_1)})^{s-1}\sum_{p_2=1}^{l}
\beta_{p_1}A^{(p_1,p_2)}\y_j^{(p_2)}\sqrt{1-t}.
\end{eqnarray}
Combining further (\ref{eq:genDanal19b}) and (\ref{eq:genDanal22}), we obtain
\begin{eqnarray}\label{eq:genDanal23}
T_{2,1,j}^d &  = & \mE_{G,{\mathcal U}_2,{\mathcal U}_1}\lp \frac{\p_1\lp(C^{(i_1)})^{s-1} A^{(i_1,i_2)}\y_j^{(i_2)} \rp}{Z^{1-\m_1}}\frac{d}{d\u_j^{(2,2)}}\lp \frac{1}{\mE_{{\mathcal U}_1} Z^{\m_1}}\rp\rp \nonumber \\
& = & -s\sqrt{1-t}\p_1\m_1\mE_{G,{\mathcal U}_2,{\mathcal U}_1}\Bigg( \Bigg. \frac{\lp(C^{(i_1)})^{s-1} A^{(i_1,i_2)}\y_j^{(i_2)} \rp}{Z^{1-\m_1}} \nonumber \\
& & \times
\lp \frac{1}{\lp\mE_{{\mathcal U}_1} Z^{\m_1}\rp^2}
\mE_{{\mathcal U}_1}  \frac{1}{Z^{1-\m_1}} \sum_{p_1=1}^{l}  (C^{(p_1)})^{s-1}\sum_{p_2=1}^{l}
\beta_{p_1}A^{(p_1,p_2)}\y_j^{(p_2)} \rp\Bigg. \Bigg)\nonumber \\
& = & -s\sqrt{1-t}\p_1\m_1\mE_{G,{\mathcal U}_2}\Bigg( \Bigg. \mE_{{\mathcal U}_1}\frac{Z^{\m_1}}{\mE_{{\mathcal U}_1} Z^{\m_1}}
 \frac{(C^{(i_1)})^{s}}{Z}  \frac{A^{(i_1,i_2)}}{C^{(i_1)}}\y_j^{(i_2)} \nonumber \\
& & \times
\lp \mE_{{\mathcal U}_1}  \frac{Z^{\m_1}}{\mE_{{\mathcal U}_1} Z^{\m_1}} \sum_{p_1=1}^{l}  \frac{(C^{(p_1)})^s}{Z}\sum_{p_2=1}^{l}
\frac{A^{(p_1,p_2)}}{C^{(p,1)}}\beta_{p_1}\y_j^{(p_2)}  \rp\Bigg. \Bigg).
\end{eqnarray}
We find it useful to note the following as well
\begin{eqnarray}\label{eq:genDanal24}
 \sum_{i_1=1}^{l}  \sum_{i_2=1}^{l} \sum_{j=1}^{m}  \beta_{i_1}\frac{T_{2,1,j}^d}{\sqrt{1-t}} &  = &  -s\p_1\m_1\mE_{G,{\mathcal U}_2}  \sum_{j=1}^{m} \Bigg( \Bigg. \mE_{{\mathcal U}_1}\frac{Z^{\m_1}}{\mE_{{\mathcal U}_1} Z^{\m_1}}
 \sum_{i_1=1}^{l} \frac{(C^{(i_1)})^{s}}{Z}  \sum_{i_2=1}^{l} \frac{A^{(i_1,i_2)}}{C^{(i_1)}} \beta_{i_1}\y_j^{(i_2)} \nonumber \\
& & \times
\lp \mE_{{\mathcal U}_1}  \frac{Z^{\m_1}}{\mE_{{\mathcal U}_1} Z^{\m_1}} \sum_{p_1=1}^{l}  \frac{(C^{(p_1)})^s}{Z}\sum_{p_2=1}^{l}
\frac{A^{(p_1,p_2)}}{C^{(p,1)}}\beta_{p_1}\y_j^{(p_2)}  \rp\Bigg. \Bigg)\nonumber\\
& = & -s\p_1\m_1\mE_{G,{\mathcal U}_2} \langle \beta_{p_1}\beta_{i_1}(\y^{(p_2)})^T\y^{(i_2)} \rangle_{\gamma_{2}^{(1)}} \nonumber \\
& = & -s\beta^2\p_1\m_1\mE_{G,{\mathcal U}_2} \langle \|\x^{(i_1)}\|_2\|\x^{(p_1)}\|_2(\y^{(p_2)})^T\y^{(i_2)} \rangle_{\gamma_{2}^{(1)}}.
\end{eqnarray}
A combination of (\ref{eq:genDanal19a}), (\ref{eq:genDanal19b1}), and (\ref{eq:genDanal24}) produces
\begin{eqnarray}\label{eq:genDanal25}
 \sum_{i_1=1}^{l}  \sum_{i_2=1}^{l} \sum_{j=1}^{m}  \beta_{i_1}\frac{T_{2,1,j}}{\sqrt{1-t}}& = & \p_1\beta^2
 \lp \mE_{G,{\mathcal U}_2}\langle \|\x^{(i_1)}\|_2^2\|\y^{(i_2)}\|_2^2\rangle_{\gamma_{01}^{(1)}} +  (s-1)\mE_{G,{\mathcal U}_2}\langle \|\x^{(i_1)}\|_2^2(\y^{(p_2)})^T\y^{(i_2)}\rangle_{\gamma_{02}^{(1)}} \rp\nonumber \\
& & - \p_1s\beta^2(1-\m_1)\mE_{G,{\mathcal U}_2}\langle \|\x^{(i_1)}\|_2\|\x^{(p_1)}\|_2(\y^{(p_2)})^T\y^{(i_2)} \rangle_{\gamma_{1}^{(1)}}\nonumber \\
 &   &
  -s\beta^2\p_1\m_1\mE_{G,{\mathcal U}_2} \langle \|\x^{(i_1)}\|_2\|\x^{(p_1)}\|_2(\y^{(p_2)})^T\y^{(i_2)} \rangle_{\gamma_{2}^{(1)}}.
\end{eqnarray}

%%%%%%%%%%%%%%%%%%%%%%%%%%%%%%%%%%%%%%%%%%%%%%%%%%%%%%%%%%%%%%%%%%%%%%%%
\underline{\textbf{\emph{Determining}} $T_{2,2}$}
\label{sec:hand1T22}
%%%%%%%%%%%%%%%%%%%%%%%%%%%%%%%%%%%%%%%%%%%%%%%%%%%%%%%%%%%%%%%%%%%%%%%%

As usual, we proceed by relying on the Gaussian integration by parts to obtain
\begin{eqnarray}\label{eq:liftgenEanal20}
T_{2,2} & = &  \mE_{G,{\mathcal U}_2} \lp \frac{1}{\mE_{{\mathcal U}_1} Z^{\m_1}}\mE_{{\mathcal U}_1}\frac{(C^{(i_1)})^{s-1} A^{(i_1,i_2)} \u^{(i_1,3,2)}}{Z^{1-\m_1}} \rp \nonumber \\
& = &  \mE_{G,{\mathcal U}_2,{\mathcal U}_1} \lp\frac{1}{\mE_{{\mathcal U}_1} Z^{\m_1}} \frac{(C^{(i_1)})^{s-1} A^{(i_1,i_2)} \u^{(i_1,3,2)}}{Z^{1-\m_1}}\rp\nonumber \\
& = & \mE_{G,{\mathcal U}_1} \lp
\mE_{{\mathcal U}_2} \lp \sum_{p_1=1}^{l}\mE_{{\mathcal U}_2}(\u^{(i_1,3,2)}\u^{(p_1,3,2)}) \frac{d}{d\u^{(p_1,3,2)}}\lp\frac{(C^{(i_1)})^{s-1} A^{(i_1,i_2)}}{Z^{1-\m_1}\mE_{{\mathcal U}_1} Z^{\m_1}}\rp\rp\rp \nonumber \\
& = & \mE_{G,{\mathcal U}_2,{\mathcal U}_1} \lp
 \frac{1}{\mE_{{\mathcal U}_1} Z^{\m_1}} \sum_{p_1=1}^{l}\mE_{{\mathcal U}_2}(\u^{(i_1,3,2)}\u^{(p_1,3,2)}) \frac{d}{d\u^{(p_1,3,2)}}\lp\frac{(C^{(i_1)})^{s-1} A^{(i_1,i_2)}}{Z^{1-\m_1}}\rp\rp \nonumber \\
& & + \mE_{G,{\mathcal U}_2,{\mathcal U}_1} \lp \frac{(C^{(i_1)})^{s-1} A^{(i_1,i_2)}}{Z^{1-\m_1}}
  \sum_{p_1=1}^{l}\mE_{{\mathcal U}_2}(\u^{(i_1,3,2)}\u^{(p_1,3,2)}) \frac{d}{d\u^{(p_1,3,2)}}\lp\frac{1}{\mE_{{\mathcal U}_1} Z^{\m_1}}\rp\rp.
 \end{eqnarray}
Later on, it will turn out as useful to write the above as
\begin{eqnarray}\label{eq:genEanal19a}
T_{2,2}   =   T_{2,2}^{c} +  T_{2,2}^{d},
\end{eqnarray}
where
\begin{eqnarray}\label{eq:genEanal19b}
T_{2,2}^c &  = &
\mE_{G,{\mathcal U}_2,{\mathcal U}_1} \lp
 \frac{1}{\mE_{{\mathcal U}_1} Z^{\m_1}} \sum_{p_1=1}^{l}\mE_{{\mathcal U}_2}(\u^{(i_1,3,2)}\u^{(p_1,3,2)}) \frac{d}{d\u^{(p_1,3,2)}}\lp\frac{(C^{(i_1)})^{s-1} A^{(i_1,i_2)}}{Z^{1-\m_1}}\rp\rp \nonumber \\
T_{2,2}^d &  = & \mE_{G,{\mathcal U}_2,{\mathcal U}_1} \lp \frac{(C^{(i_1)})^{s-1} A^{(i_1,i_2)}}{Z^{1-\m_1}}
  \sum_{p_1=1}^{l}\mE_{{\mathcal U}_2}(\u^{(i_1,3,2)}\u^{(p_1,3,2)}) \frac{d}{d\u^{(p_1,3,2)}}\lp\frac{1}{\mE_{{\mathcal U}_1} Z^{\m_1}}\rp\rp.
\end{eqnarray}
Since
\begin{eqnarray}\label{eq:genEanal19c}
\mE_{{\mathcal U}_2}(\u^{(i_1,3,2)}\u^{(p_1,3,2)}) & = & \q_1\frac{(\x^{(i_1)})^T\x^{(p_1)}}{\|\x^{(i_1)}\|_2\|\x^{(p_1)}\|_2}
=\q_1\frac{\beta^2(\x^{(i_1)})^T\x^{(p_1)}}{\beta_{i_1}\beta_{p_1}}, \nonumber \\
\mE_{{\mathcal U}_2}(\u^{(i_1,3,1)}\u^{(p_1,3,1)}) & = & (\q_0-\q_1)\frac{(\x^{(i_1)})^T\x^{(p_1)}}{\|\x^{(i_1)}\|_2\|\x^{(p_1)}\|_2}=(\q_0-\q_1)\frac{\beta^2(\x^{(i_1)})^T\x^{(p_1)}}{\beta_{i_1}\beta_{p_1}},
\end{eqnarray}
one observes that $T_{2,2}^c$ scaled by $\q_1$ is structurally identical to the term considered in the second part of Section \ref{sec:hand1T11} scaled by $(\q_0-\q_1)$. Utilizing (\ref{eq:liftgenBanal20b}) we then have
 \begin{eqnarray}\label{eq:genEanal19c1}
\sum_{i_1=1}^{l}\sum_{i_2=1}^{l} \beta_{i_1}\|\y^{(i_2)}\|_2 \frac{T_{2,2}^c}{\sqrt{1-t}} & = & \q_1 \beta^2 \Bigg( \Bigg. \mE_{G,{\mathcal U}_2}\langle \|\x^{(i_1)}\|_2^2\|\y^{(i_2)}\|_2^2\rangle_{\gamma_{01}^{(1)}} \nonumber \\
& & +   (s-1)\mE_{G,{\mathcal U}_2}\langle \|\x^{(i_1)}\|_2^2 \|\y^{(i_2)}\|_2\|\y^{(p_2)}\|_2\rangle_{\gamma_{02}^{(1)}}\Bigg.\Bigg)   \nonumber \\
& & - \q_1s\beta^2(1-\m_1)\mE_{G,{\mathcal U}_2}\langle (\x^{(p_1)})^T\x^{(i_1)}\|\y^{(i_2)}\|_2\|\y^{(p_2)}\|_2 \rangle_{\gamma_{1}^{(1)}}.
\end{eqnarray}

To determine  $T_{2,2}^d$, we start with
\begin{eqnarray}\label{eq:genEanal20}
\frac{d}{d\u^{(p_1,3,2)}}\lp\frac{1}{\mE_{{\mathcal U}_1} Z^{\m_1}}\rp  & = &
 -\frac{1}{\lp\mE_{{\mathcal U}_1} Z^{\m_1} \rp^2}
\mE_{{\mathcal U}_1}\frac{dZ^{\m_1}}{d\u^{(p_1,3)}}.
\end{eqnarray}
Moreover,
\begin{eqnarray}\label{eq:genEanal21}
\frac{dZ^{\m_1}}{d\u^{(p_1,3,2)}} & = & \frac{\m_1}{Z^{1-\m_1}} \frac{dZ}{d\u^{(p_1,3,2)}} = \frac{\m_1}{Z^{1-\m_1}}\frac{d\sum_{q_1=1}^{l}  (C^{(q_1)})^s}{d\u^{(p_1,3,2)}} = \frac{\m_1}{Z^{1-\m_1}} s  (C^{(p_1)})^{s-1}\sum_{p_2=1}^{l}\frac{d(A^{(p_1,p_2)})}{d\u^{(p_1,3,2)}} \nonumber \\
& = & \frac{\m_1}{Z^{1-\m_1}} s  (C^{(p_1)})^{s-1}\sum_{p_2=1}^{l}
\beta_{p_1}A^{(p_1,p_2)}\|\y^{(p_2)}\|_2\sqrt{1-t}.
\end{eqnarray}
Combining (\ref{eq:genEanal20}) and (\ref{eq:genEanal21}), we obtain
\begin{eqnarray}\label{eq:genEanal22}
\frac{d}{d\u^{(p_1,3,2)}}\lp\frac{1}{\mE_{{\mathcal U}_1} Z^{\m_1}}\rp  & = &
 -\frac{1}{\lp\mE_{{\mathcal U}_1} Z^{\m_1} \rp^2}
 \frac{\m_1}{Z^{1-\m_1}} s  (C^{(p_1)})^{s-1}\sum_{p_2=1}^{l}
\beta_{p_1}A^{(p_1,p_2)}\|\y^{(p_2)}\|_2\sqrt{1-t}.
\end{eqnarray}
A further combination of (\ref{eq:genEanal19b}) and (\ref{eq:genEanal22}) gives
\begin{eqnarray}\label{eq:genEanal23}
 T_{2,2}^d &  = & -\mE_{G,{\mathcal U}_2,{\mathcal U}_1} \Bigg( \Bigg. \frac{(C^{(i_1)})^{s-1} A^{(i_1,i_2)}}{Z^{1-\m_1}} \nonumber \\
 & & \times
  \sum_{p_1=1}^{l}  \q_1\frac{\beta^2(\x^{(i_1)})^T\x^{(p_1)}}{\beta_{i_1}\beta_{p_1}}
\frac{1}{\lp\mE_{{\mathcal U}_1} Z^{\m_1} \rp^2}
 \frac{\m_1}{Z^{1-\m_1}} s  (C^{(p_1)})^{s-1}\sum_{p_2=1}^{l}
\beta_{p_1}A^{(p_1,p_2)}\|\y^{(p_2)}\|_2\sqrt{1-t}  \Bigg. \Bigg) \nonumber \\
&  = & -s\sqrt{1-t}\beta^2\q_1\m_1\mE_{G,{\mathcal U}_2,{\mathcal U}_1} \Bigg( \Bigg. \frac{Z^{\m_1}}{\mE_{{\mathcal U}_1} Z^{\m_1}}
 \frac{(C^{(i_1)})^{s-1} A^{(i_1,i_2)}}{Z} \nonumber \\
 & & \times
\frac{Z^{\m_1}}{\mE_{{\mathcal U}_1} Z^{\m_1}}
  \sum_{p_1=1}^{l}
 \frac{(C^{(p_1)})^{s}}{Z}  \sum_{p_2=1}^{l}
\frac{A^{(p_1,p_2)}}{(C^{(p_1)})}\|\y^{(p_2)}\|_2\frac{(\x^{(i_1)})^T\x^{(p_1)}}{\beta_{i_1}} \Bigg. \Bigg).
\end{eqnarray}
Similarly to what was done above when considering $T_{2,1,j}^d$, we here note
\begin{eqnarray}\label{eq:genEanal24}
\sum_{i_1=1}^{l}\sum_{i_2=1}^{l} \beta_{i_1}\|\y^{(i_2)}\|_2\frac{T_{2,2}^d}{\sqrt{1-t}} &  = &   -s\beta^2\q_1\m_1\mE_{G,{\mathcal U}_2,{\mathcal U}_1} \Bigg( \Bigg. \frac{Z^{\m_1}}{\mE_{{\mathcal U}_1} Z^{\m_1}}
\sum_{i_1=1}^{l} \frac{(C^{(i_1)})^s}{Z} \sum_{i_2=1}^{l} \frac{A^{(i_1,i_2)}}{(C^{(i_1)})} \nonumber \\
 & & \times
\frac{Z^{\m_1}}{\mE_{{\mathcal U}_1} Z^{\m_1}}
  \sum_{p_1=1}^{l}
 \frac{(C^{(p_1)})^{s}}{Z}  \sum_{p_2=1}^{l}
\frac{A^{(p_1,p_2)}}{(C^{(p_1)})} \|\y^{(i_2)}\|_2\|\y^{(p_2)}\|_2(\x^{(i_1)})^T\x^{(p_1)} \Bigg. \Bigg) \nonumber \\
& = & -s\beta^2\q_1\m_1\mE_{G,{\mathcal U}_2} \langle \|\y^{(i_2)}\|_2\|\y^{(p_2)}\|_2(\x^{(i_1)})^T\x^{(p_1)}\rangle_{\gamma_{2}^{(1)}}.
\end{eqnarray}
A combination of (\ref{eq:genEanal19a}), (\ref{eq:genEanal19c1}), and
(\ref{eq:genEanal24}) gives
\begin{eqnarray}\label{eq:genEanal25}
\sum_{i_1=1}^{l}\sum_{i_2=1}^{l} \beta_{i_1}\|\y^{(i_2)}\|_2\frac{T_{2,2}}{\sqrt{1-t}} &  = &
\q_1\beta^2 \Bigg( \Bigg. \mE_{G,{\mathcal U}_2}\langle \|\x^{(i_1)}\|_2^2\|\y^{(i_2)}\|_2^2\rangle_{\gamma_{01}^{(1)}} \nonumber \\
& & +  (s-1)\mE_{G,{\mathcal U}_2}\langle \|\x^{(i_1)}\|_2^2 \|\y^{(i_2)}\|_2\|\y^{(p_2)}\|_2\rangle_{\gamma_{02}^{(1)}}\Bigg.\Bigg) \nonumber \\
& & - \q_1s\beta^2(1-\m_1)\mE_{G,{\mathcal U}_2}\langle (\x^{(p_1)})^T\x^{(i_1)}\|\y^{(i_2)}\|_2\|\y^{(p_2)}\|_2 \rangle_{\gamma_{1}^{(1)}} \nonumber \\
&  & -s\beta^2\q_1\m_1\mE_{G,{\mathcal U}_2} \langle \|\y^{(i_2)}\|_2\|\y^{(p_2)}\|_2(\x^{(i_1)})^T\x^{(p_1)}\rangle_{\gamma_{2}^{(1)}}.
\end{eqnarray}

%%%%%%%%%%%%%%%%%%%%%%%%%%%%%%%%%%%%%%%%%%%%%%%%%%%%%%%%%%%%%%%%%%%%%%%%
\underline{\textbf{\emph{Determining}} $T_{2,3}$}
\label{sec:hand1T23}
%%%%%%%%%%%%%%%%%%%%%%%%%%%%%%%%%%%%%%%%%%%%%%%%%%%%%%%%%%%%%%%%%%%%%%%%

Gaussian integration by parts also gives
\begin{eqnarray}\label{eq:genFanal21}
T_{2,3} & = &  \mE_{G,{\mathcal U}_2}\lp \frac{1}{\mE_{{\mathcal U}_1} Z^{\m_1}}\mE_{{\mathcal U}_1}\frac{(C^{(i_1)})^{s-1} A^{(i_1,i_2)} u^{(4,2)}}{Z^{1-\m_1}} \rp \nonumber \\
& = & \mE_{G,{\mathcal U}_1} \lp  \mE_{{\mathcal U}_2} \lp\mE_{{\mathcal U}_2} (u^{(4,2)}u^{(4,2)})\lp\frac{d}{du^{(4,2)}} \lp\frac{(C^{(i_1)})^{s-1} A^{(i_1,i_2)}}{Z^{1-\m_1}\mE_{{\mathcal U}_1} Z^{\m_1}}\rp \rp\rp\rp \nonumber \\
& = & \p_1\q_1\mE_{G,{\mathcal U}_2,{\mathcal U}_1} \lp \frac{1}{\mE_{{\mathcal U}_1} Z^{\m_1}} \lp\frac{d}{du^{(4,2)}} \lp\frac{(C^{(i_1)})^{s-1} A^{(i_1,i_2)}}{Z^{1-\m_1}}\rp\rp\rp \nonumber \\
& & + \p_1\q_1\mE_{G,{\mathcal U}_2,{\mathcal U}_1} \lp \frac{(C^{(i_1)})^{s-1} A^{(i_1,i_2)}}{Z^{1-\m_1}}\lp\frac{d}{du^{(4,2)}} \lp\frac{1}{\mE_{{\mathcal U}_1} Z^{\m_1}} \rp\rp\rp.
\end{eqnarray}
As usual, we find it useful to rewrite the above as
\begin{eqnarray}\label{eq:genFanal22}
T_{2,3} & = & T_{2,3}^c+T_{2,3}^d,
\end{eqnarray}
where
\begin{eqnarray}\label{eq:genFanal23}
T_{2,3}^c & = &   \p_1\q_1\mE_{G,{\mathcal U}_2,{\mathcal U}_1} \lp \frac{1}{\mE_{{\mathcal U}_1} Z^{\m_1}} \lp\frac{d}{du^{(4,2)}} \lp\frac{(C^{(i_1)})^{s-1} A^{(i_1,i_2)}}{Z^{1-\m_1}}\rp\rp\rp \nonumber \\
 T_{2,3}^d & = & \p_1\q_1\mE_{G,{\mathcal U}_2,{\mathcal U}_1} \lp \frac{(C^{(i_1)})^{s-1} A^{(i_1,i_2)}}{Z^{1-\m_1}}\lp\frac{d}{du^{(4,2)}} \lp\frac{1}{\mE_{{\mathcal U}_1} Z^{\m_1}} \rp\rp\rp.
\end{eqnarray}
It is again not that difficult to observe that $\frac{T_{2,3}^c}{\p_1\q_1}=\frac{T_{1,3}}{\p_0\q_0-\p_1\q_1}$, which together with (\ref{eq:liftgenCanal21b}) gives
 \begin{eqnarray}\label{eq:genFanal23b}
\sum_{i_1=1}^{l}\sum_{i_2=1}^{l} \beta_{i_1}\|\y^{(i_2)}\|_2 \frac{T_{2,3}^c}{\sqrt{t}} & = &
\p_1\q_1 \beta^2 \Bigg( \Bigg. \mE_{G,{\mathcal U}_2}\langle \|\x^{(i_1)}\|_2^2\|\y^{(i_2)}\|_2^2\rangle_{\gamma_{01}^{(1)}} \nonumber \\
& & \times +   (s-1)\mE_{G,{\mathcal U}_2}\langle \|\x^{(i_1)}\|_2^2 \|\y^{(i_2)}\|_2\|\y^{(p_2)}\|_2\rangle_{\gamma_{02}^{(1)}}\Bigg. \Bigg) \nonumber \\
& & - \p_1\q_1 s\beta^2(1-\m_1)\mE_{G,{\mathcal U}_2}\langle \|\x^{(i_1)}\|_2\|\x^{(p_`)}\|_2\|\y^{(i_2)}\|_2\|\y^{(p_2)}\|_2 \rangle_{\gamma_{1}^{(1)}}.
\end{eqnarray}

To find $T_{2,3}^d$, we start by writing
\begin{eqnarray}\label{eq:genFanal24}
\frac{d}{du^{(4,2)}}\lp \frac{1}{\mE_{{\mathcal U}_1} Z^{\m_1}} \rp  & = &
 -\frac{1}{\lp\mE_{{\mathcal U}_1} Z^{\m_1}\rp^2} \mE_{{\mathcal U}_1} \frac{dZ^{\m_1}}{d u^{(4,2)}}=-\frac{1}{\lp\mE_{{\mathcal U}_1} Z^{\m_1}\rp^2} \mE_{{\mathcal U}_1} \frac{\m_1}{Z^{1-\m_1}}\frac{dZ}{d u^{(4,2)}}.\nonumber \\
\end{eqnarray}
We also have
\begin{equation}\label{eq:genFanal25}
\frac{dZ}{du^{(4,2)}}=\frac{d\sum_{p_1=1}^{l}  (C^{(p_1)})^s}{du^{(4,2)}}
=s \sum_{p_1=1}^{l} (C^{(p_1)})^{s-1}\sum_{p_2=1}^{l}\frac{d(A^{(p_1,p_2)})}{du^{(4,2)}}=s \sum_{p_1=1}^{l} (C^{(p_1)})^{s-1}\sum_{p_2=1}^{l}
\beta_{p_1}A^{(p_1,p_2)}\|\y^{(p_2)}\|_2\sqrt{t}.
\end{equation}
Connecting (\ref{eq:genFanal24}) and (\ref{eq:genFanal25}), we find
\begin{eqnarray}\label{eq:genFanal26}
\frac{d}{du^{(4,2)}}\lp \frac{1}{\mE_{{\mathcal U}_1} Z^{\m_1}} \rp   =-\frac{1}{\lp\mE_{{\mathcal U}_1} Z^{\m_1}\rp^2} \mE_{{\mathcal U}_1} \frac{\m_1}{Z^{1-\m_1}}s \sum_{p_1=1}^{l} (C^{(p_1)})^{s-1}\sum_{p_2=1}^{l}
\beta_{p_1}A^{(p_1,p_2)}\|\y^{(p_2)}\|_2\sqrt{t}.\nonumber \\
\end{eqnarray}
Further connection between (\ref{eq:genFanal23}) and (\ref{eq:genFanal26}) gives
\begin{eqnarray}\label{eq:genFanal27}
T_{2,3}^d   & = & -\p_1\q_1\mE_{G,{\mathcal U}_2,{\mathcal U}_1} \Bigg( \Bigg.\frac{(C^{(i_1)})^{s-1} A^{(i_1,i_2)}}{Z^{1-\m_1}}\nonumber \\
& & \times\frac{1}{\lp\mE_{{\mathcal U}_1} Z^{\m_1}\rp^2} \mE_{{\mathcal U}_1} \frac{\m_1}{Z^{1-\m_1}}s \sum_{p_1=1}^{l} (C^{(p_1)})^{s-1}\sum_{p_2=1}^{l}
\beta_{p_1}A^{(p_1,p_2)}\|\y^{(p_2)}\|_2\sqrt{t}\Bigg. \Bigg) \nonumber \\
& = & -s\sqrt{t}\p_1\q_1\m_1\mE_{G,{\mathcal U}_2} \Bigg( \Bigg. \mE_{{\mathcal U}_1} \frac{Z^{\m_1}}{\mE_{{\mathcal U}_1} Z^{\m_1}} \frac{(C^{(i_1)})^s}{Z}\frac{A^{(i_1,i_2)}}{(C^{(i_1)})}\nonumber \\
& & \times \mE_{{\mathcal U}_1} \frac{Z^{\m_1}}{\mE_{{\mathcal U}_1} Z^{\m_1}} \sum_{p_1=1}^{l} \frac{(C^{(p_1)})^{s}}{Z}\sum_{p_2=1}^{l}
\frac{A^{(p_1,p_2)}}{(C^{(p_1)})}\beta_{p_1}\|\y^{(p_2)}\|_2\Bigg. \Bigg).
\end{eqnarray}
As earlier, we find it useful to note
\begin{eqnarray}\label{eq:genFanal28}
\sum_{i_1=1}^{l}\sum_{i_2=1}^{l} \beta_{i_1}\|\y^{(i_2)}\|_2\frac{T_{2,3}^d}{\sqrt{t}}
& = & -s\beta^2\p_1\q_1\m_1\mE_{G,{\mathcal U}_2} \Bigg( \Bigg. \mE_{{\mathcal U}_1} \frac{Z^{\m_1}}{\mE_{{\mathcal U}_1} Z^{\m_1}} \sum_{i_1=1}^{l} \frac{(C^{(i_1)})^s}{Z} \sum_{i_2=1}^{l} \frac{A^{(i_1,i_2)}}{(C^{(i_1)})}\nonumber \\
& & \times \mE_{{\mathcal U}_1} \frac{Z^{\m_1}}{\mE_{{\mathcal U}_1} Z^{\m_1}} \sum_{p_1=1}^{l} \frac{(C^{(p_1)})^{s}}{Z}\sum_{p_2=1}^{l}
\frac{A^{(p_1,p_2)}}{(C^{(p_1)})}\|\x^{(i_2)}\|_2\|\x^{(p_2)}\|_2\|\y^{(i_2)}\|_2\|\y^{(p_2)}\|_2\Bigg. \Bigg)\nonumber \\
& = & -s\beta^2\p_1\q_1\m_1\mE_{G,{\mathcal U}_2} \langle\|\x^{(i_2)}\|_2\|\x^{(p_2)}\|_2\|\y^{(i_2)}\|_2\|\y^{(p_2)}\rangle_{\gamma_{2}^{(1)}}.
\end{eqnarray}
 A combination of (\ref{eq:genFanal22}), (\ref{eq:genFanal23b}), and (\ref{eq:genFanal28}) gives
\begin{eqnarray}\label{eq:genFanal29}
\sum_{i_1=1}^{l}\sum_{i_2=1}^{l} \beta_{i_1}\|\y^{(i_2)}\|_2\frac{T_{2,3}}{\sqrt{t}}
& = &
\p_1\q_1\beta^2 \Bigg( \Bigg. \mE_{G,{\mathcal U}_2}\langle \|\x^{(i_1)}\|_2^2\|\y^{(i_2)}\|_2^2\rangle_{\gamma_{01}^{(1)}} \nonumber \\
& & +   (s-1)\mE_{G,{\mathcal U}_2}\langle \|\x^{(i_1)}\|_2^2 \|\y^{(i_2)}\|_2\|\y^{(p_2)}\|_2\rangle_{\gamma_{02}^{(1)}}\Bigg.\Bigg) \nonumber \\
& & - \p_1\q_1 s\beta^2(1-\m_1)\mE_{G,{\mathcal U}_2}\langle \|\x^{(i_1)}\|_2\|\x^{(p_`)}\|_2\|\y^{(i_2)}\|_2\|\y^{(p_2)}\|_2 \rangle_{\gamma_{1}^{(1)}} \nonumber \\
&  & -s\beta^2\p_1\q_1\m_1\mE_{G,{\mathcal U}_2} \langle\|\x^{(i_2)}\|_2\|\x^{(p_2)}\|_2\|\y^{(i_2)}\|_2\|\y^{(p_2)}\rangle_{\gamma_{2}^{(1)}}.
\end{eqnarray}

%%%%%%%%%%%%%%%%%%%%%%%%%%%%%%%%%%%%%%%%%%%%%%%%%%%%%%%%%%%%%%%%%%%%%%%%
%%%%%%%%%%%%%%%%%%%%%%%%%%%%%%%%%%%%%%%%%%%%%%%%%%%%%%%%%%%%%%%%%%%%%%%%
%%%%%%%%%%%%%%%%%%%%%%%%%%%%%%%%%%%%%%%%%%%%%%%%%%%%%%%%%%%%%%%%%%%%%%%%
%%%%%%%%%%%%%%%%%%%%%%%%%%%%%%%%%%%%%%%%%%%%%%%%%%%%%%%%%%%%%%%%%%%%%%%%
\subsubsection{Handling $T_G$--group}
\label{sec:handlTG}
%%%%%%%%%%%%%%%%%%%%%%%%%%%%%%%%%%%%%%%%%%%%%%%%%%%%%%%%%%%%%%%%%%%%%%%%
%%%%%%%%%%%%%%%%%%%%%%%%%%%%%%%%%%%%%%%%%%%%%%%%%%%%%%%%%%%%%%%%%%%%%%%%
%%%%%%%%%%%%%%%%%%%%%%%%%%%%%%%%%%%%%%%%%%%%%%%%%%%%%%%%%%%%%%%%%%%%%%%%
%%%%%%%%%%%%%%%%%%%%%%%%%%%%%%%%%%%%%%%%%%%%%%%%%%%%%%%%%%%%%%%%%%%%%%%%

We start by recalling on
\begin{eqnarray}\label{eq:genGanal1}
T_{G,j}  =   \mE_{G,{\mathcal U}_2} \lp\frac{1}{\mE_{{\mathcal U}_1} Z^{\m_1}}\mE_{{\mathcal U}_1}\frac{(C^{(i_1)})^{s-1} A^{(i_1,i_2)} \y_j^{(i_2)}\u_j^{(i_1,1)}}{Z^{1-\m_1}} \rp =   \mE_{{\mathcal U}_2,{\mathcal U}_1} \lp \mE_{G} \frac{(C^{(i_1)})^{s-1} A^{(i_1,i_2)} \y_j^{(i_2)}\u_j^{(i_1,1)}}{Z^{1-\m_1}\mE_{{\mathcal U}_1} Z^{\m_1}} \rp.
 \end{eqnarray}
Gaussian integration by parts gives
\begin{eqnarray}\label{eq:genGanal2}
T_{G,j}  & = &   \mE_{{\mathcal U}_2,{\mathcal U}_1} \lp \mE_{G} \frac{(C^{(i_1)})^{s-1} A^{(i_1,i_2)} \y_j^{(i_2)}\u_j^{(i_1,1)}}{Z^{1-\m_1}\mE_{{\mathcal U}_1} Z^{\m_1}} \rp \nonumber \\
& = &
 \mE_{{\mathcal U}_2,{\mathcal U}_1} \Bigg( \Bigg. \mE_{G} \lp \frac{1}{\mE_{{\mathcal U}_1} Z^{\m_1}} \sum_{p_1=1}^{l} \mE (\u_j^{(i_1,1)}\u_j^{(p_1,1)})\frac{d}{d\u_j^{(p_1,1)}}\lp \frac{(C^{(i_1)})^{s-1} A^{(i_1,i_2)}\y_j^{(i_2)}}{Z^{1-\m_1}}\rp \rp \nonumber \\
 & & + \mE_{G} \lp   \frac{(C^{(i_1)})^{s-1} A^{(i_1,i_2)}\y_j^{(i_2)}}{Z^{1-\m_1}} \sum_{p_1=1}^{l} \mE (\u_j^{(i_1,1)}\u_j^{(p_1,1)})\frac{d}{d\u_j^{(p_1,1)}}\lp   \frac{1}{\mE_{{\mathcal U}_1} Z^{\m_1}} \rp \rp\Bigg.\Bigg).
 \end{eqnarray}
After setting
\begin{eqnarray}\label{eq:genGanal3}
\Theta_{G,1} & = & \sum_{p_1=1}^{l} \mE (\u_j^{(i_1,1)}\u_j^{(p_1,1)})\frac{d}{d\u_j^{(p_1,1)}}\lp \frac{(C^{(i_1)})^{s-1} A^{(i_1,i_2)}\y_j^{(i_2)}}{Z^{1-\m_1}}\rp \nonumber \\
\Theta_{G,2} & = &   \lp   \frac{(C^{(i_1)})^{s-1} A^{(i_1,i_2)}\y_j^{(i_2)}}{Z^{1-\m_1}} \sum_{p_1=1}^{l} \mE (\u_j^{(i_1,1)}\u_j^{(p_1,1)})\frac{d}{d\u_j^{(p_1,1)}}\lp   \frac{1}{\mE_{{\mathcal U}_1} Z^{\m_1}} \rp \rp \nonumber \\
T_{G,j}^c & = &  \mE_{G,{\mathcal U}_2}\lp \mE_{{\mathcal U}_1}\frac{1}{\mE_{{\mathcal U}_1} Z^{\m_1}} \Theta_{G,1}\rp \nonumber \\
T_{G,j}^d & = &  \mE_{G,{\mathcal U}_2}\lp \mE_{{\mathcal U}_1}\Theta_{G,2}\rp,
 \end{eqnarray}
one has
\begin{eqnarray}\label{eq:genGanal4}
T_{G,j}  & = &  T_{G,j}^c+T_{G,j}^d.
 \end{eqnarray}
One now recognizes that $\Theta_{G,1}$ is structurally closely related to the one considered in Section 2.1.1 of \cite{Stojnicgscompyx16}. The only structural difference is the power of $Z$. Instead of $1$ we now have $1-\m_1$. Such a difference doesn't change the analysis very much. In fact, these results are used to handle the structure considered in the first part of Section 3.1 of \cite{Stojnicgscompyx16}, which is pretty much identical to the one that we have here. Following \cite{Stojnicgscompyx16}'s Section 3.1, we then have
\begin{eqnarray}\label{eq:genGanal5}
\Theta_{G,1} & = & \lp \frac{\y_j^{(i_2)}}{Z^{1-\m_1}}\lp(C^{(i_1)})^{s-1}\beta_{i_1}A^{(i_1,i_2)}\y_j^{(i_2)}\sqrt{t}+(s-1)(C^{(i_1)})^{s-2}\beta_{i_1}\sum_{p_2=1}^{l}A^{(i_1,p_2)}\y_j^{(p_2)}\sqrt{t}\rp \rp \nonumber \\
& &  -(1-\m_1)
\mE \lp\sum_{p_1=1}^{l} \frac{(\x^{(i_1)})^T\x^{(p_1)}}{\|\x^{(i_1)}\|_2\|\x^{(p_1)}\|_2}
\frac{(C^{(i_1)})^{s-1} A^{(i_1,i_2)}\y_j^{(i_2)}}{Z^{2-\m_1}}
s  (C^{(p_1)})^{s-1}\sum_{p_2=1}^{l}\beta_{p_1}A^{(p_1,p_2)}\y_j^{(p_2)}\sqrt{t}\rp.\nonumber \\
\end{eqnarray}
We can then conveniently write
\begin{eqnarray}\label{eq:genGanal6}
 \sum_{i_1=1}^{l} \sum_{i_2=1}^{l}\sum_{j=1}^{m}\beta_{i_1} \frac{T_{G,j}^c}{\sqrt{t}}
 & = & \sum_{i_1=1}^{l} \sum_{i_2=1}^{l}\sum_{j=1}^{m}\beta_{i_1} \frac{\mE_{G,{\mathcal U}_2}\lp \mE_{{\mathcal U}_1}\frac{1}{\mE_{{\mathcal U}_1} Z^{\m_1}} \Theta_{G,1}\rp}{\sqrt{t}} \nonumber\\
 & = & \beta^2 \lp \mE_{G,{\mathcal U}_2} \langle \|\x^{(i_1)}\|_2^2\|\y^{(i_2)}\|_2^2\rangle_{\gamma_{01}^{(1)}} + \mE_{G,{\mathcal U}_2}\langle \|\x^{(i_1)}\|_2^2(\y^{(p_2)})^T\y^{(i_2)}\rangle_{\gamma_{02}^{(1)}}    \rp   \nonumber \\
 &  & -s\beta^2(1-\m_1) \mE_{G,{\mathcal U}_2}\langle (\x^{(p_1)})^T\x^{(i_1)}(\y^{(p_2)})^T\y^{(i_2)}\rangle_{\gamma_{1}^{(1)}},
 \end{eqnarray}
where the first two terms are obtained directly from (\ref{eq:liftgenAanal19g}).

To determine $T_{G,j}^d$, we start by writing
\begin{eqnarray}\label{eq:genGanal7}
\frac{d}{d\u_j^{(p_1,1)}}\lp   \frac{1}{\mE_{{\mathcal U}_1} Z^{\m_1}} \rp
& = & -\frac{1}{\lp\mE_{{\mathcal U}_1} Z^{\m_1}\rp^2}\mE_{{\mathcal U}_1}  \frac{dZ^{\m_1}}{d\u_j^{(p_1,1)}}
= -\mE_{{\mathcal U}_1}  \frac{\m_1Z^{\m_1}}{\lp\mE_{{\mathcal U}_1} Z^{\m_1}\rp^2}\frac{1}{Z} \frac{dZ}{d\u_j^{(p_1,1)}}\nonumber\\
& = &  -\mE_{{\mathcal U}_1}  \frac{\m_1Z^{\m_1}}{\lp\mE_{{\mathcal U}_1} Z^{\m_1}\rp^2}\frac{1}{Z} \frac{d\sum_{q_1=1}^l(C^{(q_1)})^s}{d\u_j^{(p_1,1)}} = -\mE_{{\mathcal U}_1}  \frac{\m_1Z^{\m_1}}{\lp\mE_{{\mathcal U}_1} Z^{\m_1}\rp^2}\frac{1}{Z} s (C^{(p_1)})^{s-1}\frac{dC^{(p_1)}}{d\u_j^{(p_1,1)}}\nonumber\\
& = & -\mE_{{\mathcal U}_1}  \frac{\m_1Z^{\m_1}}{\lp\mE_{{\mathcal U}_1} Z^{\m_1}\rp^2}\frac{1}{Z} s (C^{(p_1)})^{s-1}\sum_{p_2=1}^{l}\frac{dA^{(p_1,p_2)}}{d\u_j^{(p_1,1)}}\nonumber\\
& = & -\mE_{{\mathcal U}_1}  \frac{\m_1Z^{\m_1}}{\lp\mE_{{\mathcal U}_1} Z^{\m_1}\rp^2}\frac{1}{Z} s (C^{(p_1)})^{s-1}\sum_{p_2=1}^{l}A^{(p_1,p_2)}\beta_{p_1}\y_j^{(p_2)}\sqrt{t}.
\end{eqnarray}
Combining  the expression for $\Theta_{G,2}$ from (\ref{eq:genGanal3}) with (\ref{eq:genGanal7}), we find
\begin{eqnarray}\label{eq:genGanal8}
 \Theta_{G,2} & = &  - \lp   \frac{(C^{(i_1)})^{s-1} A^{(i_1,i_2)}\y_j^{(i_2)}}{Z^{1-\m_1}} \sum_{p_1=1}^{l} \mE (\u_j^{(i_1,1)}\u_j^{(p_1,1)})\frac{d}{d\u_j^{(p_1,1)}}\lp   \frac{1}{\mE_{{\mathcal U}_1} Z^{\m_1}} \rp \rp \nonumber \\& = &
  \frac{(C^{(i_1)})^{s-1} A^{(i_1,i_2)}\y_j^{(i_2)}}{Z^{1-\m_1}} \nonumber \\
  & & \times \sum_{p_1=1}^{l} \mE (\u_j^{(i_1,1)}\u_j^{(p_1,1)})
 \mE_{{\mathcal U}_1}  \frac{\m_1Z^{\m_1}}{\lp\mE_{{\mathcal U}_1} Z^{\m_1}\rp^2}\frac{1}{Z} s (C^{(p_1)})^{s-1}\sum_{p_2=1}^{l}A^{(p_1,p_2)}\beta_{p_1}\y_j^{(p_2)}\sqrt{t}.
  \end{eqnarray}
Combining $T_{G,j}^d$ from (\ref{eq:genGanal3}) and (\ref{eq:genGanal8}), we obtain
\begin{eqnarray}\label{eq:genGanal9}
 \sum_{i_1=1}^{l} \sum_{i_2=1}^{l}\sum_{j=1}^{m}\beta_{i_1} \frac{T_{G,j}^d}{\sqrt{t}}
 & = & \sum_{i_1=1}^{l} \sum_{i_2=1}^{l}\sum_{j=1}^{m}\beta_{i_1} \frac{\mE_{G,{\mathcal U}_2}\lp \mE_{{\mathcal U}_1}  \Theta_{G,2}\rp}{\sqrt{t}} \nonumber\\
 & = & -\sum_{i_1=1}^{l} \sum_{i_2=1}^{l} \mE_{G,{\mathcal U}_2} \Bigg( \Bigg.  \mE_{{\mathcal U}_1} \frac{(C^{(i_1)})^{s-1} A^{(i_1,i_2)}}{Z^{1-\m_1}} \sum_{p_1=1}^{l} \frac{(\x^{(p_1)})^T\x^{(i_1)}}{\|\x^{(i_1)}\|_2\|\x^{(p_1)}\|_2}  \nonumber \\
 & & \times
 \mE_{{\mathcal U}_1}  \frac{\m_1Z^{\m_1}}{\lp\mE_{{\mathcal U}_1} Z^{\m_1}\rp^2}\frac{1}{Z} s (C^{(p_1)})^{s-1}\sum_{p_2=1}^{l}A^{(p_1,p_2)}\beta_{i_1}\beta_{p_1}(\y^{(p_2)})^T\y^{(i_2)} \Bigg. \Bigg) \nonumber \\
 & = & -s\beta^2\m_1\mE_{G,{\mathcal U}_2}  \Bigg( \Bigg.  \mE_{{\mathcal U}_1}  \frac{Z^{\m_1}}{\mE_{{\mathcal U}_1} Z^{\m_1}} \sum_{p_1=1}^{l}  \frac{C^{(i_1)})^{s}}{Z} \sum_{i_2=1}^{l}\frac{A^{(i_1,i_2)}}{C^{(i_1)}} \nonumber \\
 & & \times
 \mE_{{\mathcal U}_1}  \frac{Z^{\m_1}}{\mE_{{\mathcal U}_1} Z^{\m_1}} \sum_{p_1=1}^{l}  \frac{C^{(p_1)})^{s}}{Z} \sum_{p_2=1}^{l}\frac{A^{(p_1,p_2)}}{C^{(p_1)}} (\x^{(p_1)})^T\x^{(i_1)} (\y^{(p_2)})^T\y^{(i_2)} \Bigg. \Bigg) \nonumber \\
 & = & -s\beta^2\m_1\mE_{G,{\mathcal U}_2} \langle(\x^{(p_1)})^T\x^{(i_1)} (\y^{(p_2)})^T\y^{(i_2)} \rangle_{\gamma_{2}^{(1)}}.
 \end{eqnarray}
A further combination of (\ref{eq:genGanal4}), (\ref{eq:genGanal6}), and (\ref{eq:genGanal9}) gives
\begin{eqnarray}\label{eq:genGanal10}
 \sum_{i_1=1}^{l} \sum_{i_2=1}^{l}\sum_{j=1}^{m}\beta_{i_1} \frac{T_{G,j}}{\sqrt{t}}
  & = & \beta^2 \lp \mE_{G,{\mathcal U}_2} \langle \|\x^{(i_1)}\|_2^2\|\y^{(i_2)}\|_2^2\rangle_{\gamma_{01}^{(1)}} +   (s-1) \mE_{G,{\mathcal U}_2}\langle \|\x^{(i_1)}\|_2^2(\y^{(p_2)})^T\y^{(i_2)}\rangle_{\gamma_{02}^{(1)}}   \rp    \nonumber \\
 &  & -s\beta^2(1-\m_1) \mE_{G,{\mathcal U}_2}\langle (\x^{(p_1)})^T\x^{(i_1)}(\y^{(p_2)})^T\y^{(i_2)}\rangle_{\gamma_{1}^{(1)}} \nonumber \\
 & & -s\beta^2\m_1\mE_{G,{\mathcal U}_2} \langle(\x^{(p_1)})^T\x^{(i_1)} (\y^{(p_2)})^T\y^{(i_2)} \rangle_{\gamma_{2}^{(1)}}.
 \end{eqnarray}

%%%%%%%%%%%%%%%%%%%%%%%%%%%%%%%%%%%%%%%%%%%%%%%%%%%%%%%%%%%%%%%%%%%%%%%%
\subsubsection{Connecting all pieces together}
\label{sec:conalt}
%%%%%%%%%%%%%%%%%%%%%%%%%%%%%%%%%%%%%%%%%%%%%%%%%%%%%%%%%%%%%%%%%%%%%%%%

We now connect together all the pieces obtained above. To do so, we utilize (\ref{eq:genanal10e}) and (\ref{eq:genanal10f}) to write
\begin{eqnarray}\label{eq:ctp1}
\frac{d\psi(\calX,\calY,\q,\m,\beta,s,t)}{dt}  & = &       \frac{\mbox{sign}(s)}{2\beta\sqrt{n}} \lp \Omega_G+\Omega_1+\Omega_2+\Omega_3\rp,
\end{eqnarray}
where
\begin{eqnarray}\label{eq:ctp2}
\Omega_G & = & \sum_{i_1=1}^{l}  \sum_{i_2=1}^{l}\sum_{j=1}^{m} \beta_{i_1}\frac{T_{G,j}}{\sqrt{t}}  \nonumber\\
\Omega_1 & = & -\sum_{i_1=1}^{l}  \sum_{i_2=1}^{l} \sum_{j=1}^{m}\beta_{i_1}\frac{T_{2,1,j}}{\sqrt{1-t}}-\sum_{i_1=1}^{l}  \sum_{i_2=1}^{l} \sum_{j=1}^{m}\beta_{i_1}\frac{T_{1,1,j}}{\sqrt{1-t}} \nonumber\\
\Omega_2 & = & -\sum_{i_1=1}^{l}  \sum_{i_2=1}^{l}\beta_{i_1}\|\y^{(i_2)}\|_2\frac{T_{2,2}}{\sqrt{1-t}}-\sum_{i_1=1}^{l}  \sum_{i_2=1}^{l}\beta_{i_1}\|\y^{(i_2)}\|_2\frac{T_{1,2}}{\sqrt{1-t}} \nonumber\\
\Omega_3 & = & \sum_{i_1=1}^{l}  \sum_{i_2=1}^{l}\beta_{i_1}\|\y^{(i_2)}\|_2\frac{T_{2,3}}{\sqrt{t}}+ \sum_{i_1=1}^{l}  \sum_{i_2=1}^{l}\beta_{i_1}\|\y^{(i_2)}\|_2\frac{T_{1,3}}{\sqrt{t}}.
\end{eqnarray}
From (\ref{eq:genGanal10}) we have
\begin{eqnarray}\label{eq:cpt3}
\Omega_G & = & \beta^2 \lp \mE_{G,{\mathcal U}_2} \langle \|\x^{(i_1)}\|_2^2\|\y^{(i_2)}\|_2^2\rangle_{\gamma_{01}^{(1)}} +  (s-1)\mE_{G,{\mathcal U}_2}\langle \|\x^{(i_1)}\|_2^2(\y^{(p_2)})^T\y^{(i_2)}\rangle_{\gamma_{02}^{(1)}}  \rp     \nonumber \\
 &  & -s\beta^2(1-\m_1) \mE_{G,{\mathcal U}_2}\langle (\x^{(p_1)})^T\x^{(i_1)}(\y^{(p_2)})^T\y^{(i_2)}\rangle_{\gamma_{1}^{(1)}} \nonumber \\
 & & -s\beta^2\m_1\mE_{G,{\mathcal U}_2} \langle(\x^{(p_1)})^T\x^{(i_1)} (\y^{(p_2)})^T\y^{(i_2)} \rangle_{\gamma_{2}^{(1)}}.
 \end{eqnarray}
From (\ref{eq:liftgenAanal19i}) and (\ref{eq:genDanal25}), we have
\begin{eqnarray}\label{eq:cpt4}
-\Omega_1& = & (\p_0-\p_1)\beta^2 \lp \mE_{G,{\mathcal U}_2}\langle \|\x^{(i_1)}\|_2^2\|\y^{(i_2)}\|_2^2\rangle_{\gamma_{01}^{(1)}} +   (s-1)\mE_{G,{\mathcal U}_2}\langle \|\x^{(i_1)}\|_2^2(\y^{(p_2)})^T\y^{(i_2)}\rangle_{\gamma_{02}^{(1)}} \rp \nonumber \\
& & - (\p_0-\p_1)s\beta^2(1-\m_1)\mE_{G,{\mathcal U}_2}\langle \|\x^{(i_1)}\|_2\|\x^{(p_1)}\|_2(\y^{(p_2)})^T\y^{(i_2)} \rangle_{\gamma_{1}^{(1)}} \nonumber\\
& & +\p_1\beta^2 \lp \mE_{G,{\mathcal U}_2}\langle \|\x^{(i_1)}\|_2^2\|\y^{(i_2)}\|_2^2\rangle_{\gamma_{01}^{(1)}} +   (s-1)\mE_{G,{\mathcal U}_2}\langle \|\x^{(i_1)}\|_2^2(\y^{(p_2)})^T\y^{(i_2)}\rangle_{\gamma_{02}^{(1)}} \rp \nonumber \\
& & - \p_1s\beta^2(1-\m_1)\mE_{G,{\mathcal U}_2}\langle \|\x^{(i_1)}\|_2\|\x^{(p_1)}\|_2(\y^{(p_2)})^T\y^{(i_2)} \rangle_{\gamma_{1}^{(1)}}\nonumber \\
 &   &
  -s\beta^2\p_1\m_1\mE_{G,{\mathcal U}_2} \langle \|\x^{(i_1)}\|_2\|\x^{(p_1)}\|_2(\y^{(p_2)})^T\y^{(i_2)} \rangle_{\gamma_{2}^{(1)}} \nonumber \\
& = & \p_0\beta^2 \lp \mE_{G,{\mathcal U}_2}\langle \|\x^{(i_1)}\|_2^2\|\y^{(i_2)}\|_2^2\rangle_{\gamma_{01}^{(1)}} +   (s-1)\mE_{G,{\mathcal U}_2}\langle \|\x^{(i_1)}\|_2^2(\y^{(p_2)})^T\y^{(i_2)}\rangle_{\gamma_{02}^{(1)}} \rp \nonumber \\
& & - \p_0s\beta^2(1-\m_1)\mE_{G,{\mathcal U}_2}\langle \|\x^{(i_1)}\|_2\|\x^{(p_1)}\|_2(\y^{(p_2)})^T\y^{(i_2)} \rangle_{\gamma_{1}^{(1)}}\nonumber \\
 &   &
  -s\beta^2\p_1\m_1\mE_{G,{\mathcal U}_2} \langle \|\x^{(i_1)}\|_2\|\x^{(p_1)}\|_2(\y^{(p_2)})^T\y^{(i_2)} \rangle_{\gamma_{2}^{(1)}}.
\end{eqnarray}
From (\ref{eq:liftgenBanal20b}) and (\ref{eq:genEanal25}), we have
\begin{eqnarray}\label{eq:cpt5}
-\Omega_2 & = & (\q_0-\q_1)\beta^2 \lp\mE_{G,{\mathcal U}_2}\langle \|\x^{(i_1)}\|_2^2\|\y^{(i_2)}\|_2^2\rangle_{\gamma_{01}^{(1)}} +   (s-1)\mE_{G,{\mathcal U}_2}\langle \|\x^{(i_1)}\|_2^2 \|\y^{(i_2)}\|_2\|\y^{(p_2)}\|_2\rangle_{\gamma_{02}^{(1)}}\rp\nonumber \\
& & - (\q_0-\q_1)s\beta^2(1-\m_1)\mE_{G,{\mathcal U}_2}\langle (\x^{(p_1)})^T\x^{(i_1)}\|\y^{(i_2)}\|_2\|\y^{(p_2)}\|_2 \rangle_{\gamma_{1}^{(1)}}\nonumber \\
&   &+
\q_1\beta^2\lp\mE_{G,{\mathcal U}_2}\langle \|\x^{(i_1)}\|_2^2\|\y^{(i_2)}\|_2^2\rangle_{\gamma_{01}^{(1)}} +   (s-1)\mE_{G,{\mathcal U}_2}\langle \|\x^{(i_1)}\|_2^2 \|\y^{(i_2)}\|_2\|\y^{(p_2)}\|_2\rangle_{\gamma_{02}^{(1)}}\rp\nonumber \\
& & - \q_1s\beta^2(1-\m_1)\mE_{G,{\mathcal U}_2}\langle (\x^{(p_1)})^T\x^{(i_1)}\|\y^{(i_2)}\|_2\|\y^{(p_2)}\|_2 \rangle_{\gamma_{1}^{(1)}} \nonumber \\
&  & -s\beta^2\q_1\m_1\mE_{G,{\mathcal U}_2} \langle \|\y^{(i_2)}\|_2\|\y^{(p_2)}\|_2(\x^{(i_1)})^T\x^{(p_1)}\rangle_{\gamma_{2}^{(1)}} \nonumber \\
&  = &
\q_0\beta^2\lp\mE_{G,{\mathcal U}_2}\langle \|\x^{(i_1)}\|_2^2\|\y^{(i_2)}\|_2^2\rangle_{\gamma_{01}^{(1)}} +   (s-1)\mE_{G,{\mathcal U}_2}\langle \|\x^{(i_1)}\|_2^2 \|\y^{(i_2)}\|_2\|\y^{(p_2)}\|_2\rangle_{\gamma_{02}^{(1)}}\rp\nonumber \\
& & - \q_0s\beta^2(1-\m_1)\mE_{G,{\mathcal U}_2}\langle (\x^{(p_1)})^T\x^{(i_1)}\|\y^{(i_2)}\|_2\|\y^{(p_2)}\|_2 \rangle_{\gamma_{1}^{(1)}} \nonumber \\
&  & -s\beta^2\q_1\m_1\mE_{G,{\mathcal U}_2} \langle \|\y^{(i_2)}\|_2\|\y^{(p_2)}\|_2(\x^{(i_1)})^T\x^{(p_1)}\rangle_{\gamma_{1}^{(1)}}.
\end{eqnarray}
From (\ref{eq:liftgenCanal21b}) and (\ref{eq:genFanal29}), we have
  \begin{eqnarray}\label{eq:cpt6}
\Omega_3 & = & (\p_0\q_0-\p_1\q_1)\beta^2 \lp \mE_{G,{\mathcal U}_2}\langle \|\x^{(i_1)}\|_2^2\|\y^{(i_2)}\|_2^2\rangle_{\gamma_{01}^{(1)}} +   (s-1)\mE_{G,{\mathcal U}_2}\langle \|\x^{(i_1)}\|_2^2 \|\y^{(i_2)}\|_2\|\y^{(p_2)}\|_2\rangle_{\gamma_{02}^{(1)}}\rp\nonumber \\
& & - (\p_0\q_0-\p_1\q_1)s\beta^2(1-\m_1)\mE_{G,{\mathcal U}_2}\langle \|\x^{(i_1)}\|_2\|\x^{(p_`)}\|_2\|\y^{(i_2)}\|_2\|\y^{(p_2)}\|_2 \rangle_{\gamma_{1}^{(1)}}\nonumber \\
&  & +
\p_1\q_1\beta^2\lp\mE_{G,{\mathcal U}_2}\langle \|\x^{(i_1)}\|_2^2\|\y^{(i_2)}\|_2^2\rangle_{\gamma_{01}^{(1)}} +   (s-1)\mE_{G,{\mathcal U}_2}\langle \|\x^{(i_1)}\|_2^2 \|\y^{(i_2)}\|_2\|\y^{(p_2)}\|_2\rangle_{\gamma_{02}^{(1)}}\rp\nonumber \\
& & - \p_1\q_1 s\beta^2(1-\m_1)\mE_{G,{\mathcal U}_2}\langle \|\x^{(i_1)}\|_2\|\x^{(p_`)}\|_2\|\y^{(i_2)}\|_2\|\y^{(p_2)}\|_2 \rangle_{\gamma_{1}^{(1)}} \nonumber \\
&  & -s\beta^2\p_1\q_1\m_1\mE_{G,{\mathcal U}_2} \mE_{G,{\mathcal U}_2}\langle\|\x^{(i_2)}\|_2\|\x^{(p_2)}\|_2\|\y^{(i_2)}\|_2\|\y^{(p_2)}\rangle_{\gamma_{2}^{(1)}} \nonumber \\
& = &
\p_0\q_0\beta^2\lp\mE_{G,{\mathcal U}_2}\langle \|\x^{(i_1)}\|_2^2\|\y^{(i_2)}\|_2^2\rangle_{\gamma_{01}^{(1)}} +   (s-1)\mE_{G,{\mathcal U}_2}\langle \|\x^{(i_1)}\|_2^2 \|\y^{(i_2)}\|_2\|\y^{(p_2)}\|_2\rangle_{\gamma_{02}^{(1)}}\rp\nonumber \\
& & - \p_0\q_0 s\beta^2(1-\m_1)\mE_{G,{\mathcal U}_2}\langle \|\x^{(i_1)}\|_2\|\x^{(p_`)}\|_2\|\y^{(i_2)}\|_2\|\y^{(p_2)}\|_2 \rangle_{\gamma_{1}^{(1)}} \nonumber \\
&  & -s\beta^2\p_1\q_1
\m_1\mE_{G,{\mathcal U}_2} \langle\|\x^{(i_1)}\|_2\|\x^{(p_1)}\|_2\|\y^{(i_2)}\|_2\|\y^{(p_2)}\rangle_{\gamma_{2}^{(1)}}.
\end{eqnarray}
Finally, combining (\ref{eq:ctp1}) with (\ref{eq:cpt3})-(\ref{eq:cpt6}), we obtain
\begin{eqnarray}\label{eq:cpt7}
\frac{d\psi(\calX,\calY,\q,\m,\beta,s,t)}{dt}  & = &       \frac{\mbox{sign}(s)\beta}{2\sqrt{n}} \lp \phi_1^{(1)}+\phi_2^{(1)} +\phi_{01}^{(1)}+\phi_{02}^{(1)}\rp ,
 \end{eqnarray}
where
\begin{eqnarray}\label{eq:cpt8}
\phi_1^{(1)} & = &
-s(1-\m_1)\mE_{G,{\mathcal U}_2} \langle (\p_0\|\x^{(i_1)}\|_2\|\x^{(p_1)}\|_2 -(\x^{(p_1)})^T\x^{(i_1)})(\q_0\|\y^{(i_2)}\|_2\|\y^{(p_2)}\|_2 -(\y^{(p_2)})^T\y^{(i_2)})\rangle_{\gamma_{1}^{(1)}} \nonumber \\
\phi_2^{(1)} & = &
-s\m_1\mE_{G,{\mathcal U}_2} \langle (\p_1\|\x^{(i_1)}\|_2\|\x^{(p_1)}\|_2 -(\x^{(p_1)})^T\x^{(i_1)})(\q_1\|\y^{(i_2)}\|_2\|\y^{(p_2)}\|_2 -(\y^{(p_2)})^T\y^{(i_2)})\rangle_{\gamma_{2}^{(1)}} \nonumber \\
\phi_{01}^{(1)} & = & (1-\p_0)(1-\q_0)\mE_{G,{\mathcal U}_2}\langle \|\x^{(i_1)}\|_2^2\|\y^{(i_2)}\|_2^2\rangle_{\gamma_{01}^{(1)}} \nonumber\\
\phi_{02}^{(1)} & = & (s-1)(1-\p_0)\mE_{G,{\mathcal U}_2}\left\langle \|\x^{(i_1)}\|_2^2 \lp\q_0\|\y^{(i_2)}\|_2\|\y^{(p_2)}\|_2-(\y^{(p_2)})^T\y^{(i_2)}\rp\right\rangle_{\gamma_{02}^{(1)}}. \end{eqnarray}

We summarize the above into the following proposition.
\begin{proposition}
\label{thm:thm1} Consider scalar $\m_1$, vector $\p=[\p_0,\p_1,\p_2]$ with $\p_0\geq \p_1\geq \p_2=0$,  and vector $\q=[\q_0,\q_1,\q_2]$ with $\q_0\geq \q_1\geq \q_2=0$. Let $k\in\{1,2\}$ and $G\in\mR^{m \times n},u^{(4,k)}\in\mR^1,\u^{(2,k)}\in\mR^{m\times 1}$, and $\h^{(k)}\in\mR^{n\times 1}$ all have i.i.d. zero-mean normal components (they are then independent of each other as well). Let the variances of the components of $G$, $u^{(4,1)}$, $u^{(4,2)}$, $\u^{(2,1)}$, $\u^{(2,2)}$, $\h^{(1)}$, and $\h^{(2)}$ be $1$, $\p_0\q_0-\p_1\q_1$, $\p_1\q_1$, $\p_0-\p_1$, $\p_1$, $\q_0-\q_1$, and $\q_1$, respectively. Let ${\mathcal U}_k=[u^{(4,k)},\u^{(2,k)},\h^{(k)}]$. Assume that set ${\mathcal X}=\{\x^{(1)},\x^{(2)},\dots,\x^{(l)}\}$, where $\x^{(i)}\in \mR^{n},1\leq i\leq l$, and set ${\mathcal Y}=\{\y^{(1)},\y^{(2)},\dots,\y^{(l)}\}$, where $\y^{(i)}\in \mR^{m},1\leq i\leq l$ are given and that $\beta\geq 0$ and $s$ are real numbers. Consider the following function
\begin{equation}\label{eq:prop1eq1}
\psi(\calX,\calY,\p,\q,\m,\beta,s,t)  =  \mE_{G,{\mathcal U}_2} \frac{1}{\beta|s|\sqrt{n}\m_1} \log \mE_{{\mathcal U}_1} \lp \sum_{i_1=1}^{l}\lp\sum_{i_2=1}^{l}e^{\beta D_0^{(i_1,i_2)}} \rp^{s}\rp^{\m_1},
\end{equation}
where
\begin{eqnarray}\label{eq:prop1eq2}
 D_0^{(i_1,i_2)} & \triangleq & \sqrt{t}(\y^{(i_2)})^T
 G\x^{(i_1)}+\sqrt{1-t}\|\x^{(i_1)}\|_2 (\y^{(i_2)})^T(\u^{(2,1)}+\u^{(2,2)})\nonumber \\
 & & +\sqrt{t}\|\x^{(i_1)}\|_2\|\y^{(i_2)}\|_2(u^{(4,1)}+u^{(4,2)}) +\sqrt{1-t}\|\y^{(i_2)}\|_2(\h^{(1)}+\h^{(2)})^T\x^{(i_1)}.
 \end{eqnarray}
then
\begin{eqnarray}\label{eq:prop1eq3}
\frac{d\psi(\calX,\calY,\q,\m,\beta,s,t)}{dt}  & = &   \frac{\mbox{sign}(s)\beta}{2\sqrt{n}} \lp \phi_1^{(1)}+\phi_2^{(1)} +\phi_{01}^{(1)}+\phi_{02}^{(1)}\rp,
 \end{eqnarray}
where $\phi$'s are as in (\ref{eq:cpt8}).
\end{proposition}
\begin{proof}
  Follows from the above presentation.
\end{proof}

%%%%%%%%%%%%%%%%%%%%%%%%%%%%%%%%%%%%%%%%%%%%%%%%%%%%%%%%%%%%%%%%%
%%%%%%%%%%%%%%%%%%%%%%%%%%%%%%%%%%%%%%%%%%%%%%%%%%%%%%%%%%%%%%%%%
%%%%%%%%%%%%%%%%%%%%%%%%%%%%%%%%%%%%%%%%%%%%%%%%%%%%%%%%%%%%%%%%%
%%%%%%%%%%%%%%%%%%%%%%%%%%%%%%%%%%%%%%%%%%%%%%%%%%%%%%%%%%%%%%%%%
%%%%%%%%%%%%%%%%%%%%%%%%%%%%%%%%%%%%%%%%%%%%%%%%%%%%%%%%%%%%%%%%%
\section{Second level of full lifting}
\label{sec:seclev}
%%%%%%%%%%%%%%%%%%%%%%%%%%%%%%%%%%%%%%%%%%%%%%%%%%%%%%%%%%%%%%%%%
%%%%%%%%%%%%%%%%%%%%%%%%%%%%%%%%%%%%%%%%%%%%%%%%%%%%%%%%%%%%%%%%%
%%%%%%%%%%%%%%%%%%%%%%%%%%%%%%%%%%%%%%%%%%%%%%%%%%%%%%%%%%%%%%%%%
%%%%%%%%%%%%%%%%%%%%%%%%%%%%%%%%%%%%%%%%%%%%%%%%%%%%%%%%%%%%%%%%%
%%%%%%%%%%%%%%%%%%%%%%%%%%%%%%%%%%%%%%%%%%%%%%%%%%%%%%%%%%%%%%%%%
%%%%%%%%%%%%%%%%%%%%%%%%%%%%%%%%%%%%%%%%%%%%%%%%%%%%%%%%%%%%%%%%%

In this section, we show how the results of the previous section extend to the second level of lifting (this level in some scenarios can be viewed as the first non-trivial one). All key logical ingredients needed for such an extension are already present in the derivation from Section \ref{sec:gencon}. We here formalize them and effectively trace the path for full lifting with $r\in\mN$ levels. To facilitate the easiness of the exposition and following, we will try to parallel the derivation from Sectoin \ref{sec:gencon} as closely as possible. However, to avoid unnecessary repetitions, we often proceed at a much faster pace than in Section \ref{sec:gencon}.

We now consider vectors $\p=[\p_0,\p_1,\p_2,\p_3]$ with $\p_0\geq \p_1\geq \p_2\geq \p_3=0$ and $\q=[\q_0,\q_1,\q_2,\q_3]$ with $\q_0\geq \q_1\geq \q_2\geq \q_3 = 0$. As earlier, we take $(m\times n)$ dimensional matrices  $G\in \mR^{m\times n}$ with i.i.d. standard normal components and independent (of $G$ and among themselves) random variables $u^{(4,1)}\sim {\mathcal N}(0,\p_0\q_0-\p_1\q_1)$, $u^{(4,2)}\sim {\mathcal N}(0,\p_1\q_1-\p_2\q_2)$, and $u^{(4,3)}\sim {\mathcal N}(0,\p_2\q_2)$. For a vector $\m=[\m_1,\m_2]$,  we consider the following interpolating function $\psi(\cdot)$
\begin{equation}\label{eq:lev2genanal3}
\psi(\calX,\calY,\p,\q,\m,\beta,s,t)  =  \mE_{G,{\mathcal U}_3} \frac{1}{\beta|s|\sqrt{n}\m_2} \log \lp \mE_{{\mathcal U}_2}\lp\lp\mE_{{\mathcal U}_1} \lp Z^{\m_1}\rp\rp^{\frac{\m_2}{\m_1}}\rp\rp,
\end{equation}
where
\begin{eqnarray}\label{eq:lev2genanal3a}
Z & \triangleq & \sum_{i_1=1}^{l}\lp\sum_{i_2=1}^{l}e^{\beta D_0^{(i_1,i_2)}} \rp^{s} \nonumber \\
 D_0^{(i_1,i_2)} & \triangleq & \sqrt{t}(\y^{(i_2)})^T
 G\x^{(i_1)}+\sqrt{1-t}\|\x^{(i_1)}\|_2 (\y^{(i_2)})^T(\u^{(2,1)}+\u^{(2,2)})\nonumber \\
 & & +\sqrt{t}\|\x^{(i_1)}\|_2\|\y^{(i_2)}\|_2(u^{(4,1)}+u^{(4,2)}) +\sqrt{1-t}\|\y^{(i_2)}\|_2(\h^{(1)}+\h^{(2)})^T\x^{(i_1)},
 \end{eqnarray}
 and ${\mathcal U}_k=[u^{(4,k)},\u^{(2,k)},\h^{(2k)}]$, $k\in\{1,2,3\}$. In (\ref{eq:lev2genanal3}), $\u^{(2,1)}$, $\u^{(2,2)}$, and $\u^{(2,3)}$  are $m$ dimensional vectors of i.i.d. zero-mean Gaussians with variances $\p_0-\p_1$, $\p_1-\p_2$, and , $\p_2$, respectively. Similarly, $\h^{(1)}$, $\h^{(2)}$, and $\h^{(3)}$ are $n$ dimensional vectors of i.i.d. zero-mean Gaussians with variances $\q_0-\q_1$, $\q_1-\q_2$, and $\q_2$, respectively. These six vectors are assumed to be independent among themselves and of $G$, $u^{(4,1)}$, $u^{(4,2)}$, and $u^{(4,3)}$ as well.

For the convenience of the exposition, we, analogously to (\ref{eq:genanal4}), set
\begin{eqnarray}\label{eq:lev2genanal4}
\u^{(i_1,1)} & =  & \frac{G\x^{(i_1)}}{\|\x^{(i_1)}\|_2} \nonumber \\
\u^{(i_1,3,1)} & =  & \frac{(\h^{(1)})^T\x^{(i_1)}}{\|\x^{(i_1)}\|_2} \nonumber \\
\u^{(i_1,3,2)} & =  & \frac{(\h^{(2)})^T\x^{(i_1)}}{\|\x^{(i_1)}\|_2} \nonumber \\
\u^{(i_1,3,3)} & =  & \frac{(\h^{(3)})^T\x^{(i_1)}}{\|\x^{(i_1)}\|_2}.
\end{eqnarray}
and recall that
\begin{eqnarray}\label{eq:lev2genanal5}
\u_j^{(i_1,1)} & =  & \frac{G_{j,1:n}\x^{(i_1)}}{\|\x^{(i_1)}\|_2},1\leq j\leq m.
\end{eqnarray}
Clearly, for any fixed $i_1$, the elements of $\u^{(i_1,1)}$ are i.i.d. standard normals, the elements of $\u^{(2,1)}$, $\u^{(2,2)}$, and $\u^{(2,3)}$ are zero-mean Gaussians with respective variances $\p_0-\p_1$, $\p_1-\p_2$, and $\p_2$, and the elements of $\u^{(i_1,3,1)}$, $\u^{(i_1,3,2)}$, and $\u^{(i_1,3,3)}$ are zero-mean Gaussians with  respective variances $\q_0-\q_1$, $\q_1-\q_2$, and $\q_2$. One then rewrites (\ref{eq:lev2genanal3}) as
\begin{equation}\label{eq:lev2genanal6}
\psi(\calX,\calY,\p,\q,\m,\beta,s,t)  =  \mE_{G,{\mathcal U}_2} \frac{1}{\beta|s|\sqrt{n}\m_2} \log \lp\mE_{{\mathcal U}_2}  \lp \mE_{{\mathcal U}_1} \lp \sum_{i_1=1}^{l}\lp\sum_{i_2=1}^{l}A^{(i_1,i_2)} \rp^{s}\rp^{\m_1}\rp^{\frac{\m_2}{\m_1}}\rp,
\end{equation}
where $\beta_{i_1}=\beta\|\x^{(i_1)}\|_2$ and now
\begin{eqnarray}\label{eq:lev2genanal7}
B^{(i_1,i_2)} & \triangleq &  \sqrt{t}(\y^{(i_2)})^T\u^{(i_1,1)}+\sqrt{1-t} (\y^{(i_2)})^T(\u^{(2,1)}+\u^{(2,2)}) \nonumber \\
D^{(i_1,i_2)} & \triangleq &  (B^{(i_1,i_2)}+\sqrt{t}\|\y^{(i_2)}\|_2 (u^{(4,1)}+u^{(4,2)}+u^{(4,3)})+\sqrt{1-t}\|\y^{(i_2)}\|_2(\u^{(i_1,3,1)}+\u^{(i_1,3,2)}+\u^{(i_1,3,3)})) \nonumber \\
A^{(i_1,i_2)} & \triangleq &  e^{\beta_{i_1}D^{(i_1,i_2)}}\nonumber \\
C^{(i_1)} & \triangleq &  \sum_{i_2=1}^{l}A^{(i_1,i_2)}\nonumber \\
Z & \triangleq & \sum_{i_1=1}^{l} \lp \sum_{i_2=1}^{l} A^{(i_1,i_2)}\rp^s =\sum_{i_1=1}^{l}  (C^{(i_1)})^s.
\end{eqnarray}
Since we are again interested in studying monotonicity of $\psi(\calX,\calY,\p,\q,\m,\beta,s,t)$, we start by considering its derivative
\begin{eqnarray}\label{eq:lev2genanal9}
\frac{d\psi(\calX,\calY,\q,\m,\beta,s,t)}{dt} & = &  \frac{d}{dt}\lp\mE_{G,{\mathcal U}_3} \frac{1}{\beta|s|\sqrt{n}\m_2} \log \lp \mE_{{\mathcal U}_2}\lp\lp\mE_{{\mathcal U}_1} \lp Z^{\m_1}\rp\rp^{\frac{\m_2}{\m_1}}\rp\rp\rp\nonumber \\
& = &  \mE_{G,{\mathcal U}_3} \frac{1}{\beta|s|\sqrt{n}\m_2\mE_{{\mathcal U}_2}\lp\lp\mE_{{\mathcal U}_1} \lp Z^{\m_1}\rp\rp^{\frac{\m_2}{\m_1}}\rp}
\frac{d\lp \mE_{{\mathcal U}_2}\lp\lp\mE_{{\mathcal U}_1} \lp Z^{\m_1}\rp\rp^{\frac{\m_2}{\m_1}}\rp\rp}{dt}\nonumber \\
& = &  \mE_{G,{\mathcal U}_3,{\mathcal U}_2} \frac{\lp\mE_{{\mathcal U}_1} \lp Z^{\m_1}\rp\rp^{\frac{\m_2}{\m_1}-1}}{\beta|s|\sqrt{n}\m_1\mE_{{\mathcal U}_2}\lp\lp\mE_{{\mathcal U}_1} \lp Z^{\m_1}\rp\rp^{\frac{\m_2}{\m_1}}\rp}
\frac{d \mE_{{\mathcal U}_1} Z^{\m_1} }{dt}\nonumber \\
& = &  \mE_{G,{\mathcal U}_3,{\mathcal U}_2} \frac{\lp\mE_{{\mathcal U}_1} \lp Z^{\m_1}\rp\rp^{\frac{\m_2}{\m_1}-1}}{\beta|s|\sqrt{n}\m_1\mE_{{\mathcal U}_2}\lp\lp\mE_{{\mathcal U}_1} \lp Z^{\m_1}\rp\rp^{\frac{\m_2}{\m_1}}\rp}
\mE_{{\mathcal U}_1} \frac{1}{Z^{1-\m_1}}\frac{d Z}{dt}\nonumber \\
& = &  \mE_{G,{\mathcal U}_3,{\mathcal U}_2} \frac{\lp\mE_{{\mathcal U}_1} \lp Z^{\m_1}\rp\rp^{\frac{\m_2}{\m_1}-1}}{\beta|s|\sqrt{n}\mE_{{\mathcal U}_2}\lp\lp\mE_{{\mathcal U}_1} \lp Z^{\m_1}\rp\rp^{\frac{\m_2}{\m_1}}\rp}
\mE_{{\mathcal U}_1} \frac{1}{Z^{1-\m_1}} \frac{d\lp \sum_{i_1=1}^{l} \lp \sum_{i_2=1}^{l} A^{(i_1,i_2)}\rp^s \rp }{dt}\nonumber \\
& = &  \mE_{G,{\mathcal U}_3,{\mathcal U}_2} \frac{\lp\mE_{{\mathcal U}_1} \lp Z^{\m_1}\rp\rp^{\frac{\m_2}{\m_1}-1}}{\beta|s|\sqrt{n}\mE_{{\mathcal U}_2}\lp\lp\mE_{{\mathcal U}_1} \lp Z^{\m_1}\rp\rp^{\frac{\m_2}{\m_1}}\rp}\mE_{{\mathcal U}_1} \frac{1}{Z^{1-\m_1}}  \sum_{i=1}^{l} (C^{(i_1)})^{s-1} \nonumber \\
& & \times \sum_{i_2=1}^{l}\beta_{i_1}A^{(i_1,i_2)}\frac{dD^{(i_1,i_2)}}{dt},
\end{eqnarray}
where
\begin{eqnarray}\label{eq:lev2genanal9a}
\frac{dD^{(i_1,i_2)}}{dt}= \lp \frac{dB^{(i_1,i_2)}}{dt}+\frac{\|\y^{(i_2)}\|_2 (u^{(4,1)}+u^{(4,2)}+u^{(4,3)})}{2\sqrt{t}}-\frac{\|\y^{(i_2)}\|_2 (\u^{(i_1,3,1)}+\u^{(i_1,3,2)}+\u^{(i_1,3,3)})}{2\sqrt{1-t}}\rp.
\end{eqnarray}
Utilizing (\ref{eq:lev2genanal7}) we find
\begin{eqnarray}\label{eq:lev2genanal10}
\frac{dB^{(i_1,i_2)}}{dt} & = &   \frac{d\lp\sqrt{t}(\y^{(i_2)})^T\u^{(i_1,1)}+\sqrt{1-t} (\y^{(i_2)})^T(\u^{(2,1)}+\u^{(2,2)}+\u^{(2,3)})\rp}{dt} \nonumber \\
 & = &
\sum_{j=1}^{m}\lp \frac{\y_j^{(i_2)}\u_j^{(i_1,1)}}{2\sqrt{t}}-\frac{\y_j^{(i_2)}\u_j^{(2,1)}}{2\sqrt{1-t}}-\frac{\y_j^{(i_2)}\u_j^{(2,2)}}{2\sqrt{1-t}}
-\frac{\y_j^{(i_2)}\u_j^{(2,3)}}{2\sqrt{1-t}}\rp.
\end{eqnarray}
One can then write analogously to (\ref{eq:genanal10e})
\begin{equation}\label{eq:lev2genanal10e}
\frac{d\psi(\calX,\calY,\q,\m,\beta,s,t)}{dt}  =       \frac{\mbox{sign}(s)}{2\beta\sqrt{n}} \sum_{i_1=1}^{l}  \sum_{i_2=1}^{l}
\beta_{i_1}\lp T_G + T_2+ T_1\rp,
\end{equation}
where
\begin{eqnarray}\label{eq:lev2genanal10f}
T_G & = & \sum_{j=1}^{m}\frac{T_{G,j}}{\sqrt{t}}  \nonumber\\
T_3 & = & -\sum_{j=1}^{m}\frac{T_{3,1,j}}{\sqrt{1-t}}-\|\y^{(i_2)}\|_2\frac{T_{3,2}}{\sqrt{1-t}}+\|\y^{(i_2)}\|_2\frac{T_{3,3}}{\sqrt{t}} \nonumber\\
T_2 & = & -\sum_{j=1}^{m}\frac{T_{2,1,j}}{\sqrt{1-t}}-\|\y^{(i_2)}\|_2\frac{T_{2,2}}{\sqrt{1-t}}+\|\y^{(i_2)}\|_2\frac{T_{2,3}}{\sqrt{t}} \nonumber\\
T_1 & = & -\sum_{j=1}^{m}\frac{T_{1,1,j}}{\sqrt{1-t}}-\|\y^{(i_2)}\|_2\frac{T_{1,2}}{\sqrt{1-t}}+\|\y^{(i_2)}\|_2\frac{T_{1,3}}{\sqrt{t}},
\end{eqnarray}
and for $k\in\{1,2,3\}$
\begin{eqnarray}\label{eq:lev2genanal10g}
T_{G,j} & = &  \mE_{G,{\mathcal U}_3,{\mathcal U}_2} \lp\frac{\lp\mE_{{\mathcal U}_1} \lp Z^{\m_1}\rp\rp^{\frac{\m_2}{\m_1}-1}}{\mE_{{\mathcal U}_2}\lp\lp\mE_{{\mathcal U}_1} \lp Z^{\m_1}\rp\rp^{\frac{\m_2}{\m_1}}\rp}
  \mE_{{\mathcal U}_1}\frac{(C^{(i_1)})^{s-1} A^{(i_1,i_2)} \y_j^{(i_2)}\u_j^{(i_1,1)}}{Z^{1-\m_1}} \rp \nonumber \\
T_{k,1,j} & = &   \mE_{G,{\mathcal U}_3,{\mathcal U}_2} \lp\frac{\lp\mE_{{\mathcal U}_1} \lp Z^{\m_1}\rp\rp^{\frac{\m_2}{\m_1}-1}}{\mE_{{\mathcal U}_2}\lp\lp\mE_{{\mathcal U}_1} \lp Z^{\m_1}\rp\rp^{\frac{\m_2}{\m_1}}\rp}
  \mE_{{\mathcal U}_1}\frac{(C^{(i_1)})^{s-1} A^{(i_1,i_2)} \y_j^{(i_2)}\u_j^{(2,k)}}{Z^{1-\m_1}} \rp \nonumber \\
T_{k,2} & = &   \mE_{G,{\mathcal U}_3,{\mathcal U}_2} \lp\frac{\lp\mE_{{\mathcal U}_1} \lp Z^{\m_1}\rp\rp^{\frac{\m_2}{\m_1}-1}}{\mE_{{\mathcal U}_2}\lp\lp\mE_{{\mathcal U}_1} \lp Z^{\m_1}\rp\rp^{\frac{\m_2}{\m_1}}\rp}
  \mE_{{\mathcal U}_1}\frac{(C^{(i_1)})^{s-1} A^{(i_1,i_2)} \u^{(i_1,3,k)}}{Z^{1-\m_1}} \rp \nonumber \\
T_{k,3} & = &   \mE_{G,{\mathcal U}_3,{\mathcal U}_2} \lp\frac{\lp\mE_{{\mathcal U}_1} \lp Z^{\m_1}\rp\rp^{\frac{\m_2}{\m_1}-1}}{\mE_{{\mathcal U}_2}\lp\lp\mE_{{\mathcal U}_1} \lp Z^{\m_1}\rp\rp^{\frac{\m_2}{\m_1}}\rp}
  \mE_{{\mathcal U}_1}\frac{(C^{(i_1)})^{s-1} A^{(i_1,i_2)} u^{(4,k)}}{Z^{1-\m_1}} \rp.
\end{eqnarray}
We will handle each of the above ten terms separately heavily relying on and paralleling as closely as possible the presentation of Section \ref{sec:gencon}.

%%%%%%%%%%%%%%%%%%%%%%%%%%%%%%%%%%%%%%%%%%%%%%%%%%%%%%%%%%%%%%%%%%%%%%%%
\subsection{Computing $\frac{d\psi(\calX,\calY,\p,\q,\m,\beta,s,t)}{dt}$ on the second level of full lifting}
\label{sec:lev2compderivative}
%%%%%%%%%%%%%%%%%%%%%%%%%%%%%%%%%%%%%%%%%%%%%%%%%%%%%%%%%%%%%%%%%%%%%%%%

As in Section \ref{sec:gencon}, we carefully choose the order in which we handle each of the above terms. We group the three terms indexed by $k$ into three groups, $T_k$, $k\in\{1,2,3\}$, and first handle $k=1$ scenario, then $k=2$, and then $k=3$. After that, we handle $T_{G,j}$ as well.

%%%%%%%%%%%%%%%%%%%%%%%%%%%%%%%%%%%%%%%%%%%%%%%%%%%%%%%%%%%%%%%%%%%%%%%%
\subsubsection{Handling $T_k$--groups --- $k=1$}
\label{sec:lev2handlTk1}
%%%%%%%%%%%%%%%%%%%%%%%%%%%%%%%%%%%%%%%%%%%%%%%%%%%%%%%%%%%%%%%%%%%%%%%%

We handle separately each of the three terms from $T_1$--group.

%%%%%%%%%%%%%%%%%%%%%%%%%%%%%%%%%%%%%%%%%%%%%%%%%%%%%%%%%%%%%%%%%%%%%%%%
\underline{\textbf{\emph{Determining}} $T_{1,1,j}$}
\label{sec:lev2hand1T11}
%%%%%%%%%%%%%%%%%%%%%%%%%%%%%%%%%%%%%%%%%%%%%%%%%%%%%%%%%%%%%%%%%%%%%%%%

Analogously to (\ref{eq:liftgenAanal19}), the Gaussian integration by parts gives
 \begin{eqnarray}\label{eq:lev2liftgenAanal19}
T_{1,1,j} & = & \mE_{G,{\mathcal U}_3,{\mathcal U}_2} \lp\frac{\lp\mE_{{\mathcal U}_1} \lp Z^{\m_1}\rp\rp^{\frac{\m_2}{\m_1}-1}}{\mE_{{\mathcal U}_2}\lp\lp \mE_{{\mathcal U}_1} \lp Z^{\m_1}\rp\rp^{\frac{\m_2}{\m_1}}\rp}      \mE_{{\mathcal U}_1}  \frac{(C^{(i_1)})^{s-1} A^{(i_1,i_2)}\y_j^{(i_2)}}{Z^{1-\m_1}}\rp \nonumber \\
& = & \mE_{G,{\mathcal U}_3,{\mathcal U}_2} \lp\frac{\lp\mE_{{\mathcal U}_1} \lp Z^{\m_1}\rp\rp^{\frac{\m_2}{\m_1}-1}}{\mE_{{\mathcal U}_2}\lp\lp   \mE_{{\mathcal U}_1} \lp Z^{\m_1}\rp\rp^{\frac{\m_2}{\m_1}}\rp}   \mE_{{\mathcal U}_1}\lp  \frac{d}{du_j^{(2,1)}}\lp\frac{(C^{(i_1)})^{s-1} A^{(i_1,i_2)}\y_j^{(i_2)}}{Z^{1-\m_1}}\rp\rp\rp.
\end{eqnarray}
Moreover, analogously to (\ref{eq:liftgenAanal19a}), we also have
\begin{eqnarray}\label{eq:lev2liftgenAanal19a}
T_{1,1,j}
& = &  (\p_0-\p_1)
\mE_{G,{\mathcal U}_3,{\mathcal U}_2} \lp\frac{\lp\mE_{{\mathcal U}_1} \lp Z^{\m_1}\rp\rp^{\frac{\m_2}{\m_1}-1}}{\mE_{{\mathcal U}_2}\lp\lp \mE_{{\mathcal U}_1} \lp Z^{\m_1}\rp\rp^{\frac{\m_2}{\m_1}}\rp}   \lp \Theta_1+\Theta_2 \rp\rp.
\end{eqnarray}
where $\Theta_1$ and $\Theta_2$ are as in (\ref{eq:liftgenAanal19c}). Analogously to (\ref{eq:liftgenAanal19d}), we conveniently observe the following as well
 \begin{eqnarray}\label{eq:lev2liftgenAanal19d}
\sum_{i_1=1}^{l}\sum_{i_2=1}^{l}\sum_{j=1}^{m} \mE_{{\mathcal U}_2} \lp\frac{\lp\mE_{{\mathcal U}_1} \lp Z^{\m_1}\rp\rp^{\frac{\m_2}{\m_1}-1}}{\mE_{{\mathcal U}_2}\lp\lp \mE_{{\mathcal U}_1} \lp Z^{\m_1}\rp\rp^{\frac{\m_2}{\m_1}}\rp} \frac{\beta_{i_1}\Theta_1}{\sqrt{1-t}}\rp
&  = &  L^{(2)}_1+L^{(2)}_2,
 \end{eqnarray}
where
\begin{eqnarray}\label{eq:lev2liftgenAanal19d1}
L^{(1)} &  = & \mE_{{\mathcal U}_2} \lp \mE_{{\mathcal U}_1} \frac{\lp\mE_{{\mathcal U}_1} \lp Z^{\m_1}\rp\rp^{\frac{\m_2}{\m_1}-1}Z^{\m_1}}{\mE_{{\mathcal U}_2}\lp\lp \mE_{{\mathcal U}_1} \lp Z^{\m_1}\rp\rp^{\frac{\m_2}{\m_1}}\rp}  \sum_{i_1=1}^{l}\frac{(C^{(i_1)})^s}{Z}\sum_{i_2=1}^{l}\frac{A^{(i_1,i_2)}}{C^{(i_1)}}\beta_{i_1}^2\|\y^{(i_2)}\|_2^2\rp \nonumber\\
L^{(1)} &  = & \mE_{{\mathcal U}_2} \lp \mE_{{\mathcal U}_1} \frac{\lp\mE_{{\mathcal U}_1} \lp Z^{\m_1}\rp\rp^{\frac{\m_2}{\m_1}-1}Z^{\m_1}}{\mE_{{\mathcal U}_2}\lp\lp \mE_{{\mathcal U}_1} \lp Z^{\m_1}\rp\rp^{\frac{\m_2}{\m_1}}\rp}    \sum_{i_1=1}^{l}\frac{(s-1)(C^{(i_1)})^s}{Z}\sum_{i_2=1}^{l}\sum_{p_2=1}^{l}\frac{A^{(i_1,i_2)}A^{(i_1,p_2)}}{(C^{(i_1)})^2}\beta_{i_1}^2(\y^{(p_2)})^T\y^{(i_2)}\rp.\nonumber \\
 \end{eqnarray}
 We set the operators
\begin{eqnarray}\label{eq:lev2genAanal19d2}
 \Phi_{{\mathcal U}_1} & \triangleq &  \mE_{{\mathcal U}_1} \frac{Z^{\m_1}}{\mE_{{\mathcal U}_1}\lp Z^{\m_1}\rp} \nonumber \\
 \Phi_{{\mathcal U}_2} & \triangleq & \mE_{{\mathcal U}_2}   \frac{\lp\mE_{{\mathcal U}_1} \lp Z^{\m_1}\rp\rp^{\frac{\m_2}{\m_1}}}{\mE_{{\mathcal U}_2}\lp\lp \mE_{{\mathcal U}_1} \lp Z^{\m_1}\rp\rp^{\frac{\m_2}{\m_1}}\rp}\nonumber \\
  \Phi_{{\mathcal U}_2,{\mathcal U}_1} & \triangleq &   \Phi_{{\mathcal U}_2}  \Phi_{{\mathcal U}_1} = \mE_{{\mathcal U}_2}   \frac{\lp\mE_{{\mathcal U}_1} \lp Z^{\m_1}\rp\rp^{\frac{\m_2}{\m_1}}}{\mE_{{\mathcal U}_2}\lp\lp \mE_{{\mathcal U}_1} \lp Z^{\m_1}\rp\rp^{\frac{\m_2}{\m_1}}\rp}
\mE_{{\mathcal U}_1} \frac{Z^{\m_1}}{\mE_{{\mathcal U}_1}\lp Z^{\m_1}\rp},
 \end{eqnarray}
and consider the following weighted/product Gibbs measures
\begin{eqnarray}\label{eq:lev2genAanal19e}
\gamma_0(i_1,i_2) & = &
\frac{(C^{(i_1)})^{s}}{Z}  \frac{A^{(i_1,i_2)}}{C^{(i_1)}} \nonumber \\
\gamma_{01}^{(2)}  & = & \Phi_{{\mathcal U}_2,{\mathcal U}_1} (\gamma_0(i_1,i_2)) \nonumber \\
\gamma_{02}^{(2)}  & = & \Phi_{{\mathcal U}_2,{\mathcal U}_1} (\gamma_0(i_1,i_2)\times \gamma_0(i_1,p_2)) \nonumber \\
\gamma_1^{(2)}  & = & \Phi_{{\mathcal U}_2,{\mathcal U}_1} (\gamma_0(i_1,i_2)\times \gamma_0(p_1,p_2)) \nonumber\\
\gamma_2^{(2)}  & = & \Phi_{{\mathcal U}_2} (\Phi_{{\mathcal U}_1}\gamma_0(i_1,i_2)\times \Phi_{{\mathcal U}_1}\gamma_0(p_1,p_2))\nonumber \\
\gamma_3^{(2)}  & = & (\Phi_{{\mathcal U}_2}\Phi_{{\mathcal U}_1}\gamma_0(i_1,i_2)\times \Phi_{{\mathcal U}_2}\Phi_{{\mathcal U}_1}\gamma_0(p_1,p_2))
\end{eqnarray}
As in Section \ref{sec:gencon}, all of the above $\gamma$'s are indeed valid measures. For example, we can write the following for $\gamma_1^{(2)}$
\begin{eqnarray}\label{eq:lev2genAanal19f}
 \sum_{i_1=1}^{l}  \sum_{i_2=1}^{l} \gamma_1^{(2)} & = & \sum_{i_1=1}^{l}  \sum_{i_2=1}^{l}\mE_{{\mathcal U}_2}   \frac{\lp\mE_{{\mathcal U}_1} \lp Z^{\m_1}\rp\rp^{\frac{\m_2}{\m_1}}}{\mE_{{\mathcal U}_2}\lp\lp \mE_{{\mathcal U}_1} \lp Z^{\m_1}\rp\rp^{\frac{\m_2}{\m_1}}\rp}
\mE_{{\mathcal U}_1} \frac{Z^{\m_1}}{\mE_{{\mathcal U}_1}\lp Z^{\m_1}\rp}
\frac{(C^{(i_1)})^{s}}{Z}  \frac{A^{(i_1,i_2)}}{C^{(i_1)}}\nonumber \\
& = &
\mE_{{\mathcal U}_2}   \frac{\lp\mE_{{\mathcal U}_1} \lp Z^{\m_1}\rp\rp^{\frac{\m_2}{\m_1}}}{\mE_{{\mathcal U}_2}\lp\lp \mE_{{\mathcal U}_1} \lp Z^{\m_1}\rp\rp^{\frac{\m_2}{\m_1}}\rp}
\mE_{{\mathcal U}_1} \frac{Z^{\m_1}}{\mE_{{\mathcal U}_1}\lp Z^{\m_1}\rp}
  \sum_{i_1=1}^{l} \frac{(C^{(i_1)})^{s}}{Z}  \sum_{i_2=1}^{l} \frac{A^{(i_1,i_2)}}{C^{(i_1)}}\nonumber\\
 & = &
\mE_{{\mathcal U}_2}   \frac{\lp\mE_{{\mathcal U}_1} \lp Z^{\m_1}\rp\rp^{\frac{\m_2}{\m_1}}}{\mE_{{\mathcal U}_2}\lp\lp \mE_{{\mathcal U}_1} \lp Z^{\m_1}\rp\rp^{\frac{\m_2}{\m_1}}\rp}
\mE_{{\mathcal U}_1} \frac{Z^{\m_1}}{\mE_{{\mathcal U}_1}\lp Z^{\m_1}\rp} \nonumber \\
 & = & 1,
 \end{eqnarray}
where the second to last equality follows by the definitions of $Z$ and $C^{(i_1)}$  from (\ref{eq:lev2genanal7}). Together with the trivial fact $\gamma_1^{(2)}\geq 0$, (\ref{eq:lev2genAanal19f})  ensures that $\gamma_1^{(2)}$ is indeed a measure. The proofs that $\gamma_2^{(2)}$ and $\gamma_3^{(2)}$ are also valid measures trivially proceed along these lines as well and we skip them. We continue the practice established in Section \ref{sec:gencon}, and denote by $\langle \cdot \rangle_{a}$ the average with respect to measure $a$. Analogously to (\ref{eq:liftgenAanal19g}), we from (\ref{eq:lev2liftgenAanal19d}) obtain
\begin{eqnarray}\label{eq:lev2liftgenAanal19g}
\sum_{i_1=1}^{l}\sum_{i_2=1}^{l}\sum_{j=1}^{m} \mE_{{\mathcal U}_2} \lp\frac{\lp\mE_{{\mathcal U}_1} \lp Z^{\m_1}\rp\rp^{\frac{\m_2}{\m_1}-1}}{\mE_{{\mathcal U}_2}\lp\lp \mE_{{\mathcal U}_1} \lp Z^{\m_1}\rp\rp^{\frac{\m_2}{\m_1}}\rp} \frac{\beta_{i_1}\Theta_1}{\sqrt{1-t}}\rp
 & = &  \beta^2 \Bigg(\Bigg. \langle \|\x^{(i_1)}\|_2^2\|\y^{(i_2)}\|_2^2\rangle_{\gamma_{01}^{(2)}} \nonumber \\
   & & +   (s-1)\langle \|\x^{(i_1)}\|_2^2(\y^{(p_2)})^T\y^{(i_2)}\rangle_{\gamma_{02}^{(2)}} \Bigg. \Bigg).
 \end{eqnarray}
Using (\ref{eq:liftgenAanal19c}), we also have
\begin{eqnarray}\label{eq:lev2liftgenAanal19h}
\sum_{i_1=1}^{l}\sum_{i_2=1}^{l}\sum_{j=1}^{m} \mE_{{\mathcal U}_2} \lp\frac{\lp\mE_{{\mathcal U}_1} \lp Z^{\m_1}\rp\rp^{\frac{\m_2}{\m_1}-1}}{\mE_{{\mathcal U}_2}\lp\lp \mE_{{\mathcal U}_1} \lp Z^{\m_1}\rp\rp^{\frac{\m_2}{\m_1}}\rp} \frac{\beta_{i_1}\Theta_2}{\sqrt{1-t}}\rp
  =  L^{(2)}_3,
\end{eqnarray}
where
\begin{eqnarray}\label{eq:lev2liftgenAanal19h1}
L^{(2)}_3 & = & -s(1-\m_1)\mE_{{\mathcal U}_2} \Bigg( \Bigg. \frac{\lp\mE_{{\mathcal U}_1} \lp Z^{\m_1}\rp\rp^{\frac{\m_2}{\m_1}-1}Z^{\m_1}}{\mE_{{\mathcal U}_2}\lp\lp \mE_{{\mathcal U}_1} \lp Z^{\m_1}\rp\rp^{\frac{\m_2}{\m_1}}\rp}   \sum_{i_1=1}^{l}\frac{(C^{(i_1)})^s}{Z}\sum_{i_2=1}^{l}
\frac{A^{(i_1,i_2)}}{C^{(i_1)}} \nonumber \\
& & \times
  \sum_{p_1=1}^{l} \frac{(C^{(p_1)})^s}{Z}\sum_{p_2=1}^{l}\frac{A^{(p_1,p_2)}}{C^{(p_1)}} \beta_{i_1}\beta_{p_1}(\y^{(p_2)})^T\y^{(i_2)} \Bigg.\Bigg)\nonumber \\
& =& -s\beta^2(1-\m_1)\langle \|\x^{(i_1)}\|_2\|\x^{(p_1)}\|_2(\y^{(p_2)})^T\y^{(i_2)} \rangle_{\gamma_{1}^{(2)}}.
\end{eqnarray}
Combining  (\ref{eq:lev2liftgenAanal19a}), (\ref{eq:lev2liftgenAanal19g}), and (\ref{eq:lev2liftgenAanal19h}) we obtain, analogously to (\ref{eq:liftgenAanal19i}),
\begin{eqnarray}\label{eq:lev2liftgenAanal19i}
\sum_{i_1=1}^{l}\sum_{i_2=1}^{l}\sum_{j=1}^{m} \beta_{i_1}\frac{T_{1,1,j}}{\sqrt{1-t
}}
& = & (\p_0-\p_1)\beta^2 \Bigg( \Bigg. \mE_{G,{\mathcal U}_3}\langle \|\x^{(i_1)}\|_2^2\|\y^{(i_2)}\|_2^2\rangle_{\gamma_{01}^{(2)}} \nonumber \\
& & +   (s-1)\mE_{G,{\mathcal U}_3}\langle \|\x^{(i_1)}\|_2^2(\y^{(p_2)})^T\y^{(i_2)}\rangle_{\gamma_{02}^{(2)}} \Bigg. \Bigg) \nonumber \\
& & - (\p_0-\p_1)s\beta^2(1-\m_1)\mE_{G,{\mathcal U}_3}\langle \|\x^{(i_1)}\|_2\|\x^{(p_1)}\|_2(\y^{(p_2)})^T\y^{(i_2)} \rangle_{\gamma_{1}^{(2)}}.
\end{eqnarray}

%%%%%%%%%%%%%%%%%%%%%%%%%%%%%%%%%%%%%%%%%%%%%%%%%%%%%%%%%%%%%%%%%%%%%%%%
\underline{\textbf{\emph{Determining}} $T_{1,2}$}
\label{sec:lev2hand1T12}
%%%%%%%%%%%%%%%%%%%%%%%%%%%%%%%%%%%%%%%%%%%%%%%%%%%%%%%%%%%%%%%%%%%%%%%%

Integration by parts gives
{\small \begin{eqnarray}\label{eq:lev2liftgenBanal20}
T_{1,2} & = & \mE_{G,{\mathcal U}_3,{\mathcal U}_2} \lp\frac{\lp\mE_{{\mathcal U}_1} \lp Z^{\m_1}\rp\rp^{\frac{\m_2}{\m_1}-1}}{\mE_{{\mathcal U}_2}\lp\lp\mE_{{\mathcal U}_1} \lp Z^{\m_1}\rp\rp^{\frac{\m_2}{\m_1}}\rp}\mE_{{\mathcal U}_1} \frac{(C^{(i_1)})^{s-1} A^{(i_1,i_2)}\u^{(i_1,3,1)}}{Z^{1-\m_1}}\rp \nonumber \\
& = & (\q_0-\q_1) \mE_{G,{\mathcal U}_3,{\mathcal U}_2} \lp\frac{\lp\mE_{{\mathcal U}_1} \lp Z^{\m_1}\rp\rp^{\frac{\m_2}{\m_1}-1}}{\mE_{{\mathcal U}_2}\lp\lp\mE_{{\mathcal U}_1} \lp Z^{\m_1}\rp\rp^{\frac{\m_2}{\m_1}}\rp}
\mE_{{\mathcal U}_1} \sum_{p_1=1}^{l}\frac{(\x^{(i_1)})^T\x^{(p_1)}}{\|\x^{(i_1)}\|_2\|\x^{(p_1)}\|_2} \frac{d}{d\u^{(p_1,3,1)}}\lp\frac{(C^{(i_1)})^{s-1} A^{(i_1,i_2)}}{Z^{1-\m_1}}\rp\rp.\nonumber \\
\end{eqnarray}}
Analogously to (\ref{eq:liftgenBanal20b}) we obtain
{\small \begin{eqnarray}\label{eq:lev2liftgenBanal20b}
\sum_{i_1=1}^{l}\sum_{i_2=1}^{l} \beta_{i_1}\|\y^{(i_2)}\|_2 \frac{T_{1,2}}{\sqrt{1-t}}
& = & (\q_0-\q_1) \beta^2 \Bigg( \Bigg. \mE_{G,{\mathcal U}_3}\langle \|\x^{(i_1)}\|_2^2\|\y^{(i_2)}\|_2^2\rangle_{\gamma_{01}^{(2)}} \nonumber \\
& & +  (s-1)\mE_{G,{\mathcal U}_3}\langle \|\x^{(i_1)}\|_2^2 \|\y^{(i_2)}\|_2\|\y^{(p_2)}\|_2\rangle_{\gamma_{02}^{(2)}}\Bigg.\Bigg)\nonumber \\
& & - (\q_0-\q_1)s\beta^2(1-\m_1)\mE_{G,{\mathcal U}_3}\langle (\x^{(p_1)})^T\x^{(i_1)}\|\y^{(i_2)}\|_2\|\y^{(p_2)}\|_2 \rangle_{\gamma_{1}^{(2)}}.
\end{eqnarray}}

%%%%%%%%%%%%%%%%%%%%%%%%%%%%%%%%%%%%%%%%%%%%%%%%%%%%%%%%%%%%%%%%%%%%%%%%
\underline{\textbf{\emph{Determining}} $T_{1,3}$}
\label{sec:lev2hand1T13}
%%%%%%%%%%%%%%%%%%%%%%%%%%%%%%%%%%%%%%%%%%%%%%%%%%%%%%%%%%%%%%%%%%%%%%%%

We again proceed via Gaussian integration by parts to obtain
\begin{eqnarray}\label{eq:lev2liftgenCanal21}
T_{1,3} & = & \mE_{G,{\mathcal U}_3,{\mathcal U}_2} \lp\frac{\lp\mE_{{\mathcal U}_1} \lp Z^{\m_1}\rp\rp^{\frac{\m_2}{\m_1}-1}}{\mE_{{\mathcal U}_2}\lp\lp\mE_{{\mathcal U}_1} \lp Z^{\m_1}\rp\rp^{\frac{\m_2}{\m_1}}\rp}  \mE_{{\mathcal U}_1}  \frac{(C^{(i_1)})^{s-1} A^{(i_1,i_2)}u^{(4,1)}}{Z^{1-\m_1}} \rp \nonumber \\
& = & (\p_0\q_0-\p_1\q_1) \mE_{G,{\mathcal U}_3,{\mathcal U}_2} \lp\frac{\lp\mE_{{\mathcal U}_1} \lp Z^{\m_1}\rp\rp^{\frac{\m_2}{\m_1}-1}}{\mE_{{\mathcal U}_2}\lp\lp\mE_{{\mathcal U}_1} \lp Z^{\m_1}\rp\rp^{\frac{\m_2}{\m_1}}\rp}    \mE_{{\mathcal U}_1} \lp\frac{d}{du^{(4,1)}} \lp\frac{(C^{(i_1)})^{s-1} A^{(i_1,i_2)}u^{(4,1)}}{Z^{1-\m_1}}\rp\rp\rp, \nonumber \\
\end{eqnarray}
and  analogously to (\ref{eq:liftgenCanal21b})
\begin{eqnarray}\label{eq:lev2liftgenCanal21b}
\sum_{i_1=1}^{l}\sum_{i_2=1}^{l} \beta_{i_1}\|\y^{(i_2)}\|_2 \frac{T_{1,3}}{\sqrt{t}}
& = & (\p_0\q_0-\p_1\q_1) \beta^2 \Bigg(\Bigg. \mE_{G,{\mathcal U}_3}\langle \|\x^{(i_1)}\|_2^2\|\y^{(i_2)}\|_2^2\rangle_{\gamma_{01}^{(2)}}\nonumber \\
 & & +  (s-1)\mE_{G,{\mathcal U}_3}\langle \|\x^{(i_1)}\|_2^2 \|\y^{(i_2)}\|_2\|\y^{(p_2)}\|_2\rangle_{\gamma_{02}^{(2)}}\Bigg.\Bigg) \nonumber \\
& & - (\p_0\q_0-\p_1\q_1)s\beta^2(1-\m_1)\mE_{G,{\mathcal U}_3}\langle \|\x^{(i_1)}\|_2\|\x^{(p_`)}\|_2\|\y^{(i_2)}\|_2\|\y^{(p_2)}\|_2 \rangle_{\gamma_{1}^{(2)}}. \nonumber \\
\end{eqnarray}

%%%%%%%%%%%%%%%%%%%%%%%%%%%%%%%%%%%%%%%%%%%%%%%%%%%%%%%%%%%%%%%%%%%%%%%%
%%%%%%%%%%%%%%%%%%%%%%%%%%%%%%%%%%%%%%%%%%%%%%%%%%%%%%%%%%%%%%%%%%%%%%%%
\subsubsection{Handling $T_k$--groups --- $k=2$}
\label{sec:lev2handlT2}
%%%%%%%%%%%%%%%%%%%%%%%%%%%%%%%%%%%%%%%%%%%%%%%%%%%%%%%%%%%%%%%%%%%%%%%%
%%%%%%%%%%%%%%%%%%%%%%%%%%%%%%%%%%%%%%%%%%%%%%%%%%%%%%%%%%%%%%%%%%%%%%%%

As usual, we handle separately each of the three terms contributing to  $T_2$ group.

%%%%%%%%%%%%%%%%%%%%%%%%%%%%%%%%%%%%%%%%%%%%%%%%%%%%%%%%%%%%%%%%%%%%%%%%
\underline{\textbf{\emph{Determining}} $T_{2,1,j}$}
\label{sec:lev2hand1T21}
%%%%%%%%%%%%%%%%%%%%%%%%%%%%%%%%%%%%%%%%%%%%%%%%%%%%%%%%%%%%%%%%%%%%%%%%

Relying on the Gaussian integration by parts we write
\begin{eqnarray}\label{eq:lev2genDanal19}
T_{2,1,j}& = &  \mE_{G,{\mathcal U}_3,{\mathcal U}_2} \lp\frac{\lp\mE_{{\mathcal U}_1} \lp Z^{\m_1}\rp\rp^{\frac{\m_2}{\m_1}-1}}{\mE_{{\mathcal U}_2}\lp\lp\mE_{{\mathcal U}_1} \lp Z^{\m_1}\rp\rp^{\frac{\m_2}{\m_1}}\rp}   \mE_{{\mathcal U}_1}\frac{(C^{(i_1)})^{s-1} A^{(i_1,i_2)} \y_j^{(i_2)}\u_j^{(2,2)}}{Z^{1-\m_1}} \rp \nonumber \\
& = &  \mE_{G,{\mathcal U}_3,{\mathcal U}_2,{\mathcal U}_1} \lp\frac{1}{\mE_{{\mathcal U}_2}\lp\lp\mE_{{\mathcal U}_1} \lp Z^{\m_1}\rp\rp^{\frac{\m_2}{\m_1}}\rp}       \frac{(C^{(i_1)})^{s-1} A^{(i_1,i_2)} \y_j^{(i_2)}\u_j^{(2,2)}}{Z^{1-\m_1}\lp\mE_{{\mathcal U}_1} \lp Z^{\m_1}\rp\rp^{1-\frac{\m_2}{\m_1}}} \rp \nonumber \\
& = &  \mE_{G,{\mathcal U}_3,{\mathcal U}_1} \lp\frac{1}{\mE_{{\mathcal U}_2}\lp\lp\mE_{{\mathcal U}_1} \lp Z^{\m_1}\rp\rp^{\frac{\m_2}{\m_1}}\rp}     \mE_{{\mathcal U}_2}  \frac{(C^{(i_1)})^{s-1} A^{(i_1,i_2)} \y_j^{(i_2)}\u_j^{(2,2)}}{Z^{1-\m_1}\lp\mE_{{\mathcal U}_1} \lp Z^{\m_1}\rp\rp^{\frac{1-\m_2}{\m_1}}} \rp \nonumber \\
 & = &
\mE_{G,{\mathcal U}_3,{\mathcal U}_1} \lp\frac{1}{\mE_{{\mathcal U}_2}\lp\lp\mE_{{\mathcal U}_1} \lp Z^{\m_1}\rp\rp^{\frac{\m_2}{\m_1}}\rp}  \mE_{{\mathcal U}_2}\lp\mE_{{\mathcal U}_2} (\u_j^{(2,2)}\u_j^{(2,2)})\frac{d}{d\u_j^{(2,2)}}\lp \frac{(C^{(i_1)})^{s-1} A^{(i_1,i_2)}\y_j^{(i_2)}}{Z^{1-\m_1}\lp\mE_{{\mathcal U}_1} \lp Z^{\m_1}\rp\rp^{1-\frac{\m_2}{\m_1}}}\rp\rp\rp \nonumber \\
& = &
 (\p_1-\p_2) \mE_{G,{\mathcal U}_3,{\mathcal U}_2,{\mathcal U}_1} \lp\frac{\lp\mE_{{\mathcal U}_1} \lp Z^{\m_1}\rp\rp^{\frac{\m_2}{\m_1}-1}}{\mE_{{\mathcal U}_2}\lp\lp\mE_{{\mathcal U}_1} \lp Z^{\m_1}\rp\rp^{\frac{\m_2}{\m_1}}\rp} \frac{d}{d\u_j^{(2,2)}}\lp \frac{(C^{(i_1)})^{s-1} A^{(i_1,i_2)}\y_j^{(i_2)}}{Z^{1-\m_1}}\rp\rp \nonumber \\
& & +
 (\p_1-\p_2)\mE_{G,{\mathcal U}_3,{\mathcal U}_2,{\mathcal U}_1} \lp\frac{Z^{\m_1-1}(C^{(i_1)})^{s-1} A^{(i_1,i_2)}\y_j^{(i_2)}}{\mE_{{\mathcal U}_2}\lp\lp\mE_{{\mathcal U}_1} \lp Z^{\m_1}\rp\rp^{\frac{\m_2}{\m_1}}\rp}  \lp \frac{d}{d\u_j^{(2,2)}}\lp \frac{1}{\lp\mE_{{\mathcal U}_1} \lp Z^{\m_1}\rp\rp^{1-\frac{\m_2}{\m_1}}}\rp\rp\rp.
\end{eqnarray}
As in (\ref{eq:genDanal19a}), we split $T_{2,1,j}$ into two components
\begin{eqnarray}\label{eq:lev2genDanal19a}
T_{2,1,j}   =   T_{2,1,j}^{c} +  T_{2,1,j}^{d},
\end{eqnarray}
where
\begin{eqnarray}\label{eq:lev2genDanal19b}
T_{2,1,j}^c &  = &
 (\p_1-\p_2)\mE_{G,{\mathcal U}_3,{\mathcal U}_2,{\mathcal U}_1} \lp\frac{\lp\mE_{{\mathcal U}_1} \lp Z^{\m_1}\rp\rp^{\frac{\m_2}{\m_1}-1}}{\mE_{{\mathcal U}_2}\lp\lp\mE_{{\mathcal U}_1} \lp Z^{\m_1}\rp\rp^{\frac{\m_2}{\m_1}}\rp} \frac{d}{d\u_j^{(2,2)}}\lp \frac{(C^{(i_1)})^{s-1} A^{(i_1,i_2)}\y_j^{(i_2)}}{Z^{1-\m_1}}\rp\rp \nonumber \\
T_{2,1,j}^d &  = & (\p_1-\p_2) \mE_{G,{\mathcal U}_3,{\mathcal U}_2,{\mathcal U}_1} \lp\frac{Z^{\m_1-1}(C^{(i_1)})^{s-1} A^{(i_1,i_2)}\y_j^{(i_2)}}{\mE_{{\mathcal U}_2}\lp\lp\mE_{{\mathcal U}_1} \lp Z^{\m_1}\rp\rp^{\frac{\m_2}{\m_1}}\rp}  \lp \frac{d}{d\u_j^{(2,2)}}\lp \frac{1}{\lp\mE_{{\mathcal U}_1} \lp Z^{\m_1}\rp\rp^{1-\frac{\m_2}{\m_1}}}\rp\rp\rp. \nonumber \\
\end{eqnarray}
One now observes that $T_{2,1,j}^c$ scaled by $(\p_1-\p_2)$ is structurally identical to the term considered in the first part of Section \ref{sec:lev2hand1T11} scaled by $(\p_0-\p_1)$. That means that we have
\begin{eqnarray}\label{eq:lev2genDanal19b1}
\sum_{i_1=1}^{l}\sum_{i_2=1}^{l}\sum_{j=1}^{m} \beta_{i_1}\frac{T_{2,1,j}^c}{\sqrt{1-t
}}
& = &  \sum_{i_1=1}^{l}\sum_{i_2=1}^{l}\sum_{j=1}^{m} \beta_{i_1}\frac{T_{1,1,j}}{\sqrt{1-t}}\frac{\p_1-\p_2}{\p_0-\p_1}\nonumber\\
& = & (\p_1-\p_2)\beta^2 \nonumber \\
& & \times \lp \mE_{G,{\mathcal U}_3}\langle \|\x^{(i_1)}\|_2^2\|\y^{(i_2)}\|_2^2\rangle_{\gamma_{1}^{(2)}} +  (s-1)\mE_{G,{\mathcal U}_3}\langle \|\x^{(i_1)}\|_2^2(\y^{(p_2)})^T\y^{(i_2)}\rangle_{\gamma_{1}^{(2)}} \rp \nonumber \\
& & - (\p_1-\p_2)s\beta^2(1-\m_1)\mE_{G,{\mathcal U}_3,{\mathcal U}_2}\langle \|\x^{(i_1)}\|_2\|\x^{(p_1)}\|_2(\y^{(p_2)})^T\y^{(i_2)} \rangle_{\gamma_{1}^{(2)}}.
\end{eqnarray}
We now focus on $T_{2,1,j}^d$. To that end we have
\begin{eqnarray}\label{eq:lev2genDanal20}
\lp \frac{d}{d\u_j^{(2,2)}}\lp \frac{1}{\lp\mE_{{\mathcal U}_1} \lp Z^{\m_1}\rp\rp^{1-\frac{\m_2}{\m_1}}}\rp\rp
=
  -\frac{1-\frac{\m_2}{\m_1}}{\lp\mE_{{\mathcal U}_1} \lp Z^{\m_1}\rp\rp^{2-\frac{\m_2}{\m_1}}}
\mE_{{\mathcal U}_1}\frac{d Z^{\m_1}}{d\u_j^{(2,2)}}.\nonumber \\
\end{eqnarray}
Utilizing (\ref{eq:genDanal21}), we write analogously to (\ref{eq:genDanal23})
\begin{eqnarray}\label{eq:lev2genDanal23}
T_{2,1,j}^d &  = & (\p_1-\p_2) \mE_{G,{\mathcal U}_3,{\mathcal U}_2,{\mathcal U}_1} \lp\frac{Z^{\m_1-1}(C^{(i_1)})^{s-1} A^{(i_1,i_2)}\y_j^{(i_2)}}{\mE_{{\mathcal U}_2}\lp\lp\mE_{{\mathcal U}_1} \lp Z^{\m_1}\rp\rp^{\frac{\m_2}{\m_1}}\rp}  \lp   -\frac{1-\frac{\m_2}{\m_1}}{\lp\mE_{{\mathcal U}_1} \lp Z^{\m_1}\rp\rp^{2-\frac{\m_2}{\m_1}}}
\mE_{{\mathcal U}_1}\frac{d Z^{\m_1}}{d\u_j^{(2,2)}}\rp\rp \nonumber \\
& = & -s\sqrt{1-t}(\p_1-\p_2)(\m_1-\m_2)\mE_{G,{\mathcal U}_3,{\mathcal U}_2,{\mathcal U}_1}\Bigg( \Bigg. \frac{Z^{\m_1-1}\lp(C^{(i_1)})^{s-1} A^{(i_1,i_2)}\y_j^{(i_2)} \rp}{\mE_{{\mathcal U}_2}\lp\lp\mE_{{\mathcal U}_1} \lp Z^{\m_1}\rp\rp^{\frac{\m_2}{\m_1}}\rp} \nonumber \\
& & \times
\lp \frac{1}{\lp\mE_{{\mathcal U}_1} Z^{\m_1}\rp^{2-\frac{\m_2}{\m_1}}}
\mE_{{\mathcal U}_1}  \frac{1}{Z^{1-\m_1}} \sum_{p_1=1}^{l}  (C^{(p_1)})^{s-1}\sum_{p_2=1}^{l}
\beta_{p_1}A^{(p_1,p_2)}\y_j^{(p_2)} \rp\Bigg. \Bigg)\nonumber \\
& = & -s\sqrt{1-t}(\p_1-\p_2)(\m_1-\m_2)\mE_{G,{\mathcal U}_3}\Bigg( \Bigg. \mE_{{\mathcal U}_2}
\frac{\lp\lp\mE_{{\mathcal U}_1} \lp Z^{\m_1}\rp\rp^{\frac{\m_2}{\m_1}}\rp}{\mE_{{\mathcal U}_2}\lp\lp\mE_{{\mathcal U}_1} \lp Z^{\m_1}\rp\rp^{\frac{\m_2}{\m_1}}\rp}
\mE_{{\mathcal U}_1}\frac{Z^{\m_1}}{\mE_{{\mathcal U}_1} Z^{\m_1}}
 \frac{(C^{(i_1)})^{s}}{Z}  \frac{A^{(i_1,i_2)}}{C^{(i_1)}}\y_j^{(i_2)} \nonumber \\
& & \times
\lp \mE_{{\mathcal U}_1}  \frac{Z^{\m_1}}{\mE_{{\mathcal U}_1} Z^{\m_1}} \sum_{p_1=1}^{l}  \frac{(C^{(p_1)})^s}{Z}\sum_{p_2=1}^{l}
\frac{A^{(p_1,p_2)}}{C^{(p,1)}}\beta_{p_1}\y_j^{(p_2)}  \rp\Bigg. \Bigg).
\end{eqnarray}
We then first observe
\begin{eqnarray}\label{eq:lev2genDanal24}
 \sum_{i_1=1}^{l}  \sum_{i_2=1}^{l} \sum_{j=1}^{m}  \beta_{i_1}\frac{T_{2,1,j}^d}{\sqrt{1-t}}
 & = & -s\beta^2(\p_1-\p_2)(\m_1-\m_2)\mE_{G,{\mathcal U}_3} \langle \|\x^{(i_1)}\|_2\|\x^{(p_1)}\|_2(\y^{(p_2)})^T\y^{(i_2)} \rangle_{\gamma_{2}^{(2)}},
\end{eqnarray}
and
\begin{eqnarray}\label{eq:lev2genDanal25}
 \sum_{i_1=1}^{l}  \sum_{i_2=1}^{l} \sum_{j=1}^{m}  \beta_{i_1}\frac{T_{2,1,j}}{\sqrt{1-t}}
 & = & (\p_1-\p_2)\beta^2 \nonumber \\
 & & \times
  \lp \mE_{G,{\mathcal U}_3}\langle \|\x^{(i_1)}\|_2^2\|\y^{(i_2)}\|_2^2\rangle_{\gamma_{01}^{(2)}} +   (s-1)\mE_{G,{\mathcal U}_3}\langle \|\x^{(i_1)}\|_2^2(\y^{(p_2)})^T\y^{(i_2)}\rangle_{\gamma_{02}^{(2)}} \rp\nonumber \\
& & - (\p_1-\p_2)s\beta^2(1-\m_1)\mE_{G,{\mathcal U}_3}\langle \|\x^{(i_1)}\|_2\|\x^{(p_1)}\|_2(\y^{(p_2)})^T\y^{(i_2)} \rangle_{\gamma_{1}^{(2)}}\nonumber \\
 &   &
  -s\beta^2(\p_1-\p_2)(\m_1-\m_2)\mE_{G,{\mathcal U}_3} \langle \|\x^{(i_1)}\|_2\|\x^{(p_1)}\|_2(\y^{(p_2)})^T\y^{(i_2)} \rangle_{\gamma_{2}^{(2)}}.
\end{eqnarray}

%%%%%%%%%%%%%%%%%%%%%%%%%%%%%%%%%%%%%%%%%%%%%%%%%%%%%%%%%%%%%%%%%%%%%%%%
\underline{\textbf{\emph{Determining}} $T_{2,2}$}
\label{sec:lev2hand1T22}
%%%%%%%%%%%%%%%%%%%%%%%%%%%%%%%%%%%%%%%%%%%%%%%%%%%%%%%%%%%%%%%%%%%%%%%%

As usual, we proceed by relying on the Gaussian integration by parts to obtain
\begin{eqnarray}\label{eq:lev2liftgenEanal20}
T_{2,2} & = &  \mE_{G,{\mathcal U}_3,{\mathcal U}_2} \lp\frac{\lp\mE_{{\mathcal U}_1} \lp Z^{\m_1}\rp\rp^{\frac{\m_2}{\m_1}-1}}{\mE_{{\mathcal U}_2}\lp\lp\mE_{{\mathcal U}_1} \lp Z^{\m_1}\rp\rp^{\frac{\m_2}{\m_1}}\rp}
\mE_{{\mathcal U}_1}\frac{(C^{(i_1)})^{s-1} A^{(i_1,i_2)} \u^{(i_1,3,2)}}{Z^{1-\m_1}} \rp \nonumber \\
& = &  \mE_{G,{\mathcal U}_3,{\mathcal U}_2,{\mathcal U}_1} \lp\frac{1}{\mE_{{\mathcal U}_2}\lp\lp\mE_{{\mathcal U}_1} \lp Z^{\m_1}\rp\rp^{\frac{\m_2}{\m_1}}\rp}
 \frac{(C^{(i_1)})^{s-1} A^{(i_1,i_2)} \u^{(i_1,3,2)}}{Z^{1-\m_1}\lp\mE_{{\mathcal U}_1} \lp Z^{\m_1}\rp\rp^{1-\frac{\m_2}{\m_1}}} \rp \nonumber \\
 & = & \mE_{G,{\mathcal U}_3,{\mathcal U}_1} \Bigg( \Bigg. \frac{1}{\mE_{{\mathcal U}_2}\lp\lp\mE_{{\mathcal U}_1} \lp Z^{\m_1}\rp\rp^{\frac{\m_2}{\m_1}}\rp}\nonumber \\
 & & \times
\mE_{{\mathcal U}_2} \lp \sum_{p_1=1}^{l}\mE_{{\mathcal U}_2}(\u^{(i_1,3,2)}\u^{(p_1,3,2)}) \frac{d}{d\u^{(p_1,3,2)}}\lp\frac{(C^{(i_1)})^{s-1} A^{(i_1,i_2)}}{Z^{1-\m_1}    \lp\mE_{{\mathcal U}_1} \lp Z^{\m_1}\rp\rp^{1-\frac{\m_2}{\m_1}}    }\rp\rp\Bigg. \Bigg) \nonumber \\
& = &   T_{2,2}^{c} +  T_{2,2}^{d},
 \end{eqnarray}
where
\begin{eqnarray}\label{eq:lev2genEanal19b}
T_{2,2}^c &  = &
\mE_{G,{\mathcal U}_2,{\mathcal U}_1} \lp
 \frac{\lp\mE_{{\mathcal U}_1} \lp Z^{\m_1}\rp\rp^{\frac{\m_2}{\m_1}-1}}{\mE_{{\mathcal U}_2}\lp\lp\mE_{{\mathcal U}_1} \lp Z^{\m_1}\rp\rp^{\frac{\m_2}{\m_1}}\rp} \sum_{p_1=1}^{l}\mE_{{\mathcal U}_2}(\u^{(i_1,3,2)}\u^{(p_1,3,2)}) \frac{d}{d\u^{(p_1,3,2)}}\lp\frac{(C^{(i_1)})^{s-1} A^{(i_1,i_2)}}{Z^{1-\m_1}}\rp\rp \nonumber \\
T_{2,2}^d &  = & \mE_{G,{\mathcal U}_2,{\mathcal U}_1} \lp \frac{Z^{\m_1-1}(C^{(i_1)})^{s-1} A^{(i_1,i_2)}}{\mE_{{\mathcal U}_2}\lp\lp\mE_{{\mathcal U}_1} \lp Z^{\m_1}\rp\rp^{\frac{\m_2}{\m_1}}\rp}
  \sum_{p_1=1}^{l}\mE_{{\mathcal U}_2}(\u^{(i_1,3,2)}\u^{(p_1,3,2)}) \frac{d}{d\u^{(p_1,3,2)}}\lp\frac{1}{\lp\mE_{{\mathcal U}_1} \lp Z^{\m_1}\rp\rp^{1-\frac{\m_2}{\m_1}}}\rp\rp.\nonumber \\
\end{eqnarray}
Since
\begin{eqnarray}\label{eq:lev2genEanal19c}
\mE_{{\mathcal U}_2}(\u^{(i_1,3,2)}\u^{(p_1,3,2)}) & = & (\q_1-\q_2)\frac{(\x^{(i_1)})^T\x^{(p_1)}}{\|\x^{(i_1)}\|_2\|\x^{(p_1)}\|_2}
=(\q_1-\q_2)\frac{\beta^2(\x^{(i_1)})^T\x^{(p_1)}}{\beta_{i_1}\beta_{p_1}}, \nonumber \\
\mE_{{\mathcal U}_2}(\u^{(i_1,3,1)}\u^{(p_1,3,1)}) & = & (\q_0-\q_1)\frac{(\x^{(i_1)})^T\x^{(p_1)}}{\|\x^{(i_1)}\|_2\|\x^{(p_1)}\|_2}=(\q_0-\q_1)\frac{\beta^2(\x^{(i_1)})^T\x^{(p_1)}}{\beta_{i_1}\beta_{p_1}},
\end{eqnarray}
one observes that $T_{2,2}^c$ scaled by $(\q_1-\q_2)$ is structurally identical to the term considered in the second part of Section \ref{sec:lev2hand1T11} scaled by $(\q_0-\q_1)$. Utilizing (\ref{eq:lev2liftgenBanal20b}) we then have
\begin{eqnarray}\label{eq:lev2genEanal19c1}
\sum_{i_1=1}^{l}\sum_{i_2=1}^{l} \beta_{i_1}\|\y^{(i_2)}\|_2 \frac{T_{2,2}^c}{\sqrt{1-t}} & = & (\q_1-\q_2)\beta^2\Bigg(\Bigg. \mE_{G,{\mathcal U}_3}\langle \|\x^{(i_1)}\|_2^2\|\y^{(i_2)}\|_2^2\rangle_{\gamma_{01}^{(2)}} \nonumber \\
& & +   (s-1)\mE_{G,{\mathcal U}_3}\langle \|\x^{(i_1)}\|_2^2 \|\y^{(i_2)}\|_2\|\y^{(p_2)}\|_2\rangle_{\gamma_{02}^{(2)}}\Bigg.\Bigg) \nonumber \\
& & - (\q_1-\q_2)s\beta^2(1-\m_1)\mE_{G,{\mathcal U}_3}\langle (\x^{(p_1)})^T\x^{(i_1)}\|\y^{(i_2)}\|_2\|\y^{(p_2)}\|_2 \rangle_{\gamma_{1}^{(2)}}.
\nonumber \\
\end{eqnarray}

To determine  $T_{2,2}^d$, we start with
\begin{eqnarray}\label{eq:lev2genEanal20}
\frac{d}{d\u^{(p_1,3,2)}}\lp\frac{1}{\lp\mE_{{\mathcal U}_1} \lp Z^{\m_1}\rp\rp^{1-\frac{\m_2}{\m_1}}}\rp
   & = &
 -\frac{1-\frac{\m_2}{\m_1}}{\lp\mE_{{\mathcal U}_1} Z^{\m_1} \rp^{2-\frac{\m_2}{\m_1}}}
\mE_{{\mathcal U}_1}\frac{dZ^{\m_1}}{d\u^{(p_1,3,2)}}.
\end{eqnarray}
Moreover, from (\ref{eq:genEanal21}),
\begin{eqnarray}\label{eq:lev2genEanal21}
\frac{dZ^{\m_1}}{d\u^{(p_1,3,2)}}  & = & \frac{\m_1}{Z^{1-\m_1}} s  (C^{(p_1)})^{s-1}\sum_{p_2=1}^{l}
\beta_{p_1}A^{(p_1,p_2)}\|\y^{(p_2)}\|_2\sqrt{1-t}.
\end{eqnarray}
Combining (\ref{eq:lev2genEanal20}) and (\ref{eq:lev2genEanal21}), we obtain
\begin{eqnarray}\label{eq:lev2genEanal22}
\frac{d}{d\u^{(p_1,3,2)}}\lp\frac{1}{\lp\mE_{{\mathcal U}_1} \lp Z^{\m_1}\rp\rp^{1-\frac{\m_2}{\m_1}}}\rp
  & = &
-\frac{1-\frac{\m_2}{\m_1}}{\lp\mE_{{\mathcal U}_1} Z^{\m_1} \rp^{2-\frac{\m_2}{\m_1}}}
 \frac{\m_1}{Z^{1-\m_1}} s  (C^{(p_1)})^{s-1}\sum_{p_2=1}^{l}
\beta_{p_1}A^{(p_1,p_2)}\|\y^{(p_2)}\|_2\sqrt{1-t}.\nonumber \\
\end{eqnarray}
A further combination of (\ref{eq:lev2genEanal19b}) and (\ref{eq:lev2genEanal22}) gives the following analogue to (\ref{eq:genEanal23})
\begin{eqnarray}\label{eq:lev2genEanal23}
 T_{2,2}^d  &  = & -s\sqrt{1-t}\beta^2(\q_1-\q_2)(\m_1-\m_2)\mE_{G,{\mathcal U}_3} \Bigg( \Bigg. \mE_{{\mathcal U}_2}\frac{\lp\mE_{{\mathcal U}_1} \lp Z^{\m_1}\rp\rp^{\frac{\m_2}{\m_1}}}{\mE_{{\mathcal U}_2}\lp\lp\mE_{{\mathcal U}_1} \lp Z^{\m_1}\rp\rp^{\frac{\m_2}{\m_1}}\rp}    \mE_{{\mathcal U}_1}\frac{Z^{\m_1}}{\mE_{{\mathcal U}_1} Z^{\m_1}}
 \frac{(C^{(i_1)})^{s-1} A^{(i_1,i_2)}}{Z} \nonumber \\
 & & \times
\mE_{{\mathcal U}_1}\frac{Z^{\m_1}}{\mE_{{\mathcal U}_1} Z^{\m_1}}
  \sum_{p_1=1}^{l}
 \frac{(C^{(p_1)})^{s}}{Z}  \sum_{p_2=1}^{l}
\frac{A^{(p_1,p_2)}}{(C^{(p_1)})}\|\y^{(p_2)}\|_2\frac{(\x^{(i_1)})^T\x^{(p_1)}}{\beta_{i_1}} \Bigg. \Bigg).
\end{eqnarray}
Similarly to what was done above when considering $T_{2,1,j}^d$, we here note
\begin{equation}\label{eq:lev2genEanal24}
\sum_{i_1=1}^{l}\sum_{i_2=1}^{l} \beta_{i_1}\|\y^{(i_2)}\|_2\frac{T_{2,2}^d}{\sqrt{1-t}}
 =
-s\beta^2(\q_1-\q_2)(\m_1-\m_2)\mE_{G,{\mathcal U}_3} \langle \|\y^{(i_2)}\|_2\|\y^{(p_2)}\|_2(\x^{(i_1)})^T\x^{(p_1)}\rangle_{\gamma_{2}^{(2)}}.
\end{equation}
A combination of (\ref{eq:lev2liftgenEanal20}), (\ref{eq:lev2genEanal19c1}), and
(\ref{eq:lev2genEanal24}) gives

\begin{eqnarray}\label{eq:lev2genEanal25}
\sum_{i_1=1}^{l}\sum_{i_2=1}^{l} \beta_{i_1}\|\y^{(i_2)}\|_2\frac{T_{2,2}}{\sqrt{1-t}}
& = & (\q_1-\q_2) \beta^2 \Bigg( \Bigg. \mE_{G,{\mathcal U}_3}\langle \|\x^{(i_1)}\|_2^2\|\y^{(i_2)}\|_2^2\rangle_{\gamma_{01}^{(2)}} \nonumber \\
& & +   (s-1)\mE_{G,{\mathcal U}_3}\langle \|\x^{(i_1)}\|_2^2 \|\y^{(i_2)}\|_2\|\y^{(p_2)}\|_2\rangle_{\gamma_{02}^{(2)}}\Bigg.\Bigg) \nonumber \\
& & - (\q_1-\q_2)s\beta^2(1-\m_1)\mE_{G,{\mathcal U}_3}\langle (\x^{(p_1)})^T\x^{(i_1)}\|\y^{(i_2)}\|_2\|\y^{(p_2)}\|_2 \rangle_{\gamma_{1}^{(2)}} \nonumber \\
&  & -s\beta^2(\q_1-\q_2)(\m_1-\m_2)\mE_{G,{\mathcal U}_3} \langle \|\y^{(i_2)}\|_2\|\y^{(p_2)}\|_2(\x^{(i_1)})^T\x^{(p_1)}\rangle.
\end{eqnarray}

%%%%%%%%%%%%%%%%%%%%%%%%%%%%%%%%%%%%%%%%%%%%%%%%%%%%%%%%%%%%%%%%%%%%%%%%
\underline{\textbf{\emph{Determining}} $T_{2,3}$}
\label{sec:lev2hand1T23}
%%%%%%%%%%%%%%%%%%%%%%%%%%%%%%%%%%%%%%%%%%%%%%%%%%%%%%%%%%%%%%%%%%%%%%%%

Gaussian integration by parts also gives
\begin{eqnarray}\label{eq:lev2genFanal21}
T_{2,3} & = &  \mE_{G,{\mathcal U}_3,{\mathcal U}_2} \lp\frac{\lp\mE_{{\mathcal U}_1} \lp Z^{\m_1}\rp\rp^{\frac{\m_2}{\m_1}-1}}{\mE_{{\mathcal U}_2}\lp\lp\mE_{{\mathcal U}_1} \lp Z^{\m_1}\rp\rp^{\frac{\m_2}{\m_1}}\rp}
    \mE_{{\mathcal U}_1}\frac{(C^{(i_1)})^{s-1} A^{(i_1,i_2)} u^{(4,2)}}{Z^{1-\m_1}} \rp \nonumber \\
& = & \mE_{G,{\mathcal U}_3,{\mathcal U}_1} \Bigg( \Bigg. \frac{1}{\mE_{{\mathcal U}_2}\lp\lp\mE_{{\mathcal U}_1} \lp Z^{\m_1}\rp\rp^{\frac{\m_2}{\m_1}}\rp} \nonumber \\
& & \times \mE_{{\mathcal U}_2} \lp\mE_{{\mathcal U}_2} (u^{(4,2)}u^{(4,2)})\lp\frac{d}{du^{(4,2)}} \lp\frac{(C^{(i_1)})^{s-1} A^{(i_1,i_2)}}{Z^{1-\m_1}\lp\mE_{{\mathcal U}_1} \lp Z^{\m_1}\rp\rp^{1-\frac{\m_2}{\m_1}}}\rp \rp\rp\Bigg. \Bigg) \nonumber \\
& = & (\p_1\q_1-\p_2\q_2)\mE_{G,{\mathcal U}_3,{\mathcal U}_2,{\mathcal U}_1} \lp \frac{\lp\mE_{{\mathcal U}_1} \lp Z^{\m_1}\rp\rp^{\frac{\m_2}{\m_1}-1}}{\mE_{{\mathcal U}_2}\lp\lp\mE_{{\mathcal U}_1} \lp Z^{\m_1}\rp\rp^{\frac{\m_2}{\m_1}}\rp} \lp\frac{d}{du^{(4,2)}} \lp\frac{(C^{(i_1)})^{s-1} A^{(i_1,i_2)}}{Z^{1-\m_1}}\rp\rp\rp \nonumber \\
& & + (\p_1\q_1-\p_2\q_2)\mE_{G,{\mathcal U}_3,{\mathcal U}_2,{\mathcal U}_1} \lp \frac{Z^{\m_1-1}(C^{(i_1)})^{s-1} A^{(i_1,i_2)}}{\mE_{{\mathcal U}_2}\lp\lp\mE_{{\mathcal U}_1} \lp Z^{\m_1}\rp\rp^{\frac{\m_2}{\m_1}}\rp}\lp\frac{d}{du^{(4,2)}} \lp\frac{1}{\lp\mE_{{\mathcal U}_1} \lp Z^{\m_1}\rp\rp^{1-\frac{\m_2}{\m_1}}} \rp\rp\rp.\nonumber \\
\end{eqnarray}
As usual, we find it useful to rewrite the above as
\begin{eqnarray}\label{eq:lev2genFanal22}
T_{2,3} & = & T_{2,3}^c+T_{2,3}^d,
\end{eqnarray}
where
\begin{eqnarray}\label{eq:lev2genFanal23}
T_{2,3}^c & = &  (\p_1\q_1-\p_2\q_2)\mE_{G,{\mathcal U}_3,{\mathcal U}_2,{\mathcal U}_1} \lp \frac{\lp\mE_{{\mathcal U}_1} \lp Z^{\m_1}\rp\rp^{\frac{\m_2}{\m_1}-1}}{\mE_{{\mathcal U}_2}\lp\lp\mE_{{\mathcal U}_1} \lp Z^{\m_1}\rp\rp^{\frac{\m_2}{\m_1}}\rp}
\lp\frac{d}{du^{(4,2)}} \lp\frac{(C^{(i_1)})^{s-1} A^{(i_1,i_2)}}{Z^{1-\m_1}}\rp\rp\rp \nonumber \\
 T_{2,3}^d & = & (\p_1\q_1-\p_2\q_2)\mE_{G,{\mathcal U}_3,{\mathcal U}_2,{\mathcal U}_1} \lp \frac{Z^{\m_1-1}(C^{(i_1)})^{s-1} A^{(i_1,i_2)}}{\mE_{{\mathcal U}_2}\lp\lp\mE_{{\mathcal U}_1} \lp Z^{\m_1}\rp\rp^{\frac{\m_2}{\m_1}}\rp}
 \lp\frac{d}{du^{(4,2)}} \lp\frac{1}{\lp\mE_{{\mathcal U}_1} \lp Z^{\m_1}\rp\rp^{1-\frac{\m_2}{\m_1}}} \rp\rp\rp. \nonumber \\
\end{eqnarray}
One again  observes that $\frac{T_{2,3}^c}{\p_1\q_1-\p_2\q_2}=\frac{T_{1,3}}{\p_0\q_0-\p_1\q_1}$, which together with (\ref{eq:lev2liftgenCanal21b}), gives the following analogue to (\ref{eq:genFanal23b})
\begin{eqnarray}\label{eq:lev2genFanal23b}
\sum_{i_1=1}^{l}\sum_{i_2=1}^{l} \beta_{i_1}\|\y^{(i_2)}\|_2 \frac{T_{2,3}^c}{\sqrt{t}} & = &
(\p_1\q_1-\p_2\q_2)\beta^2 \Bigg(\Bigg. \mE_{G,{\mathcal U}_3}\langle \|\x^{(i_1)}\|_2^2\|\y^{(i_2)}\|_2^2\rangle_{\gamma_{01}^{(2)}} \nonumber \\
& & +   (s-1)\mE_{G,{\mathcal U}_3}\langle \|\x^{(i_1)}\|_2^2 \|\y^{(i_2)}\|_2\|\y^{(p_2)}\|_2\rangle_{\gamma_{02}^{(2)}} \Bigg.\Bigg) \nonumber \\
& & - (\p_1\q_1-\p_2\q_2) s\beta^2(1-\m_1)\mE_{G,{\mathcal U}_3}\langle \|\x^{(i_1)}\|_2\|\x^{(p_`)}\|_2\|\y^{(i_2)}\|_2\|\y^{(p_2)}\|_2 \rangle_{\gamma_{1}^{(2)}}. \nonumber \\
\end{eqnarray}

To find $T_{2,3}^d$, we first write
\begin{eqnarray}\label{eq:lev2genFanal24}
\frac{d}{du^{(4,2)}} \lp\frac{1}{\lp\mE_{{\mathcal U}_1} \lp Z^{\m_1}\rp\rp^{1-\frac{\m_2}{\m_1}}} \rp & = &
 -\frac{1-\frac{\m_2}{\m_1}}{\lp\mE_{{\mathcal U}_1} Z^{\m_1}\rp^{2-\frac{\m_2}{\m_1}}} \mE_{{\mathcal U}_1} \frac{\m_1}{Z^{1-\m_1}}\frac{dZ}{d u^{(4,2)}}.\nonumber \\
\end{eqnarray}
We then also recall on (\ref{eq:genFanal25})
\begin{equation}\label{eq:lev2genFanal25}
\frac{dZ}{du^{(4,2)}}=\frac{d\sum_{p_1=1}^{l}  (C^{(p_1)})^s}{du^{(4,2)}}
=s \sum_{p_1=1}^{l} (C^{(p_1)})^{s-1}\sum_{p_2=1}^{l}\frac{d(A^{(p_1,p_2)})}{du^{(4,2)}}=s \sum_{p_1=1}^{l} (C^{(p_1)})^{s-1}\sum_{p_2=1}^{l}
\beta_{p_1}A^{(p_1,p_2)}\|\y^{(p_2)}\|_2\sqrt{t}.
\end{equation}
Connecting (\ref{eq:lev2genFanal23}), (\ref{eq:lev2genFanal24}) and (\ref{eq:lev2genFanal25}), we find
\begin{eqnarray}\label{eq:lev2genFanal27}
T_{2,3}^d   & = & -(\p_1\q_1-\p_2\q_2)\mE_{G,{\mathcal U}_3,{\mathcal U}_2,{\mathcal U}_1} \Bigg( \Bigg.
\frac{\lp\mE_{{\mathcal U}_1} \lp Z^{\m_1}\rp\rp^{\frac{\m_2}{\m_1}}}{\mE_{{\mathcal U}_2}\lp\lp\mE_{{\mathcal U}_1} \lp Z^{\m_1}\rp\rp^{\frac{\m_2}{\m_1}}\rp}
\frac{(C^{(i_1)})^{s-1} A^{(i_1,i_2)}}{Z^{1-\m_1}}\nonumber \\
& & \times\frac{1-\frac{-\m_2}{\m_1}}{\lp\mE_{{\mathcal U}_1} Z^{\m_1}\rp^2} \mE_{{\mathcal U}_1} \frac{\m_1}{Z^{1-\m_1}}s \sum_{p_1=1}^{l} (C^{(p_1)})^{s-1}\sum_{p_2=1}^{l}
\beta_{p_1}A^{(p_1,p_2)}\|\y^{(p_2)}\|_2\sqrt{t}\Bigg. \Bigg) \nonumber \\
& = & -s\sqrt{t}(\p_1\q_1-\p_2\q_2)(\m_1-\m_2)\mE_{G,{\mathcal U}_3} \Bigg( \Bigg. \mE_{{\mathcal U}_2} \frac{\lp\mE_{{\mathcal U}_1} \lp Z^{\m_1}\rp\rp^{\frac{\m_2}{\m_1}}}{\mE_{{\mathcal U}_2}\lp\lp\mE_{{\mathcal U}_1} \lp Z^{\m_1}\rp\rp^{\frac{\m_2}{\m_1}}\rp}
\mE_{{\mathcal U}_1} \frac{Z^{\m_1}}{\mE_{{\mathcal U}_1} Z^{\m_1}} \frac{(C^{(i_1)})^s}{Z}\frac{A^{(i_1,i_2)}}{(C^{(i_1)})}\nonumber \\
& & \times \mE_{{\mathcal U}_1} \frac{Z^{\m_1}}{\mE_{{\mathcal U}_1} Z^{\m_1}} \sum_{p_1=1}^{l} \frac{(C^{(p_1)})^{s}}{Z}\sum_{p_2=1}^{l}
\frac{A^{(p_1,p_2)}}{(C^{(p_1)})}\beta_{p_1}\|\y^{(p_2)}\|_2\Bigg. \Bigg).
\end{eqnarray}
As earlier, we first write
\begin{eqnarray}\label{eq:lev2genFanal28}
\sum_{i_1=1}^{l}\sum_{i_2=1}^{l} \beta_{i_1}\|\y^{(i_2)}\|_2\frac{T_{2,3}^d}{\sqrt{t}}
 & = & -s\beta^2(\p_1\q_1-\p_2\q_2)(\m_1-\m_2)\mE_{G,{\mathcal U}_3} \langle\|\x^{(i_2)}\|_2\|\x^{(p_2)}\|_2\|\y^{(i_2)}\|_2\|\y^{(p_2)}\rangle_{\gamma_{2}^{(2)}},\nonumber \\
\end{eqnarray}
and then combine it with (\ref{eq:lev2genFanal22}) and (\ref{eq:lev2genFanal23b}) to obtain the following analogue to (\ref{eq:genFanal29})
 \begin{eqnarray}\label{eq:lev2genFanal29}
\sum_{i_1=1}^{l}\sum_{i_2=1}^{l} \beta_{i_1}\|\y^{(i_2)}\|_2\frac{T_{2,3}}{\sqrt{t}}
& = &
(\p_1\q_1-\p_2\q_2)\beta^2 \Bigg( \Bigg. \mE_{G,{\mathcal U}_3}\langle \|\x^{(i_1)}\|_2^2\|\y^{(i_2)}\|_2^2\rangle_{\gamma_{01}^{(2)}} \nonumber \\
& & +  (s-1)\mE_{G,{\mathcal U}_3}\langle \|\x^{(i_1)}\|_2^2 \|\y^{(i_2)}\|_2\|\y^{(p_2)}\|_2\rangle_{\gamma_{02}^{(2)}}\Bigg.\Bigg) \nonumber \\
& & - (\p_1\q_1-\p_2\q_2) s\beta^2(1-\m_1)\mE_{G,{\mathcal U}_3}\langle \|\x^{(i_1)}\|_2\|\x^{(p_`)}\|_2\|\y^{(i_2)}\|_2\|\y^{(p_2)}\|_2 \rangle_{\gamma_{1}^{(2)}} \nonumber \\
&  & -s\beta^2(\p_1\q_1-\p_2\q_2)(\m_1-\m_2)\mE_{G,{\mathcal U}_3} \langle\|\x^{(i_2)}\|_2\|\x^{(p_2)}\|_2\|\y^{(i_2)}\|_2\|\y^{(p_2)}\rangle_{\gamma_{2}^{(2)}}.\nonumber \\
\end{eqnarray}

%%%%%%%%%%%%%%%%%%%%%%%%%%%%%%%%%%%%%%%%%%%%%%%%%%%%%%%%%%%%%%%%%%%%%%%%
%%%%%%%%%%%%%%%%%%%%%%%%%%%%%%%%%%%%%%%%%%%%%%%%%%%%%%%%%%%%%%%%%%%%%%%%
%%%%%%%%%%%%%%%%%%%%%%%%%%%%%%%%%%%%%%%%%%%%%%%%%%%%%%%%%%%%%%%%%%%%%%%%
%%%%%%%%%%%%%%%%%%%%%%%%%%%%%%%%%%%%%%%%%%%%%%%%%%%%%%%%%%%%%%%%%%%%%%%%
\subsubsection{Handling $T_k$--groups --- $k=3$}
\label{sec:lev2x3lev2handlTk2}
%%%%%%%%%%%%%%%%%%%%%%%%%%%%%%%%%%%%%%%%%%%%%%%%%%%%%%%%%%%%%%%%%%%%%%%%
%%%%%%%%%%%%%%%%%%%%%%%%%%%%%%%%%%%%%%%%%%%%%%%%%%%%%%%%%%%%%%%%%%%%%%%%
%%%%%%%%%%%%%%%%%%%%%%%%%%%%%%%%%%%%%%%%%%%%%%%%%%%%%%%%%%%%%%%%%%%%%%%%
%%%%%%%%%%%%%%%%%%%%%%%%%%%%%%%%%%%%%%%%%%%%%%%%%%%%%%%%%%%%%%%%%%%%%%%%

As usual, we handle separately each of the three terms contributing to  $T_3$ group.

%%%%%%%%%%%%%%%%%%%%%%%%%%%%%%%%%%%%%%%%%%%%%%%%%%%%%%%%%%%%%%%%%%%%%%%%
\underline{\textbf{\emph{Determining}} $T_{3,1,j}$}
\label{sec:lev2x3lev2hand1T21}
%%%%%%%%%%%%%%%%%%%%%%%%%%%%%%%%%%%%%%%%%%%%%%%%%%%%%%%%%%%%%%%%%%%%%%%%

Relying on the Gaussian integration by parts we write
\begin{eqnarray}\label{eq:lev2x3lev2genDanal19}
T_{3,1,j}& = &  \mE_{G,{\mathcal U}_3,{\mathcal U}_2} \lp\frac{\lp\mE_{{\mathcal U}_1} \lp Z^{\m_1}\rp\rp^{\frac{\m_2}{\m_1}-1}}{\mE_{{\mathcal U}_2}\lp\lp\mE_{{\mathcal U}_1} \lp Z^{\m_1}\rp\rp^{\frac{\m_2}{\m_1}}\rp}   \mE_{{\mathcal U}_1}\frac{(C^{(i_1)})^{s-1} A^{(i_1,i_2)} \y_j^{(i_2)}\u_j^{(2,3)}}{Z^{1-\m_1}} \rp \nonumber \\
& = &  \mE_{G,{\mathcal U}_2,{\mathcal U}_1} \lp \mE_{{\mathcal U}_3}       \frac{(C^{(i_1)})^{s-1} A^{(i_1,i_2)} \y_j^{(i_2)}\u_j^{(2,3)}}{Z^{1-\m_1}\lp\mE_{{\mathcal U}_1} \lp Z^{\m_1}\rp\rp^{1-\frac{\m_2}{\m_1}}\mE_{{\mathcal U}_2}\lp\lp\mE_{{\mathcal U}_1} \lp Z^{\m_1}\rp\rp^{\frac{\m_2}{\m_1}}\rp} \rp \nonumber \\
& = &  T_{3,1,j}^{c} +  T_{3,1,j}^{d},
 \end{eqnarray}
 where
\begin{eqnarray}\label{eq:lev2x3lev2genDanal19b}
T_{3,1,j}^c &  = &
\mE_{G,{\mathcal U}_3,{\mathcal U}_2,{\mathcal U}_1} \lp\frac{1}{\mE_{{\mathcal U}_2}\lp\lp\mE_{{\mathcal U}_1} \lp Z^{\m_1}\rp\rp^{\frac{\m_2}{\m_1}}\rp}  \lp\mE_{{\mathcal U}_2} (\u_j^{(2,3)}\u_j^{(2,3)})\frac{d}{d\u_j^{(2,3)}}\lp \frac{(C^{(i_1)})^{s-1} A^{(i_1,i_2)}\y_j^{(i_2)}}{Z^{1-\m_1}\lp\mE_{{\mathcal U}_1} \lp Z^{\m_1}\rp\rp^{1-\frac{\m_2}{\m_1}}}\rp\rp\rp \nonumber \\
T_{3,1,j}^d &  = & \mE_{G,{\mathcal U}_3,{\mathcal U}_2,{\mathcal U}_1} \lp   \frac{(C^{(i_1)})^{s-1} A^{(i_1,i_2)}\y_j^{(i_2)}}{Z^{1-\m_1}\lp\mE_{{\mathcal U}_1} \lp Z^{\m_1}\rp\rp^{1-\frac{\m_2}{\m_1}}}
 \lp\mE_{{\mathcal U}_2} (\u_j^{(2,3)}\u_j^{(2,3)})\frac{d}{d\u_j^{(2,3)}}
\lp \frac{1}{\mE_{{\mathcal U}_2}\lp\lp\mE_{{\mathcal U}_1} \lp Z^{\m_1}\rp\rp^{\frac{\m_2}{\m_1}}\rp}  \rp\rp\rp.\nonumber \\
\end{eqnarray}
Since $\mE_{{\mathcal U}_2} (\u_j^{(2,3)}\u_j^{(2,3)})=\p_2$, one now observes that $T_{3,1,j}^c$ scaled by $\p_2$ is structurally identical to $T_{2,1,j}^c$ scaled by $(\p_1-\p_2)$. That means that we have
\begin{eqnarray}\label{eq:lev2x3lev2genDanal19b1}
\sum_{i_1=1}^{l}\sum_{i_2=1}^{l}\sum_{j=1}^{m} \beta_{i_1}\frac{T_{3,1,j}^c}{\sqrt{1-t
}}
& = &  \sum_{i_1=1}^{l}\sum_{i_2=1}^{l}\sum_{j=1}^{m} \beta_{i_1}\frac{T_{2,1,j}}{\sqrt{1-t}}\frac{\p_2}{\p_1-\p_2}\nonumber\\
& = & \p_2 \beta^2 \lp \mE_{G,{\mathcal U}_3}\langle \|\x^{(i_1)}\|_2^2\|\y^{(i_2)}\|_2^2\rangle_{\gamma_{01}^{(2)}} +   (s-1)\mE_{G,{\mathcal U}_3}\langle \|\x^{(i_1)}\|_2^2(\y^{(p_2)})^T\y^{(i_2)}\rangle_{\gamma_{02}^{(2)}} \rp\nonumber \\
& & - \p_2s\beta^2(1-\m_1)\mE_{G,{\mathcal U}_3}\langle \|\x^{(i_1)}\|_2\|\x^{(p_1)}\|_2(\y^{(p_2)})^T\y^{(i_2)} \rangle_{\gamma_{1}^{(2)}}\nonumber \\
 &   &
  -s\beta^2\p_2(\m_1-\m_2)\mE_{G,{\mathcal U}_3} \langle \|\x^{(i_1)}\|_2\|\x^{(p_1)}\|_2(\y^{(p_2)})^T\y^{(i_2)} \rangle_{\gamma_{2}^{(2)}}.
\end{eqnarray}

We now focus on $T_{3,1,j}^d$. To that end we have
\begin{eqnarray}\label{eq:lev2x3lev2genDanal20}
\frac{d}{d\u_j^{(2,3)}} \lp \frac{1}{\mE_{{\mathcal U}_2}\lp\lp\mE_{{\mathcal U}_1} \lp Z^{\m_1}\rp\rp^{\frac{\m_2}{\m_1}}\rp}\rp
& = & -\frac{1}{\lp\mE_{{\mathcal U}_2}\lp\lp\mE_{{\mathcal U}_1} \lp Z^{\m_1}\rp\rp^{\frac{\m_2}{\m_1}}\rp\rp^2}
\mE_{{{\mathcal U}_2}}\frac{d\lp\lp\mE_{{\mathcal U}_1} \lp Z^{\m_1}\rp\rp^{\frac{\m_2}{\m_1}}\rp}{d\u_j^{(2,3)}}\nonumber\\
& = & -\frac{1}{\lp\mE_{{\mathcal U}_2}\lp\lp\mE_{{\mathcal U}_1} \lp Z^{\m_1}\rp\rp^{\frac{\m_2}{\m_1}}\rp\rp^2}
\mE_{{{\mathcal U}_2}}
\frac{\m_2}{\m_1}   \lp\mE_{{\mathcal U}_1} \lp Z^{\m_1}\rp\rp^{\frac{\m_2}{\m_1}-1}
\mE_{{\mathcal U}_1}\frac{d Z^{\m_1}}{d\u_j^{(2,3)}}\nonumber\\
& = & -\frac{1}{\lp\mE_{{\mathcal U}_2}\lp\lp\mE_{{\mathcal U}_1} \lp Z^{\m_1}\rp\rp^{\frac{\m_2}{\m_1}}\rp\rp^2}
\mE_{{{\mathcal U}_2}}
\frac{\m_2}{\m_1}   \lp\mE_{{\mathcal U}_1} \lp Z^{\m_1}\rp\rp^{\frac{\m_2}{\m_1}-1}
\mE_{{\mathcal U}_1}\frac{d Z^{\m_1}}{d\u_j^{(2,3)}}.\nonumber \\
\end{eqnarray}
From (\ref{eq:genDanal21}), we have
\begin{eqnarray}\label{eq:lev2x3lev2genDanal20a}
\frac{dZ^{\m_1}}{d\u_j^{(2,2)}}
& = & \frac{\m_1}{Z^{1-\m_1}}s  \sum_{p_1=1}^{l}  (C^{(p_1)})^{s-1}\sum_{p_2=1}^{l}
\beta_{p_1}A^{(p_1,p_2)}\y_j^{(p_2)}\sqrt{1-t}.
\end{eqnarray}
Combining (\ref{eq:lev2x3lev2genDanal19b}), (\ref{eq:lev2x3lev2genDanal20}), and (\ref{eq:lev2x3lev2genDanal20a}) we find
\begin{eqnarray}\label{eq:lev2x3lev2genDanal21}
T_{3,1,j}^d &  = & -\p_2 \mE_{G,{\mathcal U}_3,{\mathcal U}_2,{\mathcal U}_1} \Bigg( \Bigg.   \frac{(C^{(i_1)})^{s-1} A^{(i_1,i_2)}\y_j^{(i_2)}}{Z^{1-\m_1}\lp\mE_{{\mathcal U}_1} \lp Z^{\m_1}\rp\rp^{1-\frac{\m_2}{\m_1}}}
  \Bigg( \Bigg.  \frac{1}{\lp\mE_{{\mathcal U}_2}\lp\lp\mE_{{\mathcal U}_1} \lp Z^{\m_1}\rp\rp^{\frac{\m_2}{\m_1}}\rp\rp^2} \nonumber \\
& & \times \mE_{{{\mathcal U}_2}}
\frac{\m_2}{\m_1}   \lp\mE_{{\mathcal U}_1} \lp Z^{\m_1}\rp\rp^{\frac{\m_2}{\m_1}-1}
\mE_{{\mathcal U}_1} \frac{\m_1}{Z^{1-\m_1}}s  \sum_{p_1=1}^{l}  (C^{(p_1)})^{s-1}\sum_{p_2=1}^{l}
\beta_{p_1}A^{(p_1,p_2)}\y_j^{(p_2)}\sqrt{1-t}   \Bigg. \Bigg) \Bigg.\Bigg).\nonumber \\
&  = & -s\sqrt{1-t}\m_2\p_2 \mE_{G,{\mathcal U}_3,{\mathcal U}_2,{\mathcal U}_1} \Bigg( \Bigg.   \frac{(C^{(i_1)})^{s-1} A^{(i_1,i_2)}\y_j^{(i_2)}}{Z^{1-\m_1}\lp\mE_{{\mathcal U}_1} \lp Z^{\m_1}\rp\rp^{1-\frac{\m_2}{\m_1}}}
  \Bigg( \Bigg.  \frac{1}{\lp\mE_{{\mathcal U}_2}\lp\lp\mE_{{\mathcal U}_1} \lp Z^{\m_1}\rp\rp^{\frac{\m_2}{\m_1}}\rp\rp} \nonumber \\
& & \times \mE_{{{\mathcal U}_2}}
    \frac{\lp\mE_{{\mathcal U}_1} \lp Z^{\m_1}\rp\rp^{\frac{\m_2}{\m_1}}}{\mE_{{\mathcal U}_2}\lp\lp\mE_{{\mathcal U}_1} \lp Z^{\m_1}\rp\rp^{\frac{\m_2}{\m_1}}\rp}
\mE_{{\mathcal U}_1} \frac{Z^{\m_1}}{\mE_{{\mathcal U}_1} \lp Z^{\m_1}\rp}  \sum_{p_1=1}^{l}  \frac{(C^{(p_1)})^{s}}{Z}\sum_{p_2=1}^{l}
\frac{A^{(p_1,p_2)}}{C^{(p_1)}}\beta_{p_1}\y_j^{(p_2)}   \Bigg. \Bigg) \Bigg.\Bigg)\nonumber \\
&  = & -s\sqrt{1-t}\m_2\p_2 \mE_{G,{\mathcal U}_3} \Bigg( \Bigg.   \mE_{{{\mathcal U}_2}}
    \frac{\lp\mE_{{\mathcal U}_1} \lp Z^{\m_1}\rp\rp^{\frac{\m_2}{\m_1}}}{\mE_{{\mathcal U}_2}\lp\lp\mE_{{\mathcal U}_1} \lp Z^{\m_1}\rp\rp^{\frac{\m_2}{\m_1}}\rp}
\mE_{{\mathcal U}_1} \frac{Z^{\m_1}}{\mE_{{\mathcal U}_1} \lp Z^{\m_1}\rp}  \frac{(C^{(i_1)})^{s}}{Z}
\frac{A^{(i_1,i_2)}}{C^{(i_1)}}\y_j^{(i_2)}  \nonumber \\
& & \times \mE_{{{\mathcal U}_2}}
    \frac{\lp\mE_{{\mathcal U}_1} \lp Z^{\m_1}\rp\rp^{\frac{\m_2}{\m_1}}}{\mE_{{\mathcal U}_2}\lp\lp\mE_{{\mathcal U}_1} \lp Z^{\m_1}\rp\rp^{\frac{\m_2}{\m_1}}\rp}
\mE_{{\mathcal U}_1} \frac{Z^{\m_1}}{\mE_{{\mathcal U}_1} \lp Z^{\m_1}\rp}  \sum_{p_1=1}^{l}  \frac{(C^{(p_1)})^{s}}{Z}\sum_{p_2=1}^{l}
\frac{A^{(p_1,p_2)}}{C^{(p_1)}}\beta_{p_1}\y_j^{(p_2)}   \Bigg. \Bigg).
\end{eqnarray}
We then first observe
\begin{eqnarray}\label{eq:lev2x3lev2genDanal24}
 \sum_{i_1=1}^{l}  \sum_{i_2=1}^{l} \sum_{j=1}^{m}  \beta_{i_1}\frac{T_{3,1,j}^d}{\sqrt{1-t}}
 & = & -s\beta^2\p_2\m_2\mE_{G,{\mathcal U}_3} \langle \|\x^{(i_1)}\|_2\|\x^{(p_1)}\|_2(\y^{(p_2)})^T\y^{(i_2)} \rangle_{\gamma_{3}^{(2)}},
\end{eqnarray}
and
\begin{eqnarray}\label{eq:lev2x3lev2genDanal25}
 \sum_{i_1=1}^{l}  \sum_{i_2=1}^{l} \sum_{j=1}^{m}  \beta_{i_1}\frac{T_{3,1,j}}{\sqrt{1-t}}
 & = &  \sum_{i_1=1}^{l}\sum_{i_2=1}^{l}\sum_{j=1}^{m} \beta_{i_1}\frac{T_{2,1,j}}{\sqrt{1-t}}\frac{\p_2}{\p_1-\p_2}\nonumber\\
& = & \p_2 \beta^2 \lp \mE_{G,{\mathcal U}_3}\langle \|\x^{(i_1)}\|_2^2\|\y^{(i_2)}\|_2^2\rangle_{\gamma_{01}^{(2)}} +  (s-1)\mE_{G,{\mathcal U}_3}\langle \|\x^{(i_1)}\|_2^2(\y^{(p_2)})^T\y^{(i_2)}\rangle_{\gamma_{02}^{(2)}} \rp\nonumber \\
& & - \p_2s\beta^2(1-\m_1)\mE_{G,{\mathcal U}_3}\langle \|\x^{(i_1)}\|_2\|\x^{(p_1)}\|_2(\y^{(p_2)})^T\y^{(i_2)} \rangle_{\gamma_{1}^{(2)}}\nonumber \\
 &   &
  -s\beta^2\p_2(\m_1-\m_2)\mE_{G,{\mathcal U}_3} \langle \|\x^{(i_1)}\|_2\|\x^{(p_1)}\|_2(\y^{(p_2)})^T\y^{(i_2)} \rangle_{\gamma_{2}^{(2)}}\nonumber \\
 &  & -s\beta^2\p_2\m_2\mE_{G,{\mathcal U}_3} \langle \|\x^{(i_1)}\|_2\|\x^{(p_1)}\|_2(\y^{(p_2)})^T\y^{(i_2)} \rangle_{\gamma_{3}^{(2)}}.
\end{eqnarray}

%%%%%%%%%%%%%%%%%%%%%%%%%%%%%%%%%%%%%%%%%%%%%%%%%%%%%%%%%%%%%%%%%%%%%%%%
\underline{\textbf{\emph{Determining}} $T_{3,2}$}
\label{sec:lev2x3lev2hand1T22}
%%%%%%%%%%%%%%%%%%%%%%%%%%%%%%%%%%%%%%%%%%%%%%%%%%%%%%%%%%%%%%%%%%%%%%%%

As usual, we proceed by relying on the Gaussian integration by parts to obtain
\begin{eqnarray}\label{eq:lev2x3lev2liftgenEanal20}
T_{3,2} & = &  \mE_{G,{\mathcal U}_3,{\mathcal U}_2} \lp\frac{\lp\mE_{{\mathcal U}_1} \lp Z^{\m_1}\rp\rp^{\frac{\m_2}{\m_1}-1}}{\mE_{{\mathcal U}_2}\lp\lp\mE_{{\mathcal U}_1} \lp Z^{\m_1}\rp\rp^{\frac{\m_2}{\m_1}}\rp}
\mE_{{\mathcal U}_1}\frac{(C^{(i_1)})^{s-1} A^{(i_1,i_2)} \u^{(i_1,3,3)}}{Z^{1-\m_1}} \rp \nonumber \\
& = &  \mE_{G,{\mathcal U}_2,{\mathcal U}_1} \lp \mE_{{\mathcal U}_3}
 \frac{(C^{(i_1)})^{s-1} A^{(i_1,i_2)} \u^{(i_1,3,3)}}{Z^{1-\m_1}\lp\mE_{{\mathcal U}_1} \lp Z^{\m_1}\rp\rp^{1-\frac{\m_2}{\m_1}}\mE_{{\mathcal U}_2}\lp\lp\mE_{{\mathcal U}_1} \lp Z^{\m_1}\rp\rp^{\frac{\m_2}{\m_1}}\rp} \rp \nonumber \\
 & = &   T_{3,2}^{c} +  T_{3,2}^{d},
 \end{eqnarray}
where

{\small \begin{eqnarray}\label{eq:lev2x3lev2genEanal19b}
T_{3,2}^c &  = &
\mE_{G,{\mathcal U}_2,{\mathcal U}_1} \lp
 \frac{1}{\mE_{{\mathcal U}_2}\lp\lp\mE_{{\mathcal U}_1} \lp Z^{\m_1}\rp\rp^{\frac{\m_2}{\m_1}}\rp} \sum_{p_1=1}^{l}\mE_{{\mathcal U}_3}(\u^{(i_1,3,3)}\u^{(p_1,3,3)}) \frac{d}{d\u^{(p_1,3,3)}}\lp\frac{(C^{(i_1)})^{s-1} A^{(i_1,i_2)}}{Z^{1-\m_1}\lp\mE_{{\mathcal U}_1} \lp Z^{\m_1}\rp\rp^{1-\frac{\m_2}{\m_1}}}\rp\rp \nonumber \\
T_{3,2}^d &  = & \mE_{G,{\mathcal U}_2,{\mathcal U}_1} \lp \frac{(C^{(i_1)})^{s-1} A^{(i_1,i_2)}}{Z^{1-\m_1}\lp\mE_{{\mathcal U}_1} \lp Z^{\m_1}\rp\rp^{1-\frac{\m_2}{\m_1}}}
  \sum_{p_1=1}^{l}\mE_{{\mathcal U}_3}(\u^{(i_1,3,3)}\u^{(p_1,3,3)}) \frac{d}{d\u^{(p_1,3,3)}}\lp  \frac{1}{\mE_{{\mathcal U}_2}\lp\lp\mE_{{\mathcal U}_1} \lp Z^{\m_1}\rp\rp^{\frac{\m_2}{\m_1}}\rp}  \rp\rp.\nonumber \\
\end{eqnarray}}

\noindent Since
\begin{eqnarray}\label{eq:lev2x3lev2genEanal19c}
\mE_{{\mathcal U}_3}(\u^{(i_1,3,3)}\u^{(p_1,3,3)}) & = & \q_2\frac{(\x^{(i_1)})^T\x^{(p_1)}}{\|\x^{(i_1)}\|_2\|\x^{(p_1)}\|_2}
= \q_2\frac{\beta^2(\x^{(i_1)})^T\x^{(p_1)}}{\beta_{i_1}\beta_{p_1}},
\end{eqnarray}
one observes that $T_{3,2}^c$ scaled by $\q_2$ is structurally identical to $T_{2,2}$ scaled by $(\q_0-\q_1)$, which implies
\begin{eqnarray}\label{eq:lev2x3lev2genEanal19c1}
\sum_{i_1=1}^{l}\sum_{i_2=1}^{l} \beta_{i_1}\|\y^{(i_2)}\|_2 \frac{T_{3,2}^c}{\sqrt{1-t}} & = & 
\q_2 \beta^2 
\Bigg.\Bigg( \mE_{G,{\mathcal U}_3}\langle \|\x^{(i_1)}\|_2^2\|\y^{(i_2)}\|_2^2\rangle_{\gamma_{01}^{(2)}}
 \nonumber \\
 & & 
 + (s-1)\mE_{G,{\mathcal U}_3}\langle \|\x^{(i_1)}\|_2^2 \|\y^{(i_2)}\|_2\|\y^{(p_2)}\|_2\rangle_{\gamma_{02}^{(2)}}\Bigg.\Bigg)
\nonumber \\
& & - \q_2s\beta^2(1-\m_1)\mE_{G,{\mathcal U}_3}\langle (\x^{(p_1)})^T\x^{(i_1)}\|\y^{(i_2)}\|_2\|\y^{(p_2)}\|_2 \rangle_{\gamma_{1}^{(2)}} \nonumber \\
&  & -s\beta^2\q_2(\m_1-\m_2)\mE_{G,{\mathcal U}_3} \langle \|\y^{(i_2)}\|_2\|\y^{(p_2)}\|_2(\x^{(i_1)})^T\x^{(p_1)}\rangle_{\gamma_{2}^{(2)}}.
\end{eqnarray}

To determine  $T_{3,2}^d$, we start with
\begin{eqnarray}\label{eq:lev2x3lev2genEanal20}
\frac{d}{d\u_j^{(p_1,3,3)}} \lp \frac{1}{\mE_{{\mathcal U}_2}\lp\lp\mE_{{\mathcal U}_1} \lp Z^{\m_1}\rp\rp^{\frac{\m_2}{\m_1}}\rp}\rp
& = & -\frac{1}{\lp\mE_{{\mathcal U}_2}\lp\lp\mE_{{\mathcal U}_1} \lp Z^{\m_1}\rp\rp^{\frac{\m_2}{\m_1}}\rp\rp^2}
\mE_{{{\mathcal U}_2}}\frac{d\lp\lp\mE_{{\mathcal U}_1} \lp Z^{\m_1}\rp\rp^{\frac{\m_2}{\m_1}}\rp}{d\u_j^{(p_1,3,3)}}\nonumber\\
& = & -\frac{1}{\lp\mE_{{\mathcal U}_2}\lp\lp\mE_{{\mathcal U}_1} \lp Z^{\m_1}\rp\rp^{\frac{\m_2}{\m_1}}\rp\rp^2}\nonumber\\
& & \times
\mE_{{{\mathcal U}_2}}
\frac{\m_2}{\m_1}   \lp\mE_{{\mathcal U}_1} \lp Z^{\m_1}\rp\rp^{\frac{\m_2}{\m_1}-1}
\mE_{{\mathcal U}_1}\frac{d Z^{\m_1}}{d\u_j^{(p_1,3,3)}}\nonumber\\
& = & -\frac{1}{\lp\mE_{{\mathcal U}_2}\lp\lp\mE_{{\mathcal U}_1} \lp Z^{\m_1}\rp\rp^{\frac{\m_2}{\m_1}}\rp\rp^2}\nonumber\\
& & \times \mE_{{{\mathcal U}_2}}
\frac{\m_2}{\m_1}   \lp\mE_{{\mathcal U}_1} \lp Z^{\m_1}\rp\rp^{\frac{\m_2}{\m_1}-1}
\mE_{{\mathcal U}_1}\frac{d Z^{\m_1}}{d\u_j^{(p_1,3,3)}}.\nonumber \\
\end{eqnarray}
 Moreover,
\begin{eqnarray}\label{eq:lev2x3lev2genEanal21}
\frac{dZ^{\m_1}}{d\u^{(p_1,3,3)}}  & = & \frac{\m_1}{Z^{1-\m_1}} s  (C^{(p_1)})^{s-1}\sum_{p_2=1}^{l}
\beta_{p_1}A^{(p_1,p_2)}\|\y^{(p_2)}\|_2\sqrt{1-t}.
\end{eqnarray}
A combination of  (\ref{eq:lev2x3lev2genEanal19b}), (\ref{eq:lev2x3lev2genEanal20}), and (\ref{eq:lev2x3lev2genEanal21}) gives
{\small \begin{eqnarray}\label{eq:lev2x3lev2genEanal21b}
 T_{3,2}^d &  = & -\q_2\mE_{G,{\mathcal U}_3,{\mathcal U}_2,{\mathcal U}_1} \Bigg( \Bigg. \frac{(C^{(i_1)})^{s-1} A^{(i_1,i_2)}}{Z^{1-\m_1}\lp\mE_{{\mathcal U}_1} \lp Z^{\m_1}\rp\rp^{1-\frac{\m_2}{\m_1}}}
  \sum_{p_1=1}^{l}\frac{(\x^{(i_1)})^T\x^{(p_1)}}{\beta_{i_1}}
   \frac{1}{\lp\mE_{{\mathcal U}_2}\lp\lp\mE_{{\mathcal U}_1} \lp Z^{\m_1}\rp\rp^{\frac{\m_2}{\m_1}}\rp\rp^2}\nonumber\\
& & \times \mE_{{{\mathcal U}_2}}
\frac{\m_2}{\m_1}   \lp\mE_{{\mathcal U}_1} \lp Z^{\m_1}\rp\rp^{\frac{\m_2}{\m_1}-1}
\mE_{{\mathcal U}_1}\frac{d Z^{\m_1}}{d\u_j^{(p_1,3,3)}}\Bigg. \Bigg)\nonumber \\
&  = & -\q_2\mE_{G,{\mathcal U}_3,{\mathcal U}_2,{\mathcal U}_1} \Bigg( \Bigg. \frac{(C^{(i_1)})^{s-1} A^{(i_1,i_2)}}{Z^{1-\m_1}\lp\mE_{{\mathcal U}_1} \lp Z^{\m_1}\rp\rp^{1-\frac{\m_2}{\m_1}}}
  \sum_{p_1=1}^{l}\frac{(\x^{(i_1)})^T\x^{(p_1)}}{\beta_{i_1}}
   \frac{1}{\lp\mE_{{\mathcal U}_2}\lp\lp\mE_{{\mathcal U}_1} \lp Z^{\m_1}\rp\rp^{\frac{\m_2}{\m_1}}\rp\rp^2}\nonumber\\
& & \times \mE_{{{\mathcal U}_2}}
\frac{\m_2}{\m_1}   \lp\mE_{{\mathcal U}_1} \lp Z^{\m_1}\rp\rp^{\frac{\m_2}{\m_1}-1}
\mE_{{\mathcal U}_1}
\frac{\m_1}{Z^{1-\m_1}} s  (C^{(p_1)})^{s-1}\sum_{p_2=1}^{l}
\beta_{p_1}A^{(p_1,p_2)}\|\y^{(p_2)}\|_2\sqrt{1-t}\Bigg. \Bigg) \nonumber \\
&  = & -s\sqrt{1-t}\m_2\q_2 \mE_{G,{\mathcal U}_3} \Bigg( \Bigg.   \mE_{{{\mathcal U}_2}}
    \frac{\lp\mE_{{\mathcal U}_1} \lp Z^{\m_1}\rp\rp^{\frac{\m_2}{\m_1}}}{\mE_{{\mathcal U}_2}\lp\lp\mE_{{\mathcal U}_1} \lp Z^{\m_1}\rp\rp^{\frac{\m_2}{\m_1}}\rp}
\mE_{{\mathcal U}_1} \frac{Z^{\m_1}}{\mE_{{\mathcal U}_1} \lp Z^{\m_1}\rp}  \frac{(C^{(i_1)})^{s}}{Z}
\frac{A^{(i_1,i_2)}}{C^{(i_1)}}  \nonumber \\
& & \times \mE_{{{\mathcal U}_2}}
    \frac{\lp\mE_{{\mathcal U}_1} \lp Z^{\m_1}\rp\rp^{\frac{\m_2}{\m_1}}}{\mE_{{\mathcal U}_2}\lp\lp\mE_{{\mathcal U}_1} \lp Z^{\m_1}\rp\rp^{\frac{\m_2}{\m_1}}\rp}
\mE_{{\mathcal U}_1} \frac{Z^{\m_1}}{\mE_{{\mathcal U}_1} \lp Z^{\m_1}\rp}  \sum_{p_1=1}^{l}  \frac{(C^{(p_1)})^{s}}{Z}\sum_{p_2=1}^{l}
\frac{A^{(p_1,p_2)}}{C^{(p_1)}}\frac{(\x^{(i_1)})^T\x^{(p_1)}}{\beta_{i_1}}\|\y^{(p_2)}\|_2   \Bigg. \Bigg).
\end{eqnarray}}

\noindent We then first observe
\begin{equation}\label{eq:lev2x3lev2genEanal24}
\sum_{i_1=1}^{l}\sum_{i_2=1}^{l} \beta_{i_1}\|\y^{(i_2)}\|_2\frac{T_{3,2}^d}{\sqrt{1-t}}
 =
-s\beta^2 \q_2 \m_2\mE_{G,{\mathcal U}_3} \langle \|\y^{(i_2)}\|_2\|\y^{(p_2)}\|_2(\x^{(i_1)})^T\x^{(p_1)}\rangle_{\gamma_{3}^{(2)}},
\end{equation}
and after a combination of (\ref{eq:lev2x3lev2liftgenEanal20}), (\ref{eq:lev2x3lev2genEanal19c1}), and
(\ref{eq:lev2x3lev2genEanal24}) obtain
\begin{eqnarray}\label{eq:lev2x3lev2genEanal25}
\sum_{i_1=1}^{l}\sum_{i_2=1}^{l} \beta_{i_1}\|\y^{(i_2)}\|_2\frac{T_{3,2}}{\sqrt{1-t}}
& = & \q_2\beta^2
\Bigg(\Bigg. \mE_{G,{\mathcal U}_3}\langle \|\x^{(i_1)}\|_2^2\|\y^{(i_2)}\|_2^2\rangle_{\gamma_{01}^{(2)}}  \nonumber \\
& & +   (s-1)\mE_{G,{\mathcal U}_3}\langle \|\x^{(i_1)}\|_2^2 \|\y^{(i_2)}\|_2\|\y^{(p_2)}\|_2\rangle_{\gamma_{02}^{(2)}}\Bigg.\Bigg)\nonumber \\
& & - \q_2s\beta^2(1-\m_1)\mE_{G,{\mathcal U}_3}\langle (\x^{(p_1)})^T\x^{(i_1)}\|\y^{(i_2)}\|_2\|\y^{(p_2)}\|_2 \rangle_{\gamma_{1}^{(2)}} \nonumber \\
&  & -s\beta^2\q_2(\m_1-\m_2)\mE_{G,{\mathcal U}_3} \langle \|\y^{(i_2)}\|_2\|\y^{(p_2)}\|_2(\x^{(i_1)})^T\x^{(p_1)}\rangle_{\gamma_{2}^{(2)}} \nonumber \\
& & -s\beta^2 \q_2 \m_2\mE_{G,{\mathcal U}_3} \langle \|\y^{(i_2)}\|_2\|\y^{(p_2)}\|_2(\x^{(i_1)})^T\x^{(p_1)}\rangle_{\gamma_{3}^{(2)}}.
 \end{eqnarray}

%%%%%%%%%%%%%%%%%%%%%%%%%%%%%%%%%%%%%%%%%%%%%%%%%%%%%%%%%%%%%%%%%%%%%%%%
\underline{\textbf{\emph{Determining}} $T_{3,3}$}
\label{sec:lev2x3lev2hand1T23}
%%%%%%%%%%%%%%%%%%%%%%%%%%%%%%%%%%%%%%%%%%%%%%%%%%%%%%%%%%%%%%%%%%%%%%%%

Gaussian integration by parts also gives
\begin{eqnarray}\label{eq:lev2x3lev2genFanal21}
T_{3,3}
 & = &  \mE_{G,{\mathcal U}_3,{\mathcal U}_2} \lp\frac{\lp\mE_{{\mathcal U}_1} \lp Z^{\m_1}\rp\rp^{\frac{\m_2}{\m_1}-1}}{\mE_{{\mathcal U}_2}\lp\lp\mE_{{\mathcal U}_1} \lp Z^{\m_1}\rp\rp^{\frac{\m_2}{\m_1}}\rp}
    \mE_{{\mathcal U}_1}\frac{(C^{(i_1)})^{s-1} A^{(i_1,i_2)} u^{(4,3)}}{Z^{1-\m_1}} \rp \nonumber \\
 & = &  \mE_{G,{\mathcal U}_2,{\mathcal U}_1} \lp
    \mE_{{\mathcal U}_3}    \frac{(C^{(i_1)})^{s-1} A^{(i_1,i_2)} u^{(4,3)}}{Z^{1-\m_1}\lp\mE_{{\mathcal U}_1} \lp Z^{\m_1}\rp\rp^{1-\frac{\m_2}{\m_1}}\mE_{{\mathcal U}_2}\lp\lp\mE_{{\mathcal U}_1} \lp Z^{\m_1}\rp\rp^{\frac{\m_2}{\m_1}}\rp} \rp.
\end{eqnarray}
As usual, we find it useful to rewrite the above as
\begin{eqnarray}\label{eq:lev2x3lev2genFanal22}
T_{3,3} & = & T_{3,3}^c+T_{3,3}^d,
\end{eqnarray}
where
\begin{eqnarray}\label{eq:lev2x3lev2genFanal23}
T_{3,3}^c & = &  \p_2\q_2\mE_{G,{\mathcal U}_3,{\mathcal U}_2,{\mathcal U}_1} \lp
\frac{1}{\mE_{{\mathcal U}_2}\lp\lp\mE_{{\mathcal U}_1} \lp Z^{\m_1}\rp\rp^{\frac{\m_2}{\m_1}}\rp}
\lp\frac{d}{du^{(4,3)}} \lp \frac{(C^{(i_1)})^{s-1} A^{(i_1,i_2)} u^{(4,3)}}{Z^{1-\m_1}\lp\mE_{{\mathcal U}_1} \lp Z^{\m_1}\rp\rp^{1-\frac{\m_2}{\m_1}} }\rp\rp\rp \nonumber \\
 T_{3,3}^d & = & \p_2\q_2\mE_{G,{\mathcal U}_3,{\mathcal U}_2,{\mathcal U}_1} \lp
 \frac{(C^{(i_1)})^{s-1} A^{(i_1,i_2)} u^{(4,3)}}{Z^{1-\m_1}\lp\mE_{{\mathcal U}_1} \lp Z^{\m_1}\rp\rp^{1-\frac{\m_2}{\m_1}} }
 \lp\frac{d}{du^{(4,3)}} \lp  \frac{1}{\mE_{{\mathcal U}_2}\lp\lp\mE_{{\mathcal U}_1} \lp Z^{\m_1}\rp\rp^{\frac{\m_2}{\m_1}}\rp}
   \rp\rp\rp.
\end{eqnarray}
One again  observes that $\frac{T_{3,3}^c}{\p_2\q_2}=\frac{T_{2,3}}{\p_1\q_1-\p_2\q_2}$, which gives
\begin{eqnarray}\label{eq:lev2x3lev2genFanal23b}
\sum_{i_1=1}^{l}\sum_{i_2=1}^{l} \beta_{i_1}\|\y^{(i_2)}\|_2 \frac{T_{3,3}^c}{\sqrt{t}}
& = &
\p_2\q_2\beta^2
 \nonumber \\
 & & \times
 \lp \mE_{G,{\mathcal U}_3}\langle \|\x^{(i_1)}\|_2^2\|\y^{(i_2)}\|_2^2\rangle_{\gamma_{01}^{(2)}} +  (s-1)\mE_{G,{\mathcal U}_3}\langle \|\x^{(i_1)}\|_2^2 \|\y^{(i_2)}\|_2\|\y^{(p_2)}\|_2\rangle_{\gamma_{02}^{(2)}}\rp\nonumber \\
& & -  \p_2\q_2 s\beta^2(1-\m_1)\mE_{G,{\mathcal U}_3}\langle \|\x^{(i_1)}\|_2\|\x^{(p_`)}\|_2\|\y^{(i_2)}\|_2\|\y^{(p_2)}\|_2 \rangle_{\gamma_{1}^{(2)}} \nonumber \\
&  & -s\beta^2 \p_2\q_2(\m_1-\m_2)\mE_{G,{\mathcal U}_3} \langle\|\x^{(i_2)}\|_2\|\x^{(p_2)}\|_2\|\y^{(i_2)}\|_2\|\y^{(p_2)}\rangle_{\gamma_{2}^{(2)}}.
\end{eqnarray}

To find $T_{3,3}^d$, we first write
\begin{eqnarray}\label{eq:lev2x3lev2genFanal24}
\frac{d}{d u^{(4,3)}} \lp \frac{1}{\mE_{{\mathcal U}_2}\lp\lp\mE_{{\mathcal U}_1} \lp Z^{\m_1}\rp\rp^{\frac{\m_2}{\m_1}}\rp}\rp
 & = & -\frac{1}{\lp\mE_{{\mathcal U}_2}\lp\lp\mE_{{\mathcal U}_1} \lp Z^{\m_1}\rp\rp^{\frac{\m_2}{\m_1}}\rp\rp^2}\nonumber\\
& & \times \mE_{{{\mathcal U}_2}}
\frac{\m_2}{\m_1}   \lp\mE_{{\mathcal U}_1} \lp Z^{\m_1}\rp\rp^{\frac{\m_2}{\m_1}-1}
\mE_{{\mathcal U}_1}\frac{d Z^{\m_1}}{d u^{(4,3)}}.
\end{eqnarray}
We then also have
\begin{equation}\label{eq:lev2x3lev2genFanal25}
\frac{dZ^{\m_1}}{du^{(4,3)}} =\m_1Z^{\m_1-1}s \sum_{p_1=1}^{l} (C^{(p_1)})^{s-1}\sum_{p_2=1}^{l}
\beta_{p_1}A^{(p_1,p_2)}\|\y^{(p_2)}\|_2\sqrt{t},
\end{equation}
and
\begin{eqnarray}\label{eq:lev2x3lev2genFanal25a}
  T_{3,3}^d & = & -\p_2\q_2\mE_{G,{\mathcal U}_3,{\mathcal U}_2,{\mathcal U}_1} \Bigg( \Bigg.
 \frac{(C^{(i_1)})^{s-1} A^{(i_1,i_2)} u^{(4,3)}}{Z^{1-\m_1}\lp\mE_{{\mathcal U}_1} \lp Z^{\m_1}\rp\rp^{1-\frac{\m_2}{\m_1}} }
   \frac{1}{\lp\mE_{{\mathcal U}_2}\lp\lp\mE_{{\mathcal U}_1} \lp Z^{\m_1}\rp\rp^{\frac{\m_2}{\m_1}}\rp\rp^2}\nonumber\\
& & \times \mE_{{{\mathcal U}_2}}
\frac{\m_2}{\m_1}   \lp\mE_{{\mathcal U}_1} \lp Z^{\m_1}\rp\rp^{\frac{\m_2}{\m_1}-1}
\mE_{{\mathcal U}_1}\frac{d Z^{\m_1}}{d u^{(4,3)}} \Bigg. \Bigg) \nonumber \\
& = & -\p_2\q_2\mE_{G,{\mathcal U}_3,{\mathcal U}_2,{\mathcal U}_1} \Bigg( \Bigg.
 \frac{(C^{(i_1)})^{s-1} A^{(i_1,i_2)} u^{(4,3)}}{Z^{1-\m_1}\lp\mE_{{\mathcal U}_1} \lp Z^{\m_1}\rp\rp^{1-\frac{\m_2}{\m_1}} }
   \frac{1}{\lp\mE_{{\mathcal U}_2}\lp\lp\mE_{{\mathcal U}_1} \lp Z^{\m_1}\rp\rp^{\frac{\m_2}{\m_1}}\rp\rp^2}\nonumber\\
& & \times \mE_{{{\mathcal U}_2}}
\frac{\m_2}{\m_1}   \lp\mE_{{\mathcal U}_1} \lp Z^{\m_1}\rp\rp^{\frac{\m_2}{\m_1}-1}
\mE_{{\mathcal U}_1}
\m_1Z^{\m_1-1}s \sum_{p_1=1}^{l} (C^{(p_1)})^{s-1}\sum_{p_2=1}^{l}
\beta_{p_1}A^{(p_1,p_2)}\|\y^{(p_2)}\|_2\sqrt{t}
\Bigg. \Bigg) \nonumber \\
&  = & -s\sqrt{t}\m_2\q_2 \mE_{G,{\mathcal U}_3} \Bigg( \Bigg.   \mE_{{{\mathcal U}_2}}
    \frac{\lp\mE_{{\mathcal U}_1} \lp Z^{\m_1}\rp\rp^{\frac{\m_2}{\m_1}}}{\mE_{{\mathcal U}_2}\lp\lp\mE_{{\mathcal U}_1} \lp Z^{\m_1}\rp\rp^{\frac{\m_2}{\m_1}}\rp}
\mE_{{\mathcal U}_1} \frac{Z^{\m_1}}{\mE_{{\mathcal U}_1} \lp Z^{\m_1}\rp}  \frac{(C^{(i_1)})^{s}}{Z}
\frac{A^{(i_1,i_2)}}{C^{(i_1)}}  \nonumber \\
& & \times \mE_{{{\mathcal U}_2}}
    \frac{\lp\mE_{{\mathcal U}_1} \lp Z^{\m_1}\rp\rp^{\frac{\m_2}{\m_1}}}{\mE_{{\mathcal U}_2}\lp\lp\mE_{{\mathcal U}_1} \lp Z^{\m_1}\rp\rp^{\frac{\m_2}{\m_1}}\rp}
\mE_{{\mathcal U}_1} \frac{Z^{\m_1}}{\mE_{{\mathcal U}_1} \lp Z^{\m_1}\rp}  \sum_{p_1=1}^{l}  \frac{(C^{(p_1)})^{s}}{Z}\sum_{p_2=1}^{l}
\frac{A^{(p_1,p_2)}}{C^{(p_1)}} \beta_{p_1}\|\y^{(p_2)}\|_2   \Bigg. \Bigg).
\end{eqnarray}
As earlier, we  write
\begin{eqnarray}\label{eq:lev2x3lev2genFanal28}
\sum_{i_1=1}^{l}\sum_{i_2=1}^{l} \beta_{i_1}\|\y^{(i_2)}\|_2\frac{T_{3,3}^d}{\sqrt{t}}
 & = & -s\beta^2 \p_2\q_2 \m_2\mE_{G,{\mathcal U}_3} \langle\|\x^{(i_2)}\|_2\|\x^{(p_2)}\|_2\|\y^{(i_2)}\|_2\|\y^{(p_2)}\rangle_{\gamma_{3}^{(2)}},\nonumber \\
\end{eqnarray}
and then combine it with (\ref{eq:lev2x3lev2genFanal22}) and (\ref{eq:lev2x3lev2genFanal23b}) to obtain the following
\begin{eqnarray}\label{eq:lev2x3lev2genFanal29}
\sum_{i_1=1}^{l}\sum_{i_2=1}^{l} \beta_{i_1}\|\y^{(i_2)}\|_2\frac{T_{3,3}}{\sqrt{t}}
& = &
\p_2\q_2 \beta^2
 \nonumber \\
 & & \times
 \lp \mE_{G,{\mathcal U}_3}\langle \|\x^{(i_1)}\|_2^2\|\y^{(i_2)}\|_2^2\rangle_{\gamma_{01}^{(2)}} +   (s-1)\mE_{G,{\mathcal U}_3}\langle \|\x^{(i_1)}\|_2^2 \|\y^{(i_2)}\|_2\|\y^{(p_2)}\|_2\rangle_{\gamma_{02}^{(2)}}\rp\nonumber \\
& & -  \p_2\q_2 s\beta^2(1-\m_1)\mE_{G,{\mathcal U}_3}\langle \|\x^{(i_1)}\|_2\|\x^{(p_`)}\|_2\|\y^{(i_2)}\|_2\|\y^{(p_2)}\|_2 \rangle_{\gamma_{1}^{(2)}} \nonumber \\
&  & -s\beta^2 \p_2\q_2(\m_1-\m_2)\mE_{G,{\mathcal U}_3} \langle\|\x^{(i_2)}\|_2\|\x^{(p_2)}\|_2\|\y^{(i_2)}\|_2\|\y^{(p_2)}\rangle_{\gamma_{2}^{(2)}} \nonumber \\
 &  & -s\beta^2 \p_2\q_2 \m_2\mE_{G,{\mathcal U}_3} \langle\|\x^{(i_2)}\|_2\|\x^{(p_2)}\|_2\|\y^{(i_2)}\|_2\|\y^{(p_2)}\rangle_{\gamma_{3}^{(2)}}.
\end{eqnarray}

%%%%%%%%%%%%%%%%%%%%%%%%%%%%%%%%%%%%%%%%%%%%%%%%%%%%%%%%%%%%%%%%%%%%%%%%
%%%%%%%%%%%%%%%%%%%%%%%%%%%%%%%%%%%%%%%%%%%%%%%%%%%%%%%%%%%%%%%%%%%%%%%%
\subsubsection{Handling $T_G$--group}
\label{sec:lev2handlTG}
%%%%%%%%%%%%%%%%%%%%%%%%%%%%%%%%%%%%%%%%%%%%%%%%%%%%%%%%%%%%%%%%%%%%%%%%
%%%%%%%%%%%%%%%%%%%%%%%%%%%%%%%%%%%%%%%%%%%%%%%%%%%%%%%%%%%%%%%%%%%%%%%%

We start by recalling on
\begin{eqnarray}\label{eq:lev2genGanal1}
T_{G,j} & = &  \mE_{G,{\mathcal U}_3,{\mathcal U}_2} \lp\frac{\lp\mE_{{\mathcal U}_1} \lp Z^{\m_1}\rp\rp^{\frac{\m_2}{\m_1}-1}}{\mE_{{\mathcal U}_2}\lp\lp\mE_{{\mathcal U}_1} \lp Z^{\m_1}\rp\rp^{\frac{\m_2}{\m_1}}\rp}
  \mE_{{\mathcal U}_1}\frac{(C^{(i_1)})^{s-1} A^{(i_1,i_2)} \y_j^{(i_2)}\u_j^{(i_1,1)}}{Z^{1-\m_1}} \rp \nonumber \\
& = &  \mE_{{\mathcal U}_3,{\mathcal U}_2,{\mathcal U}_1} \lp
   \mE_{G}\frac{(C^{(i_1)})^{s-1} A^{(i_1,i_2)} \y_j^{(i_2)}\u_j^{(i_1,1)}}{Z^{1-\m_1}\lp\mE_{{\mathcal U}_1} \lp Z^{\m_1}\rp\rp^{1-\frac{\m_2}{\m_1}}  \mE_{{\mathcal U}_2}\lp\lp\mE_{{\mathcal U}_1} \lp Z^{\m_1}\rp\rp^{\frac{\m_2}{\m_1}}\rp} \rp \nonumber \\
& = &  \mE_{{\mathcal U}_3,{\mathcal U}_2,{\mathcal U}_1} \Bigg(\Bigg.
   \mE_{G}  \sum_{p_1=1}^{l} \mE_G (\u_j^{(i_1,1)}\u_j^{(p_1,1)}) \nonumber \\
    & & \times \frac{d}{d\u_j^{(p_1,1)}} \Bigg(\Bigg. \frac{(C^{(i_1)})^{s-1} A^{(i_1,i_2)} \y_j^{(i_2)}}{Z^{1-\m_1}       \lp\mE_{{\mathcal U}_1} \lp Z^{\m_1}\rp\rp^{1-\frac{\m_2}{\m_1}}     \mE_{{\mathcal U}_2}\lp\lp\mE_{{\mathcal U}_1} \lp Z^{\m_1}\rp\rp^{\frac{\m_2}{\m_1}}\rp} \Bigg.\Bigg) \Bigg.\Bigg). \nonumber \\
 \end{eqnarray}
After setting
\begin{eqnarray}\label{eq:lev2genGanal3}
\Theta_{G,1}^{(2)} & = & \sum_{p_1=1}^{l} \mE_G (\u_j^{(i_1,1)}\u_j^{(p_1,1)})\frac{d}{d\u_j^{(p_1,1)}}\lp \frac{(C^{(i_1)})^{s-1} A^{(i_1,i_2)}\y_j^{(i_2)}}{Z^{1-\m_1}}\rp \nonumber \\
\Theta_{G,2}^{(2)} & = &   \lp   \frac{(C^{(i_1)})^{s-1} A^{(i_1,i_2)}\y_j^{(i_2)}}{Z^{1-\m_1}        \mE_{{\mathcal U}_2}\lp\lp\mE_{{\mathcal U}_1} \lp Z^{\m_1}\rp\rp^{\frac{\m_2}{\m_1}}\rp}          \sum_{p_1=1}^{l} \mE_G (\u_j^{(i_1,1)}\u_j^{(p_1,1)})\frac{d}{d\u_j^{(p_1,1)}}\lp   \frac{1}{ \lp\mE_{{\mathcal U}_1} \lp Z^{\m_1}\rp\rp^{1-\frac{\m_2}{\m_1}} } \rp \rp \nonumber \\
\Theta_{G,3}^{(2)} & = &   \lp   \frac{(C^{(i_1)})^{s-1} A^{(i_1,i_2)}\y_j^{(i_2)}}{Z^{1-\m_1}    \lp\mE_{{\mathcal U}_1} \lp Z^{\m_1}\rp\rp^{1-\frac{\m_2}{\m_1}}       }          \sum_{p_1=1}^{l} \mE_G (\u_j^{(i_1,1)}\u_j^{(p_1,1)})\frac{d}{d\u_j^{(p_1,1)}}\lp   \frac{1}{
 \mE_{{\mathcal U}_2}\lp\lp\mE_{{\mathcal U}_1} \lp Z^{\m_1}\rp\rp^{\frac{\m_2}{\m_1}}\rp       } \rp \rp \nonumber \\
T_{G,j}^c & = &  \mE_{G,{\mathcal U}_3} \lp \mE_{{\mathcal U}_2}
 \frac{\lp\mE_{{\mathcal U}_1} \lp Z^{\m_1}\rp\rp^{\frac{\m_2}{\m_1}}}{
 \mE_{{\mathcal U}_2}\lp\lp   \mE_{{\mathcal U}_1} \lp Z^{\m_1}\rp\rp^{\frac{\m_2}{\m_1}}\rp       }
 \mE_{{\mathcal U}_1}\frac{1}{\mE_{{\mathcal U}_1} Z^{\m_1}} \Theta_{G,1}^{(2)}  \rp \nonumber \\
T_{G,j}^{d_1} & = &  \mE_{G,{\mathcal U}_3}\lp \mE_{{\mathcal U}_2}\mE_{{\mathcal U}_1}  \Theta_{G,2}^{(2)} \rp \nonumber \\
T_{G,j}^{d_2} & = &  \mE_{G,{\mathcal U}_3}\lp \mE_{{\mathcal U}_2}\mE_{{\mathcal U}_1}  \Theta_{G,3}^{(2)} \rp,
 \end{eqnarray}
Gaussian integration by parts gives
 \begin{eqnarray}\label{eq:lev2genGanal4}
T_{G,j}  & = &  T_{G,j}^c+T_{G,j}^{d_1}+T_{G,j}^{d_2}.
 \end{eqnarray}
Following (\ref{eq:genGanal5})
\begin{eqnarray}\label{eq:lev2genGanal5}
\Theta_{G,1}^{(2)} & = & \lp \frac{\y_j^{(i_2)}}{Z^{1-\m_1}}\lp(C^{(i_1)})^{s-1}\beta_{i_1}A^{(i_1,i_2)}\y_j^{(i_2)}\sqrt{t}+(s-1)(C^{(i_1)})^{s-2}\beta_{i_1}\sum_{p_2=1}^{l}A^{(i_1,p_2)}\y_j^{(p_2)}\sqrt{t}\rp \rp \nonumber \\
& &  -(1-\m_1)
\mE \lp\sum_{p_1=1}^{l} \frac{(\x^{(i_1)})^T\x^{(p_1)}}{\|\x^{(i_1)}\|_2\|\x^{(p_1)}\|_2}
\frac{(C^{(i_1)})^{s-1} A^{(i_1,i_2)}\y_j^{(i_2)}}{Z^{2-\m_1}}
s  (C^{(p_1)})^{s-1}\sum_{p_2=1}^{l}\beta_{p_1}A^{(p_1,p_2)}\y_j^{(p_2)}\sqrt{t}\rp,\nonumber \\
\end{eqnarray}
and then analogously to (\ref{eq:genGanal6})
\begin{eqnarray}\label{eq:lev2genGanal6}
 \sum_{i_1=1}^{l} \sum_{i_2=1}^{l}\sum_{j=1}^{m}\beta_{i_1} \frac{T_{G,j}^c}{\sqrt{t}}
 & = & \sum_{i_1=1}^{l} \sum_{i_2=1}^{l}\sum_{j=1}^{m}\beta_{i_1} \frac{\mE_{G,{\mathcal U}_3}\lp \mE_{{\mathcal U}_2}
 \frac{\lp\mE_{{\mathcal U}_1} \lp Z^{\m_1}\rp\rp^{\frac{\m_2}{\m_1}}}{
 \mE_{{\mathcal U}_2}\lp\lp   \mE_{{\mathcal U}_1} \lp Z^{\m_1}\rp\rp^{\frac{\m_2}{\m_1}}\rp}   \mE_{{\mathcal U}_1}\frac{1}{\mE_{{\mathcal U}_1} Z^{\m_1}} \Theta_{G,1}^{(2)} \rp}{\sqrt{t}} \nonumber\\
 & = & \beta^2 \lp \mE_{G,{\mathcal U}_3} \langle \|\x^{(i_1)}\|_2^2\|\y^{(i_2)}\|_2^2\rangle_{\gamma_{01}^{(2)}} +   (s-1)\mE_{G,{\mathcal U}_3}\langle \|\x^{(i_1)}\|_2^2(\y^{(p_2)})^T\y^{(i_2)}\rangle_{\gamma_{02}^{(2)}}   \rp    \nonumber \\
 &  & -s\beta^2(1-\m_1) \mE_{G,{\mathcal U}_3}\langle (\x^{(p_1)})^T\x^{(i_1)}(\y^{(p_2)})^T\y^{(i_2)}\rangle_{\gamma_{1}^{(2)}}.
 \end{eqnarray}

To determine $T_{G,j}^{d_1}$, we start by writing
\begin{eqnarray}\label{eq:lev2genGanal7}
\frac{d}{d\u_j^{(p_1,1)}}\lp   \frac{1}{\lp\mE_{{\mathcal U}_1} Z^{\m_1}\rp^{1-\frac{\m_2}{\m_1}}} \rp
& = & -\frac{1-\frac{\m_2}{\m_1}}{\lp\mE_{{\mathcal U}_1} Z^{\m_1}\rp^{2-\frac{\m_2}{\m_1}}}\mE_{{\mathcal U}_1}  \frac{dZ^{\m_1}}{d\u_j^{(p_1,1)}},
\end{eqnarray}
and
\begin{eqnarray}\label{eq:lev2genGanal7a}
  \frac{dZ^{\m_1}}{d\u_j^{(p_1,1)}}
& = & \m_1 \frac{Z^{\m_1}}{Z} \frac{dZ}{d\u_j^{(p_1,1)}}=  \m_1 \frac{Z^{\m_1}}{Z}      \frac{d\sum_{q_1=1}^l(C^{(q_1)})^s}{d\u_j^{(p_1,1)}}
 = \m_1   \frac{Z^{\m_1}}{Z}      s (C^{(p_1)})^{s-1}\frac{dC^{(p_1)}}{d\u_j^{(p_1,1)}}\nonumber\\
& = &   \m_1 \frac{Z^{\m_1}}{Z} s (C^{(p_1)})^{s-1}\sum_{p_2=1}^{l}A^{(p_1,p_2)}\beta_{p_1}\y_j^{(p_2)}\sqrt{t}.
\end{eqnarray}
From (\ref{eq:lev2genGanal7}) and (\ref{eq:lev2genGanal7a}) we obtain
\begin{eqnarray}\label{eq:lev2genGanal7b}
\frac{d}{d\u_j^{(p_1,1)}}\lp   \frac{1}{\lp\mE_{{\mathcal U}_1} Z^{\m_1}\rp^{1-\frac{\m_2}{\m_1}}} \rp
& = &  -\mE_{{\mathcal U}_1}  \frac{(\m_1-\m_2)Z^{\m_1}}{\lp\mE_{{\mathcal U}_1} Z^{\m_1}\rp^{2-\frac{\m_2}{\m_1}}}   \frac{1}{Z} s (C^{(p_1)})^{s-1}\sum_{p_2=1}^{l}A^{(p_1,p_2)}\beta_{p_1}\y_j^{(p_2)}\sqrt{t}.
\end{eqnarray}
Combining  the expression for $\Theta_{G,2}^{(2)}$ from (\ref{eq:lev2genGanal3}) with (\ref{eq:lev2genGanal7b}), we find
\begin{eqnarray}\label{eq:lev2genGanal8}
 \mE_{{\mathcal U}_2} \Theta_{G,2}^{(2)}  & = &
 \Bigg( \Bigg.  \mE_{{\mathcal U}_2}
 \frac{\lp\mE_{{\mathcal U}_1} \lp Z^{\m_1}\rp\rp^{\frac{\m_2}{\m_1}}}{
 \mE_{{\mathcal U}_2}\lp\lp   \mE_{{\mathcal U}_1} \lp Z^{\m_1}\rp\rp^{\frac{\m_2}{\m_1}}\rp}
 \frac{(C^{(i_1)})^{s-1} A^{(i_1,i_2)}\y_j^{(i_2)}}{Z^{1-\m_1}} 
 \nonumber \\
& & \times \sum_{p_1=1}^{l} \mE (\u_j^{(i_1,1)}\u_j^{(p_1,1)})
 \mE_{{\mathcal U}_1}  \frac{(\m_1-\m_2)Z^{\m_1}}{\lp\mE_{{\mathcal U}_1} Z^{\m_1}\rp^2}\frac{1}{Z} s (C^{(p_1)})^{s-1}\sum_{p_2=1}^{l}A^{(p_1,p_2)}\beta_{p_1}\y_j^{(p_2)}\sqrt{t} \Bigg. \Bigg) \nonumber \\
& = & \mE_{{\mathcal U}_2}
 \frac{\lp\mE_{{\mathcal U}_1} \lp Z^{\m_1}\rp\rp^{\frac{\m_2}{\m_1}}}{
 \mE_{{\mathcal U}_2}\lp\lp   \mE_{{\mathcal U}_1} \lp Z^{\m_1}\rp\rp^{\frac{\m_2}{\m_1}}\rp}
\frac{\m_1-\m_2}{\m_1} \Theta_{G,2} \nonumber\\
& = & \Phi_{{\mathcal U}_2}\frac{\m_1-\m_2}{\m_1} \Theta_{G,2}.
  \end{eqnarray}
Connecting (\ref{eq:genGanal9}), $T_{G,j}^{d_1}$ from (\ref{eq:lev2genGanal3}), and (\ref{eq:lev2genGanal8}), we also obtain
\begin{eqnarray}\label{eq:lev2genGanal9}
 \sum_{i_1=1}^{l} \sum_{i_2=1}^{l}\sum_{j=1}^{m}\beta_{i_1} \frac{T_{G,j}^{d_1}}{\sqrt{t}}
 & = & \sum_{i_1=1}^{l} \sum_{i_2=1}^{l}\sum_{j=1}^{m}\beta_{i_1} \frac{\mE_{G,{\mathcal U}_3}\lp \mE_{{\mathcal U}_2} \mE_{{\mathcal U}_1}  \Theta_{G,2}^{(2)}\rp}{\sqrt{t}} \nonumber\\
 & = &  \mE_{G,{\mathcal U}_3}\lp \mE_{{\mathcal U}_2} \mE_{{\mathcal U}_1} \sum_{i_1=1}^{l} \sum_{i_2=1}^{l}\sum_{j=1}^{m}\beta_{i_1} \frac{\Theta_{G,2}^{(2)}}{\sqrt{t}}\rp \nonumber\\
 & = &  \mE_{G,{\mathcal U}_3}\lp \Phi_{{\mathcal U}_2}\frac{\m_1-\m_2}{\m_1} \mE_{{\mathcal U}_1} \sum_{i_1=1}^{l} \sum_{i_2=1}^{l}\sum_{j=1}^{m}\beta_{i_1} \frac{\Theta_{G,2}}{\sqrt{t}}\rp \nonumber\\
 & = &  -\mE_{G,{\mathcal U}_3}\lp \Phi_{{\mathcal U}_2}\frac{\m_1-\m_2}{\m_1} s\beta^2\m_1\langle(\x^{(p_1)})^T\x^{(i_1)} (\y^{(p_2)})^T\y^{(i_2)} \rangle_{\gamma_{2}^{(1)}}\rp \nonumber\\
  & = & -s\beta^2(\m_1-\m_2)\mE_{G,{\mathcal U}_3} \langle(\x^{(p_1)})^T\x^{(i_1)} (\y^{(p_2)})^T\y^{(i_2)} \rangle_{\gamma_{2}^{(2)}}.
 \end{eqnarray}

To determine $T_{G,j}^{d_2}$, we start by writing similarly to (\ref{eq:lev2x3lev2genFanal24})
\begin{eqnarray}\label{eq:lev2genGanal10}
\frac{d}{d \u_j^{(p_1,1)}} \lp \frac{1}{\mE_{{\mathcal U}_2}\lp\lp\mE_{{\mathcal U}_1} \lp Z^{\m_1}\rp\rp^{\frac{\m_2}{\m_1}}\rp}\rp
 & = & -\frac{1}{\lp\mE_{{\mathcal U}_2}\lp\lp\mE_{{\mathcal U}_1} \lp Z^{\m_1}\rp\rp^{\frac{\m_2}{\m_1}}\rp\rp^2}
   \mE_{{{\mathcal U}_2}}
\frac{\m_2}{\m_1}   \lp\mE_{{\mathcal U}_1} \lp Z^{\m_1}\rp\rp^{\frac{\m_2}{\m_1}-1}
\mE_{{\mathcal U}_1}\frac{d Z^{\m_1}}{d \u_j^{(p_1,1)}} \nonumber \\
 & = & -\frac{1}{\lp\mE_{{\mathcal U}_2}\lp\lp\mE_{{\mathcal U}_1} \lp Z^{\m_1}\rp\rp^{\frac{\m_2}{\m_1}}\rp\rp^2}
 \times \mE_{{{\mathcal U}_2}}
\frac{\m_2}{\m_1}   \lp\mE_{{\mathcal U}_1} \lp Z^{\m_1}\rp\rp^{\frac{\m_2}{\m_1}-1}\nonumber\\
& &
\times \mE_{{\mathcal U}_1}   \m_1   \frac{Z^{\m_1}}{Z} s (C^{(p_1)})^{s-1}\sum_{p_2=1}^{l}A^{(p_1,p_2)}\beta_{p_1}\y_j^{(p_2)}\sqrt{t}. \nonumber \\
\end{eqnarray}
Combining  the expression for $\Theta_{G,3}^{(2)}$ from (\ref{eq:lev2genGanal3}) with (\ref{eq:lev2genGanal10}), we find
{\small \begin{eqnarray}\label{eq:lev2genGanal11}
 \mE_{{\mathcal U}_2,{\mathcal U}_1} \Theta_{G,3}^{(2)}  & = &
s\beta^2\m_2 \Bigg( \Bigg.  \mE_{{\mathcal U}_2}
 \frac{\lp\mE_{{\mathcal U}_1} \lp Z^{\m_1}\rp\rp^{\frac{\m_2}{\m_1}}}{
 \mE_{{\mathcal U}_2}\lp\lp   \mE_{{\mathcal U}_1} \lp Z^{\m_1}\rp\rp^{\frac{\m_2}{\m_1}}\rp}
\mE_{{\mathcal U}_1} \frac{Z^{\m_1}}{\lp\mE_{{\mathcal U}_1} \lp Z^{\m_1}\rp\rp^{\frac{\m_2}{\m_1}}}
  \frac{(C^{(i_1)})^s}{Z}
  \frac{ A^{(i_1,i_2)}}{C^{(i_1)}}\y_j^{(i_2)}
  \sum_{p_1=1}^{l} \frac{(\x^{(p_1)})^T\x^{(i_1)}}{\beta_{i_1}}
 \nonumber \\
& & \times
\mE_{{\mathcal U}_2}
 \frac{\lp\mE_{{\mathcal U}_1} \lp Z^{\m_1}\rp\rp^{\frac{\m_2}{\m_1}}}{
 \mE_{{\mathcal U}_2}\lp\lp   \mE_{{\mathcal U}_1} \lp Z^{\m_1}\rp\rp^{\frac{\m_2}{\m_1}}\rp}
\mE_{{\mathcal U}_1} \frac{Z^{\m_1}}{\lp\mE_{{\mathcal U}_1} \lp Z^{\m_1}\rp\rp^{\frac{\m_2}{\m_1}}}
\frac{1}{Z}  (C^{(p_1)})^{s-1}\sum_{p_2=1}^{l}A^{(p_1,p_2)}\y_j^{(p_2)}\sqrt{t} \Bigg. \Bigg).
  \end{eqnarray}}

\noindent Combining $T_{G,j}^{d_2}$ from (\ref{eq:lev2genGanal3}), and (\ref{eq:lev2genGanal11}), we obtain
\begin{eqnarray}\label{eq:lev2genGanal12}
 \sum_{i_1=1}^{l} \sum_{i_2=1}^{l}\sum_{j=1}^{m}\beta_{i_1} \frac{T_{G,j}^{d_2}}{\sqrt{t}}
 & = & \sum_{i_1=1}^{l} \sum_{i_2=1}^{l}\sum_{j=1}^{m}\beta_{i_1} \frac{\mE_{G,{\mathcal U}_3}\lp \mE_{{\mathcal U}_2} \mE_{{\mathcal U}_1}  \Theta_{G,3}^{(2)}\rp}{\sqrt{t}} \nonumber\\
   & = & -s\beta^2\m_2\mE_{G,{\mathcal U}_3} \langle(\x^{(p_1)})^T\x^{(i_1)} (\y^{(p_2)})^T\y^{(i_2)} \rangle_{\gamma_{3}^{(2)}}.
 \end{eqnarray}
A further combination of (\ref{eq:lev2genGanal4}), (\ref{eq:lev2genGanal6}), (\ref{eq:lev2genGanal9}), and (\ref{eq:lev2genGanal12}) gives
\begin{eqnarray}\label{eq:lev2genGanal10}
 \sum_{i_1=1}^{l} \sum_{i_2=1}^{l}\sum_{j=1}^{m}\beta_{i_1} \frac{T_{G,j}}{\sqrt{t}}
 & = &  \sum_{i_1=1}^{l} \sum_{i_2=1}^{l}\sum_{j=1}^{m}\beta_{i_1} \frac{T_{G,j}^c+T_{G,j}^{d_1}+T_{G,j}^{d_2}}{\sqrt{t}} \nonumber \\
& = & \beta^2  \lp \mE_{G,{\mathcal U}_3} \langle \|\x^{(i_1)}\|_2^2\|\y^{(i_2)}\|_2^2\rangle_{\gamma_{01}^{(2)}} +   (s-1) \mE_{G,{\mathcal U}_3}\langle \|\x^{(i_1)}\|_2^2(\y^{(p_2)})^T\y^{(i_2)}\rangle_{\gamma_{02}^{(2)}}   \rp    \nonumber \\
 &  & -s\beta^2(1-\m_1) \mE_{G,{\mathcal U}_3}\langle (\x^{(p_1)})^T\x^{(i_1)}(\y^{(p_2)})^T\y^{(i_2)}\rangle_{\gamma_{1}^{(2)}} \nonumber \\
 &  & -s\beta^2(\m_1-\m_2)\mE_{G,{\mathcal U}_3} \langle(\x^{(p_1)})^T\x^{(i_1)} (\y^{(p_2)})^T\y^{(i_2)} \rangle_{\gamma_{2}^{(2)}} \nonumber \\
  &  & -s\beta^2\m_2\mE_{G,{\mathcal U}_3} \langle(\x^{(p_1)})^T\x^{(i_1)} (\y^{(p_2)})^T\y^{(i_2)} \rangle_{\gamma_{3}^{(2)}}.
 \end{eqnarray}

%%%%%%%%%%%%%%%%%%%%%%%%%%%%%%%%%%%%%%%%%%%%%%%%%%%%%%%%%%%%%%%%%%%%%%%%
\subsubsection{Connecting all pieces together}
\label{sec:lev2conalt}
%%%%%%%%%%%%%%%%%%%%%%%%%%%%%%%%%%%%%%%%%%%%%%%%%%%%%%%%%%%%%%%%%%%%%%%%

As in Section \ref{sec:gencon}, we now connect together all the pieces obtained above. To do so, we utilize (\ref{eq:lev2genanal10e}) and (\ref{eq:lev2genanal10f}) to write
\begin{eqnarray}\label{eq:lev2ctp1}
\frac{d\psi(\calX,\calY,\q,\m,\beta,s,t)}{dt}  & = &       \frac{\mbox{sign}(s)}{2\beta\sqrt{n}} \lp \Omega_G+\Omega_1+\Omega_2+\Omega_3\rp,
\end{eqnarray}
where
\begin{eqnarray}\label{eq:lev2ctp2}
\Omega_G & = & \sum_{i_1=1}^{l}  \sum_{i_2=1}^{l}\sum_{j=1}^{m} \beta_{i_1}\frac{T_{G,j}}{\sqrt{t}}  \nonumber\\
\Omega_1 & = &
-\sum_{i_1=1}^{l}  \sum_{i_2=1}^{l} \sum_{j=1}^{m}\beta_{i_1}\frac{T_{3,1,j}}{\sqrt{1-t}}-\sum_{i_1=1}^{l}  \sum_{i_2=1}^{l} \sum_{j=1}^{m}\beta_{i_1}\frac{T_{2,1,j}}{\sqrt{1-t}}
-\sum_{i_1=1}^{l}  \sum_{i_2=1}^{l} \sum_{j=1}^{m}\beta_{i_1}\frac{T_{1,1,j}}{\sqrt{1-t}} \nonumber\\
\Omega_2 & = &
-\sum_{i_1=1}^{l}  \sum_{i_2=1}^{l}\beta_{i_1}\|\y^{(i_2)}\|_2\frac{T_{3,2}}{\sqrt{1-t}}
-\sum_{i_1=1}^{l}  \sum_{i_2=1}^{l}\beta_{i_1}\|\y^{(i_2)}\|_2\frac{T_{2,2}}{\sqrt{1-t}}
-\sum_{i_1=1}^{l}  \sum_{i_2=1}^{l}\beta_{i_1}\|\y^{(i_2)}\|_2\frac{T_{1,2}}{\sqrt{1-t}} \nonumber\\
\Omega_3 & = &
\sum_{i_1=1}^{l}  \sum_{i_2=1}^{l}\beta_{i_1}\|\y^{(i_2)}\|_2\frac{T_{3,3}}{\sqrt{t}}
\sum_{i_1=1}^{l}  \sum_{i_2=1}^{l}\beta_{i_1}\|\y^{(i_2)}\|_2\frac{T_{2,3}}{\sqrt{t}}
+ \sum_{i_1=1}^{l}  \sum_{i_2=1}^{l}\beta_{i_1}\|\y^{(i_2)}\|_2\frac{T_{1,3}}{\sqrt{t}}.
\end{eqnarray}
From (\ref{eq:lev2genGanal10}) we have
\begin{eqnarray}\label{eq:lev2cpt3}
\Omega_G & = & \beta^2 \lp \mE_{G,{\mathcal U}_3} \langle \|\x^{(i_1)}\|_2^2\|\y^{(i_2)}\|_2^2\rangle_{\gamma_{01}^{(2)}} +  (s-1) \mE_{G,{\mathcal U}_3}\langle \|\x^{(i_1)}\|_2^2(\y^{(p_2)})^T\y^{(i_2)}\rangle_{\gamma_{02}^{(2)}}   \rp    \nonumber \\
 &  & -s\beta^2(1-\m_1) \mE_{G,{\mathcal U}_3}\langle (\x^{(p_1)})^T\x^{(i_1)}(\y^{(p_2)})^T\y^{(i_2)}\rangle_{\gamma_{1}^{(2)}} \nonumber \\
 &  & -s\beta^2(\m_1-\m_2)\mE_{G,{\mathcal U}_3} \langle(\x^{(p_1)})^T\x^{(i_1)} (\y^{(p_2)})^T\y^{(i_2)} \rangle_{\gamma_{2}^{(2)}} \nonumber \\
  &  & -s\beta^2\m_2\mE_{G,{\mathcal U}_3} \langle(\x^{(p_1)})^T\x^{(i_1)} (\y^{(p_2)})^T\y^{(i_2)} \rangle_{\gamma_{3}^{(2)}}.
 \end{eqnarray}
From (\ref{eq:lev2liftgenAanal19i}), (\ref{eq:lev2genDanal25}), and (\ref{eq:lev2x3lev2genDanal25}) we have
\begin{eqnarray}\label{eq:lev2cpt4}
-\Omega_1
& = & (\p_0-\p_1)\beta^2 \lp \mE_{G,{\mathcal U}_3}\langle \|\x^{(i_1)}\|_2^2\|\y^{(i_2)}\|_2^2\rangle_{\gamma_{01}^{(2)}} +  (s-1)\mE_{G,{\mathcal U}_3}\langle \|\x^{(i_1)}\|_2^2(\y^{(p_2)})^T\y^{(i_2)}\rangle_{\gamma_{02}^{(2)}} \rp \nonumber \\
& & - (\p_0-\p_1)s\beta^2(1-\m_1)\mE_{G,{\mathcal U}_3}\langle \|\x^{(i_1)}\|_2\|\x^{(p_1)}\|_2(\y^{(p_2)})^T\y^{(i_2)} \rangle_{\gamma_{1}^{(2)}} \nonumber \\
&  & +(\p_1-\p_2)\beta^2\lp \mE_{G,{\mathcal U}_3}\langle \|\x^{(i_1)}\|_2^2\|\y^{(i_2)}\|_2^2\rangle_{\gamma_{01}^{(2)}} +  (s-1)\mE_{G,{\mathcal U}_3}\langle \|\x^{(i_1)}\|_2^2(\y^{(p_2)})^T\y^{(i_2)}\rangle_{\gamma_{02}^{(2)}} \rp\nonumber \\
& & - (\p_1-\p_2)s\beta^2(1-\m_1)\mE_{G,{\mathcal U}_3}\langle \|\x^{(i_1)}\|_2\|\x^{(p_1)}\|_2(\y^{(p_2)})^T\y^{(i_2)} \rangle_{\gamma_{1}^{(2)}}\nonumber \\
 &   &
  -s\beta^2(\p_1-\p_2)(\m_1-\m_2)\mE_{G,{\mathcal U}_3} \langle \|\x^{(i_1)}\|_2\|\x^{(p_1)}\|_2(\y^{(p_2)})^T\y^{(i_2)} \rangle_{\gamma_{2}^{(2)}} \nonumber \\
&  & +\p_2 \beta^2\lp \mE_{G,{\mathcal U}_3}\langle \|\x^{(i_1)}\|_2^2\|\y^{(i_2)}\|_2^2\rangle_{\gamma_{01}^{(2)}} +  (s-1)\mE_{G,{\mathcal U}_3}\langle \|\x^{(i_1)}\|_2^2(\y^{(p_2)})^T\y^{(i_2)}\rangle_{\gamma_{02}^{(2)}} \rp\nonumber \\
& & - \p_2s\beta^2(1-\m_1)\mE_{G,{\mathcal U}_3}\langle \|\x^{(i_1)}\|_2\|\x^{(p_1)}\|_2(\y^{(p_2)})^T\y^{(i_2)} \rangle_{\gamma_{1}^{(2)}}\nonumber \\
 &   &
  -s\beta^2\p_2(\m_1-\m_2)\mE_{G,{\mathcal U}_3} \langle \|\x^{(i_1)}\|_2\|\x^{(p_1)}\|_2(\y^{(p_2)})^T\y^{(i_2)} \rangle_{\gamma_{2}^{(2)}}\nonumber \\
 &  & -s\beta^2\p_2\m_2\mE_{G,{\mathcal U}_3} \langle \|\x^{(i_1)}\|_2\|\x^{(p_1)}\|_2(\y^{(p_2)})^T\y^{(i_2)} \rangle_{\gamma_{3}^{(2)}} \nonumber \\
 & = & \p_0 \beta^2\lp\mE_{G,{\mathcal U}_3}\langle \|\x^{(i_1)}\|_2^2\|\y^{(i_2)}\|_2^2\rangle_{\gamma_{01}^{(2)}} +  (s-1)\mE_{G,{\mathcal U}_3}\langle \|\x^{(i_1)}\|_2^2(\y^{(p_2)})^T\y^{(i_2)}\rangle_{\gamma_{02}^{(2)}} \rp\nonumber \\
& & - \p_0s\beta^2(1-\m_1)\mE_{G,{\mathcal U}_3}\langle \|\x^{(i_1)}\|_2\|\x^{(p_1)}\|_2(\y^{(p_2)})^T\y^{(i_2)} \rangle_{\gamma_{1}^{(2)}}\nonumber \\
 &   &
  -s\beta^2\p_1(\m_1-\m_2)\mE_{G,{\mathcal U}_3} \langle \|\x^{(i_1)}\|_2\|\x^{(p_1)}\|_2(\y^{(p_2)})^T\y^{(i_2)} \rangle_{\gamma_{2}^{(2)}}\nonumber \\
 &  & -s\beta^2\p_2\m_2\mE_{G,{\mathcal U}_3} \langle \|\x^{(i_1)}\|_2\|\x^{(p_1)}\|_2(\y^{(p_2)})^T\y^{(i_2)} \rangle_{\gamma_{3}^{(2)}}.
   \end{eqnarray}
From (\ref{eq:lev2liftgenBanal20b}), (\ref{eq:lev2genEanal25}), and (\ref{eq:lev2x3lev2genEanal25}), we have
\begin{eqnarray}\label{eq:lev2cpt5}
-\Omega_2 & = & (\q_0-\q_1)\beta^2\lp\mE_{G,{\mathcal U}_3}\langle \|\x^{(i_1)}\|_2^2\|\y^{(i_2)}\|_2^2\rangle_{\gamma_{01}^{(2)}} +  (s-1)\mE_{G,{\mathcal U}_3}\langle \|\x^{(i_1)}\|_2^2 \|\y^{(i_2)}\|_2\|\y^{(p_2)}\|_2\rangle_{\gamma_{02}^{(2)}}\rp\nonumber \\
& & - (\q_0-\q_1)s\beta^2(1-\m_1)\mE_{G,{\mathcal U}_3}\langle (\x^{(p_1)})^T\x^{(i_1)}\|\y^{(i_2)}\|_2\|\y^{(p_2)}\|_2 \rangle_{\gamma_{1}^{(2)}}\nonumber \\
&  & +(\q_1-\q_2)\beta^2 \lp \mE_{G,{\mathcal U}_3}\langle \|\x^{(i_1)}\|_2^2\|\y^{(i_2)}\|_2^2\rangle_{\gamma_{01}^{(2)}} +   (s-1)\mE_{G,{\mathcal U}_3}\langle \|\x^{(i_1)}\|_2^2 \|\y^{(i_2)}\|_2\|\y^{(p_2)}\|_2\rangle_{\gamma_{02}^{(2)}}\rp\nonumber \\
& & - (\q_1-\q_2)s\beta^2(1-\m_1)\mE_{G,{\mathcal U}_3}\langle (\x^{(p_1)})^T\x^{(i_1)}\|\y^{(i_2)}\|_2\|\y^{(p_2)}\|_2 \rangle_{\gamma_{1}^{(2)}} \nonumber \\
&  & -s\beta^2(\q_1-\q_2)(\m_1-\m_2)\mE_{G,{\mathcal U}_2} \langle \|\y^{(i_2)}\|_2\|\y^{(p_2)}\|_2(\x^{(i_1)})^T\x^{(p_1)}\rangle_{\gamma_{2}^{(2)}} \nonumber \\
&  & +\q_2\beta^2\lp\mE_{G,{\mathcal U}_3}\langle \|\x^{(i_1)}\|_2^2\|\y^{(i_2)}\|_2^2\rangle_{\gamma_{01}^{(2)}} +  (s-1)\mE_{G,{\mathcal U}_3}\langle \|\x^{(i_1)}\|_2^2 \|\y^{(i_2)}\|_2\|\y^{(p_2)}\|_2\rangle_{\gamma_{02}^{(2)}}\rp\nonumber \\
& & - \q_2s\beta^2(1-\m_1)\mE_{G,{\mathcal U}_3}\langle (\x^{(p_1)})^T\x^{(i_1)}\|\y^{(i_2)}\|_2\|\y^{(p_2)}\|_2 \rangle_{\gamma_{1}^{(2)}} \nonumber \\
&  & -s\beta^2\q_2(\m_1-\m_2)\mE_{G,{\mathcal U}_3} \langle \|\y^{(i_2)}\|_2\|\y^{(p_2)}\|_2(\x^{(i_1)})^T\x^{(p_1)}\rangle_{\gamma_{2}^{(2)}} \nonumber \\
& & -s\beta^2 \q_2 \m_2\mE_{G,{\mathcal U}_3} \langle \|\y^{(i_2)}\|_2\|\y^{(p_2)}\|_2(\x^{(i_1)})^T\x^{(p_1)}\rangle_{\gamma_{3}^{(2)}} \nonumber \\
& = & \q_0 \beta^2 \lp \mE_{G,{\mathcal U}_3}\langle \|\x^{(i_1)}\|_2^2\|\y^{(i_2)}\|_2^2\rangle_{\gamma_{01}^{(2)}} +   (s-1)\mE_{G,{\mathcal U}_3}\langle \|\x^{(i_1)}\|_2^2 \|\y^{(i_2)}\|_2\|\y^{(p_2)}\|_2\rangle_{\gamma_{02}^{(2)}}\rp\nonumber \\
& & - \q_0s\beta^2(1-\m_1)\mE_{G,{\mathcal U}_3}\langle (\x^{(p_1)})^T\x^{(i_1)}\|\y^{(i_2)}\|_2\|\y^{(p_2)}\|_2 \rangle_{\gamma_{1}^{(2)}} \nonumber \\
&  & -s\beta^2\q_1(\m_1-\m_2)\mE_{G,{\mathcal U}_3} \langle \|\y^{(i_2)}\|_2\|\y^{(p_2)}\|_2(\x^{(i_1)})^T\x^{(p_1)}\rangle_{\gamma_{2}^{(2)}} \nonumber \\
& & -s\beta^2 \q_2 \m_2\mE_{G,{\mathcal U}_3} \langle \|\y^{(i_2)}\|_2\|\y^{(p_2)}\|_2(\x^{(i_1)})^T\x^{(p_1)}\rangle_{\gamma_{3}^{(2)}}.
\end{eqnarray}

From (\ref{eq:lev2liftgenCanal21b}) and (\ref{eq:lev2genFanal29}), we have
  \begin{eqnarray}\label{eq:lev2cpt6}
\Omega_3
& = & (\p_0\q_0-\p_1\q_1)\beta^2 \lp \mE_{G,{\mathcal U}_3}\langle \|\x^{(i_1)}\|_2^2\|\y^{(i_2)}\|_2^2\rangle_{\gamma_{01}^{(2)}} +   (s-1)\mE_{G,{\mathcal U}_3}\langle \|\x^{(i_1)}\|_2^2 \|\y^{(i_2)}\|_2\|\y^{(p_2)}\|_2\rangle_{\gamma_{02}^{(2)}}\rp\nonumber \\
& & - (\p_0\q_0-\p_1\q_1)s\beta^2(1-\m_1)\mE_{G,{\mathcal U}_3}\langle \|\x^{(i_1)}\|_2\|\x^{(p_`)}\|_2\|\y^{(i_2)}\|_2\|\y^{(p_2)}\|_2 \rangle_{\gamma_{1}^{(2)}} \nonumber \\
&  &
+(\p_1\q_1-\p_2\q_2)\beta^2\lp\mE_{G,{\mathcal U}_3}\langle \|\x^{(i_1)}\|_2^2\|\y^{(i_2)}\|_2^2\rangle_{\gamma_{01}^{(2)}} +   (s-1)\mE_{G,{\mathcal U}_3}\langle \|\x^{(i_1)}\|_2^2 \|\y^{(i_2)}\|_2\|\y^{(p_2)}\|_2\rangle_{\gamma_{02}^{(2)}}\rp\nonumber \\
& & - (\p_1\q_1-\p_2\q_2) s\beta^2(1-\m_1)\mE_{G,{\mathcal U}_3}\langle \|\x^{(i_1)}\|_2\|\x^{(p_`)}\|_2\|\y^{(i_2)}\|_2\|\y^{(p_2)}\|_2 \rangle_{\gamma_{1}^{(2)}} \nonumber \\
&  & -s\beta^2(\p_1\q_1-\p_2\q_2)(\m_1-\m_2)\mE_{G,{\mathcal U}_3} \langle\|\x^{(i_2)}\|_2\|\x^{(p_2)}\|_2\|\y^{(i_2)}\|_2\|\y^{(p_2)}\rangle_{\gamma_{2}^{(2)}} \nonumber \\
&  &
+\p_2\q_2\beta^2\lp\mE_{G,{\mathcal U}_3}\langle \|\x^{(i_1)}\|_2^2\|\y^{(i_2)}\|_2^2\rangle_{\gamma_{01}^{(2)}} +  (s-1)\mE_{G,{\mathcal U}_3}\langle \|\x^{(i_1)}\|_2^2 \|\y^{(i_2)}\|_2\|\y^{(p_2)}\|_2\rangle_{\gamma_{02}^{(2)}}\rp\nonumber \\
& & -  \p_2\q_2 s\beta^2(1-\m_1)\mE_{G,{\mathcal U}_3}\langle \|\x^{(i_1)}\|_2\|\x^{(p_`)}\|_2\|\y^{(i_2)}\|_2\|\y^{(p_2)}\|_2 \rangle_{\gamma_{1}^{(2)}} \nonumber \\
&  & -s\beta^2 \p_2\q_2(\m_1-\m_2)\mE_{G,{\mathcal U}_3} \langle\|\x^{(i_2)}\|_2\|\x^{(p_2)}\|_2\|\y^{(i_2)}\|_2\|\y^{(p_2)}\rangle_{\gamma_{2}^{(2)}} \nonumber \\
 &  & -s\beta^2 \p_2\q_2 \m_2\mE_{G,{\mathcal U}_3} \langle\|\x^{(i_2)}\|_2\|\x^{(p_2)}\|_2\|\y^{(i_2)}\|_2\|\y^{(p_2)}\rangle_{\gamma_{3}^{(2)}} \nonumber\\
 & = &
\p_0\q_0\beta^2\lp\mE_{G,{\mathcal U}_3}\langle \|\x^{(i_1)}\|_2^2\|\y^{(i_2)}\|_2^2\rangle_{\gamma_{01}^{(2)}} +   (s-1)\mE_{G,{\mathcal U}_3}\langle \|\x^{(i_1)}\|_2^2 \|\y^{(i_2)}\|_2\|\y^{(p_2)}\|_2\rangle_{\gamma_{02}^{(2)}}\rp\nonumber \\
& & -  \p_0\q_0 s\beta^2(1-\m_1)\mE_{G,{\mathcal U}_3}\langle \|\x^{(i_1)}\|_2\|\x^{(p_`)}\|_2\|\y^{(i_2)}\|_2\|\y^{(p_2)}\|_2 \rangle_{\gamma_{1}^{(2)}} \nonumber \\
&  & -s\beta^2 \p_1\q_1(\m_1-\m_2)\mE_{G,{\mathcal U}_3} \langle\|\x^{(i_2)}\|_2\|\x^{(p_2)}\|_2\|\y^{(i_2)}\|_2\|\y^{(p_2)}\rangle_{\gamma_{2}^{(2)}} \nonumber \\
 &  & -s\beta^2 \p_2\q_2 \m_2\mE_{G,{\mathcal U}_3} \langle\|\x^{(i_2)}\|_2\|\x^{(p_2)}\|_2\|\y^{(i_2)}\|_2\|\y^{(p_2)}\rangle_{\gamma_{3}^{(2)}}.
\end{eqnarray}
Finally, combining (\ref{eq:lev2ctp1}) with (\ref{eq:lev2cpt3})-(\ref{eq:lev2cpt6}), we obtain
\begin{eqnarray}\label{eq:lev2cpt7}
\frac{d\psi(\calX,\calY,\q,\m,\beta,s,t)}{dt}  & = &       \frac{\mbox{sign}(s)\beta}{2\sqrt{n}} \lp \phi_1^{(2)}+\phi_2^{(2)}+\phi_3^{(2)} +\phi_{01}^{(2)} +\phi_{02}^{(2)} \rp ,
 \end{eqnarray}
where
\begin{eqnarray}\label{eq:lev2cpt8}
\phi_1^{(2)} & = &
-s(1-\m_1)\mE_{G,{\mathcal U}_3} \langle (\p_0\|\x^{(i_1)}\|_2\|\x^{(p_1)}\|_2 -(\x^{(p_1)})^T\x^{(i_1)})(\q_0\|\y^{(i_2)}\|_2\|\y^{(p_2)}\|_2 -(\y^{(p_2)})^T\y^{(i_2)})\rangle_{\gamma_{1}^{(2)}} \nonumber \\
\phi_2^{(2)} & = &
-s(\m_1-\m_2)\mE_{G,{\mathcal U}_3} \langle (\p_1\|\x^{(i_1)}\|_2\|\x^{(p_1)}\|_2 -(\x^{(p_1)})^T\x^{(i_1)})(\q_1\|\y^{(i_2)}\|_2\|\y^{(p_2)}\|_2 -(\y^{(p_2)})^T\y^{(i_2)})\rangle_{\gamma_{2}^{(2)}} \nonumber \\
\phi_3^{(2)} & = &
-s\m_2\mE_{G,{\mathcal U}_3} \langle (\p_2\|\x^{(i_1)}\|_2\|\x^{(p_1)}\|_2 -(\x^{(p_1)})^T\x^{(i_1)})(\q_2\|\y^{(i_2)}\|_2\|\y^{(p_2)}\|_2 -(\y^{(p_2)})^T\y^{(i_2)})\rangle_{\gamma_{3}^{(2)}} \nonumber \\
\phi_{01}^{(2)} & = & (1-\p_0)(1-\q_0)\mE_{G,{\mathcal U}_3}\langle \|\x^{(i_1)}\|_2^2\|\y^{(i_2)}\|_2^2\rangle_{\gamma_{01}^{(2)}} \nonumber\\
\phi_{02}^{(2)} & = & (1-\p_0)\mE_{G,{\mathcal U}_3}\left\langle \|\x^{(i_1)}\|_2^2 \lp\q_0\|\y^{(i_2)}\|_2\|\y^{(p_2)}\|_2-(\y^{(p_2)})^T\y^{(i_2)}\rp\right\rangle_{\gamma_{02}^{(2)}}. \end{eqnarray}

We summarize the above into the following proposition.
\begin{proposition}
\label{thm:thm2}
Let $k=\{1,2,3\}$ and $G\in\mR^{m \times n},u^{(4,k)}\in\mR^1,\u^{(2,k)}\in\mR^{m\times 1}$, and $\h^{(k)}\in\mR^{n\times 1}$ all have i.i.d. zero-mean normal components (they are then independent of each other as well). For vectors $\m=[\m_1,\m_2]$, $\p=[\p_0,\p_1,\p_2,\p_3]$ with $\p_0\geq \p_1\geq \p_2\geq \p_3= 0$ and $\q=[\q_0,\q_1,\q_2,\q_3]$ with $\q_0\geq \q_1\geq \q_2\geq \q_3= 0$, let the variances of the components of $G$, $u^{(4,k)}$, $\u^{(2,k)}$, and $\h^{(k)}$ be $1$, $\p_{k-1}\q_{k-1}-\p_k\q_k$, $\p_{k-1}-\p_{k}$, $\q_{k-1}-\q_{k}$, respectively.
 Assume that set ${\mathcal X}=\{\x^{(1)},\x^{(2)},\dots,\x^{(l)}\}$, where $\x^{(i)}\in \mR^{n},1\leq i\leq l$, and set ${\mathcal Y}=\{\y^{(1)},\y^{(2)},\dots,\y^{(l)}\}$, where $\y^{(i)}\in \mR^{m},1\leq i\leq l$ are given and that $\beta\geq 0$ and $s$ are real numbers. Set ${\mathcal U}_k=[u^{(4,k)},\u^{(2,k)},\h^{(2k)}]$
 and consider the following function
\begin{equation}\label{eq:thm2eq1}
\psi(\calX,\calY,\p,\q,\m,\beta,s,t)  =  \mE_{G,{\mathcal U}_3} \frac{1}{\beta|s|\sqrt{n}\m_2} \log \lp \mE_{{\mathcal U}_2}\lp\lp\mE_{{\mathcal U}_1} \lp Z^{\m_1}\rp\rp^{\frac{\m_2}{\m_1}}\rp\rp,
\end{equation}
where
\begin{eqnarray}\label{eq:thm2eq2}
Z & \triangleq & \sum_{i_1=1}^{l}\lp\sum_{i_2=1}^{l}e^{\beta D_0^{(i_1,i_2)}} \rp^{s} \nonumber \\
 D_0^{(i_1,i_2)} & \triangleq & \sqrt{t}(\y^{(i_2)})^T
 G\x^{(i_1)}+\sqrt{1-t}\|\x^{(i_1)}\|_2 (\y^{(i_2)})^T(\u^{(2,1)}+\u^{(2,2)}+\u^{(2,3)})\nonumber \\
 & & +\sqrt{t}\|\x^{(i_1)}\|_2\|\y^{(i_2)}\|_2(u^{(4,1)}+u^{(4,2)}+u^{(4,3)}) +\sqrt{1-t}\|\y^{(i_2)}\|_2(\h^{(1)}+\h^{(2)}+\h^{(3)})^T\x^{(i_1)}.\nonumber \\
 \end{eqnarray}
Then
\begin{eqnarray}\label{eq:prop1eq3}
\frac{d\psi(\calX,\calY,\q,\m,\beta,s,t)}{dt}  & = &        \frac{\mbox{sign}(s)\beta}{2\sqrt{n}} \lp \phi_1^{(2)}+\phi_2^{(2)}+\phi_3^{(2)} +\phi_{01}^{(2)} +\phi_{02}^{(2)} \rp,
 \end{eqnarray}
where $\phi$'s are as in (\ref{eq:lev2cpt8}) and $\gamma$'s are as defined in (\ref{eq:lev2genAanal19d2}) and (\ref{eq:lev2genAanal19e}).
\end{proposition}
\begin{proof}
  Follows from the above presentation.
\end{proof}

%%%%%%%%%%%%%%%%%%%%%%%%%%%%%%%%%%%%%%%%%%%%%%%%%%%%%%%%%%%%%%%%%
%%%%%%%%%%%%%%%%%%%%%%%%%%%%%%%%%%%%%%%%%%%%%%%%%%%%%%%%%%%%%%%%%
%%%%%%%%%%%%%%%%%%%%%%%%%%%%%%%%%%%%%%%%%%%%%%%%%%%%%%%%%%%%%%%%%
%%%%%%%%%%%%%%%%%%%%%%%%%%%%%%%%%%%%%%%%%%%%%%%%%%%%%%%%%%%%%%%%%
%%%%%%%%%%%%%%%%%%%%%%%%%%%%%%%%%%%%%%%%%%%%%%%%%%%%%%%%%%%%%%%%%
\section{$r$-th level of full lifting}
\label{sec:rthlev}
%%%%%%%%%%%%%%%%%%%%%%%%%%%%%%%%%%%%%%%%%%%%%%%%%%%%%%%%%%%%%%%%%
%%%%%%%%%%%%%%%%%%%%%%%%%%%%%%%%%%%%%%%%%%%%%%%%%%%%%%%%%%%%%%%%%
%%%%%%%%%%%%%%%%%%%%%%%%%%%%%%%%%%%%%%%%%%%%%%%%%%%%%%%%%%%%%%%%%
%%%%%%%%%%%%%%%%%%%%%%%%%%%%%%%%%%%%%%%%%%%%%%%%%%%%%%%%%%%%%%%%%
%%%%%%%%%%%%%%%%%%%%%%%%%%%%%%%%%%%%%%%%%%%%%%%%%%%%%%%%%%%%%%%%%
%%%%%%%%%%%%%%%%%%%%%%%%%%%%%%%%%%%%%%%%%%%%%%%%%%%%%%%%%%%%%%%%%

We now provide generalization to an arbitrary level, $r\in\mN$,  of full lifting. Section \ref{sec:gencon} provided all key principles that were used in Section \ref{sec:seclev} to present all formal mathematical steps needed to proceed to second and any higher level of lifting. We here finally formalize the final results of applying the procedure from Section \ref{sec:seclev} to any level of lifting $r\in\mN$.

We now consider vectors $\m=[\m_1,\m_2,...,\m_r]$,
$\p=[\p_0,\p_1,...,\p_r,\p_{r+1}]$ with $\p_0\geq \p_1\geq \p_2\geq \dots \geq \p_r\geq\p_{r+1} = 0$ and $\q=[\q_0,\q_1,\q_2,\dots,\q_r,\q_{r+1}]$ with $\q_0\geq \q_1\geq \q_2\geq \dots \geq \q_r\geq \q_{r+1} = 0$. As earlier, we take $G\in\mR^{m \times n}$ and, for any $k\in\{1,2,\dots,r+1\}$,  $u^{(4,k)}\in\mR^1,\u^{(2,k)}\in\mR^{m\times 1}$, and $\h^{(k)}\in\mR^{n\times 1}$ to all have i.i.d. zero-mean normal components (they are then independent of each other as well). Also, we assume that the variances of the components of $G$, $u^{(4,k)}$, $\u^{(2,k)}$, and $\h^{(k)}$ are $1$, $\p_{k-1}\q_{k-1}-\p_k\q_k$, $\p_{k-1}-\p_{k}$, $\q_{k-1}-\q_{k}$, respectively, and ${\mathcal U}_k\triangleq [u^{(4,k)},\u^{(2,k)},\h^{(2k)}]$.  Assuming that set ${\mathcal X}=\{\x^{(1)},\x^{(2)},\dots,\x^{(l)}\}$, where $\x^{(i)}\in \mR^{n},1\leq i\leq l$, and set ${\mathcal Y}=\{\y^{(1)},\y^{(2)},\dots,\y^{(l)}\}$, where $\y^{(i)}\in \mR^{m},1\leq i\leq l$ are given and that $\beta\geq 0$ and $s$ are real numbers, we are interested in the following function
\begin{equation}\label{eq:rthlev2genanal3}
\psi(\calX,\calY,\p,\q,\m,\beta,s,t)  =  \mE_{G,{\mathcal U}_{r+1}} \frac{1}{\beta|s|\sqrt{n}\m_r} \log
\lp \mE_{{\mathcal U}_{r}} \lp \dots \lp \mE_{{\mathcal U}_2}\lp\lp\mE_{{\mathcal U}_1} \lp Z^{\m_1}\rp\rp^{\frac{\m_2}{\m_1}}\rp\rp^{\frac{\m_3}{\m_2}} \dots \rp^{\frac{\m_{r}}{\m_{r-1}}}\rp,
\end{equation}
where
\begin{eqnarray}\label{eq:rthlev2genanal3a}
Z & \triangleq & \sum_{i_1=1}^{l}\lp\sum_{i_2=1}^{l}e^{\beta D_0^{(i_1,i_2)}} \rp^{s} \nonumber \\
 D_0^{(i_1,i_2)} & \triangleq & \sqrt{t}(\y^{(i_2)})^T
 G\x^{(i_1)}+\sqrt{1-t}\|\x^{(i_1)}\|_2 (\y^{(i_2)})^T\lp\sum_{k=1}^{r+1}\u^{(2,k)}\rp\nonumber \\
 & & +\sqrt{t}\|\x^{(i_1)}\|_2\|\y^{(i_2)}\|_2\lp\sum_{k=1}^{r+1}u^{(4,k)}\rp +\sqrt{1-t}\|\y^{(i_2)}\|_2\lp\sum_{k=1}^{r+1}\h^{(k)}\rp^T\x^{(i_1)}
 \end{eqnarray}
As in previous sections, for the convenience of the exposition, we, analogously to (\ref{eq:genanal4}), set
\begin{eqnarray}\label{eq:rthlev2genanal4}
\u^{(i_1,1)} & =  & \frac{G\x^{(i_1)}}{\|\x^{(i_1)}\|_2} \nonumber \\
\u^{(i_1,3,k)} & =  & \frac{(\h^{(k)})^T\x^{(i_1)}}{\|\x^{(i_1)}\|_2},
\end{eqnarray}
and recall that
\begin{eqnarray}\label{eq:rthlev2genanal5}
\u_j^{(i_1,1)} & =  & \frac{G_{j,1:n}\x^{(i_1)}}{\|\x^{(i_1)}\|_2},1\leq j\leq m.
\end{eqnarray}
Clearly, for any fixed $i_1$, the elements of $\u^{(i_1,1)}$ are i.i.d. standard normals, the elements of $\u^{(2,k)}$  are zero-mean Gaussians with variances $\p_{k-1}-\p_k$, the elements of $\u^{(i_1,3,k)}$ are zero-mean Gaussians with variances $\q_{k-1}-\q_k$. One then rewrites (\ref{eq:rthlev2genanal3}) as
\begin{equation}\label{eq:rthlev2genanal6}
\psi(\calX,\calY,\p,\q,\m,\beta,s,t)  =  \mE_{G,{\mathcal U}_{r+1}} \frac{1}{\beta|s|\sqrt{n}\m_r} \log
\lp \mE_{{\mathcal U}_{r}} \lp \dots \lp \mE_{{\mathcal U}_2}\lp\lp\mE_{{\mathcal U}_1} \lp Z^{\m_1}\rp\rp^{\frac{\m_2}{\m_1}}\rp\rp^{\frac{\m_3}{\m_2}} \dots \rp^{\frac{\m_{r}}{\m_{r-1}}}\rp,
\end{equation}
where $\beta_{i_1}=\beta\|\x^{(i_1)}\|_2$ and now
\begin{eqnarray}\label{eq:rthlev2genanal7}
B^{(i_1,i_2)} & \triangleq &  \sqrt{t}(\y^{(i_2)})^T\u^{(i_1,1)}+\sqrt{1-t} (\y^{(i_2)})^T\lp \sum_{k=1}^{r+1}\u^{(2,k)} \rp \nonumber \\
D^{(i_1,i_2)} & \triangleq &  B^{(i_1,i_2)}+\sqrt{t}\|\y^{(i_2)}\|_2 \lp \sum_{k=1}^{r+1}u^{(4,k)}\rp+\sqrt{1-t}\|\y^{(i_2)}\|_2 \lp \sum_{k=1}^{r+1}\u^{(i_1,3,k)}  \rp   \nonumber \\
A^{(i_1,i_2)} & \triangleq &  e^{\beta_{i_1}D^{(i_1,i_2)}}\nonumber \\
C^{(i_1)} & \triangleq &  \sum_{i_2=1}^{l}A^{(i_1,i_2)}\nonumber \\
Z & \triangleq & \sum_{i_1=1}^{l} \lp \sum_{i_2=1}^{l} A^{(i_1,i_2)}\rp^s =\sum_{i_1=1}^{l}  (C^{(i_1)})^s.
\end{eqnarray}
We set $\m_0=1$ and
\begin{eqnarray}\label{eq:rthlev2genanal7a}
\zeta_r\triangleq \mE_{{\mathcal U}_{r}} \lp \dots \lp \mE_{{\mathcal U}_2}\lp\lp\mE_{{\mathcal U}_1} \lp Z^{\frac{\m_1}{\m_0}}\rp\rp^{\frac{\m_2}{\m_1}}\rp\rp^{\frac{\m_3}{\m_2}} \dots \rp^{\frac{\m_{r}}{\m_{r-1}}}, r\geq 1.
\end{eqnarray}
Then
\begin{eqnarray}\label{eq:rthlev2genanal7b}
\zeta_k = \mE_{{\mathcal U}_{k}} \lp  \zeta_{k-1} \rp^{\frac{\m_{k}}{\m_{k-1}}}, k\geq 2,\quad \mbox{and} \quad
\zeta_1=\mE_{{\mathcal U}_1} \lp Z^{\frac{\m_1}{\m_0}}\rp,
\end{eqnarray}
and, for the completeness, we also set $\zeta_0=Z$. We can then write
\begin{eqnarray}\label{eq:rthlev2genanal9}
\frac{d\psi(\calX,\calY,\q,\m,\beta,s,t)}{dt}
& = &  \frac{d}{dt}\lp\mE_{G,{\mathcal U}_{r+1}} \frac{1}{\beta|s|\sqrt{n}\m_r} \log \lp \zeta_r\rp \rp \nonumber \\
& = &  \mE_{G,{\mathcal U}_{r+1}} \frac{1}{\beta|s|\sqrt{n}\m_r\zeta_r} \frac{d\zeta_r}{dt} \nonumber \\
& = &  \mE_{G,{\mathcal U}_{r+1}} \frac{1}{\beta|s|\sqrt{n}\m_{r-1}\zeta_r} \mE_{{\mathcal U}_{r}} \zeta_{r-1}^{\frac{\m_r}{\m_{r-1}}-1} \frac{d\zeta_{r-1}}{dt} \nonumber \\
& = &  \mE_{G,{\mathcal U}_{r+1}} \frac{1}{\beta|s|\sqrt{n}\m_{r-2}\zeta_r} \mE_{{\mathcal U}_{r}} \zeta_{r-1}^{\frac{\m_r}{\m_{r-1}}-1}
 \mE_{{\mathcal U}_{r-1}} \zeta_{r-2}^{\frac{\m_{r-1}}{\m_{r-2}}-1}
\frac{d\zeta_{r-2}}{dt} \nonumber \\
& = &  \mE_{G,{\mathcal U}_{r+1}} \frac{1}{\beta|s|\sqrt{n}\m_1\zeta_r}
\prod_{k=r}^{2}\mE_{{\mathcal U}_{k}} \zeta_{k-1}^{\frac{\m_k}{\m_{k-1}}-1}
 \frac{d\zeta_{1}}{dt} \nonumber \\
& = &  \mE_{G,{\mathcal U}_{r+1}} \frac{1}{\beta|s|\sqrt{n}\m_1\zeta_r}
\prod_{k=r}^{2}\mE_{{\mathcal U}_{k}} \zeta_{k-1}^{\frac{\m_k}{\m_{k-1}}-1}
 \frac{d \mE_{{\mathcal U}_1} Z^{\m_1} }{dt} \nonumber \\
 & = &
\mE_{G,{\mathcal U}_{r+1}} \frac{1}{\beta|s|\sqrt{n}\zeta_r}
\prod_{k=r}^{2}\mE_{{\mathcal U}_{k}} \zeta_{k-1}^{\frac{\m_k}{\m_{k-1}}-1}
 \mE_{{\mathcal U}_1} \frac{1}{Z^{1-\m_1}}  \sum_{i=1}^{l} (C^{(i_1)})^{s-1} \nonumber \\
& & \times \sum_{i_2=1}^{l}\beta_{i_1}A^{(i_1,i_2)}\frac{dD^{(i_1,i_2)}}{dt},
\end{eqnarray}
where the product runs in an \emph{index decreasing order} and
\begin{eqnarray}\label{eq:rthlev2genanal9a}
\frac{dD^{(i_1,i_2)}}{dt}= \lp \frac{dB^{(i_1,i_2)}}{dt}+\frac{\|\y^{(i_2)}\|_2 (\sum_{k=1}^{r+1} u^{(4,k)})}{2\sqrt{t}}-\frac{\|\y^{(i_2)}\|_2 (\sum_{k=1}^{r+1}\u^{(i_1,3,k)})}{2\sqrt{1-t}}\rp.
\end{eqnarray}
Utilizing (\ref{eq:rthlev2genanal7}) we find
\begin{eqnarray}\label{eq:rthlev2genanal10}
\frac{dB^{(i_1,i_2)}}{dt} & = &   \frac{d\lp\sqrt{t}(\y^{(i_2)})^T\u^{(i_1,1)}+\sqrt{1-t} (\y^{(i_2)})^T(\sum_{k=1}^{r+1} \u^{(2,k)})\rp}{dt} \nonumber \\
 & = &
\sum_{j=1}^{m}\lp \frac{\y_j^{(i_2)}\u_j^{(i_1,1)}}{2\sqrt{t}}-\frac{\y_j^{(i_2)} \sum_{k=1}^{r+1} \u_j^{(2,k)}}{2\sqrt{1-t}}\rp.
\end{eqnarray}
One can then write analogously to (\ref{eq:genanal10e})
\begin{equation}\label{eq:rthlev2genanal10e}
\frac{d\psi(\calX,\calY,\q,\m,\beta,s,t)}{dt}  =       \frac{\mbox{sign}(s)}{2\beta\sqrt{n}} \sum_{i_1=1}^{l}  \sum_{i_2=1}^{l}
\beta_{i_1}\lp T_G + \sum_{k=1}^{r+1}T_k\rp,
\end{equation}
where
\begin{eqnarray}\label{eq:rthlev2genanal10f}
T_G & = & \sum_{j=1}^{m}\frac{T_{G,j}}{\sqrt{t}}  \nonumber\\
T_k & = & -\sum_{j=1}^{m}\frac{T_{k,1,j}}{\sqrt{1-t}}-\|\y^{(i_2)}\|_2\frac{T_{k,2}}{\sqrt{1-t}}+\|\y^{(i_2)}\|_2\frac{T_{k,3}}{\sqrt{t}}, k\in\{1,2,\dots,r+1\}.
 \end{eqnarray}
 \begin{eqnarray}\label{eq:rthlev2genanal10g}
T_{G,j} & = &  \mE_{G,{\mathcal U}_{r+1}} \lp
\zeta_r^{-1}\prod_{v=r}^{2}\mE_{{\mathcal U}_{v}} \zeta_{v-1}^{\frac{\m_v}{\m_{v-1}}-1}
  \mE_{{\mathcal U}_1}\frac{(C^{(i_1)})^{s-1} A^{(i_1,i_2)} \y_j^{(i_2)}\u_j^{(i_1,1)}}{Z^{1-\m_1}} \rp \nonumber \\
T_{k,1,j} & = &   \mE_{G,{\mathcal U}_{r+1}} \lp
\zeta_r^{-1}\prod_{v=r}^{2}\mE_{{\mathcal U}_{v}} \zeta_{v-1}^{\frac{\m_v}{\m_{v-1}}-1}
  \mE_{{\mathcal U}_1}\frac{(C^{(i_1)})^{s-1} A^{(i_1,i_2)} \y_j^{(i_2)}\u_j^{(2,k)}}{Z^{1-\m_1}} \rp \nonumber \\
T_{k,2} & = &   \mE_{G,{\mathcal U}_{r+1}} \lp
\zeta_r^{-1}\prod_{v=r}^{2}\mE_{{\mathcal U}_{v}} \zeta_{v-1}^{\frac{\m_v}{\m_{v-1}}-1}
  \mE_{{\mathcal U}_1}\frac{(C^{(i_1)})^{s-1} A^{(i_1,i_2)} \u^{(i_1,3,k)}}{Z^{1-\m_1}} \rp \nonumber \\
T_{k,3} & = &  \mE_{G,{\mathcal U}_{r+1}} \lp
\zeta_r^{-1}\prod_{v=r}^{2}\mE_{{\mathcal U}_{v}} \zeta_{v-1}^{\frac{\m_v}{\m_{v-1}}-1}
  \mE_{{\mathcal U}_1}\frac{(C^{(i_1)})^{s-1} A^{(i_1,i_2)} u^{(4,k)}}{Z^{1-\m_1}} \rp.
\end{eqnarray}

We will first handle the 3 $r+1$-tuplets, $\lp T_{k,1,j}\rp_{k=1:r+1}$, $\lp T_{k,2}\rp_{k=1:r+1}$, and $\lp T_{k,3}\rp_{k=1:r+1}$ and then $T_{G,j}$. As the discussion in Section \ref{sec:seclev} showed, for all our purposes, the internal connections among the components within each of these three sequences are identical. We will below consider one of them, say $\lp T_{k,1,j}\rp_{k=1:r+1}$ and show how the results, obtained in Section \ref{sec:seclev}, for $r=2$, immediately extend to any $r$. To avoid unnecessarily presenting the same strategy three times, for the remaining two $r+1$-tuplets, $\lp T_{k,2}\rp_{k=1:r+1}$, and $\lp T_{k,3}\rp_{k=1:r+1}$, we will then quickly deduce the final results.

%%%%%%%%%%%%%%%%%%%%%%%%%%%%%%%%%%%%%%%%%%%%%%%%%%%%%%%%%%%%%%%%%%%%%%%%%%%%%%%%%%
%%%%%%%%%%%%%%%%%%%%%%%%%%%%%%%%%%%%%%%%%%%%%%%%%%%%%%%%%%%%%%%%%%%%%%%%%%%%%%%%%%
%%%%%%%%%%%%%%%%%%%%%%%%%%%%%%%%%%%%%%%%%%%%%%%%%%%%%%%%%%%%%%%%%%%%%%%%%%%%%%%%%%
%%%%%%%%%%%%%%%%%%%%%%%%%%%%%%%%%%%%%%%%%%%%%%%%%%%%%%%%%%%%%%%%%%%%%%%%%%%%%%%%%%
\subsection{Scaled canceling out --- successive scaling + reweightedly averaged new terms}
\label{sec:scaledcanout}
%%%%%%%%%%%%%%%%%%%%%%%%%%%%%%%%%%%%%%%%%%%%%%%%%%%%%%%%%%%%%%%%%%%%%%%%%%%%%%%%%%
%%%%%%%%%%%%%%%%%%%%%%%%%%%%%%%%%%%%%%%%%%%%%%%%%%%%%%%%%%%%%%%%%%%%%%%%%%%%%%%%%%
%%%%%%%%%%%%%%%%%%%%%%%%%%%%%%%%%%%%%%%%%%%%%%%%%%%%%%%%%%%%%%%%%%%%%%%%%%%%%%%%%%
%%%%%%%%%%%%%%%%%%%%%%%%%%%%%%%%%%%%%%%%%%%%%%%%%%%%%%%%%%%%%%%%%%%%%%%%%%%%%%%%%%

We take a component of the sequence $\lp T_{k,1,j}\rp_{k=1:r+1}$, say $T_{k_1,1,j}$, $k_1\geq 2$, and show how it relates to the very next one, $T_{k_1+1,1,j}$. As we will formalize it below, there are two key parts of such a recursive/inductive relation: 1) \emph{successive scaling}; and 2) appearance of an appropriately \emph{reweightedly averaged (over a $\gamma$ measure) new term}.

We start by noting that
 \begin{eqnarray}\label{eq:rthlev2genanal11}
 T_{k_1,1,j} & = &   \mE_{G,{\mathcal U}_{r+1}} \lp
\zeta_r^{-1}\prod_{v=r}^{2}\mE_{{\mathcal U}_{v}} \zeta_{v-1}^{\frac{\m_v}{\m_{v-1}}-1}
  \mE_{{\mathcal U}_1}\frac{(C^{(i_1)})^{s-1} A^{(i_1,i_2)} \y_j^{(i_2)}\u_j^{(2,k_1)}}{Z^{1-\m_1}} \rp \nonumber \\
 & = &   \mE_{G,{\mathcal U}_{r+1}} \lp
\zeta_r^{-1}\prod_{v=r}^{k_1+1}\mE_{{\mathcal U}_{v}} \zeta_{v-1}^{\frac{\m_v}{\m_{v-1}}-1}
\prod_{v=k_1}^{2}\mE_{{\mathcal U}_{v}} \zeta_{v-1}^{\frac{\m_v}{\m_{v-1}}-1}
  \mE_{{\mathcal U}_1}\frac{(C^{(i_1)})^{s-1} A^{(i_1,i_2)} \y_j^{(i_2)}\u_j^{(2,k_1)}}{Z^{1-\m_1}} \rp \nonumber \\
    & = &   \mE_{G,{\mathcal U}_{r+1}} \Bigg(\Bigg.
\zeta_r^{-1}\prod_{v=r}^{k_1+1}\mE_{{\mathcal U}_{v}} \zeta_{v-1}^{\frac{\m_v}{\m_{v-1}}-1}
\mE_{{\mathcal U}_{k_1}} \mE_{{\mathcal U}_{k_1}} \lp \u_j^{(2,k_1)}\u_j^{(2,k_1)}\rp \nonumber \\
& & \times
\frac{d}{d\u_j^{(2,k_1)}} \lp \zeta_{k_1-1}^{\frac{\m_{k_1}}{\m_{k_1-1}}-1}\prod_{v=k_1-1}^{2}\mE_{{\mathcal U}_{v}} \zeta_{v-1}^{\frac{\m_v}{\m_{v-1}}-1}
  \mE_{{\mathcal U}_1}\frac{(C^{(i_1)})^{s-1} A^{(i_1,i_2)} \y_j^{(i_2)}\u_j^{(2,k_1)}}{Z^{1-\m_1}} \rp \Bigg.\Bigg) \nonumber \\
   & = &  (\p_{k_1-1}-\p_{k_1}) \mE_{G,{\mathcal U}_{r+1}} \Bigg(\Bigg.
\zeta_r^{-1}\prod_{v=r}^{k_1+1}\mE_{{\mathcal U}_{v}} \zeta_{v-1}^{\frac{\m_v}{\m_{v-1}}-1}
\mE_{{\mathcal U}_{k_1}}  \nonumber \\
& & \times
\frac{d}{d\u_j^{(2,k_1)}} \lp \zeta_{k_1-1}^{\frac{\m_{k_1}}{\m_{k_1-1}}-1}\prod_{v=k_1-1}^{2}\mE_{{\mathcal U}_{v}} \zeta_{v-1}^{\frac{\m_v}{\m_{v-1}}-1}
  \mE_{{\mathcal U}_1}\frac{(C^{(i_1)})^{s-1} A^{(i_1,i_2)} \y_j^{(i_2)}\u_j^{(2,k_1)}}{Z^{1-\m_1}} \rp \Bigg.\Bigg)  ,
\end{eqnarray}
and also
 \begin{eqnarray}\label{eq:rthlev2genanal12}
 T_{k_1+1,1,j}
  & = &   \mE_{G,{\mathcal U}_{r+1}} \lp
\zeta_r^{-1}\prod_{v=r}^{k_1+2}\mE_{{\mathcal U}_{v}} \zeta_{v-1}^{\frac{\m_v}{\m_{v-1}}-1}
\prod_{v=k_1+1}^{2}\mE_{{\mathcal U}_{v}} \zeta_{v-1}^{\frac{\m_v}{\m_{v-1}}-1}
  \mE_{{\mathcal U}_1}\frac{(C^{(i_1)})^{s-1} A^{(i_1,i_2)} \y_j^{(i_2)}\u_j^{(2,k_1+1)}}{Z^{1-\m_1}} \rp \nonumber \\
    & = &   \mE_{G,{\mathcal U}_{r+1}} \Bigg( \Bigg.
\zeta_r^{-1}\prod_{v=r}^{k_1+2}\mE_{{\mathcal U}_{v}} \zeta_{v-1}^{\frac{\m_v}{\m_{v-1}}-1}
\mE_{{\mathcal U}_{k_1+1}}   \nonumber \\
& & \times  \zeta_{k_1}^{\frac{\m_{k_1+1}}{\m_{k_1}}-1}
\prod_{v=k_1}^{2}\mE_{{\mathcal U}_{v}} \zeta_{v-1}^{\frac{\m_v}{\m_{v-1}}-1}
  \mE_{{\mathcal U}_1}\frac{(C^{(i_1)})^{s-1} A^{(i_1,i_2)} \y_j^{(i_2)}\u_j^{(2,k_1+1)}}{Z^{1-\m_1}} \Bigg. \Bigg) \nonumber \\
      & = &   \mE_{G,{\mathcal U}_{r+1}} \Bigg( \Bigg.
\zeta_r^{-1}\prod_{v=r}^{k_1+2}\mE_{{\mathcal U}_{v}} \zeta_{v-1}^{\frac{\m_v}{\m_{v-1}}-1}
\mE_{{\mathcal U}_{k_1+1}} \mE_{{\mathcal U}_{k_1+1}} (\u_j^{(2,k_1+1)}\u_j^{(2,k_1+1)}) \nonumber \\
& & \times  \frac{d}{d\u_j^{(2,k_1+1)}}  \lp \zeta_{k_1}^{\frac{\m_{k_1+1}}{\m_{k_1}}-1}
\prod_{v=k_1}^{2}\mE_{{\mathcal U}_{v}} \zeta_{v-1}^{\frac{\m_v}{\m_{v-1}}-1}
  \mE_{{\mathcal U}_1}\frac{(C^{(i_1)})^{s-1} A^{(i_1,i_2)} \y_j^{(i_2)}}{Z^{1-\m_1}} \rp   \Bigg. \Bigg)   \nonumber \\
      & = & (\p_{k_1} -\p_{k_1+1})
  \mE_{G,{\mathcal U}_{r+1}} \Bigg( \Bigg.
\zeta_r^{-1}\prod_{v=r}^{k_1+2}\mE_{{\mathcal U}_{v}} \zeta_{v-1}^{\frac{\m_v}{\m_{v-1}}-1}
\mE_{{\mathcal U}_{k_1+1}}  \nonumber \\
& & \times  \frac{d}{d\u_j^{(2,k_1+1)}}  \lp \zeta_{k_1}^{\frac{\m_{k_1+1}}{\m_{k_1}}-1}
\prod_{v=k_1}^{2}\mE_{{\mathcal U}_{v}} \zeta_{v-1}^{\frac{\m_v}{\m_{v-1}}-1}
  \mE_{{\mathcal U}_1}\frac{(C^{(i_1)})^{s-1} A^{(i_1,i_2)} \y_j^{(i_2)}}{Z^{1-\m_1}} \rp   \Bigg. \Bigg)   \nonumber \\
& = &  \Pi_{k_1+1,1,j}^{(1)}
+  \Pi_{k_1+1,1,j}^{(2)},
 \end{eqnarray}
where
 \begin{eqnarray}\label{eq:rthlev2genanal13}
 \Pi_{k_1+1,1,j}^{(1)}
      & = & (\p_{k_1}-\p_{k_1+1})
  \mE_{G,{\mathcal U}_{r+1}} \Bigg( \Bigg.
\zeta_r^{-1}\prod_{v=r}^{k_1+2}\mE_{{\mathcal U}_{v}} \zeta_{v-1}^{\frac{\m_v}{\m_{v-1}}-1}
\mE_{{\mathcal U}_{k_1+1}}  \nonumber \\
& & \times \zeta_{k_1}^{\frac{\m_{k_1+1}}{\m_{k_1}}-1}
 \frac{d}{d\u_j^{(2,k_1+1)}}  \lp
\prod_{v=k_1}^{2}\mE_{{\mathcal U}_{v}} \zeta_{v-1}^{\frac{\m_v}{\m_{v-1}}-1}
  \mE_{{\mathcal U}_1}\frac{(C^{(i_1)})^{s-1} A^{(i_1,i_2)} \y_j^{(i_2)}}{Z^{1-\m_1}} \rp   \Bigg. \Bigg)   \nonumber \\
\Pi_{k_1+1,1,j}^{(2)}
& = &(\p_{k_1}-\p_{k_1+1})
  \mE_{G,{\mathcal U}_{r+1}} \Bigg( \Bigg.
\zeta_r^{-1}\prod_{v=r}^{k_1+2}\mE_{{\mathcal U}_{v}} \zeta_{v-1}^{\frac{\m_v}{\m_{v-1}}-1}
\mE_{{\mathcal U}_{k_1+1}}  \nonumber \\
& & \times \prod_{v=k_1}^{2}\mE_{{\mathcal U}_{v}} \zeta_{v-1}^{\frac{\m_v}{\m_{v-1}}-1}
  \mE_{{\mathcal U}_1}\frac{(C^{(i_1)})^{s-1} A^{(i_1,i_2)} \y_j^{(i_2)}}{Z^{1-\m_1}}
   \frac{d}{d\u_j^{(2,k_1+1)}}  \lp \zeta_{k_1}^{\frac{\m_{k_1+1}}{\m_{k_1}}-1}
 \rp   \Bigg. \Bigg).
\end{eqnarray}
It is useful to note here that $\Pi_{k_1+1,1,j}^{(1)}$ relates to the first part mentioned above -- \emph{successive scaling}, and that $\Pi_{k_1+1,1,j}^{(2)}$ relates to the second part -- introducing a \emph{reweightedly averaged (over a $\gamma$ measure) new term}.

For $\Pi_{k_1+1,1,j}^{(1)}$ we further have
 \begin{eqnarray}\label{eq:rthlev2genanal14}
 \Pi_{k_1+1,1,j}^{(1)}
      & = & (\p_{k_1}-\p_{k_1+1})
  \mE_{G,{\mathcal U}_{r+1}} \Bigg( \Bigg.
\zeta_r^{-1}\prod_{v=r}^{k_1+2}\mE_{{\mathcal U}_{v}} \zeta_{v-1}^{\frac{\m_v}{\m_{v-1}}-1}
\mE_{{\mathcal U}_{k_1+1}}  \nonumber \\
& & \times \zeta_{k_1}^{\frac{\m_{k_1+1}}{\m_{k_1}}-1}
 \frac{d}{d\u_j^{(2,k_1+1)}}  \lp
\prod_{v=k_1}^{2}\mE_{{\mathcal U}_{v}} \zeta_{v-1}^{\frac{\m_v}{\m_{v-1}}-1}
  \mE_{{\mathcal U}_1}\frac{(C^{(i_1)})^{s-1} A^{(i_1,i_2)} \y_j^{(i_2)}}{Z^{1-\m_1}} \rp   \Bigg. \Bigg)   \nonumber \\
      & = & (\p_{k_1}-\p_{k_1+1})
  \mE_{G,{\mathcal U}_{r+1}} \Bigg( \Bigg.
\zeta_r^{-1}\prod_{v=r}^{k_1+1}\mE_{{\mathcal U}_{v}} \zeta_{v-1}^{\frac{\m_v}{\m_{v-1}}-1}
\nonumber \\
& & \times
 \frac{d}{d\u_j^{(2,k_1+1)}}  \lp
\prod_{v=k_1}^{2}\mE_{{\mathcal U}_{v}} \zeta_{v-1}^{\frac{\m_v}{\m_{v-1}}-1}
  \mE_{{\mathcal U}_1}\frac{(C^{(i_1)})^{s-1} A^{(i_1,i_2)} \y_j^{(i_2)}}{Z^{1-\m_1}} \rp   \Bigg. \Bigg)   \nonumber \\
      & = & (\p_{k_1}-\p_{k_1+1})
  \mE_{G,{\mathcal U}_{r+1}} \Bigg( \Bigg.
\zeta_r^{-1}\prod_{v=r}^{k_1+1}\mE_{{\mathcal U}_{v}} \zeta_{v-1}^{\frac{\m_v}{\m_{v-1}}-1}
\nonumber \\
& & \times
 \frac{d}{d\u_j^{(2,k_1+1)}}  \lp
\mE_{{\mathcal U}_{k_1}} \zeta_{k_1-1}^{\frac{\m_{k_1}}{\m_{k_1-1}}-1}
\prod_{v=k_1-1}^{2}\mE_{{\mathcal U}_{v}} \zeta_{v-1}^{\frac{\m_v}{\m_{v-1}}-1}
  \mE_{{\mathcal U}_1}\frac{(C^{(i_1)})^{s-1} A^{(i_1,i_2)} \y_j^{(i_2)}}{Z^{1-\m_1}} \rp   \Bigg. \Bigg)   \nonumber \\
      & = & (\p_{k_1}-\p_{k_1+1})
  \mE_{G,{\mathcal U}_{r+1}} \Bigg( \Bigg.
\zeta_r^{-1}\prod_{v=r}^{k_1+1}\mE_{{\mathcal U}_{v}} \zeta_{v-1}^{\frac{\m_v}{\m_{v-1}}-1}
\nonumber \\
& & \times
\mE_{{\mathcal U}_{k_1}}  \frac{d}{d\u_j^{(2,k_1+1)}}  \lp
\zeta_{k_1-1}^{\frac{\m_{k_1}}{\m_{k_1-1}}-1}
\prod_{v=k_1-1}^{2}\mE_{{\mathcal U}_{v}} \zeta_{v-1}^{\frac{\m_v}{\m_{v-1}}-1}
  \mE_{{\mathcal U}_1}\frac{(C^{(i_1)})^{s-1} A^{(i_1,i_2)} \y_j^{(i_2)}}{Z^{1-\m_1}} \rp   \Bigg. \Bigg).   \nonumber \\
 \end{eqnarray}
One now observes from (\ref{eq:rthlev2genanal7}) that $\u_j^{(2,k_1+1)}$ appears only in $B^{(i_1,i_2)}$, and then via $B^{(i_1,i_2)}$ in $D^{(i_1,i_2)}$, via $D^{(i_1,i_2)}$ in $A^{(i_1,i_2)}$, via $A^{(i_1,i_2)}$ in $C^{(i_1,i_2)}$, via $C^{(i_1)}$ in $Z$, and via $Z$ in $\zeta_k$, $k\in\{1,2,\dots,r\}$. This implies that for a function $f_c(B^{(i_1,i_2)})$
 \begin{eqnarray}\label{eq:rthlev2genanal15}
 f_c(B^{(i_1,i_2)}) = \lp
\zeta_{k_1-1}^{\frac{\m_{k_1}}{\m_{k_1-1}}-1}
\prod_{v=k_1-1}^{2}\mE_{{\mathcal U}_{v}} \zeta_{v-1}^{\frac{\m_v}{\m_{v-1}}-1}
  \mE_{{\mathcal U}_1}\frac{(C^{(i_1)})^{s-1} A^{(i_1,i_2)} \y_j^{(i_2)}}{Z^{1-\m_1}} \rp,
 \end{eqnarray}
we have
 \begin{eqnarray}\label{eq:rthlev2genanal16}
   \frac{d}{d\u_j^{(2,k_1+1)}}  \lp f_c(B^{(i_1,i_2)}) \rp=    \frac{d}{d B^{(i_1,i_2)}}  \lp f_c(B^{(i_1,i_2)})  \rp  \frac{d B^{(i_1,i_2)}}{d\u_j^{(2,k_1+1)}} = \sqrt{1-t}\y_j^{(i_2)}.
 \end{eqnarray}
Since there is nothing specific about $k_1$ one also has
 \begin{eqnarray}\label{eq:rthlev2genanal17}
   \frac{d}{d\u_j^{(2,k_1)}}  \lp f_c(B^{(i_1,i_2)}) \rp=    \frac{d}{d B^{(i_1,i_2)}}  \lp f_c(B^{(i_1,i_2)})  \rp  \frac{d B^{(i_1,i_2)}}{d\u_j^{(2,k_1+1)}} = \sqrt{1-t}\y_j^{(i_2)},
 \end{eqnarray}
and
 \begin{eqnarray}\label{eq:rthlev2genanal18}
   \frac{d}{d\u_j^{(2,k_1)}}  \lp f_c(B^{(i_1,i_2)}) \rp=
   \frac{d}{d\u_j^{(2,k_1+1)}}  \lp f_c(B^{(i_1,i_2)}) \rp.
 \end{eqnarray}
A combination of (\ref{eq:rthlev2genanal11}), (\ref{eq:rthlev2genanal14}), and (\ref{eq:rthlev2genanal18}) then gives the \emph{successive scaling} relation (which is at the heart of the \emph{canceling out} mechanism that will be visible a bit later)
 \begin{eqnarray}\label{eq:rthlev2genanal19}
 \Pi_{k_1+1,1,j}^{(1)}=\frac{(\p_{k_1} -\p_{k_1+1})}{(\p_{k_1-1}-\p_{k_1})}T_{k_1,1,j}.
 \end{eqnarray}

For $\Pi_{k_1+1,1,j}^{(2)}$, we have
\begin{eqnarray}\label{eq:rthlev2genanal20}
 \Pi_{k_1+1,1,j}^{(2)}
& = &(\p_{k_1}-\p_{k_1+1})
  \mE_{G,{\mathcal U}_{r+1}} \Bigg( \Bigg.
\zeta_r^{-1}\prod_{v=r}^{k_1+2}\mE_{{\mathcal U}_{v}} \zeta_{v-1}^{\frac{\m_v}{\m_{v-1}}-1}
\mE_{{\mathcal U}_{k_1+1}}  \nonumber \\
& & \times \prod_{v=k_1}^{2}\mE_{{\mathcal U}_{v}} \zeta_{v-1}^{\frac{\m_v}{\m_{v-1}}-1}
  \mE_{{\mathcal U}_1}\frac{(C^{(i_1)})^{s-1} A^{(i_1,i_2)} \y_j^{(i_2)}}{Z^{1-\m_1}}
   \frac{d}{d\u_j^{(2,k_1+1)}}  \lp \zeta_{k_1}^{\frac{\m_{k_1+1}}{\m_{k_1}}-1}
 \rp   \Bigg. \Bigg),
\end{eqnarray}
where based on the sequence of equalities of (\ref{eq:rthlev2genanal12}) we find
\begin{eqnarray}\label{eq:rthlev2genanal21}
    \frac{d}{d\u_j^{(2,k_1+1)}}  \lp \zeta_{k_1}^{\frac{\m_{k_1+1}}{\m_{k_1}}-1}
 \rp   \Bigg. \Bigg)
 & = &   \lp \frac{\m_{k_1+1}-\m_{k_1}}{\m_{k_1}}\rp \zeta_{k_1}^{\frac{\m_{k_1+1}}{\m_{k_1}}-2}
    \frac{d  \zeta_{k_1}}{d\u_j^{(2,k_1+1)}}  \nonumber \\
 & =  & \lp \frac{\m_{k_1+1}-\m_{k_1}}{\m_{1}}\rp \zeta_{k_1}^{\frac{\m_{k_1+1}}{\m_{k_1}}-2}
\prod_{k=k_1}^{2}\mE_{{\mathcal U}_{k}} \zeta_{k-1}^{\frac{\m_k}{\m_{k-1}}-1}
 \frac{d \mE_{{\mathcal U}_1} Z^{\m_1} }{d\u_j^{(2,k_1+1)}} \nonumber \\
 & =  & s \lp \m_{k_1+1} - \m_{k_1}\rp \zeta_{k_1}^{\frac{\m_{k_1+1}}{\m_{k_1}}-2}
\prod_{k=k_1}^{2}\mE_{{\mathcal U}_{k}} \zeta_{k-1}^{\frac{\m_k}{\m_{k-1}}-1}
 \mE_{{\mathcal U}_1} \frac{1}{Z^{1-\m_1}}  \sum_{p_1=1}^{l} (C^{(p_1)})^{s-1} \nonumber \\
& & \times \sum_{p_2=1}^{l}\beta_{p_1}A^{(p_1,p_2)}\frac{dD^{(p_1,p_2)}}{d\u_j^{(2,k_1+1)}}.
\end{eqnarray}
Since
\begin{eqnarray}\label{eq:rthlev2genanal22}
 \frac{dD^{(p_1,p_2)}}{d\u_j^{(2,k_1+1)}}=
 \frac{dB^{(p_1,p_2)}}{d\u_j^{(2,k_1+1)}}=\sqrt{1-t} \y_j^{(p_2)},
\end{eqnarray}
we, from (\ref{eq:rthlev2genanal21}), also have
\begin{eqnarray}\label{eq:rthlev2genanal23}
    \frac{d}{d\u_j^{(2,k_1+1)}}  \lp \zeta_{k_1}^{\frac{\m_{k_1+1}}{\m_{k_1}}-1}
 \rp   \Bigg. \Bigg)
  & =  & s \lp \m_{k_1+1} - \m_{k_1} \rp \zeta_{k_1}^{\frac{\m_{k_1+1}}{\m_{k_1}}-2}
\prod_{k=k_1}^{2}\mE_{{\mathcal U}_{k}} \zeta_{k-1}^{\frac{\m_k}{\m_{k-1}}-1}
 \mE_{{\mathcal U}_1} \frac{1}{Z^{1-\m_1}}  \sum_{p_1=1}^{l} (C^{(p_1)})^{s-1} \nonumber \\
& & \times \sum_{p_2=1}^{l}\beta_{p_1}A^{(p_1,p_2)}\sqrt{1-t} \y_j^{(p_2)}.
\end{eqnarray}
A combination of (\ref{eq:rthlev2genanal20}) and (\ref{eq:rthlev2genanal23}) gives
\begin{eqnarray}\label{eq:rthlev2genanal24}
 \Pi_{k_1+1,1,j}^{(2)}
& = &(\p_{k_1} -\p_{k_1+1})
  \mE_{G,{\mathcal U}_{r+1}} \Bigg( \Bigg.
\zeta_r^{-1}\prod_{v=r}^{k_1+2}\mE_{{\mathcal U}_{v}} \zeta_{v-1}^{\frac{\m_v}{\m_{v-1}}-1}
\mE_{{\mathcal U}_{k_1+1}}  \nonumber \\
& & \times \prod_{v=k_1}^{2}\mE_{{\mathcal U}_{v}} \zeta_{v-1}^{\frac{\m_v}{\m_{v-1}}-1}
  \mE_{{\mathcal U}_1}\frac{(C^{(i_1)})^{s-1} A^{(i_1,i_2)} \y_j^{(i_2)}}{Z^{1-\m_1}}
s \lp \m_{k_1+1} - \m_1 \rp \zeta_{k_1}^{\frac{\m_{k_1+1}}{\m_{k_1}}-2} \nonumber \\
& & \times \prod_{k=k_1}^{2}\mE_{{\mathcal U}_{k}} \zeta_{k-1}^{\frac{\m_k}{\m_{k-1}}-1}
 \mE_{{\mathcal U}_1} \frac{1}{Z^{1-\m_1}}  \sum_{p_1=1}^{l} (C^{(p_1)})^{s-1}
 \sum_{p_2=1}^{l}\beta_{p_1}A^{(p_1,p_2)}\sqrt{1-t} \y_j^{(p_2)}
   \Bigg. \Bigg) \nonumber \\
& = & s(\p_{k_1}-\p_{k_1+1}) \lp \m_{k_1+1} - \m_{k_1} \rp
  \mE_{G,{\mathcal U}_{r+1}} \Bigg( \Bigg.
\zeta_r^{-1}\prod_{v=r}^{k_1+2}\mE_{{\mathcal U}_{v}} \zeta_{v-1}^{\frac{\m_v}{\m_{v-1}}-1}
\mE_{{\mathcal U}_{k_1+1}}  \zeta_{k_1}^{\frac{\m_{k_1+1}}{\m_{k_1}}}  \nonumber \\
& & \times \prod_{v=k_1}^{2}\mE_{{\mathcal U}_{v}} \frac{\zeta_{v-1}^{\frac{\m_v}{\m_{v-1}}}}{\zeta_v}
   \mE_{{\mathcal U}_1}\frac{Z^{\m_1}}{\zeta_1}\frac{(C^{(i_1)})^{s-1} A^{(i_1,i_2)} \y_j^{(i_2)}}{Z}
\nonumber \\
& & \times \prod_{k=k_1}^{2}  \mE_{{\mathcal U}_{k}} \frac{\zeta_{k-1}^{\frac{\m_k}{\m_{k-1}}}}{\zeta_k}
 \mE_{{\mathcal U}_1}\frac{Z^{\m_1}}{\zeta_1}  \sum_{p_1=1}^{l} \frac{(C^{(p_1)})^s}{Z}
 \sum_{p_2=1}^{l} \frac{A^{(p_1,p_2)}}{C^{(p_1)}}\sqrt{1-t} \beta_{p_1} \y_j^{(p_2)}
   \Bigg. \Bigg) \nonumber \\
& = & s(\p_{k_1}-\p_{k_1+1}) \lp \m_{k_1+1} - \m_{k_1} \rp
  \mE_{G,{\mathcal U}_{r+1}} \Bigg( \Bigg.
\prod_{v=r}^{k_1+1} \mE_{{\mathcal U}_{v}} \frac{\zeta_{v-1}^{\frac{\m_v}{\m_{v-1}}}}{\zeta_v} \nonumber \\
& & \times \prod_{v=k_1}^{2}\mE_{{\mathcal U}_{v}} \frac{\zeta_{v-1}^{\frac{\m_v}{\m_{v-1}}}}{\zeta_v}
   \mE_{{\mathcal U}_1}\frac{Z^{\m_1}}{\zeta_1}   \frac{(C^{(i_1)})^{s} }{Z}
\frac{ A^{(i_1,i_2)} }{C^{(i_1)}}\y_j^{(i_2)}
\nonumber \\
& & \times \prod_{k=k_1}^{2}  \mE_{{\mathcal U}_{k}} \frac{\zeta_{k-1}^{\frac{\m_k}{\m_{k-1}}}}{\zeta_k}
 \mE_{{\mathcal U}_1}\frac{Z^{\m_1}}{\zeta_1}  \sum_{p_1=1}^{l} \frac{(C^{(p_1)})^s}{Z}
 \sum_{p_2=1}^{l} \frac{A^{(p_1,p_2)}}{C^{(p_1)}}\sqrt{1-t} \beta_{p_1} \y_j^{(p_2)}
   \Bigg. \Bigg).
\end{eqnarray}
After setting
\begin{eqnarray}\label{eq:rthlev2genanal25}
 \Phi_{{\mathcal U}_k} & \triangleq &  \mE_{{\mathcal U}_{k}} \frac{\zeta_{k-1}^{\frac{\m_k}{\m_{k-1}}}}{\zeta_k} \nonumber \\
 \gamma_0(i_1,i_2) & = &
\frac{(C^{(i_1)})^{s}}{Z}  \frac{A^{(i_1,i_2)}}{C^{(i_1)}} \nonumber \\
\gamma_{01}^{(r)}  & = & \prod_{k=r}^{1}\Phi_{{\mathcal U}_k} (\gamma_0(i_1,i_2)) \nonumber \\
\gamma_{02}^{(r)}  & = & \prod_{k=r}^{1}\Phi_{{\mathcal U}_k} (\gamma_0(i_1,i_2)\times \gamma_0(i_1,p_2)) \nonumber \\
\gamma_{k_1+1}^{(r)}  & = & \prod_{k=r}^{k_1+1}\Phi_{{\mathcal U}_k} \lp \prod_{k=k_1}^{1}\Phi_{{\mathcal U}_k}\gamma_0(i_1,i_2)\times \prod_{k=k_1}^{1} \Phi_{{\mathcal U}_k}\gamma_0(p_1,p_2) \rp,
 \end{eqnarray}
and recalling that $\zeta_0=Z$ and $\m_0=1$, we, from (\ref{eq:rthlev2genanal24}), obtain
\begin{eqnarray}\label{eq:rthlev2genanal26}
 \sum_{i_1=1}^{l}\sum_{i_2=1}^{l} \sum_{j=1}^{m}
\beta_{i_1} \frac{\Pi_{k_1+1,1,j}^{(2)}}{\sqrt{1-t}}
 & = & -s(\p_{k_1}-\p_{k_1+1}) \lp \m_{k_1}  - \m_{k_1+1} \rp
  \mE_{G,{\mathcal U}_{r+1}} \Bigg( \Bigg.
\prod_{v=r}^{k_1+1} \Phi_{{\mathcal U}_v} \nonumber \\
& & \times \prod_{v=k_1}^{2} \Phi_{{\mathcal U}_v}\sum_{i_1=1}^{l}\sum_{i_2=1}^{l} \frac{(C^{(i_1)})^{s} }{Z}
\frac{ A^{(i_1,i_2)} }{C^{(i_1)}}\y_j^{(i_2)}
\nonumber \\
& & \times \prod_{k=k_1}^{2} \Phi_{{\mathcal U}_k} \sum_{p_1=1}^{l} \frac{(C^{(p_1)})^s}{Z}
 \sum_{p_2=1}^{l} \frac{A^{(p_1,p_2)}}{C^{(p_1)}} \beta_{i_1} \beta_{p_1} (\y^{(p_2)})^T \y^{(i_2)}
   \Bigg. \Bigg) \nonumber \\
 & = & -s\beta^2(\p_{k_1}-\p_{k_1+1}) \lp \m_{k_1} - \m_{k_1+1} \rp \nonumber\\
 & & \times
  \mE_{G,{\mathcal U}_{r+1}} \left \langle \|\x^{(i_1)}\|_2\|\x^{(p_1)}\|_2 (\y^{(p_2)})^T \y^{(i_2)} \right \rangle_{\gamma_{k_1+1}^{(r)}}.
\end{eqnarray}
Combining (\ref{eq:rthlev2genanal12}), (\ref{eq:rthlev2genanal19}), and (\ref{eq:rthlev2genanal26}), we find
 \begin{eqnarray}\label{eq:rthlev2genanal27}
 \sum_{i_1=1}^{l}\sum_{i_2=1}^{l} \sum_{j=1}^{m}
\beta_{i_1} \frac{T_{k_1+1,1,j}}{\sqrt{1-t}}
& = &  \sum_{i_1=1}^{l}\sum_{i_2=1}^{l} \sum_{j=1}^{m}
\beta_{i_1} \frac{\Pi_{k_1+1,1,j}^{(1)}}{\sqrt{1-t}} +  \sum_{i_1=1}^{l}\sum_{i_2=1}^{l} \sum_{j=1}^{m}
\beta_{i_1} \frac{\Pi_{k_1+1,1,j}^{(2)}}{\sqrt{1-t}} \nonumber \\
& = &  \frac{\p_{k_1}-\p_{k_1+1}}{\p_{k_1-1}-\p_{k_1}}\sum_{i_1=1}^{l}\sum_{i_2=1}^{l} \sum_{j=1}^{m}
\beta_{i_1} \frac{T_{k_1,1,j}}{\sqrt{1-t}} \nonumber \\
 &  & - \Bigg( \Bigg. s\beta^2(\p_{k_1}-\p_{k_1+1}) \lp \m_{k_1} -  \m_{k_1+1} \rp  \nonumber \\
 & & \times
  \mE_{G,{\mathcal U}_{r+1}} \left \langle \|\x^{(i_1)}\|_2\|\x^{(p_1)}\|_2 (\y^{(p_2)})^T \y^{(i_2)} \right \rangle_{\gamma_{k_1+1}^{(r)} } \Bigg. \Bigg).
\end{eqnarray}
Keeping in mind the results for first and second level of full lifting, one can then analogously easily write for the other two sequences $\lp T_{k,2}\rp_{k=1:r+1}$ and $\lp T_{k,3}\rp_{k=1:r+1}$
\begin{eqnarray}\label{eq:rthlev2genanal28}
 \sum_{i_1=1}^{l}\sum_{i_2=1}^{l}
\beta_{i_1} \|\y^{(i_2)}\|_2 \frac{T_{k_1+1,2}}{\sqrt{1-t}}
 & = &  \frac{\q_{k_1}-\q_{k_1+1}}{\q_{k_1-1}-\q_{k_1}}\sum_{i_1=1}^{l}\sum_{i_2=1}^{l}
\beta_{i_1} \|\y^{(i_2)}\|_2 \frac{T_{k_1,2}}{\sqrt{1-t}} \nonumber \\
 &  & - \Bigg( \Bigg. s\beta^2(\q_{k_1}-\q_{k_1+1}) \lp \m_{k_1} - \m_{k_1+1} \rp \nonumber \\
 & &
\times  \mE_{G,{\mathcal U}_{r+1}} \left \langle   (\x^{(p_1)})^T \x^{(i_1)} \|\y^{(i_2)}\|_2\|\y^{(p_2)}\|_2  \right \rangle_{\gamma_{k_1+1}^{(r)} } \Bigg.\Bigg), \nonumber \\
\end{eqnarray}
and
 \begin{eqnarray}\label{eq:rthlev2genanal29}
 \sum_{i_1=1}^{l}\sum_{i_2=1}^{l}
\beta_{i_1} \|\y^{(i_2)}\|_2 \frac{T_{k_1+1,3}}{\sqrt{1-t}}
 & = &  \frac{\p_{k_1}\q_{k_1}-\p_{k_1+1}\q_{k_1+1}}{\p_{k_1-1}\q_{k_1-1}-\p_{k_1}\q_{k_1}}\sum_{i_1=1}^{l}\sum_{i_2=1}^{l}
\beta_{i_1} \|\y^{(i_2)}\|_2 \frac{T_{k_1,3}}{\sqrt{1-t}} \nonumber \\
 &  & - \Bigg( \Bigg. s\beta^2(\p_{k_1}\q_{k_1}-\p_{k_1+1}\q_{k_1+1}) \lp  \m_{k_1} - \m_{k_1+1} \rp \nonumber \\
& & \times   \mE_{G,{\mathcal U}_{r+1}} \left \langle   \|\x^{(i_1)}\|_2\|\x^{(p_1)}\|_2 \|\y^{(i_2)}\|_2\|\y^{(p_2)}\|_2  \right \rangle_{\gamma_{k_1+1}^{(r)} } \Bigg. \Bigg).
\end{eqnarray}
One now observes that adding the elements in each of these three sequences results in canceling out parts of the first and second summands. A bit later on, we will make this observation fully formal.

%%%%%%%%%%%%%%%%%%%%%%%%%%%%%%%%%%%%%%%%%%%%%%%%%%%%%%%%%%%%%%%%%%%%%%%%%%%%%%%%%%
%%%%%%%%%%%%%%%%%%%%%%%%%%%%%%%%%%%%%%%%%%%%%%%%%%%%%%%%%%%%%%%%%%%%%%%%%%%%%%%%%%
%%%%%%%%%%%%%%%%%%%%%%%%%%%%%%%%%%%%%%%%%%%%%%%%%%%%%%%%%%%%%%%%%%%%%%%%%%%%%%%%%%
%%%%%%%%%%%%%%%%%%%%%%%%%%%%%%%%%%%%%%%%%%%%%%%%%%%%%%%%%%%%%%%%%%%%%%%%%%%%%%%%%%
\subsection{Handling $T_{G,j}$ for $r$ levels of full lifting}
\label{sec:rthTG}
%%%%%%%%%%%%%%%%%%%%%%%%%%%%%%%%%%%%%%%%%%%%%%%%%%%%%%%%%%%%%%%%%%%%%%%%%%%%%%%%%%
%%%%%%%%%%%%%%%%%%%%%%%%%%%%%%%%%%%%%%%%%%%%%%%%%%%%%%%%%%%%%%%%%%%%%%%%%%%%%%%%%%
%%%%%%%%%%%%%%%%%%%%%%%%%%%%%%%%%%%%%%%%%%%%%%%%%%%%%%%%%%%%%%%%%%%%%%%%%%%%%%%%%%
%%%%%%%%%%%%%%%%%%%%%%%%%%%%%%%%%%%%%%%%%%%%%%%%%%%%%%%%%%%%%%%%%%%%%%%%%%%%%%%%%%

Introducing $\m_{r+1}=0$, we recall from (\ref{eq:rthlev2genanal10g})
 \begin{eqnarray}\label{eq:rthlev2genanal30}
T_{G,j} & = &  \mE_{G,{\mathcal U}_{r+1}} \lp
\zeta_r^{-1}\prod_{v=r}^{2}\mE_{{\mathcal U}_{v}} \zeta_{v-1}^{\frac{\m_v}{\m_{v-1}}-1}
  \mE_{{\mathcal U}_1}\frac{(C^{(i_1)})^{s-1} A^{(i_1,i_2)} \y_j^{(i_2)}\u_j^{(i_1,1)}}{Z^{1-\m_1}} \rp \nonumber \\
 & = &  \mE_{{\mathcal U}_{r+1},{{\mathcal U}_{r}},\dots,{{\mathcal U}_{1}}} \lp
\mE_{G}\zeta_r^{-1}\prod_{v=r}^{2} \zeta_{v-1}^{\frac{\m_v}{\m_{v-1}}-1}
    \frac{(C^{(i_1)})^{s-1} A^{(i_1,i_2)} \y_j^{(i_2)}\u_j^{(i_1,1)}}{Z^{1-\m_1}} \rp \nonumber \\
 & = &  \mE_{{\mathcal U}_{r+1},{{\mathcal U}_{r}},\dots,{{\mathcal U}_{1}}} \lp
\mE_{G} \prod_{v=r+1}^{2} \zeta_{v-1}^{\frac{\m_v}{\m_{v-1}}-1}
    \frac{(C^{(i_1)})^{s-1} A^{(i_1,i_2)} \y_j^{(i_2)}\u_j^{(i_1,1)}}{Z^{1-\m_1}} \rp \nonumber \\
 & = &  \mE_{{\mathcal U}_{r+1},{{\mathcal U}_{r}},\dots,{{\mathcal U}_{1}}} \lp
\mE_{G} \sum_{p_1=1}^{l}\mE_{G} \lp \u_j^{(i_1,1)}\u_j^{(p_1,1)}\rp
\frac{d}{d\u_j^{(p_1,1)}} \lp
\prod_{v=r+1}^{2} \zeta_{v-1}^{\frac{\m_v}{\m_{v-1}}-1}
    \frac{(C^{(i_1)})^{s-1} A^{(i_1,i_2)} \y_j^{(i_2)}}{Z^{1-\m_1}} \rp \rp \nonumber \\
 & = &  \beta^2 \mE_{G,{\mathcal U}_{r+1},{{\mathcal U}_{r}},\dots,{{\mathcal U}_{1}}} \lp
 \sum_{p_1=1}^{l} (\x^{(p_1)})^T\x^{(i_1)}
\frac{d}{d\u_j^{(p_1,1)}} \lp
\prod_{v=r+1}^{2} \zeta_{v-1}^{\frac{\m_v}{\m_{v-1}}-1}
    \frac{(C^{(i_1)})^{s-1} A^{(i_1,i_2)} \y_j^{(i_2)}}{Z^{1-\m_1}} \rp \rp \nonumber \\
 & = &  \beta^2 \mE_{G,{\mathcal U}_{r+1},{{\mathcal U}_{r}},\dots,{{\mathcal U}_{1}}} \lp
 \sum_{p_1=1}^{l} (\x^{(p_1)})^T\x^{(i_1)}
\prod_{v=r+1}^{2} \zeta_{v-1}^{\frac{\m_v}{\m_{v-1}}-1}
\frac{d}{d\u_j^{(p_1,1)}} \lp
    \frac{(C^{(i_1)})^{s-1} A^{(i_1,i_2)} \y_j^{(i_2)}}{Z^{1-\m_1}} \rp \rp \nonumber \\
 &  & + \sum_{k_1=2}^{r+1} \beta^2 \mE_{G,{\mathcal U}_{r+1},{{\mathcal U}_{r}},\dots,{{\mathcal U}_{1}}} \Bigg(\Bigg.
 \sum_{p_1=1}^{l} (\x^{(p_1)})^T\x^{(i_1)}
\prod_{v=r+1,v\neq k_1}^{2} \zeta_{v-1}^{\frac{\m_v}{\m_{v-1}}-1}
    \frac{(C^{(i_1)})^{s-1} A^{(i_1,i_2)} \y_j^{(i_2)}}{Z^{1-\m_1}}  \nonumber \\
& &
\times \frac{d}{d\u_j^{(p_1,1)}} \lp
 \zeta_{k_1-1}^{\frac{\m_{k_1}}{\m_{k_1-1}}-1}
  \rp \Bigg.\Bigg) \nonumber \\
& = & T_{G,j}^c+\sum_{k_1=2}^{r+1} T_{G,j}^{d_{k_1-1}},
 \end{eqnarray}
 where
 \begin{eqnarray}\label{eq:rthlev2genanal31}
T_{G,j}^c   & = &
\beta^2 \mE_{G,{\mathcal U}_{r+1},{{\mathcal U}_{r}},\dots,{{\mathcal U}_{1}}} \lp
 \sum_{p_1=1}^{l} (\x^{(p_1)})^T\x^{(i_1)}
\prod_{v=r+1}^{2} \zeta_{v-1}^{\frac{\m_v}{\m_{v-1}}-1}
\frac{d}{d\u_j^{(p_1,1)}} \lp
    \frac{(C^{(i_1)})^{s-1} A^{(i_1,i_2)} \y_j^{(i_2)}}{Z^{1-\m_1}} \rp \rp \nonumber \\
& = &
\beta^2 \mE_{G,{\mathcal U}_{r+1},{{\mathcal U}_{r}},\dots,{{\mathcal U}_{1}}} \lp
\prod_{v=r+1}^{2} \zeta_{v-1}^{\frac{\m_v}{\m_{v-1}}-1}
\Theta_{G,1}^{(2)} \rp,
\end{eqnarray}
and
 \begin{eqnarray}\label{eq:rthlev2genanal32}
T_{G,j}^{d_{k_1-1}} & = &  \beta^2 \mE_{G,{\mathcal U}_{r+1},{{\mathcal U}_{r}},\dots,{{\mathcal U}_{1}}} \Bigg(\Bigg.
 \sum_{p_1=1}^{l} (\x^{(p_1)})^T\x^{(i_1)}
\prod_{v=r+1,v\neq k_1}^{2} \zeta_{v-1}^{\frac{\m_v}{\m_{v-1}}-1}
    \frac{(C^{(i_1)})^{s-1} A^{(i_1,i_2)} \y_j^{(i_2)}}{Z^{1-\m_1}}  \nonumber \\
& &
\times \frac{d}{d\u_j^{(p_1,1)}} \lp
 \zeta_{k_1-1}^{\frac{\m_{k_1}}{\m_{k_1-1}}-1}
  \rp \Bigg.\Bigg).
 \end{eqnarray}
Combining (\ref{eq:lev2genGanal5}) and  (\ref{eq:rthlev2genanal31}), we obtain
\begin{eqnarray}\label{eq:rthlev2genanal33}
T_{G,j}^c   & = &
 \beta^2 \mE_{G,{\mathcal U}_{r+1},{{\mathcal U}_{r}},\dots,{{\mathcal U}_{1}}} \lp
\prod_{v=r+1}^{2} \zeta_{v-1}^{\frac{\m_v}{\m_{v-1}}-1}
\Theta_{G,1}^{(2)} \rp  \nonumber \\
 & = &
 \beta^2 \mE_{G,{\mathcal U}_{r+1},{{\mathcal U}_{r}},\dots,{{\mathcal U}_{1}}} \Bigg( \Bigg.
\prod_{v=r+1}^{2} \zeta_{v-1}^{\frac{\m_v}{\m_{v-1}}-1} \nonumber \\
& &  \times
 \Bigg( \Bigg. \frac{\y_j^{(i_2)}}{Z^{1-\m_1}}\lp(C^{(i_1)})^{s-1}\beta_{i_1}A^{(i_1,i_2)}\y_j^{(i_2)}\sqrt{t}+(s-1)(C^{(i_1)})^{s-2}\beta_{i_1}\sum_{p_2=1}^{l}A^{(i_1,p_2)}\y_j^{(p_2)}\sqrt{t}\rp \nonumber \\
& &  -(1-\m_1)
  \lp\sum_{p_1=1}^{l} \frac{(\x^{(i_1)})^T\x^{(p_1)}}{\|\x^{(i_1)}\|_2\|\x^{(p_1)}\|_2}
\frac{(C^{(i_1)})^{s-1} A^{(i_1,i_2)}\y_j^{(i_2)}}{Z^{2-\m_1}}
s  (C^{(p_1)})^{s-1}\sum_{p_2=1}^{l}\beta_{p_1}A^{(p_1,p_2)}\y_j^{(p_2)}\sqrt{t}\rp \Bigg. \Bigg)\Bigg. \Bigg) \nonumber \\
 & = &
 \beta^2 \mE_{G,{\mathcal U}_{r+1},{{\mathcal U}_{r}},\dots,{{\mathcal U}_{1}}} \Bigg( \Bigg.
\prod_{v=r}^{2} \frac{\zeta_{v-1}^{\frac{\m_v}{\m_{v-1}}}}{\zeta_v} \nonumber \\
& &  \times
 \lp
\frac{Z^{\m_1}}{\zeta_1}
 \frac{\y_j^{(i_2)}}{Z}\lp(C^{(i_1)})^{s-1}\beta_{i_1}A^{(i_1,i_2)}\y_j^{(i_2)}\sqrt{t}+(s-1)(C^{(i_1)})^{s-2}\beta_{i_1}\sum_{p_2=1}^{l}A^{(i_1,p_2)}\y_j^{(p_2)}\sqrt{t}\rp \rp \nonumber \\
& &  -(1-\m_1)
\frac{Z^{\m_1}}{\zeta_1}
\nonumber \\
& & \times
  \lp\sum_{p_1=1}^{l} \frac{(\x^{(i_1)})^T\x^{(p_1)}}{\|\x^{(i_1)}\|_2\|\x^{(p_1)}\|_2}
\frac{(C^{(i_1)})^{s-1} A^{(i_1,i_2)}\y_j^{(i_2)}}{Z^{2}}
s  (C^{(p_1)})^{s-1}\sum_{p_2=1}^{l}\beta_{p_1}A^{(p_1,p_2)}\y_j^{(p_2)}\sqrt{t}\rp \Bigg. \Bigg) \nonumber \\
 & = &
 \beta^2 \mE_{G,{\mathcal U}_{r+1}} \Bigg( \Bigg.
\prod_{v=r}^{2} \mE_{{\mathcal U}_{v}} \frac{\zeta_{v-1}^{\frac{\m_v}{\m_{v-1}}}}{\zeta_v} \nonumber \\
& &  \times
 \lp
\mE_{{\mathcal U}_{1}}\frac{Z^{\m_1}}{\zeta_1}
 \frac{\y_j^{(i_2)}}{Z}\lp(C^{(i_1)})^{s-1}\beta_{i_1}A^{(i_1,i_2)}\y_j^{(i_2)}\sqrt{t}+(s-1)(C^{(i_1)})^{s-2}\beta_{i_1}\sum_{p_2=1}^{l}A^{(i_1,p_2)}\y_j^{(p_2)}\sqrt{t}\rp \rp \nonumber \\
& &  -(1-\m_1)
\mE_{{\mathcal U}_{1}}\frac{Z^{\m_1}}{\zeta_1}
\nonumber \\
& & \times
  \lp\sum_{p_1=1}^{l} \frac{(\x^{(i_1)})^T\x^{(p_1)}}{\|\x^{(i_1)}\|_2\|\x^{(p_1)}\|_2}
\frac{(C^{(i_1)})^{s-1} A^{(i_1,i_2)}\y_j^{(i_2)}}{Z^{2}}
s  (C^{(p_1)})^{s-1}\sum_{p_2=1}^{l}\beta_{p_1}A^{(p_1,p_2)}\y_j^{(p_2)}\sqrt{t}\rp \Bigg. \Bigg) \nonumber \\
 & = &
 \beta^2 \mE_{G,{\mathcal U}_{r+1}} \Bigg( \Bigg.
\prod_{v=r}^{2} \Phi_{{\mathcal U}_{v}}   \nonumber \\
& &  \times
 \lp
\Phi_{{\mathcal U}_{1}}\frac{Z^{\m_1}}{\zeta_1}
 \frac{\y_j^{(i_2)}}{Z}\lp(C^{(i_1)})^{s-1}\beta_{i_1}A^{(i_1,i_2)}\y_j^{(i_2)}\sqrt{t}+(s-1)(C^{(i_1)})^{s-2}\beta_{i_1}\sum_{p_2=1}^{l}A^{(i_1,p_2)}\y_j^{(p_2)}\sqrt{t}\rp \rp \nonumber \\
& &  -(1-\m_1)
\Phi_{{\mathcal U}_{1}}
\frac{Z^{\m_1}}{\zeta_1}
\nonumber \\
& & \times
  \lp\sum_{p_1=1}^{l} \frac{(\x^{(i_1)})^T\x^{(p_1)}}{\|\x^{(i_1)}\|_2\|\x^{(p_1)}\|_2}
\frac{(C^{(i_1)})^{s-1} A^{(i_1,i_2)}\y_j^{(i_2)}}{Z^{2}}
s  (C^{(p_1)})^{s-1}\sum_{p_2=1}^{l}\beta_{p_1}A^{(p_1,p_2)}\y_j^{(p_2)}\sqrt{t}\rp \Bigg. \Bigg) \nonumber \\
& = &  \beta^2  \mE_{G,{\mathcal U}_{r+1}} \langle \|\x^{(i_1)}\|_2^2\|\y^{(i_2)}\|_2^2\rangle_{\gamma_{01}^{(r)}} +   (s-1)\beta^2 \mE_{G,{\mathcal U}_{r+1}}\langle \|\x^{(i_1)}\|_2^2(\y^{(p_2)})^T\y^{(i_2)}\rangle_{\gamma_{02}^{(r)}}       \nonumber \\
 &  & -s\beta^2(1-\m_1) \mE_{G,{\mathcal U}_{r+1}}\langle (\x^{(p_1)})^T\x^{(i_1)}(\y^{(p_2)})^T\y^{(i_2)}\rangle_{\gamma_1^{(r)}},
\end{eqnarray}
and
\begin{eqnarray}\label{eq:rthlev2genanal34}
\sum_{i_1=1}^{l}\sum_{i_2=1}^{l}\sum_{j=1}^{m} \beta_{i_1} \frac{T_{G,j}^c}{\sqrt{t}}   & = &  \beta^2 \lp  \mE_{G,{\mathcal U}_{r+1}} \langle \|\x^{(i_1)}\|_2^2\|\y^{(i_2)}\|_2^2\rangle_{\gamma_{01}^{(r)}} +   (s-1)  \mE_{G,{\mathcal U}_{r+1}}\langle \|\x^{(i_1)}\|_2^2(\y^{(p_2)})^T\y^{(i_2)}\rangle_{\gamma_{02}^{(r)}} \rp      \nonumber \\
 &  & -s\beta^2(1-\m_1) \mE_{G,{\mathcal U}_{r+1}}\langle (\x^{(p_1)})^T\x^{(i_1)}(\y^{(p_2)})^T\y^{(i_2)}\rangle_{\gamma_1^{(r)}}.
\end{eqnarray}

To determine $T_{G,j}^{d_{k_1-1}}$ we write
 \begin{eqnarray}\label{eq:rthlev2genanal35}
T_{G,j}^{d_{k_1-1}} & = &  \beta^2 \mE_{G,{\mathcal U}_{r+1},{{\mathcal U}_{r}},\dots,{{\mathcal U}_{1}}} \Bigg(\Bigg.
 \sum_{p_1=1}^{l} (\x^{(p_1)})^T\x^{(i_1)}
\prod_{v=r+1,v\neq k_1}^{2} \zeta_{v-1}^{\frac{\m_v}{\m_{v-1}}-1}
    \frac{(C^{(i_1)})^{s-1} A^{(i_1,i_2)} \y_j^{(i_2)}}{Z^{1-\m_1}}  \nonumber \\
& &
\times \frac{d}{d\u_j^{(p_1,1)}} \lp
 \zeta_{k_1-1}^{\frac{\m_{k_1}}{\m_{k_1-1}}-1}
  \rp \Bigg.\Bigg) \nonumber \\
& = &  \beta^2 \mE_{G,{\mathcal U}_{r+1},{{\mathcal U}_{r}},\dots,{{\mathcal U}_{1}}} \Bigg(\Bigg.
 \sum_{p_1=1}^{l} (\x^{(p_1)})^T\x^{(i_1)}
\prod_{v=r+1,v\neq k_1}^{2} \zeta_{v-1}^{\frac{\m_v}{\m_{v-1}}-1}
    \frac{(C^{(i_1)})^{s-1} A^{(i_1,i_2)} \y_j^{(i_2)}}{Z^{1-\m_1}}  \nonumber \\
& &
\times   s \lp \m_{k_1} - \m_{k_1-1}\rp \zeta_{k_1-1}^{\frac{\m_{k_1}}{\m_{k_1-1}}-2}
\prod_{k=k_1-1}^{2}\mE_{{\mathcal U}_{k}} \zeta_{k-1}^{\frac{\m_k}{\m_{k-1}}-1}
 \mE_{{\mathcal U}_1} \frac{1}{Z^{1-\m_1}}  (C^{(p_1)})^{s-1} \nonumber \\
& & \times \sum_{p_2=1}^{l}\beta_{p_1}A^{(p_1,p_2)}\frac{dD^{(p_1,p_2)}}{d\u_j^{(p_1,1)}}
 \Bigg.\Bigg).
 \end{eqnarray}
 Since
\begin{eqnarray}\label{eq:rthlev2genanal36}
 \frac{dD^{(p_1,p_2)}}{d\u_j^{(p_1,1)}}=
 \frac{dB^{(p_1,p_2)}}{d\u_j^{(p_1,1)}}=\sqrt{t} \y_j^{(p_2)},
\end{eqnarray}
we, from (\ref{eq:rthlev2genanal21}), also have
\begin{eqnarray}\label{eq:rthlev2genanal37}
T_{G,j}^{d_{k_1-1}}
& = &  \beta^2 \mE_{G,{\mathcal U}_{r+1},{{\mathcal U}_{r}},\dots,{{\mathcal U}_{1}}} \Bigg(\Bigg.
 \sum_{p_1=1}^{l} (\x^{(p_1)})^T\x^{(i_1)}
\prod_{v=r+1,v\neq k_1}^{2} \zeta_{v-1}^{\frac{\m_v}{\m_{v-1}}-1}
    \frac{(C^{(i_1)})^{s-1} A^{(i_1,i_2)} \y_j^{(i_2)}}{Z^{1-\m_1}}  \nonumber \\
& &
\times   s \lp \m_{k_1} - \m_{k_1-1}\rp \zeta_{k_1-1}^{\frac{\m_{k_1}}{\m_{k_1-1}}-2}
\prod_{k=k_1}^{2}\mE_{{\mathcal U}_{k}} \zeta_{k-1}^{\frac{\m_k}{\m_{k-1}}-1}
 \mE_{{\mathcal U}_1} \frac{1}{Z^{1-\m_1}}  (C^{(p_1)})^{s-1} \nonumber \\
& & \times \sum_{p_2=1}^{l}\beta_{p_1}A^{(p_1,p_2)}\sqrt{t} \y_j^{(p_2)}
 \Bigg.\Bigg) \nonumber \\
 & = &  \beta^2 \mE_{G,{\mathcal U}_{r+1}} \Bigg(\Bigg.
\prod_{k=r}^{k_1+1}\mE_{{\mathcal U}_{k}} \frac{\zeta_{k-1}^{\frac{\m_k}{\m_{k-1}}}}{\zeta_k}
 \prod_{k=k_1}^{2}\mE_{{\mathcal U}_{k}} \frac{\zeta_{k-1}^{\frac{\m_k}{\m_{k-1}}}}{\zeta_k}
 \mE_{{\mathcal U}_1} \frac{Z^{\m_1}}{\zeta_1}
    \frac{(C^{(i_1)})^{s} }{Z}     \frac{ A^{(i_1,i_2)} }{C^{(i_1)}}  \nonumber \\
& &
\times   s \lp \m_{k_1} - \m_{k_1-1}\rp
\prod_{k=k_1}^{2}\mE_{{\mathcal U}_{k}} \frac{\zeta_{k-1}^{\frac{\m_k}{\m_{k-1}}}}{\zeta_k}
 \mE_{{\mathcal U}_1} \frac{Z^{\m_1}}{\zeta_1}   \sum_{p_1=1}^{l}\frac{(C^{(p_1)})^{s} }{Z}  \nonumber \\
& & \times \sum_{p_2=1}^{l} \frac{A^{(p_1,p_2)}}{C^{(p_1)}}\sqrt{t}\beta_{p_1}  (\x^{(p_1)})^T\x^{(i_1)}  \y_j^{(p_2)}\y_j^{(i_2)}
 \Bigg.\Bigg) \nonumber \\
  & = &  \beta^2 \mE_{G,{\mathcal U}_{r+1}} \Bigg(\Bigg.
\prod_{k=r}^{k_1+1}\Phi_{{\mathcal U}_{k}}
 \prod_{k=k_1}^{1}\Phi_{{\mathcal U}_{k}}
    \frac{(C^{(i_1)})^{s} }{Z}     \frac{ A^{(i_1,i_2)} }{C^{(i_1)}}  \nonumber \\
& &
\times   s \lp \m_{k_1} - \m_{k_1-1}\rp
\prod_{k=k_1}^{1}\Phi_{{\mathcal U}_{k}}    \sum_{p_1=1}^{l}\frac{(C^{(p_1)})^{s} }{Z}  \sum_{p_2=1}^{l} \frac{A^{(p_1,p_2)}}{C^{(p_1)}}\sqrt{t}\beta_{p_1}  (\x^{(p_1)})^T\x^{(i_1)}  \y_j^{(p_2)}\y_j^{(i_2)}
 \Bigg.\Bigg). \nonumber \\
\end{eqnarray}
Combining (\ref{eq:rthlev2genanal30}), (\ref{eq:rthlev2genanal34}), and (\ref{eq:rthlev2genanal37}) we obtain
\begin{eqnarray}\label{eq:rthlev2genanal38}
\sum_{i_1=1}^{l}\sum_{i_2=1}^{l}\sum_{j=1}^{m} \beta_{i_1 }\frac{T_{G,j}}{\sqrt{t}}
& = & \sum_{i_1=1}^{l}\sum_{i_2=1}^{l}\sum_{j=1}^{m} \beta_{i_1 }\frac{(T_{G,j}^c+\sum_{k_1=2}^{r+1} T_{G,j}^{d_{k_1-1}})}{\sqrt{t}} \nonumber \\
& = & \beta^2 \lp \mE_{G,{\mathcal U}_{r+1}} \langle \|\x^{(i_1)}\|_2^2\|\y^{(i_2)}\|_2^2\rangle_{\gamma_{01}^{(r)}} +   (s-1)\mE_{G,{\mathcal U}_{r+1}}\langle \|\x^{(i_1)}\|_2^2(\y^{(p_2)})^T\y^{(i_2)}\rangle_{\gamma_{02}^{(r)}}   \rp    \nonumber \\
 &  & -s\beta^2(1-\m_1) \mE_{G,{\mathcal U}_{r+1}}\langle (\x^{(p_1)})^T\x^{(i_1)}(\y^{(p_2)})^T\y^{(i_2)}\rangle_{\gamma_1^{(r)}}
  \nonumber \\
  &  & -s\beta^2 \sum_{k_1=2}^{r+1}  (\m_{k_1-1}-\m_{k_1}) \mE_{G,{\mathcal U}_{r+1}}\langle (\x^{(p_1)})^T\x^{(i_1)}(\y^{(p_2)})^T\y^{(i_2)}\rangle_{\gamma_{k_1}^{(r)}}.
\end{eqnarray}
With $\m_0=1$, we can also write
\begin{eqnarray}\label{eq:rthlev2genanal39}
\sum_{i_1=1}^{l}\sum_{i_2=1}^{l}\sum_{j=1}^{m} \beta_{i_1 }\frac{T_{G,j}}{\sqrt{t}}
 & = & \beta^2  \lp \mE_{G,{\mathcal U}_{r+1}} \langle \|\x^{(i_1)}\|_2^2\|\y^{(i_2)}\|_2^2\rangle_{\gamma_{01}^{(r)}} +   (s-1)\mE_{G,{\mathcal U}_{r+1}}\langle \|\x^{(i_1)}\|_2^2(\y^{(p_2)})^T\y^{(i_2)}\rangle_{\gamma_{02}^{(r)}}   \rp    \nonumber \\
   &  & -s\beta^2 \sum_{k_1=1}^{r+1}  (\m_{k_1-1}-\m_{k_1}) \mE_{G,{\mathcal U}_{r+1}}\langle (\x^{(p_1)})^T\x^{(i_1)}(\y^{(p_2)})^T\y^{(i_2)}\rangle_{\gamma_{k_1}^{(r)}}.
\end{eqnarray}

%%%%%%%%%%%%%%%%%%%%%%%%%%%%%%%%%%%%%%%%%%%%%%%%%%%%%%%%%%%%%%%%%%%%%%%%%%%%%%%%%%
%%%%%%%%%%%%%%%%%%%%%%%%%%%%%%%%%%%%%%%%%%%%%%%%%%%%%%%%%%%%%%%%%%%%%%%%%%%%%%%%%%
%%%%%%%%%%%%%%%%%%%%%%%%%%%%%%%%%%%%%%%%%%%%%%%%%%%%%%%%%%%%%%%%%%%%%%%%%%%%%%%%%%
%%%%%%%%%%%%%%%%%%%%%%%%%%%%%%%%%%%%%%%%%%%%%%%%%%%%%%%%%%%%%%%%%%%%%%%%%%%%%%%%%%
\subsection{Connecting everything together}
\label{sec:connect}
%%%%%%%%%%%%%%%%%%%%%%%%%%%%%%%%%%%%%%%%%%%%%%%%%%%%%%%%%%%%%%%%%%%%%%%%%%%%%%%%%%
%%%%%%%%%%%%%%%%%%%%%%%%%%%%%%%%%%%%%%%%%%%%%%%%%%%%%%%%%%%%%%%%%%%%%%%%%%%%%%%%%%
%%%%%%%%%%%%%%%%%%%%%%%%%%%%%%%%%%%%%%%%%%%%%%%%%%%%%%%%%%%%%%%%%%%%%%%%%%%%%%%%%%
%%%%%%%%%%%%%%%%%%%%%%%%%%%%%%%%%%%%%%%%%%%%%%%%%%%%%%%%%%%%%%%%%%%%%%%%%%%%%%%%%%

We now have all the ingredients to formulate the following theorem.
\begin{theorem}
\label{thm:thm3}
Let $r\in\mN$ and $k\in\{1,2,\dots,r+1\}$. For vectors $\m=[\m_0,\m_1,\m_2,...,\m_r,\m_{r+1}]$ with $\m_0=1$ and $\m_{r+1}=0$,
$\p=[\p_0,\p_1,...,\p_r,\p_{r+1}]$ with $1\geq\p_0\geq \p_1\geq \p_2\geq \dots \geq \p_r\geq \p_{r+1} =0$ and $\q=[\q_0,\q_1,\q_2,\dots,\q_r,\q_{r+1}]$ with $1\geq\q_0\geq \q_1\geq \q_2\geq \dots \geq \q_r\geq \q_{r+1} = 0$, let the variances of the zero-mean independent normal components of $G\in\mR^{m\times n}$, $u^{(4,k)}\in\mR$, $\u^{(2,k)}\in\mR^m$, and $\h^{(k)}\in\mR^n$ be $1$, $\p_{k-1}\q_{k-1}-\p_k\q_k$, $\p_{k-1}-\p_{k}$, $\q_{k-1}-\q_{k}$, respectively. Also, let ${\mathcal U}_k\triangleq [u^{(4,k)},\u^{(2,k)},\h^{(2k)}]$.  Assuming that set ${\mathcal X}=\{\x^{(1)},\x^{(2)},\dots,\x^{(l)}\}$, where $\x^{(i)}\in \mR^{n},1\leq i\leq l$, set ${\mathcal Y}=\{\y^{(1)},\y^{(2)},\dots,\y^{(l)}\}$, where $\y^{(i)}\in \mR^{m},1\leq i\leq l$, and scalars $\beta\geq 0$ and $s\in\mR$ are given, and consider the following function
\begin{equation}\label{eq:thm3eq1}
\psi(\calX,\calY,\p,\q,\m,\beta,s,t)  =  \mE_{G,{\mathcal U}_{r+1}} \frac{1}{\beta|s|\sqrt{n}\m_r} \log
\lp \mE_{{\mathcal U}_{r}} \lp \dots \lp \mE_{{\mathcal U}_2}\lp\lp\mE_{{\mathcal U}_1} \lp Z^{\m_1}\rp\rp^{\frac{\m_2}{\m_1}}\rp\rp^{\frac{\m_3}{\m_2}} \dots \rp^{\frac{\m_{r}}{\m_{r-1}}}\rp,
\end{equation}
where
\begin{eqnarray}\label{eq:thm3eq2}
Z & \triangleq & \sum_{i_1=1}^{l}\lp\sum_{i_2=1}^{l}e^{\beta D_0^{(i_1,i_2)}} \rp^{s} \nonumber \\
 D_0^{(i_1,i_2)} & \triangleq & \sqrt{t}(\y^{(i_2)})^T
 G\x^{(i_1)}+\sqrt{1-t}\|\x^{(i_1)}\|_2 (\y^{(i_2)})^T\lp\sum_{k=1}^{r+1}\u^{(2,k)}\rp\nonumber \\
 & & +\sqrt{t}\|\x^{(i_1)}\|_2\|\y^{(i_2)}\|_2\lp\sum_{k=1}^{r+1}u^{(4,k)}\rp +\sqrt{1-t}\|\y^{(i_2)}\|_2\lp\sum_{k=1}^{r+1}\h^{(k)}\rp^T\x^{(i_1)}
 \end{eqnarray}
Let $\zeta_0=Z$ and let $\zeta_k,k\geq 1$, be as defined in (\ref{eq:rthlev2genanal7a}) and (\ref{eq:rthlev2genanal7b}). Moreover, consider the operators
\begin{eqnarray}\label{eq:thm3eq3}
 \Phi_{{\mathcal U}_k} & \triangleq &  \mE_{{\mathcal U}_{k}} \frac{\zeta_{k-1}^{\frac{\m_k}{\m_{k-1}}}}{\zeta_k},
 \end{eqnarray}
and set
\begin{eqnarray}\label{eq:thm3eq4}
  \gamma_0(i_1,i_2) & = &
\frac{(C^{(i_1)})^{s}}{Z}  \frac{A^{(i_1,i_2)}}{C^{(i_1)}} \nonumber \\
\gamma_{01}^{(r)}  & = & \prod_{k=r}^{1}\Phi_{{\mathcal U}_k} (\gamma_0(i_1,i_2)) \nonumber \\
\gamma_{02}^{(r)}  & = & \prod_{k=r}^{1}\Phi_{{\mathcal U}_k} (\gamma_0(i_1,i_2)\times \gamma_0(i_1,p_2)) \nonumber \\
\gamma_{k_1+1}^{(r)}  & = & \prod_{k=r}^{k_1+1}\Phi_{{\mathcal U}_k} \lp \prod_{k=k_1}^{1}\Phi_{{\mathcal U}_k}\gamma_0(i_1,i_2)\times \prod_{k=k_1}^{1} \Phi_{{\mathcal U}_k}\gamma_0(p_1,p_2) \rp.
 \end{eqnarray}
Let also
\begin{eqnarray}\label{eq:thm3eq5}
 \phi_{k_1}^{(r)} & = &
-s(\m_{k_1-1}-\m_{k_1}) \nonumber \\
&  & \times
\mE_{G,{\mathcal U}_{r+1}} \langle (\p_{k_1-1}\|\x^{(i_1)}\|_2\|\x^{(p_1)}\|_2 -(\x^{(p_1)})^T\x^{(i_1)})(\q_{k_1-1}\|\y^{(i_2)}\|_2\|\y^{(p_2)}\|_2 -(\y^{(p_2)})^T\y^{(i_2)})\rangle_{\gamma_{k_1}^{(r)}} \nonumber \\
 \phi_{01}^{(r)} & = & (1-\p_0)(1-\q_0)\mE_{G,{\mathcal U}_{r+1}}\langle \|\x^{(i_1)}\|_2^2\|\y^{(i_2)}\|_2^2\rangle_{\gamma_{01}^{(r)}} \nonumber\\
\phi_{02}^{(r)} & = & (s-1)(1-\p_0)\mE_{G,{\mathcal U}_{r+1}}\left\langle \|\x^{(i_1)}\|_2^2 \lp\q_0\|\y^{(i_2)}\|_2\|\y^{(p_2)}\|_2-(\y^{(p_2)})^T\y^{(i_2)}\rp\right\rangle_{\gamma_{02}^{(r)}}. \end{eqnarray}

\noindent Then
\begin{eqnarray}\label{eq:thm3eq6}
\frac{d\psi(\calX,\calY,\p,\q,\m,\beta,s,t)}{dt}  & = &       \frac{\mbox{sign}(s)\beta}{2\sqrt{n}} \lp  \lp\sum_{k_1=1}^{r+1} \phi_{k_1}^{(r)}\rp +\phi_{01}^{(r)}+\phi_{02}^{(r)}\rp.
 \end{eqnarray}
It particular, choosing $\p_0=\q_0=1$, one also has
\begin{eqnarray}\label{eq:rthlev2genanal43}
\frac{d\psi(\calX,\calY,\p,\q,\m,\beta,s,t)}{dt}  & = &       \frac{\mbox{sign}(s)\beta}{2\sqrt{n}} \sum_{k_1=1}^{r+1} \phi_{k_1}^{(r)} .
 \end{eqnarray}
 \end{theorem}
\begin{proof}
  Follows through a combination of (\ref{eq:rthlev2genanal10e})-(\ref{eq:rthlev2genanal10g}), (\ref{eq:rthlev2genanal27})-(\ref{eq:rthlev2genanal29}), (\ref{eq:rthlev2genanal39}), after summing and canceling scaled terms and noting that, analogously to  (\ref{eq:liftgenAanal19i}) and (\ref{eq:lev2liftgenAanal19i}) for the first and second level, we have for the $r$-th level
\begin{eqnarray}\label{eq:rthlev2genanal40}
\sum_{i_1=1}^{l}\sum_{i_2=1}^{l}\sum_{j=1}^{m} \beta_{i_1}\frac{T_{1,1,j}}{\sqrt{1-t
}}
& = & \beta^2(\p_0-\p_1) \Bigg( \Bigg. \mE_{G,{\mathcal U}_{r+1}}\langle \|\x^{(i_1)}\|_2^2\|\y^{(i_2)}\|_2^2\rangle_{\gamma_{01}^{(r)}} \nonumber \\
& & +  (s-1)\mE_{G,{\mathcal U}_{r+1}}\langle \|\x^{(i_1)}\|_2^2(\y^{(p_2)})^T\y^{(i_2)}\rangle_{\gamma_{02}^{(r)}} \Bigg.\Bigg)  \nonumber \\
& & - (\p_0-\p_1)s\beta^2(1-\m_1)\mE_{G,{\mathcal U}_{r+1}}\langle \|\x^{(i_1)}\|_2\|\x^{(p_1)}\|_2(\y^{(p_2)})^T\y^{(i_2)} \rangle_{\gamma_1^{(r)}},
\end{eqnarray}
and analogously to (\ref{eq:liftgenBanal20b}) and (\ref{eq:lev2liftgenBanal20b}) for the first and second level, we have for the $r$-th level
\begin{eqnarray}\label{eq:rthlev2genanal41}
\sum_{i_1=1}^{l}\sum_{i_2=1}^{l} \beta_{i_1}\|\y^{(i_2)}\|_2 \frac{T_{1,2}}{\sqrt{1-t}}
& = & \beta^2(\q_0-\q_1)\Bigg( \Bigg.  \mE_{G,{\mathcal U}_{r+1}}\langle \|\x^{(i_1)}\|_2^2\|\y^{(i_2)}\|_2^2\rangle_{\gamma_{01}^{(r)}} \nonumber \\
& &
+   (s-1)\mE_{G,{\mathcal U}_{r+1}}\langle \|\x^{(i_1)}\|_2^2 \|\y^{(i_2)}\|_2\|\y^{(p_2)}\|_2\rangle_{\gamma_{02}^{(r)}}\Bigg.\Bigg) \nonumber \\
& & - (\q_0-\q_1)s\beta^2(1-\m_1)\mE_{G,{\mathcal U}_{r+1}}\langle (\x^{(p_1)})^T\x^{(i_1)}\|\y^{(i_2)}\|_2\|\y^{(p_2)}\|_2 \rangle_{\gamma_1^{(r)}},\nonumber \\
\end{eqnarray}
and analogously to  (\ref{eq:liftgenCanal21b}) and (\ref{eq:lev2liftgenCanal21b})  for the first and second level, we have for the $r$-th level
  \begin{eqnarray}\label{eq:rthlev2genanal42}
\sum_{i_1=1}^{l}\sum_{i_2=1}^{l} \beta_{i_1}\|\y^{(i_2)}\|_2 \frac{T_{1,3}}{\sqrt{t}}
& = & \beta^2(\p_0\q_0-\p_1\q_1)\Bigg( \Bigg. \mE_{G,{\mathcal U}_{r+1}}\langle \|\x^{(i_1)}\|_2^2\|\y^{(i_2)}\|_2^2\rangle_{\gamma_{01}^{(r)}} \nonumber \\
& & +   (s-1)\mE_{G,{\mathcal U}_{r+1}}\langle \|\x^{(i_1)}\|_2^2 \|\y^{(i_2)}\|_2\|\y^{(p_2)}\|_2\rangle_{\gamma_{02}^{(r)}}\Bigg.\Bigg)    \nonumber \\
& & - (\p_0\q_0-\p_1\q_1)s\beta^2(1-\m_1)\mE_{G,{\mathcal U}_{r+1}}\langle \|\x^{(i_1)}\|_2\|\x^{(p_`)}\|_2\|\y^{(i_2)}\|_2\|\y^{(p_2)}\|_2 \rangle_{\gamma_1^{(r)}}.\nonumber \\
\end{eqnarray}
The summing  and canceling out works out in an identical way for all three sequences, $\lp T_{k,1,j}\rp_{k=1:r+1}$, $\lp T_{k,2}\rp_{k=1:r+1}$, and $\lp T_{k,3}\rp_{k=1:r+1}$. To avoid unnecessary repetition of the same mechanism, we show how it works only for the first one. Following (\ref{eq:rthlev2genanal27}), we first write
 \begin{eqnarray}\label{eq:rthlev2genanal43}
 \sum_{i_1=1}^{l}\sum_{i_2=1}^{l} \sum_{j=1}^{m}
\beta_{i_1} \frac{T_{2,1,j}}{\sqrt{1-t}}
 & = &  \frac{\p_{1}-\p_{2}}{\p_{0}-\p_{1}}\sum_{i_1=1}^{l}\sum_{i_2=1}^{l} \sum_{j=1}^{m}
\beta_{i_1} \frac{T_{1,1,j}}{\sqrt{1-t}} \nonumber \\
 &  & - \Bigg( \Bigg. s\beta^2(\p_{1}-\p_{2}) \lp \m_{1} -  \m_{2} \rp \mE_{G,{\mathcal U}_{r+1}} \left \langle \|\x^{(i_1)}\|_2\|\x^{(p_1)}\|_2 (\y^{(p_2)})^T \y^{(i_2)} \right \rangle_{\gamma_{2}^{(r)} } \Bigg. \Bigg) \nonumber \\
& = & \beta^2\frac{\p_{1}-\p_{2}}{\p_{0}-\p_{1}} \Bigg( \Bigg.(\p_0-\p_1) \Bigg( \Bigg. \mE_{G,{\mathcal U}_{r+1}}\langle \|\x^{(i_1)}\|_2^2\|\y^{(i_2)}\|_2^2\rangle_{\gamma_{01}^{(r)}} \nonumber \\
& & +  (s-1)\mE_{G,{\mathcal U}_{r+1}}\langle \|\x^{(i_1)}\|_2^2(\y^{(p_2)})^T\y^{(i_2)}\rangle_{\gamma_{02}^{(r)}} \Bigg.\Bigg)  \nonumber \\
& & - (\p_0-\p_1)s\beta^2(1-\m_1)\mE_{G,{\mathcal U}_{r+1}}\langle \|\x^{(i_1)}\|_2\|\x^{(p_1)}\|_2(\y^{(p_2)})^T\y^{(i_2)} \rangle_{\gamma_1^{(r)}} \Bigg.\Bigg) \nonumber \\
 &  & - \Bigg( \Bigg. s\beta^2(\p_{1}-\p_{2}) \lp \m_{1} -  \m_{2} \rp
  \mE_{G,{\mathcal U}_{r+1}} \left \langle \|\x^{(i_1)}\|_2\|\x^{(p_1)}\|_2 (\y^{(p_2)})^T \y^{(i_2)} \right \rangle_{\gamma_{2}^{(r)} } \Bigg. \Bigg) \nonumber \\
& = & \beta^2 (\p_1-\p_2) \Bigg( \Bigg. \mE_{G,{\mathcal U}_{r+1}}\langle \|\x^{(i_1)}\|_2^2\|\y^{(i_2)}\|_2^2\rangle_{\gamma_{01}^{(r)}} \nonumber \\
 & & +  (s-1)\mE_{G,{\mathcal U}_{r+1}}\langle \|\x^{(i_1)}\|_2^2(\y^{(p_2)})^T\y^{(i_2)}\rangle_{\gamma_{02}^{(r)}}   \nonumber \\
& & - \Bigg( \Bigg. s(1-\m_1)\mE_{G,{\mathcal U}_{r+1}}\langle \|\x^{(i_1)}\|_2\|\x^{(p_1)}\|_2(\y^{(p_2)})^T\y^{(i_2)} \rangle_{\gamma_1^{(r)}} \Bigg.\Bigg) \nonumber \\
 &  & - \Bigg( \Bigg. s \lp \m_1 -  \m_{2} \rp
  \mE_{G,{\mathcal U}_{r+1}} \left \langle \|\x^{(i_1)}\|_2\|\x^{(p_1)}\|_2 (\y^{(p_2)})^T \y^{(i_2)} \right \rangle_{\gamma_{2}^{(r)} } \Bigg. \Bigg).
\end{eqnarray}
In general, assuming that
 \begin{eqnarray}\label{eq:rthlev2genanal44}
 \sum_{i_1=1}^{l}\sum_{i_2=1}^{l} \sum_{j=1}^{m}
\beta_{i_1} \frac{T_{k_1,1,j}}{\sqrt{1-t}}
 & = & \beta^2 (\p_{k_1-1}-\p_{k_1}) \Bigg( \Bigg. \mE_{G,{\mathcal U}_{r+1}}\langle \|\x^{(i_1)}\|_2^2\|\y^{(i_2)}\|_2^2\rangle_{\gamma_{01}^{(r)}} \nonumber \\
 & & +  (s-1)\mE_{G,{\mathcal U}_{r+1}}\langle \|\x^{(i_1)}\|_2^2(\y^{(p_2)})^T\y^{(i_2)}\rangle_{\gamma_{02}^{(r)}}   \nonumber \\
 &  & - s \sum_{v=1}^{k_1} \lp \m_{v-1} -  \m_{v} \rp
  \mE_{G,{\mathcal U}_{r+1}} \left \langle \|\x^{(i_1)}\|_2\|\x^{(p_1)}\|_2 (\y^{(p_2)})^T \y^{(i_2)} \right \rangle_{\gamma_{k_1}^{(r)} } \Bigg. \Bigg),\nonumber \\
\end{eqnarray}
we, from (\ref{eq:rthlev2genanal27}), have
 \begin{eqnarray}\label{eq:rthlev2genanal45}
 \sum_{i_1=1}^{l}\sum_{i_2=1}^{l} \sum_{j=1}^{m}
\beta_{i_1} \frac{T_{k_1+1,1,j}}{\sqrt{1-t}}
 & = &  \frac{\p_{k_1}-\p_{k_1+1}}{\p_{k_1-1}-\p_{k_1}}\sum_{i_1=1}^{l}\sum_{i_2=1}^{l} \sum_{j=1}^{m}
\beta_{i_1} \frac{T_{k_1,1,j}}{\sqrt{1-t}} \nonumber \\
 &  & - \Bigg( \Bigg. s\beta^2(\p_{k_1}-\p_{k_1+1}) \lp \m_{k_1} -  \m_{k_1+1} \rp
  \nonumber \\
  & & \times
  \mE_{G,{\mathcal U}_{r+1}} \left \langle \|\x^{(i_1)}\|_2\|\x^{(p_1)}\|_2 (\y^{(p_2)})^T \y^{(i_2)} \right \rangle_{\gamma_{k_1+1}^{(r)} } \Bigg. \Bigg) \nonumber \\
 & = &  \frac{\p_{k_1}-\p_{k_1+1}}{\p_{k_1-1}-\p_{k_1}}\beta^2 (\p_{k_1-1}-\p_{k_1}) \Bigg( \Bigg. \mE_{G,{\mathcal U}_{r+1}}\langle \|\x^{(i_1)}\|_2^2\|\y^{(i_2)}\|_2^2\rangle_{\gamma_{01}^{(r)}} \nonumber \\
 & & +  (s-1)\mE_{G,{\mathcal U}_{r+1}}\langle \|\x^{(i_1)}\|_2^2(\y^{(p_2)})^T\y^{(i_2)}\rangle_{\gamma_{02}^{(r)}}   \nonumber \\
 &  & - s \sum_{v=1}^{k_1} \lp \m_{v-1} -  \m_{v} \rp
  \nonumber \\
  & & \times
  \mE_{G,{\mathcal U}_{r+1}} \left \langle \|\x^{(i_1)}\|_2\|\x^{(p_1)}\|_2 (\y^{(p_2)})^T \y^{(i_2)} \right \rangle_{\gamma_{v}^{(r)} } \Bigg. \Bigg) \nonumber \\
   &  & - \Bigg( \Bigg. s\beta^2(\p_{k_1}-\p_{k_1+1}) \lp \m_{k_1} -  \m_{k_1+1} \rp
  \nonumber \\
  & & \times
  \mE_{G,{\mathcal U}_{r+1}} \left \langle \|\x^{(i_1)}\|_2\|\x^{(p_1)}\|_2 (\y^{(p_2)})^T \y^{(i_2)} \right \rangle_{\gamma_{k_1+1}^{(r)} } \Bigg. \Bigg) \nonumber \\
 & = &  \beta^2 (\p_{k_1}-\p_{k_1+1}) \Bigg( \Bigg. \mE_{G,{\mathcal U}_{r+1}}\langle \|\x^{(i_1)}\|_2^2\|\y^{(i_2)}\|_2^2\rangle_{\gamma_{01}^{(r)}} \nonumber \\
 & & +  (s-1)\mE_{G,{\mathcal U}_{r+1}}\langle \|\x^{(i_1)}\|_2^2(\y^{(p_2)})^T\y^{(i_2)}\rangle_{\gamma_{02}^{(r)}}   \nonumber \\
 &  & - s \sum_{v=1}^{k_1+1} \lp \m_{v-1} -  \m_{v} \rp
  \mE_{G,{\mathcal U}_{r+1}} \left \langle \|\x^{(i_1)}\|_2\|\x^{(p_1)}\|_2 (\y^{(p_2)})^T \y^{(i_2)} \right \rangle_{\gamma_{v}^{(r)} } \Bigg. \Bigg). \nonumber \\
    \end{eqnarray}
Since (\ref{eq:rthlev2genanal44}) holds for $k_1=1$ (and, as shown in (\ref{eq:rthlev2genanal43}),  even for $k_1=2$), one then, based on (\ref{eq:rthlev2genanal45}), has that it holds for any $k_1\in\{1,2,\dots,r+1\}$. Finally one has for the sum over $k_1$
 \begin{eqnarray}\label{eq:rthlev2genanal46}
 \sum_{k_1=1}^{r+1}\sum_{i_1=1}^{l}\sum_{i_2=1}^{l} \sum_{j=1}^{m}
\beta_{i_1} \frac{T_{k_1,1,j}}{\sqrt{1-t}}
 & = & \beta^2  \sum_{k_1=1}^{r+1}  (\p_{k_1-1}-\p_{k_1}) \Bigg( \Bigg. \mE_{G,{\mathcal U}_{r+1}}\langle \|\x^{(i_1)}\|_2^2\|\y^{(i_2)}\|_2^2\rangle_{\gamma_{01}^{(r)}} \nonumber \\
 & & +  (s-1)\mE_{G,{\mathcal U}_{r+1}}\langle \|\x^{(i_1)}\|_2^2(\y^{(p_2)})^T\y^{(i_2)}\rangle_{\gamma_{02}^{(r)}}   \nonumber \\
 &  & - s \sum_{v=1}^{k_1} \lp \m_{v-1} -  \m_{v} \rp
  \mE_{G,{\mathcal U}_{r+1}} \left \langle \|\x^{(i_1)}\|_2\|\x^{(p_1)}\|_2 (\y^{(p_2)})^T \y^{(i_2)} \right \rangle_{\gamma_{k_1}^{(r)} } \Bigg. \Bigg)\nonumber \\
 & = & \beta^2 (\p_{0}-\p_{r+1}) \Bigg( \Bigg. \mE_{G,{\mathcal U}_{r+1}}\langle \|\x^{(i_1)}\|_2^2\|\y^{(i_2)}\|_2^2\rangle_{\gamma_{01}^{(r)}} \nonumber \\
 & & +  (s-1)\mE_{G,{\mathcal U}_{r+1}}\langle \|\x^{(i_1)}\|_2^2(\y^{(p_2)})^T\y^{(i_2)}\rangle_{\gamma_{02}^{(r)}} \Bigg.\Bigg)  \nonumber \\
 &  & - s  \sum_{k_1=1}^{r+1}  (\p_{k_1-1}-\p_{k_1})\nonumber \\
 & &  \times  \sum_{v=1}^{k_1} \lp \m_{v-1} -  \m_{v} \rp
  \mE_{G,{\mathcal U}_{r+1}} \left \langle \|\x^{(i_1)}\|_2\|\x^{(p_1)}\|_2 (\y^{(p_2)})^T \y^{(i_2)} \right \rangle_{\gamma_{v}^{(r)} } \nonumber \\
  & = & \beta^2 (\p_{0}-\p_{r+1}) \Bigg( \Bigg. \mE_{G,{\mathcal U}_{r+1}}\langle \|\x^{(i_1)}\|_2^2\|\y^{(i_2)}\|_2^2\rangle_{\gamma_{01}^{(r)}} \nonumber \\
 & & +  (s-1)\mE_{G,{\mathcal U}_{r+1}}\langle \|\x^{(i_1)}\|_2^2(\y^{(p_2)})^T\y^{(i_2)}\rangle_{\gamma_{02}^{(r)}}  \Bigg. \Bigg) \nonumber \\
 &  & - s  \sum_{v=1}^{r+1} \lp \m_{v-1} -  \m_{v} \rp \mE_{G,{\mathcal U}_{r+1}} \left \langle \|\x^{(i_1)}\|_2\|\x^{(p_1)}\|_2 (\y^{(p_2)})^T \y^{(i_2)} \right \rangle_{\gamma_{v}^{(r)} }
\nonumber \\
 & &  \times  \sum_{k_1=v}^{r+1}  (\p_{k_1-1}-\p_{k_1}) \nonumber \\
  & = & \beta^2 (\p_{0}-\p_{r+1}) \Bigg( \Bigg. \mE_{G,{\mathcal U}_{r+1}}\langle \|\x^{(i_1)}\|_2^2\|\y^{(i_2)}\|_2^2\rangle_{\gamma_{01}^{(r)}} \nonumber \\
 & & +  (s-1)\mE_{G,{\mathcal U}_{r+1}}\langle \|\x^{(i_1)}\|_2^2(\y^{(p_2)})^T\y^{(i_2)}\rangle_{\gamma_{02}^{(r)}}  \Bigg. \Bigg) \nonumber \\
 &  & - s  \sum_{v=1}^{r+1} \p_{v-1}\lp \m_{v-1} -  \m_{v} \rp \mE_{G,{\mathcal U}_{r+1}} \left \langle \|\x^{(i_1)}\|_2\|\x^{(p_1)}\|_2 (\y^{(p_2)})^T \y^{(i_2)} \right \rangle_{\gamma_{v}^{(r)} }.\nonumber \\
\end{eqnarray}
One then analogously has for the other two sequences
 \begin{eqnarray}\label{eq:rthlev2genanal47}
 \sum_{k_1=1}^{r+1}\sum_{i_1=1}^{l}\sum_{i_2=1}^{l}
\beta_{i_1}\|\y^{(i_2)}\|_2 \frac{T_{k_1,2}}{\sqrt{1-t}}
   & = & \beta^2 (\q_{0}-\q_{r+1}) \Bigg( \Bigg. \mE_{G,{\mathcal U}_{r+1}}\langle \|\x^{(i_1)}\|_2^2\|\y^{(i_2)}\|_2^2\rangle_{\gamma_{01}^{(r)}} \nonumber \\
 & & +  (s-1)\mE_{G,{\mathcal U}_{r+1}}\langle \|\x^{(i_1)}\|_2^2 \|\y^{(i_2)}\|_2\|\y^{(p_2)}\|_2\rangle_{\gamma_{02}^{(r)}} \Bigg. \Bigg)  \nonumber \\
 &  & - s  \sum_{v=1}^{r+1} \q_{v-1}\lp \m_{v-1} -  \m_{v} \rp \mE_{G,{\mathcal U}_{r+1}} \left \langle \|\y^{(i_2)}\|_2\|\y^{(p_2)}\|_2 (\x^{(p_1)})^T \x^{(i_1)} \right \rangle_{\gamma_{v}^{(r)} }, \nonumber \\
\end{eqnarray}
and
 \begin{eqnarray}\label{eq:rthlev2genanal48}
 \sum_{k_1=1}^{r+1}\sum_{i_1=1}^{l}\sum_{i_2=1}^{l}
\beta_{i_1}\|\y^{(i_2)}\|_2 \frac{T_{k_1,3}}{\sqrt{t}}
   & = & \beta^2 (\p_{0}\q_{0}-\p_{r+1}\q_{r+1}) \Bigg( \Bigg. \mE_{G,{\mathcal U}_{r+1}}\langle \|\x^{(i_1)}\|_2^2\|\y^{(i_2)}\|_2^2\rangle_{\gamma_{01}^{(r)}} \nonumber \\
 & & +  (s-1)\mE_{G,{\mathcal U}_{r+1}}\langle \|\x^{(i_1)}\|_2^2\|\y^{(i_2)}\|_2\|\y^{(p_2)}\|_2\rangle_{\gamma_{02}^{(r)}} \Bigg. \Bigg)  \nonumber \\
 &  & - s  \sum_{v=1}^{r+1} \p_{v-1}\q_{v-1}\lp \m_{v-1} -  \m_{v} \rp 
 \nonumber \\
 & & \times
 \mE_{G,{\mathcal U}_{r+1}} \left \langle \|\y^{(i_2)}\|_2\|\y^{(p_2)}\|_2 \|\x^{(i_1)}\|_2\|\x^{(p_1)}\|_2 \right \rangle_{\gamma_{v}^{(r)} }, \nonumber \\
\end{eqnarray}
A combination of (\ref{eq:rthlev2genanal10e})-(\ref{eq:rthlev2genanal10g}), (\ref{eq:rthlev2genanal39}), and (\ref{eq:rthlev2genanal46})-(\ref{eq:rthlev2genanal48}) finally gives (\ref{eq:thm3eq5}) and (\ref{eq:thm3eq6}).
\end{proof}

%%%%%%%%%%%%%%%%%%%%%%%%%%%%%%%%%%%%%%%%%%%%%%%%%%%%%%%%%%%%%%%%%%%%%%%%%%%%%%%%
%%%%%%%%%%%%%%%%%%%%%%%%%%%%%%%%%%%%%%%%%%%%%%%%%%%%%%%%%%%%%%%%%%%%%%%%%%%%%%%%
%%%%%%%%%%%%%%%%%%%%%%%%%%%%%%%%%%%%%%%%%%%%%%%%%%%%%%%%%%%%%%%%%%%%%%%%%%%%%%%%
%%%%%%%%%%%%%%%%%%%%%%%%%%%%%%%%%%%%%%%%%%%%%%%%%%%%%%%%%%%%%%%%%%%%%%%%%%%%%%%%
\section{A few practical examples}
\label{sec:examples}
%%%%%%%%%%%%%%%%%%%%%%%%%%%%%%%%%%%%%%%%%%%%%%%%%%%%%%%%%%%%%%%%%%%%%%%%%%%%%%%%
%%%%%%%%%%%%%%%%%%%%%%%%%%%%%%%%%%%%%%%%%%%%%%%%%%%%%%%%%%%%%%%%%%%%%%%%%%%%%%%%
%%%%%%%%%%%%%%%%%%%%%%%%%%%%%%%%%%%%%%%%%%%%%%%%%%%%%%%%%%%%%%%%%%%%%%%%%%%%%%%%
%%%%%%%%%%%%%%%%%%%%%%%%%%%%%%%%%%%%%%%%%%%%%%%%%%%%%%%%%%%%%%%%%%%%%%%%%%%%%%%%

As mentioned earlier, considered models and underlying biliearly indexed random processes encompass many well known random structures and optimization problems. Here, we briefly mention only a few and leave a more thorough discussion about them and many others for companion separate papers.

We first recall on the following key object of interest from Theorem \ref{thm:thm3}
\begin{equation}\label{eq:exampleseq1}
\psi(\calX,\calY,\p,\q,\m,\beta,s,t)  =  \mE_{G,{\mathcal U}_{r+1}} \frac{1}{\beta|s|\sqrt{n}\m_r} \log
\lp \mE_{{\mathcal U}_{r}} \lp \dots \lp \mE_{{\mathcal U}_2}\lp\lp\mE_{{\mathcal U}_1} \lp Z^{\m_1}\rp\rp^{\frac{\m_2}{\m_1}}\rp\rp^{\frac{\m_3}{\m_2}} \dots \rp^{\frac{\m_{r}}{\m_{r-1}}}\rp,
\end{equation}
where
\begin{eqnarray}\label{eq:exampleseq2}
Z & \triangleq & \sum_{i_1=1}^{l}\lp\sum_{i_2=1}^{l}e^{\beta D_0^{(i_1,i_2)}} \rp^{s} \nonumber \\
 D_0^{(i_1,i_2)} & \triangleq & \sqrt{t}(\y^{(i_2)})^T
 G\x^{(i_1)}+\sqrt{1-t}\|\x^{(i_1)}\|_2 (\y^{(i_2)})^T\lp\sum_{k=1}^{r+1}\u^{(2,k)}\rp\nonumber \\
 & & +\sqrt{t}\|\x^{(i_1)}\|_2\|\y^{(i_2)}\|_2\lp\sum_{k=1}^{r+1}u^{(4,k)}\rp +\sqrt{1-t}\|\y^{(i_2)}\|_2\lp\sum_{k=1}^{r+1}\h^{(k)}\rp^T\x^{(i_1)},
 \end{eqnarray}
and
${\mathcal X}=\{\x^{(1)},\x^{(2)},\dots,\x^{(l)}\}$ with $\x^{(i)}\in \mR^{n},1\leq i\leq l$ and ${\mathcal Y}=\{\y^{(1)},\y^{(2)},\dots,\y^{(l)}\}$ with $\y^{(i)}\in \mR^{m},1\leq i\leq l$. For the concreteness, we assume the so-called linear (proportional growth) regime, $\frac{m}{n}=\alpha$, where $\alpha$ remains constant as $n\rightarrow\infty$.

%%%%%%%%%%%%%%%%%%%%%%%%%%%%%%%%%%%%%%%%%%%%%%%%%%%%%%%%%%%%%%%%%%%%%%%%%%%%%%%%
\subsection{Hopfield models}
\label{sec:hop}
%%%%%%%%%%%%%%%%%%%%%%%%%%%%%%%%%%%%%%%%%%%%%%%%%%%%%%%%%%%%%%%%%%%%%%%%%%%%%%%%

The above function $\psi(\cdot)$ for $s=1$, ${\mathcal X}=\{-\frac{1}{\sqrt{n}},\frac{1}{\sqrt{n}}\}^n$, and ${\mathcal Y}=\mS^m$ (where $\mS^m$ is $m$-dimensional unit sphere) i.e.,  $\psi(\{-\frac{1}{\sqrt{n}},\frac{1}{\sqrt{n}}\}^n,\mS^m,\p,\q,\m,\beta,1,1)$, is the so called free energy of the positive square root Hopfield model. Clearly, $\psi(\{-\frac{1}{\sqrt{n}},\frac{1}{\sqrt{n}}\}^n,\mS^m,\p,\q,\m,\beta,1,0)$ is then its intended decoupled comparative counterpart. Of particular interest in random optimizations is often the so-called ground state regime, where $\beta\rightarrow\infty$. It is then rather clear that
\begin{equation}\label{eq:exampleseq3}
  \lim_{n,\beta\rightarrow\infty}\psi\lp\left \{-\frac{1}{\sqrt{n}},\frac{1}{\sqrt{n}}\right \}^n,\mS^m,\p,\q,\m,\beta,1,1\rp=  \lim_{n\rightarrow\infty} \frac{\max_{\x\in\{-\frac{1}{\sqrt{n}},\frac{1}{\sqrt{n}}\}^n}\|G\x\|_2}{\sqrt{n}}
\end{equation}
 and  $\lim_{n,\beta\rightarrow\infty}\psi(\{-\frac{1}{\sqrt{n}},\frac{1}{\sqrt{n}}\}^n,\mS^m,\p,\q,\m,\beta,1,0)$ are precisely the ground state energy of the  positive square root Hopfield model and its decoupled counterpart in the thermodynamic limit. More on the foundations, importance, relevance, and current state of the art results related to key aspects of the Hopfield models can be found in, e.g.,  \cite{Hop82,PasFig78,Hebb49,PasShchTir94,ShchTir93,BarGenGueTan10,BarGenGueTan12,Tal98,StojnicMoreSophHopBnds10,BovGay98,TalBook11a}.

Similarly for $s=-1$ one has analogously
\begin{equation}\label{eq:exampleseq4}
  \lim_{n,\beta\rightarrow\infty}\psi\lp\left \{-\frac{1}{\sqrt{n}},\frac{1}{\sqrt{n}}\right \}^n,\mS^m,\p,\q,\m,\beta,-1,1\rp=   - \lim_{n\rightarrow\infty} \frac{\min_{\x\in\{-\frac{1}{\sqrt{n}},\frac{1}{\sqrt{n}}\}^n}\|G\x\|_2}{\sqrt{n}}
\end{equation}
 and  $\lim_{n,\beta\rightarrow\infty}\psi(\{-\frac{1}{\sqrt{n}},\frac{1}{\sqrt{n}}\}^n,\mS^m,\p,\q,\m,\beta,-1,0)$ as the ground state energy of the negative square root Hopfield model and its decoupled counterpart in the thermodynamic limit.

%%%%%%%%%%%%%%%%%%%%%%%%%%%%%%%%%%%%%%%%%%%%%%%%%%%%%%%%%%%%%%%%%%%%%%%%%%%%%%%%
\subsection{Asymmetric Little models}
\label{sec:litt}
%%%%%%%%%%%%%%%%%%%%%%%%%%%%%%%%%%%%%%%%%%%%%%%%%%%%%%%%%%%%%%%%%%%%%%%%%%%%%%%%

The above function $\psi(\cdot)$ for $s=1$, ${\mathcal X}=\{-\frac{1}{\sqrt{n}},\frac{1}{\sqrt{n}}\}^n$, and ${\mathcal X}={\mathcal Y}=\{-\frac{1}{\sqrt{m}},\frac{1}{\sqrt{m}}\}^m$ cam ne viewed as a variant of the so called free energy of the positive asymmetric Little model. Clearly, $\psi(\{-\frac{1}{\sqrt{n}},\frac{1}{\sqrt{n}}\}^n,\{-\frac{1}{\sqrt{m}},\frac{1}{\sqrt{m}}\}^m,\p,\q,\m,\beta,1,0)$ is then its intended decoupled comparative counterpart. In random optimizations one is, again, often interested in
\begin{equation}\label{eq:exampleseq5}
  \lim_{n,\beta\rightarrow\infty}\psi\lp\left \{-\frac{1}{\sqrt{n}},\frac{1}{\sqrt{n}}\right \}^n,\left \{-\frac{1}{\sqrt{m}},\frac{1}{\sqrt{m}}\right \}^n,\p,\q,\m,\beta,1,1\rp=  \lim_{n\rightarrow\infty} \frac{\max_{\x\in\{-\frac{1}{\sqrt{n}},\frac{1}{\sqrt{n}}\}^n}\|G\x\|_1}{\sqrt{nm}}
\end{equation}
 and  $\lim_{n,\beta\rightarrow\infty}\psi(\{-\frac{1}{\sqrt{n}},\frac{1}{\sqrt{n}}\}^n,\{-\frac{1}{\sqrt{m}},\frac{1}{\sqrt{m}} \}^n,\p,\q,\m,\beta,1,0)$ which are precisely the ground state energy of the  positive asymmetric Little model and its decoupled counterpart in the thermodynamic limit. More on the foundations, importance, relevance, and key state of the art results related to various aspects of Little models can be found in, e.g.,
\cite{BruParRit92,Little74,BarGenGue11bip,CabMarPaoPar88,AmiGutSom85,StojnicAsymmLittBnds11}.

Similarly for $s=-1$ one has analogously
\begin{equation}\label{eq:exampleseq6}
  \lim_{n,\beta\rightarrow\infty}\psi\lp\left \{-\frac{1}{\sqrt{n}},\frac{1}{\sqrt{n}}\right \}^n,\left \{-\frac{1}{\sqrt{m}},\frac{1}{\sqrt{m}}\right \}^n,\p,\q,\m,\beta,-1,1\rp= - \lim_{n\rightarrow\infty} \frac{\min_{\x\in\{-\frac{1}{\sqrt{n}},\frac{1}{\sqrt{n}}\}^n}\|G\x\|_1}{\sqrt{nm}}
\end{equation}
 and  $\lim_{n,\beta\rightarrow\infty}\psi(\{-\frac{1}{\sqrt{n}},\frac{1}{\sqrt{n}}\}^n,\{-\frac{1}{\sqrt{m}},\frac{1}{\sqrt{m}} \}^n,\p,\q,\m,\beta,-1,0)$ as the negative asymmetric Little model ground state energy and its decoupled counterpart in the thermodynamic limit.

%%%%%%%%%%%%%%%%%%%%%%%%%%%%%%%%%%%%%%%%%%%%%%%%%%%%%%%%%%%%%%%%%%%%%%%%%%%%%%%%
\subsection{Perceptrons}
\label{sec:perc}
%%%%%%%%%%%%%%%%%%%%%%%%%%%%%%%%%%%%%%%%%%%%%%%%%%%%%%%%%%%%%%%%%%%%%%%%%%%%%%%%

%%%%%%%%%%%%%%%%%%%%%%%%%%%%%%%%%%%%%%%%%%%%%%%%%%%%%%%%%%%%%%%%%%%%%%%%%%%%%%%%
\subsubsection{Spherical perceptron}
\label{sec:sphperc}
%%%%%%%%%%%%%%%%%%%%%%%%%%%%%%%%%%%%%%%%%%%%%%%%%%%%%%%%%%%%%%%%%%%%%%%%%%%%%%%%

For $s=-1$, ${\mathcal X}=\mS^n$, and ${\mathcal Y}=\mS_+^n$  (where $\mS_+^m$ is the positive orthant portion of the $m$-dimensional unit sphere), one has that $\psi(\cdot)$ in (\ref{eq:exampleseq1}) is (a properly adjusted) free energy associated with the classical, so-called, spherical perceptron (see, e.g., \cite{StojnicGardGen13,StojnicGardSphErr13,StojnicGardSphNeg13,GarDer88,Gar88,SchTir02,SchTir03}).
Clearly, $\psi(\mS^m,\mS_+^m,\p,\q,\m,\beta,1,0)$ is then its intended decoupled comparative counterpart. In random optimizations (see, e.g., \cite{FPSUZ17,FraHwaUrb19,FraPar16,FraSclUrb19,FraSclUrb20,AlaSel20,StojnicGardGen13,StojnicGardSphErr13,StojnicGardSphNeg13,GarDer88,Gar88,Schlafli,Cover65,Winder,Winder61,Wendel62,Cameron60,Joseph60,BalVen87,Ven86,SchTir02,SchTir03}), one is, again, often interested in
\begin{equation}\label{eq:exampleseq7}
  \lim_{n,\beta\rightarrow\infty}\psi\lp \psi(\mS^m,\mS_+^m,\p,\q,\m,\beta,1,1\rp=  \lim_{n\rightarrow\infty} \frac{\min_{\x\in\mS^m}\max_{\y\in\mS_+^m} \y^TG\x}{\sqrt{n}}
\end{equation}
 and  $\lim_{n,\beta\rightarrow\infty}\psi\lp \psi(\mS^m,\mS_+^m,\p,\q,\m,\beta,1,0\rp$ which are precisely the spherical perceptron associated ground state energy and its decoupled counterpart in the thermodynamic limit.

%%%%%%%%%%%%%%%%%%%%%%%%%%%%%%%%%%%%%%%%%%%%%%%%%%%%%%%%%%%%%%%%%%%%%%%%%%%%%%%%
\subsubsection{Binary perceptron}
\label{sec:binperc}
%%%%%%%%%%%%%%%%%%%%%%%%%%%%%%%%%%%%%%%%%%%%%%%%%%%%%%%%%%%%%%%%%%%%%%%%%%%%%%%%

For $s=-1$, ${\mathcal X}=\{-\frac{1}{\sqrt{n}},\frac{1}{\sqrt{n}}\}^n$, and ${\mathcal Y}=\mS_+^m$, one has that $\psi(\cdot)$ in (\ref{eq:exampleseq1}) is (a properly adjusted) free energy associated with the, so-called, binary perceptron (see, e.g., \cite{StojnicGardGen13,GarDer88,Gar88,StojnicDiscPercp13,KraMez89,GutSte90,KimRoc98,TalBook11a,NakSun23,BoltNakSunXu22,PerkXu21,CXu21,DingSun19}). $\psi(\{-\frac{1}{\sqrt{m}},\frac{1}{\sqrt{m}} \}^n,\mS_+^m,\p,\q,\m,\beta,1,0)$ is then its intended decoupled comparative counterpart. In random optimizations (see, e.g., \cite{StojnicGardGen13,GarDer88,Gar88,StojnicDiscPercp13,KraMez89,GutSte90,KimRoc98,TalBook11a,NakSun23,BoltNakSunXu22,PerkXu21,CXu21,DingSun19})), one is, again, often interested in
\begin{equation}\label{eq:exampleseq8}
  \lim_{n,\beta\rightarrow\infty}\psi\lp \psi(\left \{-\frac{1}{\sqrt{n}},\frac{1}{\sqrt{n}}\right \}^n,\mS_+^m,\p,\q,\m,\beta,1,1\rp=  \lim_{n\rightarrow\infty} \frac{\min_{\x\in\left \{-\frac{1}{\sqrt{m}},\frac{1}{\sqrt{m}}\right \}^n}\max_{\y\in\mS_+^m} \y^TG\x}{\sqrt{n}}
\end{equation}
 and  $\lim_{n,\beta\rightarrow\infty}\psi\lp \psi( \{-\frac{1}{\sqrt{m}},\frac{1}{\sqrt{m}}\}^n,\mS_+^m,\p,\q,\m,\beta,1,0\rp$ which are precisely the spherical perceptron associated ground state energy and its decoupled counterpart in the thermodynamic limit.

The above examples are only a few illustrative ones from a rather unlimited collection (the cited references contain a large number of closely related relevant ones as well). Although small, this set provides a pretty good hint as to how wide could be the range of potential applications of our results. Further studying of the above presented interpolating concepts within the context of each of these applications is therefore of great interest. Such a studying is usually problem specific and we will present many interesting results that can be obtained in these directions in separate papers.

%%%%%%%%%%%%%%%%%%%%%%%%%%%%%%%%%%%%%%%%%%%%%%%%%%%%%%%%%%%%%%%%%%%%%%%%%%%%%%%%
\section{Conclusion}
\label{sec:lev2x3lev2liftconc}
%%%%%%%%%%%%%%%%%%%%%%%%%%%%%%%%%%%%%%%%%%%%%%%%%%%%%%%%%%%%%%%%%%%%%%%%%%%%%%%%

A very powerful statistical interpolating/comparison mechanism, to which we refer as \emph{fully lifted} (fl), is presented. The concept is a substantial upgrade of the fully bilinear mechanism introduced in \cite{Stojnicgscompyx16}, which corresponds to the zero level of full or one level of partial lifting. As it contains the mechanism of \cite{Stojnicgscompyx16} as a special case, it also contains all special subcases of \cite{Stojnicgscompyx16}. In particular, the basic well known Slepian's max and Gordon's minmax comparison principles are special cases of \cite{Stojnicgscompyx16} and therefore of our work as well. Since many of the results that we created in recent years relied on these well-known principles, it is clear that the generality and magnitude of the results presented here will enable a substantial progress in further studying all of them.

To give a hint regarding the range of potential applications, we showed how one can deduce several classical optimization problems as special cases of our model. These are of course only a few illustrative ones. We should, however, emphasize that the presented mechanism is a very generic and extremely powerful self-sustainable tool which can be used in various extensions as well. Since these are problem specific, we discuss them as well as their final results in separate papers.

%\newpage1
%\setcounter{page}{1}
\begin{singlespace}
\bibliographystyle{plain}
\bibliography{nflgscompyxRefs}
\end{singlespace}

\end{document}